%% file: MehU22-ppt.tex
\documentclass[11pt,reqno]{amsart}

\usepackage[english]{babel}
\usepackage[utf8]{inputenc}
\usepackage[T1]{fontenc}
\usepackage{lmodern} \normalfont %to load T1lmr.fd 
\DeclareFontShape{T1}{lmr}{bx}{sc} { <-> ssub * cmr/bx/sc }{}
\usepackage{microtype}	% for improved spacing between words and letters
\usepackage{xspace}

\usepackage{amssymb}
\usepackage{amsmath}
\usepackage{amsthm}
\usepackage{mathtools}
\usepackage{mathrsfs}
\usepackage{accents} % correct spacing for accents in sub- and superscript
\usepackage{etoolbox}
\usepackage{siunitx}
\usepackage{dsfont} % for special 1, \amthds{1}
\usepackage{graphicx}
\usepackage{color}
\usepackage[dvipsnames]{xcolor}
\usepackage{tikz}
\usetikzlibrary{calc,positioning,shapes}
\usetikzlibrary{patterns,decorations.pathmorphing,decorations.markings}
\usepackage[oldvoltagedirection]{circuitikz}
\usepackage{pgfplots}
\pgfplotsset{compat=newest}
\usepackage[margin=10pt,font=small,labelfont=bf,labelsep=endash]{caption}
\usepackage{subcaption}

\usepackage{booktabs}
\usepackage{paralist}
\usepackage{soul}

\numberwithin{equation}{section}

\usepackage{algorithm}
\usepackage{algorithmicx}
\usepackage{algpseudocode}

\usepackage{cite}
\usepackage{multirow} % for multirows in tabular environment
\usepackage{enumitem} % for labelling items in enumerate
\setlist[enumerate]{label=(\roman*)}

\usepackage[colorlinks,linkcolor=teal,citecolor=teal,urlcolor=teal]{hyperref}
\usepackage[nameinlink,noabbrev]{cleveref}

\theoremstyle{plain}
\newtheorem{theorem}{Theorem}[section]

\newtheorem{lemma}[theorem]{Lemma}
\newtheorem{corollary}[theorem]{Corollary}
\newtheorem{remark}[theorem]{Remark}
\newtheorem{definition}[theorem]{Definition}
\newtheorem{assumption}[theorem]{Assumption}
\newtheorem{example}[theorem]{Example}
\newtheorem{hypothesis}[theorem]{Hypothesis}

\newcommand{\lineWidth}{2pt} % line widths in tikx-Pictures

\input{def}

\newcommand{\pHODE}{\textsf{pHODE}\xspace}
\newcommand{\pHODEs}{\textsf{pHODEs}\xspace}
\newcommand{\pHDAE}{\textsf{pHDAE}\xspace}
\newcommand{\pHDAEs}{\textsf{pHDAEs}\xspace}
\newcommand{\DAE}{\textsf{DAE}\xspace}
\newcommand{\DAEs}{\textsf{DAEs}\xspace}
\newcommand{\dHDAE}{\textsf{dHDAE}\xspace}
\newcommand{\ECRM}{\textsf{ECRM}\xspace}
\newcommand{\FCRM}{\textsf{FCRM}\xspace}
\newcommand{\IRKA}{\textsf{IRKA}\xspace}
\newcommand{\LTI}{\textsf{LTI}\xspace}
\newcommand{\LTV}{\textsf{LTV}\xspace}
\newcommand{\MM}{\textsf{MM}\xspace}
\newcommand{\MOR}{\textsf{MOR}\xspace}
\newcommand{\ROM}{\textsf{ROM}\xspace}
\newcommand{\ROMs}{\textsf{ROMs}\xspace}
\newcommand{\ODE}{\textsf{ODE}\xspace}
\newcommand{\ODEs}{\textsf{ODEs}\xspace}
\newcommand{\pH}{\textsf{pH}\xspace}
\newcommand{\PDE}{\textsf{PDE}\xspace}
\newcommand{\PDEs}{\textsf{PDEs}\xspace}
\newcommand{\MSO}{\textsf{MSO}\xspace}

\title[Control of port-Hamiltonian DAEs]{Control of  port-Hamiltonian differential-algebraic systems and applications}
\author{Volker Mehrmann${}^\dagger$ \and Benjamin Unger${}^\star$}
\address{${}^{\dagger}$  Institut für Mathematik, Technische Universität Berlin\\ Strasse des 17. Juni 136, 10623 Berlin, Germany}
\email{mehrmann@math.tu-berlin.de}
\address{${}^{\star}$ Stuttgart Center for Simulation Science (SC SimTech), University of Stuttgart, Universit\"{a}tsstr.~32, 70569 Stuttgart, Germany}
\email{benjamin.unger@simtech.uni-stuttgart.de}
\date{\today}
\keywords{port-Hamiltonian systems; structure-preserving model-order reduction; passivity; spectral factorization; $\mathcal{H}_2$-optimal}

\begin{document}

\begin{abstract}
	    The modeling framework of port-Hamiltonian descriptor systems and their use in numerical simulation and control are discussed. The structure is ideal for automated network-based modeling since it is invariant under power-conserving interconnection, congruence transformations, and Galerkin projection. Moreover, stability and passivity properties are easily shown. Condensed forms under orthogonal transformations present easy analysis tools for existence, uniqueness, regularity, and numerical methods to check these properties.  
    
    After recalling the concepts for general linear and nonlinear descriptor systems, we demonstrate that many difficulties that arise in general descriptor systems can be easily overcome within the port-Hamiltonian framework. The properties of port-Hamiltonian descriptor systems are analyzed,  time-discretization, and numerical linear algebra techniques are discussed.  Structure-preserving regularization procedures for descriptor systems are presented to make them suitable for simulation and control. Model reduction techniques that preserve the structure and stabilization and optimal control techniques are discussed.
    
    The properties of port-Hamiltonian descriptor systems and their use in modeling simulation and control methods are illustrated with several examples from different physical domains. The survey concludes with open problems and research topics that deserve further attention.
\end{abstract}

\maketitle
{\footnotesize \textsc{Keywords:} port-Hamiltonian systems, descriptor system, differential-algebraic equation, energy-based modeling, passivity, stability, interconnectability, condensed form, Dirac structure, structure-preserving model-order reduction, time discretization, linear system solves, optimal control,  feedback control}

{\footnotesize \textsc{AMS subject classification:} 37J06, 37M99,49M05, 65L80, 65P99,93A30,93A15, 93B11, 93B17, 93B52}
%
%37J06 = General theory of finite-dimensional Hamiltonian and Lagrangian systems, Hamiltonian and Lagrangian structures, symmetries, invariants
%37M99 = Dynamical systems and ergodic theory
%49M05 = Numerical methods based on necessary conditions (optimal control)
%65L80 = Numerical methods for differential-algebraic equations
%65P99 = Numerical problems in dynamical systems
%93A30 = Mathematical modeling (models of systems, model-matching, etc.)
%93A15 = Large-scale systems
%93B11 = System structure simplification
%93B17 = Transformations (systems theory; control)
%93B52 = Feedback control  

\section{Introduction}
\label{sec:introduction}
Modern key technologies in science and technology require modeling, simulation, optimization or control (\MSO) of complex dynamical systems. Most real-world systems are multi-physics systems, combining components from different physical domains, and with different accuracies and scales in the components. To address these requirements, there exist many commercial and open source \MSO software packages for simulation and control in all physical domains, \eg Abaqus\footnote{\url{https://www.3ds.com/products-services/simulia/products/abaqus/}}, Ansys\footnote{\url{https://www.ansys.com/}}, COMSOL\footnote{\url{https://www.comsol.com}}, Dymola\footnote{\url{https://www.3ds.com/products-services/catia/products/dymola/}}, FEniCS\footnote{\url{https://fenicsproject.org}}, and Simulink\footnote{\url{https://www.mathworks.com/products/simulink.html}}.
Several of these also have multi-physics components, but all are still very limited when it comes to applications, such as digital twins, which require a cross-domain evolutionary modeling process, the coupling of different domain-specific tools, the incorporation of model hierarchies consisting of coarse and fine discretizations and reduced-order models as well as the incorporation of (optimal) control techniques. The latter point, in particular,  requires tools to be open to performing easy and automatized model modifications.

Furthermore, flexible compromises between different accuracies and computational speed have to be possible to allow uncertainty quantification procedures, as well as error estimates that balance model, discretization, optimization, approximation, or roundoff errors, combined with sensitivity, stability, and robustness measures. Finally, with modern data science tools becoming increasingly powerful, it is necessary to have models and methods that allow pure data-based approaches and have the flexibility to recycle and reuse components in different applications. On top of all these requirements, the \MSO tools cannot be separated from the available computing environments, ranging from process controllers, data, sensor, and visualization interfaces, linked up with high performance and cloud computing facilities.

To address all these challenges in the future and in an increasingly digitized world, a fundamental paradigm shift in \MSO is necessary. For every scientific and technological product or process, and the whole life cycle from the design phase to the waste recycling, it is necessary to build digital twins with multi-fidelity model hierarchies or catalogs of several models that range from very fine descriptions that help to understand the behavior via detailed and repeated simulations, to very coarse (reduced or surrogate) models used for real-time control and optimization. Furthermore, the \MSO tools should as much as possible be open for interaction, automatized, and allow the linking of subsystems or numerical methods in a network fashion. They should also allow the combination with methods that deal with large nets of real-time data that can and should be employed in a modeling or data assimilation process. Because of all this, it is necessary that mathematical modeling, analysis, numerics, control, optimization, model reduction methods, and data science techniques work hand in hand.

To illustrate these general comments, let us consider a major societal application.
In order to reduce global warming, it is necessary to replace the emissions arising in the production of energy from fossil sources by increasing renewable energy production such as wind or solar energy. At the same time it is essential to allow energy-efficient multi-directional  sector coupling  such as power-to-heat or power-to-mobility, see \Cref{figcoupling}.

\begin{figure}
	\centering
	\includegraphics[width=.7\linewidth]{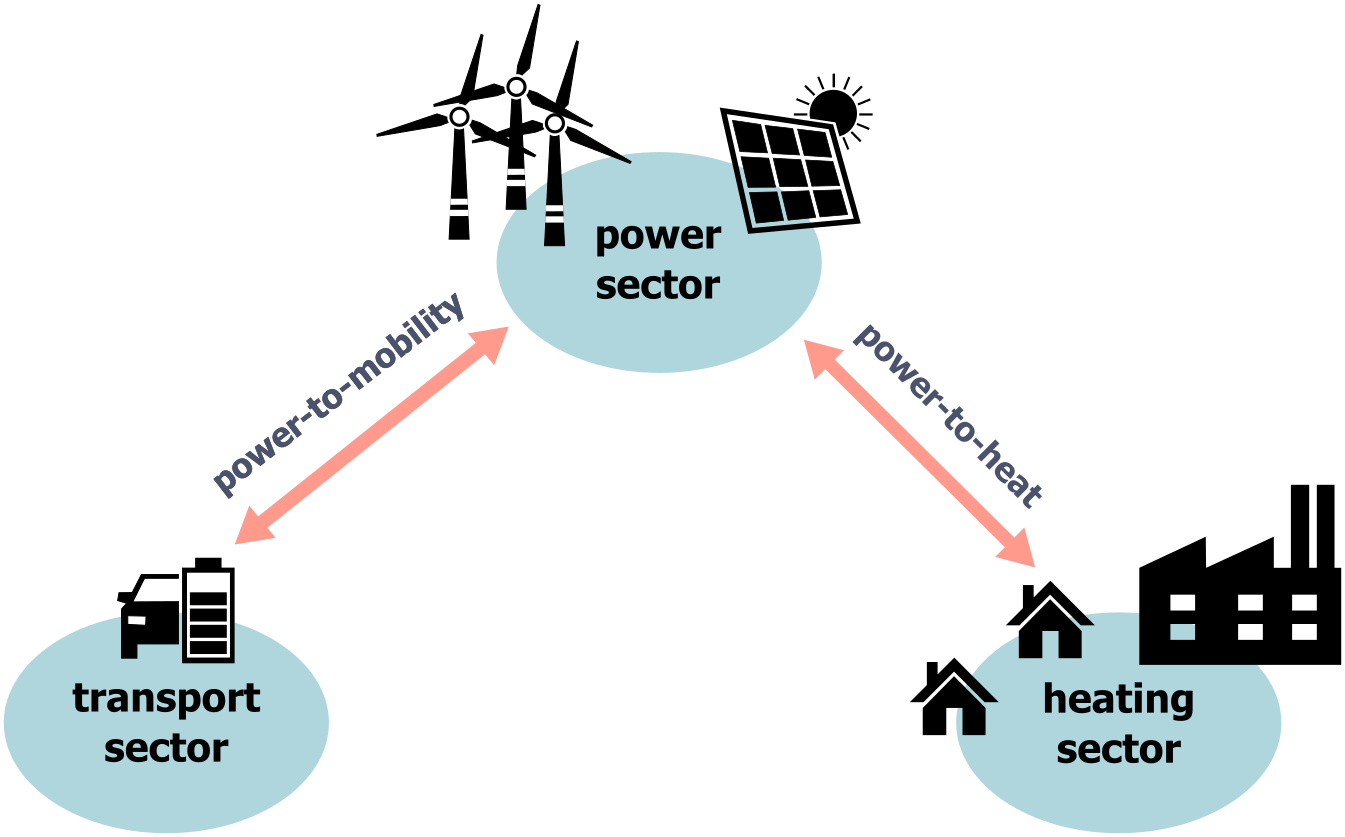}
	\caption{Sector coupling and the power-to-X concept}
	\label{figcoupling}
\end{figure}

The coupling of energy sectors includes the storage or transformation to another energy carrier (like hydrogen) of superfluous electrical energy, as well the layout and operation of energy transportation networks,  see \eg \cite{BroSKG18,ConCC20,RamHAW21}.

On the mathematical/computational side, challenges arise because mathematical models of different energy conversion processes and energy transport networks live on very different time scales such as \eg, gas transport networks or electrical power networks. Furthermore, while most energy transport networks are currently operated in a stationary regime, in the future dynamic approaches are required that allow control and optimization of energy production and transport in real time, see \eg \cite{Bie15,MacLB20}. 
These further challenges lead to the following \emph{model class wish-list} for a new flexible modeling, simulation, optimization, and control framework.
\begin{itemize}
	\item The model class should allow for automated modeling, in a modularized, and network based fashion, including pure data-based models.
	\item The mathematical representation should allow coupling of mathematical models across different scales and physical domains, in continuous and discrete time.
	\item The mathematical models should be close to the real (open or closed) physical system and easily extendable if further effects have to be considered.
	\item The model class should have  nice algebraic, geometric, and analytical properties. The models should be easy to analyze concerning existence, uniqueness, robustness, stability, uncertainty, perturbation theory, and  error analysis.
	\item The class should be invariant under local variable (coordinate) transformations (in space and time) to understand the local behavior, e.g. via local normal forms.
	\item The model class should allow for simple space and time  discretization and model reduction methods as well as fast solvers for the resulting linear and nonlinear systems of equations.
    \item The model class should be feasible for simulation, control, optimization and data assimilation.
\end{itemize}
Can there be such a \emph{Jack of all trades}? 
The main goal of this paper is to show that, even though many aspects are still under investigation, energy-based modeling via the model class of \emph{dissipative port-Hamiltonian (\pH) descriptor systems} has many great features and comes very close to being such a model class. 
Let us emphasize that the field of \pHDAE systems is a highly active research area and many developments are just taking place. In this survey we thus focus only on selected topics.

\subsection*{Structure of the manuscript}
The paper is organized as follows. We first review 
general nonlinear descriptor systems and its solution theory in \Cref{sec:modelClass} and associated control theoretical results in \Cref{sec:controlaspects}. Dissipative port-Hamiltonian (\pH) descriptor systems are introduced in \Cref{sec:phdae} and illustrated with several examples from different application domains in \Cref{sec:examples}. In \Cref{sec:properties} we start to analyze \pH descriptor systems in terms of our model class wish-list by discussing the inherent properties of \pH descriptor systems. Condensed forms are presented in \Cref{sec:normalForm}. We then turn to structure-preserving model order reduction in \Cref{sec:MOR} and discuss time-discretization and associated linear system solves in \Cref{sec:TimeDiscretization}. We conclude our presentation with a discussion of control methods in \Cref{sec:control} and a summary (including open problems and future work) in \Cref{sec:summary}. We emphasize that within this manuscript, we mainly focus on finite-dimensional problems. Nevertheless, the \pH model class can be extended to the infinite-dimensional setting, and we provide a brief discussion and an (incomplete) list of references at the end of \Cref{sec:summary}. 

\subsection*{Notation}
The sets $\N$, $\N_0$, $\R$, and $\C$ denote the natural numbers, non-negative integers, real numbers,  and complex numbers, respectively. For a complex number~$z\in\C$ we denote its real part as $\real(z)$. The set of $n\times m$ matrices with values in a field $\mathbb{F}$ is denoted with $\mathbb{F}^{n,m}$. The symbol $I$ is used for the identity matrix, whose dimension is clear from the context. The rank and corank of a matrix $M\in\mathbb{F}^{n,m}$ are denoted by $\rank M$ and $\corank M$, where the latter is defined as 
\begin{equation}
    \label{eqn:corank}
    \corank M \vcentcolon= n-\rank M.
\end{equation}
The transpose of a matrix and the conjugate transpose (if $\mathbb{F} = \C$) are denoted by $M^T$ and $M^H$, respectively. To indicate that a matrix $M\in\mathbb{F}^{n,n}$ is positive definite, or positive semi-definite, we write $M>0$ and $M\geq 0$, respectively. The Moore-Penrose inverse  of a matrix $M\in\mathbb{F}^{m,n}$, \ie the unique matrix $A$ satisfying 
$MAM=M$, $AMA=A$, $(AM)^H =AM$, and $(MA)^H=MA$,  is denoted by $A=M^\dagger$. A block diagonal matrix with diagonal blocks $B_1,\ldots,B_k$ is denoted by $\diag(B_1,\ldots,B_k)$, and the span of a list of vectors $v_1,\ldots,v_k$ is denoted by $\mathrm{span}(v_1,\ldots,v_k)$.

The spaces of continuous and $k$-times continuously differentiable functions (with $k\in\N$) from the time interval $\timeInt$ to some Banach space $\mathcal{X}$ are denoted by $\mathcal{C}(\timeInt,\mathcal{X})$ and $\mathcal{C}^k(\timeInt,\mathcal{X})$, respectively. For a function $f\in\mathcal{C}^1(\timeInt,\mathcal{X})$ we write $\dot{f} \vcentcolon= \tfrac{\mathrm{d}}{\mathrm{d}t}$ to denote the (time) derivative. Similarly, we use the notation $\ddot{f}$ for the second derivative and $f^{(k)}$ for the $k$th derivative.
The Jacobian of a function $f\colon \R^\ell \to \R^\stateDim$ is denoted by $\tfrac{\partial}{\partial {\state}}f(\state)$.

\subsection*{Abbreviations}
Throughout the manuscript, we use the following abbreviations.

\begin{center}
\begin{tabular}{rl}
    \DAE & differential-algebraic equation\\
    \dHDAE & dissipative Hamiltonian differential-algebraic equation\\
    \ECRM & effort constraint reduction method\\
    \FCRM & flow constraint reduction method\\
    \IRKA & iterative rational Krylov algorithm\\
    \LTI & linear time-invariant\\
    \LTV & linear time-varying\\
    \MM & moment matching\\
    \MSO & modeling, simulation, and optimization\\
    \MOR & model order reduction\\
    \ODE & ordinary differential equation\\
    \PDE & partial differential equation\\
    \pH & port-Hamiltonian\\
    \pHODE & port-Hamiltonian ordinary differential equation\\
    \pHDAE & port-Hamiltonian differential-algebraic equation\\
    \ROM & reduced order model\\
\end{tabular}
\end{center}

%%%%%%%%%%%%%%%%%%%%%%%%%%%%%%%%%%%%%%%%%%%%%%%%
\section{The model class of descriptor systems}
\label{sec:modelClass}
To allow network-based automated modularized modeling via interconnection, constraint-preserving simulation, optimization, and control of dynamic models,  it is common practice in many application domains to use the class of (implicit) control systems, called \emph{descriptor systems} or \emph{differential-algebraic equation (\DAE) systems}, of the form
\begin{subequations}
	\label{eqn:descriptorSystem}
	\begin{align}
		\label{eqn:descriptorSystem:stateEq}F(t,\state(t),\dot{\state}(t),\inpVar(t)) &= 0,\\
		\label{eqn:descriptorSystem:output}\outVar(t) - G(t,\state(t),\inpVar(t)) &= 0,
	\end{align}
\end{subequations}
on some time interval $\timeInt \vcentcolon= [t_0,t_\mathrm{f}]$
with
\begin{align*}
    F\colon \timeInt\times\mathbb{D}_{\mathrm{\state}}\times\mathbb{D}_{\dot{\mathrm{\state}}}\times\mathbb{D}_{\mathrm{\inpVar}}\to \R^\ell\qquad\text{and}\qquad
    G\colon\timeInt\times\mathbb{D}_{\mathrm{\state}}\times\mathbb{D}_{\mathrm{\inpVar}}\to \R^p,
\end{align*}
with open domains, vector spaces, or manifolds
$\mathbb{D}_{\mathrm{\state}},\mathbb{D}_{\dot{\mathrm{\state}}},\mathbb{D}_{\mathrm{\inpVar}}$. In the finite dimensional case of real systems, which is predominantly discussed in this paper, we assume for the ease of presentation that
\begin{displaymath}
    F\colon \timeInt\times\R^{\stateDim}\times\R^{\stateDim}\times\R^{m}\to\R^\ell\qquad \text{and}\qquad G\colon\timeInt\times\R^{\stateDim}\times\R^m\to\R^p.
\end{displaymath}
We refer to $\state$, $\inpVar$, and $\outVar$, as the \emph{state}, \emph{input}, and \emph{output}, respectively.

Note that there can be very different roles of inputs in different applications, \eg, to deal with control actions, interconnections, or disturbances,  and different roles of outputs \eg, for measurements, interconnection, or observer design.  Also, the models typically have parameters and/or may have random uncertain components such as \eg, unmodeled quantities, uncertainty in parameters, or disturbances. 

We remark that most of the results and methods we present also hold for complex systems, but we restrict ourselves to the real case in this survey. In the following, for the sake of a simpler presentation, we will also often omit the time- or space argument whenever this is appropriate or clear from the context.

%%%%%%%%%%%%%%%%%%%%%%%%%%%%%%%%%%%%%%%%%%%%%%%%
\subsection{Solution concept}
It is clear that depending on the application, different solution concepts for \eqref{eqn:descriptorSystem} may be  necessary, see \eg \cite{BreCP96,KunM06,LamMT13}. 

In the finite dimensional setting we restrict ourselves to classical function spaces of continuous or continuously differentiable functions. For control problems as in~\eqref{eqn:descriptorSystem} we often follow the behavior framework, see for instance \cite{PolW98}, in which a new combined state vector 
\begin{equation}
    \label{eqn:behaviorState}
    \xi \vcentcolon= \begin{bmatrix}\state^T & \inpVar^T & \outVar^T\end{bmatrix}^T
\end{equation}
is introduced (or $\xi \vcentcolon= [\state^T, \inpVar^T]^T$ if only the state-equation~\eqref{eqn:descriptorSystem:stateEq} is considered).
The descriptor system~\eqref{eqn:descriptorSystem} is then turned into  an under-determined \DAE, see for instance \cite{KunM01}, \ie, the meaning of the variables is not distinguished any more. 
\begin{definition}[Solution concept]
    \label{def:solutionConcept}
	Consider the \DAE~\eqref{eqn:descriptorSystem} on the time interval $\timeInt$ with open subsets
$\mathbb{D}_{\mathrm{\state}}\subseteq\R^{\stateDim}$, $\mathbb{D}_{\dot{\mathrm{\state}}}\subseteq\R^{\stateDim}$, $\mathbb{D}_{\mathrm{\inpVar}}\subseteq\R^m$.
\begin{enumerate}
\item[(i)] Let $u\colon\timeInt\to\R^m$ be a given input. We call a function $\state\in \calC^1(\timeInt,\R^\stateDim)$ a \emph{solution} of the \DAE~\eqref{eqn:descriptorSystem:stateEq}, if it satisfies~\eqref{eqn:descriptorSystem:stateEq} pointwise. It is called a \emph{solution of the initial value problem}~\eqref{eqn:descriptorSystem} with initial condition
\begin{equation}
		\label{eqn:initialCondition}
	\state(t_0) = \state_0\in\R^{\stateDim}
\end{equation}
if it furthermore satisfies~\eqref{eqn:initialCondition}. 
\item[(ii)] An initial value $\state_0\in\R^{\stateDim}$ is called \emph{consistent} with~\eqref{eqn:descriptorSystem:stateEq}, if the associated initial value problem has at least one solution.
\item[(iii)] The control problem~\eqref{eqn:descriptorSystem} is called \emph{consistent}, if there exists an input $u\colon\timeInt\to\R^m$, such that the resulting \DAE~\eqref{eqn:descriptorSystem:stateEq} has a solution. It is called \emph{regular}, if for every sufficiently smooth input function $u\colon\timeInt\to\R^m$ the corresponding \DAE~\eqref{eqn:descriptorSystem:stateEq} is solvable and the solution is unique for every consistent initial value.
\item[(iv)] We call $\xi=[\state^T,u^T,y^T]^T$ a \emph{behavior solution of the descriptor system}~\eqref{eqn:descriptorSystem}, if $\state$ is a solution of the \DAE~\eqref{eqn:descriptorSystem:stateEq} for this $u$ and $\xi$ satisfies~\eqref{eqn:descriptorSystem} pointwise.
	\end{enumerate}
\end{definition}

Let us emphasize that for \DAE systems, typically, not every initial value is consistent. This is due to the fact that in order to deal with algebraic constraints as well as over- and under-determined systems, we allow the Jacobian $\tfrac{\partial}{\partial \dot{\xi}} {F}$
to be singular or even rectangular. We refer to the forthcoming \Cref{sec:DAEtheory} for further details. 
For inconsistent initial values and systems with jumps in the coefficients one may still obtain a solution using weaker solution concepts, see \eg \cite{KunM06,RabR96a,RabR96b,Tre13}. However, for ease of presentation, we will not cover these weaker solution concepts in this survey.

%%%%%%%%%%%%%%%%%%%%%%%%%%%%%%%%%%%%%%%%%%%%%%%%
\subsection{Solution theory for general nonlinear descriptor systems}
\label{sec:DAEtheory}

In this subsection we recall the solution theory for general  \DAE systems
\begin{equation}
    \label{eq:rDAE}
    \mathcal{F}(t,\xi(t),\dot{\xi}(t)) = 0,
\end{equation}
with $\mathcal{F}\colon \timeInt \times \mathbb{D}_{\xi}\times\mathbb{D}_{\dot{\xi}} \to \R^{{L}}$ and open sets $\mathbb{D}_{\xi},\mathbb{D}_{\dot{\xi}}\subseteq\R^{N}$. Here $\xi$ is the standard state or an extended behavior vector as in~\eqref{eqn:behaviorState}.

If the Jacobian $\tfrac{\partial}{\partial \dot{\xi}} \mathcal{F}$ is not square or singular, then a solution $\xi$ of \eqref{eq:rDAE}, provided such a solution exists, may depend on derivatives of $\mathcal{F}$. This is illustrated in the following example.
\begin{example}
    \label{ex:DAEvanishingEquations}
    Consider a linear \DAE of the form
    \begin{equation}
        \label{eqn:DAEvanishingEquations}
        \begin{bmatrix}
            1 & 0 & 0\\
            0 & 0 & 0\\
            0 & 0 & 0
        \end{bmatrix}\begin{bmatrix}
            \dot{\state}_1(t)\\
            \dot{\state}_2(t)\\
            \dot{\state}_3(t)
        \end{bmatrix} = \begin{bmatrix}
            0 & 1 & 0\\
            1 & 0 & 0\\
            0 & 0 & 0
        \end{bmatrix}\begin{bmatrix}
            \state_1(t)\\
            \state_2(t)\\
            \state_3(t)
        \end{bmatrix} + \begin{bmatrix}
            f_1(t)\\
            f_2(t)\\
            f_3(t)
        \end{bmatrix}.
    \end{equation}
    We immediately notice that $\state_3$ does not contribute to the equations and hence can be chosen arbitrarily. Moreover, the third equation dictates $f_3\equiv 0$, thus detailing that a smooth function $f = \begin{bmatrix}
        f_1 & f_2 & f_3
    \end{bmatrix}^T$ is not sufficient for a solution to exist. The second equation yields $\state_1(t) = -f_2(t)$, and hence the only valid initial value for $\state_1$ is determined by $-f_2(t_0)$. Substituting $\state_1 = -f_2$ into the first equation yields
    \begin{equation}
        \label{eqn:DAEvanishingEquations:hiddenAlgebraic}
        \state_2(t) = -f_1(t) - \dot{f}_2(t),
    \end{equation}
    showing that the solution depends on the derivative of $f_2$. Moreover, we notice that~\eqref{eqn:DAEvanishingEquations:hiddenAlgebraic} constitutes another algebraic equation that is implicitly encoded in~\eqref{eqn:DAEvanishingEquations}.
\end{example}
The difficulties arising with these differentations are classified by so-called \emph{index} concepts, see~\cite{Meh15} for a survey. In this paper, we mainly make use of the \emph{strangeness index} concept \cite{KunM06}, which is, roughly speaking, a generalization of the \emph{differentiation index}, cf.~\cite{BreCP96}, to under- and overdetermined systems. The strangeness index is based on the \emph{derivative array} of level $\mu$, see~\cite{Cam87a},  defined as
\begin{equation}
	\label{eq:derivativeArray}
	\widetilde{\mathcal F}_\mu\left(t,\xi,\eta\right) \vcentcolon= \begin{bmatrix}
		\mathcal{F}(t,\xi,\dot{\xi})\\
		\ddt\mathcal{F}(t,\xi,\dot{\xi})\\
		\vdots\\
		\left(\ddt\right)^{\!\mu} \mathcal{F}(t,\xi,\dot{\xi})
	\end{bmatrix}\in\mathbb{R}^{(\mu+1)L}\ \text{with}\ \eta\vcentcolon=\begin{bmatrix}
		\dot{\xi}\\
		\ddot{\xi}\\
		\vdots\\
		\xi^{(\mu+1)}
\end{bmatrix}\in \mathbb R^{(\mu+1)N}.
\end{equation}

Since it is a-priori not clear, that the \DAE~\eqref{eq:rDAE} is solvable and that the dimension of the solution manifold in terms of the algebraic vaiables $t,\xi,\ldots,\xi^{(\mu+1)}$ is invariant over time, we need to assume that the set
\begin{equation}
    \label{eqn:nonlinDAE:manifold}
	\mathcal{M}_\mu \vcentcolon= \left\{\left(t,\xi,\eta\right)\in\mathbb{R}^{(\mu+2)N+1}\ \bigg|\ \widetilde{\mathcal F}_{\mu}\left(t,\xi,\eta\right) = 0\right\}
\end{equation}
is nonempty and (locally) forms a manifold. For notational convenience
we assume that $\mathcal{M}_\mu$ is a manifold of dimension $(\mu+2)N+1-r$. The number~$r$ will later correspond to the dimension of the regular part of the \DAE. Following \cite{KunM98}, we introduce the Jacobians
\begin{subequations}
\label{eqn:nonlinDAE:Jacobians}
\begin{align}
	{\mathcal{E}}_\mu &\vcentcolon= \begin{bmatrix}
		\frac{\partial \widetilde{\mathcal F}_{\mu}}{\partial \dot{\xi}} & \dots & \frac{\partial \widetilde{\mathcal F}_{\mu}}{\partial \xi^{(\mu+1)}}
	\end{bmatrix}\in\mathbb{R}^{(\mu+1)L,(\mu+1)N},\\
{\mathcal{A}}_\mu &\vcentcolon= -\begin{bmatrix}
		\frac{\partial \widetilde{\mathcal F}_{\mu}}{\partial \xi} & 0 & \dots & 0
	\end{bmatrix}\in\mathbb{R}^{(\mu+1)L,(\mu+1)N}.
\end{align}
\end{subequations}
In the following, we will make some constant rank assumptions, which in turn is the basis for a (local) smooth full rank decomposition as provided in the next theorem, see \cite[Thm.~4.3]{KunM06}.

\begin{theorem}
    \label{thm:smoothFullRankDecomposition}
    For open sets $\mathbb{M}\subseteq\mathbb{D}\subseteq\mathbb{R}^{k}$ let $E\in\calC^\mu(\mathbb{D},\mathbb{R}^{\ell,\stateDim})$. Furthermore, assume that $\rank E(\state) \equiv r$ for all $\state\in\mathbb{M}$. Then, for every $\widehat{\state}\in\mathbb{M}$ there exists a sufficiently small neighborhood $\mathbb{V}\subseteq\mathbb{D}$ of $\widehat{\state}$, and matrix functions $T\in\calC^\mu\left(\mathbb{V},\mathbb{R}^{\stateDim,(\stateDim-r)}\right)$ and $Z\in\calC^\mu\left(\mathbb{V},\mathbb{R}^{\ell,(\ell-r)}\right)$ with pointwise orthonormal columns such that
    \begin{displaymath}
            E(\state)T(\state) = 0\qquad\text{and}\qquad Z^T(\state)E(\state) = 0\qquad\text{for all $\state\in\mathbb{V}$}.
    \end{displaymath}
\end{theorem}

To analyze the nonlinear \DAE~\eqref{eq:rDAE} we now make the following assumption, taken from~\cite{KunM01} and presented similarly as in \cite{Ung20b}, to filter out the regular part. Note that we use the terminology $\corank$ to denote the difference between the size of a matrix and its rank; see also~\eqref{eqn:corank} for a formal definition.
\begin{assumption}
    \label{ass:nonlinDAE:regularPart}
    Assume that the set $\mathcal{M}_\mu$ in~\eqref{eqn:nonlinDAE:manifold} forms a manifold of dimension $(\mu+2)N+1-r$ and the Jacobians defined in~\eqref{eqn:nonlinDAE:Jacobians} satisfy
    \begin{equation}
        \label{eqn:nonlinDAE:regularPartAss}
        \rank \begin{bmatrix}
            \mathcal{E}_\mu & \mathcal{A}_\mu
        \end{bmatrix} = r \qquad \text{on $\mathcal{M}_\mu$.}
    \end{equation}
    Moreover, we have
    \begin{equation}
        \corank \begin{bmatrix}
            \mathcal{E}_\mu & \mathcal{A}_\mu
        \end{bmatrix} - \corank \begin{bmatrix}
            \mathcal{E}_{\mu-1} & \mathcal{A}_{\mu-1}
        \end{bmatrix} = v\qquad \text{on $\mathcal{M}_\mu$}
    \end{equation}
    with the convention that $\corank \tfrac{\partial F_{-1}}{\partial \xi} = 0$.
\end{assumption}

The quantity $v$ in \Cref{ass:nonlinDAE:regularPart} measures the number of equations in the original system that give rise to trivial equations $0=0$, \ie, it counts the number of redundancies in the system.
After the quantification of the regular and redundant parts of the \DAE~\eqref{eq:rDAE}, we use the next assumption to filter out the algebraic equations.

\begin{assumption}
    \label{ass:nonlinDAE:algebraicPart}
    Suppose the \DAE~\eqref{eq:rDAE} satisfies \Cref{ass:nonlinDAE:regularPart}. Then the matrix $\mathcal{E}_{\mu}$ defined in~\eqref{eqn:nonlinDAE:Jacobians} satisfies
    \begin{equation}
        \rank \mathcal{E}_{\mu} = r-a\qquad\text{on $\mathcal{M}_\mu$}.
    \end{equation}
\end{assumption}

\Cref{ass:nonlinDAE:algebraicPart} together with \Cref{thm:smoothFullRankDecomposition} ensures (locally) the existence of a smooth matrix function $Z_\mathrm{a}\colon \mathcal{M}_{\mu}\to \R^{(\mu+1)L,a}$ with pointwise maximal rank on $\mathcal{M}_\mu$ that satisfies
\begin{equation}
    \label{eqn:nonlinDAE:hyp:algebraicSelector}
    Z_\mathrm{a}^T \mathcal{E}_\mu = 0\qquad\text{on $\mathcal{M}_\mu$}.
\end{equation}
The (linearized) algebraic equations are thus encoded in the matrix function
\begin{equation}
    \label{eqn:nonlinDAE:hyp:linearizedAlgebraicEquations}
    Z_{\mathrm{a}}^T \tfrac{\partial \widetilde{\mathcal{F}}_{\mu}}{\partial \xi}.
\end{equation}
To ensure that we are able to solve the algebraic equations for $a$ unknowns,  requires the matrix in~\eqref{eqn:nonlinDAE:hyp:linearizedAlgebraicEquations} to have full rank. This is indeed the case, since \eqref{eqn:nonlinDAE:hyp:algebraicSelector} together with \Cref{ass:nonlinDAE:regularPart} implies that
\begin{displaymath}
    \rank Z_{\mathrm{a}}^T \mathcal{A}_{\mu} = \rank Z_{\mathrm{a}} \tfrac{\partial \widetilde{\mathcal{F}}_{\mu}}{\partial \xi} = a.
\end{displaymath}
Again, \Cref{thm:smoothFullRankDecomposition} implies (locally) the existence of a smooth matrix function
\begin{displaymath}
	T_\mathrm{a}\colon \mathcal{M}_\mu \to \R^{N,N-a}
\end{displaymath}
with pointwise maximal rank satisfying
\begin{equation}
    \label{eqn:Ta}
    Z_\mathrm{a}^T \tfrac{\partial \widetilde{\mathcal{F}}_\mu}{\partial \xi}T_\mathrm{a} = 0\qquad\text{on $\mathcal{M}_\mu$}.
\end{equation}
The remaining differential equations must be contained in the original \DAE~\eqref{eq:rDAE} (in contrast to the algebraic equations, which are contained in the derivative array), and thus, we make the following additional assumption.

\begin{assumption}
    \label{ass:nonlinDAE:differentialPart}
    Let Assumptions~\ref{ass:nonlinDAE:regularPart} and \ref{ass:nonlinDAE:algebraicPart} hold and let $T_\mathrm{a}$ be constructed as in \eqref{eqn:Ta}. Define $d\vcentcolon= L-a-v$ and assume that
    \begin{displaymath}
        \rank \tfrac{\partial \mathcal{F}}{\partial \dot{\xi}}T_\mathrm{a} = d\qquad \text{on $\mathcal{M}_\mu$}.
    \end{displaymath}
\end{assumption}

Once again, we employ \Cref{thm:smoothFullRankDecomposition} to (locally) obtain a smooth matrix function $Z_{\mathrm{d}}$ of size $N\times d$ with pointwise maximal rank that satisfies $Z_{\mathrm{d}}^T \tfrac{\partial \mathcal{F}}{\partial \dot{\xi}}T_\mathrm{a} = d$. The matrix function $Z_\mathrm{d}$ will later be used  to filter out the differential equations.

To summarize the previous discussion, we make the following assumption, which for historical reasons (cf.~\cite{KunM01}) and since in the linear case it is actually a theorem,  is referred to as a hypothesis. Note that due to the local character of \Cref{thm:smoothFullRankDecomposition} all assumptions hold only in a suitable neighborhood.
\begin{hypothesis}
    \label{hyp:nonLin:nonRegular}
    There exists integers $\mu$, $r$, $a$, and $v$ such that $\mathcal{M}_\mu$ defined in \eqref{eqn:nonlinDAE:manifold} is nonempty and such that for every $(t_0,\xi_0,\eta_0)\in\mathcal{M}_\mu$ there exists a (sufficiently small) neighborhood $\mathcal{U}$ in which the following properties hold:
    \begin{enumerate}
        \item[(i)] The set $\mathcal{M}_\mu$ forms a manifold of dimension $(\mu+2)N+1-r$.
        \item[(ii)] We have $\rank\begin{bmatrix}
            \mathcal{E}_\mu & \mathcal{A}_\mu
            \end{bmatrix} = r$ on $\mathcal{M}_\mu\cap \mathcal{U}$.
        \item[(iii)] We have $\corank \begin{bmatrix}
            \mathcal{E}_\mu & \mathcal{A}_\mu
            \end{bmatrix} - \corank \begin{bmatrix}
            \mathcal{E}_{\mu-1} & \mathcal{A}_{\mu-1}
            \end{bmatrix} = v$ on $\mathcal{M}_\mu\cap \mathcal{U}$ (with the convention $\corank\begin{bmatrix}
                \mathcal{E}_{-1} & \mathcal{A}_{-1}
            \end{bmatrix} = 0$).
        \item[(iv)] We have $\rank \mathcal{E}_\mu = r-a$ on $\mathcal{M}_\mu\cap\mathcal{U}$, such that there exist smooth matrix functions $Z_\mathrm{a}$ and $T_\mathrm{a}$ of size $(\mu+1)L\times a$ and $N\times (N-a)$, respectively, and pointwise maximal rank, satisfying $Z_\mathrm{a}^T\mathcal{E}_\mu = 0$, $\rank Z_{\mathrm{a}}^T\mathcal{A}_\mu = a$, and $Z_\mathrm{a}^T\tfrac{\partial \widetilde{F}_\mu}{\partial \xi}T_\mathrm{a} = 0$ on $\mathcal{M}_\mu\cap\mathcal{U}$.
        \item[(v)] We have $\rank \tfrac{\partial \mathcal{F}}{\partial \dot{\xi}}T_{\mathrm{a}} = d \vcentcolon= L-a-v$ on $\mathcal{M}_\mu\cap\mathcal{U}$ such that there exists a smooth matrix function $Z_\mathrm{d}$ of size $N\times d$ and pointwise maximal rank, satisfying $\rank Z_\mathrm{d}^T \tfrac{\partial \widetilde{F}}{\partial \dot{\xi}}T_\mathrm{a} = d$.
    \end{enumerate}
\end{hypothesis}

\begin{definition}
    Given the \DAE~\eqref{eq:rDAE}, the smallest value $\mu$ such that $\mathcal{F}$ satisfies \Cref{hyp:nonLin:nonRegular} is called the \emph{strangeness index} of~\eqref{eq:rDAE}. If $\mu=0$, then the \DAE is called \emph{strangeness-free}.
\end{definition}

\begin{remark}
    \Cref{hyp:nonLin:nonRegular} is invariant under a large class of equivalence transformations, see \cite[Sec.~4.1]{KunM06}, which is why the numbers $\mu$, $r$, $a$, $v$, and $d$ are referred to as \emph{characteristic values} for the \DAE~\eqref{eq:rDAE}.
\end{remark}

Following the discussion in \cite{KunM01}, we can use the matrix functions $Z_\mathrm{a}$ and $Z_\mathrm{d}$ to construct the \DAE
\begin{equation}
    \label{eqn:nonlinDAE:sfree}
    \widehat{\mathcal{F}}(t,\xi,\dot{\xi}) \vcentcolon= \begin{bmatrix}
    \widehat{\mathcal{F}}_\mathrm{d}(t,\xi,\dot{\xi})\\
    \widehat{\mathcal{F}}_\mathrm{a}(t,\xi)
    \end{bmatrix}
\end{equation}
with 
\begin{align*}
    \widehat{F}_\mathrm{d}(t,\xi,\dot{\xi}) \vcentcolon= (Z_\mathrm{d}^T \mathcal{F})(t,\xi,\dot{\xi}) \qquad\text{and}\qquad
    \widehat{F}_{\mathrm{a}}(t,\xi) \vcentcolon= (Z_\mathrm{a}^T \widetilde{\mathcal{F}}_\mu)(t,\xi).
\end{align*}
Note, that although the matrix functions $Z_\mathrm{a}$ and $Z_\mathrm{d}$ depend on derivatives of~$\xi$, it is possible to show (cf.~\cite{KunM01}) that the reduced quantities $\widehat{\mathcal{F}}_\mathrm{a}$ and $\widehat{\mathcal{F}}_\mathrm{d}$ are independent of higher derivatives of~$\xi$. In addition, one can show that \eqref{eqn:nonlinDAE:sfree} satisfies \Cref{hyp:nonLin:nonRegular} with characteristic values $\mu=0$, $r$, $a$, and $v$. In particular, \eqref{eqn:nonlinDAE:sfree} is strangeness-free.

In the regular case, where $N = L$ and $v=0$, we can simplify \Cref{hyp:nonLin:nonRegular} as follows, see also \cite{KunM98}.

\begin{hypothesis}
    \label{hyp:nonLin:regular}
    There exist integers $\mu$ and $a$ such that the set $\mathcal{M}_\mu$ defined in~\eqref{eqn:nonlinDAE:manifold} is nonempty and such that for every $(t_0,\xi_0,\eta_0)\in\mathcal{M}_\mu$ there exists a (sufficiently small) neighborhood $\mathcal{U}$ in which the following properties hold:
    \begin{enumerate}
        \item[(i)] We have $\rank \mathcal{E}_\mu = (\mu+1)N-a$ on     $\mathcal{M}_\mu\cap \mathcal{U}$ such that there exists a smooth matrix function $Z_\mathrm{a}$ of size $(\mu+1)N\times a$ and pointwise maximal rank such that $Z_\mathrm{a}^T\mathcal{E}_\mu = 0$ on $\mathcal{M}_\mu\cap\mathcal{U}$.
        \item[(ii)] We have $\rank Z_{\mathrm{a}}^T\mathcal{A}_\mu = a$ on $\mathcal{M}_\mu\cap\mathcal{U}$ such that there exists a smooth matrix function $T_\mathrm{a}$ of size $N\times d$ with $d\vcentcolon=N-a$ and pointwise maximal rank, satisfying $Z_\mathrm{a}^T\tfrac{\partial \widetilde{F}_\mu}{\partial \xi}T_\mathrm{a} = 0$ on $\mathcal{M}_\mu\cap\mathcal{U}$.
        \item[(iii)] We have $\rank \tfrac{\partial \mathcal{F}}{\partial \dot{\xi}}T_{\mathrm{a}} = d \vcentcolon= L-a-v$ on $\mathcal{M}_\mu\cap\mathcal{U}$ such that there exists a smooth matrix function $Z_\mathrm{d}$ of size $N\times d$ and pointwise maximal rank, satisfying $\rank Z_\mathrm{d}^T \tfrac{\partial \widetilde{F}}{\partial \dot{\xi}}T_\mathrm{a} = d$.
    \end{enumerate}
\end{hypothesis}

The relation between the original \DAE~\eqref{eq:rDAE} and the strangeness-free reformulation~\eqref{eqn:nonlinDAE:sfree} is given in the following theorem, taken from \cite[Thm.~4.11 and Thm.~4.13]{KunM06}. For the ease of presentation, we focus here on the regular case using \Cref{hyp:nonLin:regular} and remark that a similar result is also available for the more general setting described in \Cref{hyp:nonLin:nonRegular}, see \cite{KunM01} for further details.

\begin{theorem}
    Let $\mathcal{F}$ as in~\eqref{eq:rDAE} be sufficiently smooth and satisfy \Cref{hyp:nonLin:regular} with characteristic values $\mu$, $a$, and $d \vcentcolon=N-a$. Then the following statements hold.
    \begin{enumerate}
        \item[(i)] Every sufficiently smooth solution of~\eqref{eq:rDAE} is also a solution of the strangeness-free \DAE~\eqref{eqn:nonlinDAE:sfree}.
        \item[(ii)] Suppose additionally that $\mathcal{F}$ satisfies \Cref{hyp:nonLin:regular} with characteristic values $\mu+1$, $a$, and $d$. Then, for every $(t_0,\xi_0,\eta_0)\in\mathcal{M}_{\mu+1}$, the strangeness-free problem~\eqref{eqn:nonlinDAE:sfree} has a unique solution satisfying the initial condition $\xi(t_0) = \xi_0$. Moreover, this solution locally solves the original problem~\eqref{eq:rDAE}.
    \end{enumerate}
\end{theorem}

\begin{remark}\label{rem:index}
    The relation of the strangeness index concept as presented above to other index concepts commonly used in the theory of \DAEs, such as the \emph{differentiation index} \cite{CamG95}, the \emph{perturbation index} \cite{HaiW96}, the \emph{tractability index} \cite{LamMT13}, 
    the \emph{geometric index} \cite{Rhe84,Rei90}, and the \emph{structural index} \cite{Pan88,Pry01}, is discussed in \cite{Meh15}.
\end{remark}

%%%%%%%%%%%%%%%%%%%%%%%%%%%%%%%%%%%%%%%%%%%%%%%%%%%%%%
\subsection{Linear time-varying \DAE systems}
\label{sec:timevar}

If the \DAE~\eqref{eq:rDAE} is linear time-varying, \ie, of the form
\begin{equation}
    \label{eqn:dae:ltv}
    E(t)\dot{\xi}(t) = A(t)\xi(t) + f(t),\qquad \state(t_0) = \state_0,
\end{equation}
with smooth matrix functions $E,A\colon\timeInt\to\R^{L,N}$, 
then the analysis of the previous subsection can be further simplified. In this case, the Jacobians~\eqref{eqn:nonlinDAE:Jacobians} are given as
\begin{align*}
    (\mathcal{E}_\mu)_{i,j} &= {\textstyle\binom{i}{j}} E^{(i-j)}- {\textstyle\binom{i}{j+1}} A^{(i-j-1)},\quad
i,j=0,\ldots,\mu,\\
    (\mathcal{A}_\mu)_{i,j} &= \begin{cases} 
        A^{(i)} &\text{ for }  i=0,\ldots,\mu,\ j=0,\\
        0 & \text{ otherwise.}
    \end{cases}
\end{align*}
Since these matrix functions do not depend on the state variable $\state$ nor its derivatives, we can get rid of the local character of \Cref{hyp:nonLin:nonRegular} (respectively \Cref{hyp:nonLin:regular}) by using the following simplified version of the smooth rank-revealing decomposition (cf.~\Cref{thm:smoothFullRankDecomposition}), see for instance \cite[Thm.~3.9]{KunM06}.

\begin{theorem}
    \label{thm:smoothRankRevealingDecomposition}
    Let $E\in \mathcal{C}^\mu({\timeInt},\R^{L, N})$, $\mu\in\N_0\cup\{\infty\}$, with $\rank E(t)=r$ for all $t\in\timeInt$. Then there exist pointwise orthogonal functions
    $U\in \mathcal{C}^\mu(\timeInt,\R^{L,L})$ and $V\in \mathcal{C}^\mu(\timeInt,\R^{N,N})$, such that
    \begin{equation*}
        %\label{eqn:rankRevealingDecomp}
        U^T EV= \begin{bmatrix} \Sigma&0\\0&0\end{bmatrix}
    \end{equation*}
    with pointwise nonsingular $\Sigma\in \mathcal{C}^\mu(\timeInt,\R^{r,r})$.
\end{theorem}

In general, we can now proceed as in \Cref{hyp:nonLin:nonRegular} and construct the matrix functions $Z_\mathrm{a}$ and $Z_\mathrm{d}$. To avoid checking that $\mathcal{M}_\mu$ is nonempty, we further construct a matrix function $Z_\mathrm{v}$ of size $(\mu+1)N \times v$ with pointwise maximal rank that filters out equations that do not depend on $\xi$ and its derivatives. If this number of equations is nonzero, we check whether the right-hand side vanishes as well. If this is the case, then we omit these equations. If not, then $\mathcal{M}_\mu=\emptyset$ and the problem has to be regularized \cite{KunM06}. Defining
\begin{displaymath}
    g_\mu \vcentcolon= \begin{bmatrix}
        f^T &
        (\ddt f)^T & 
        \cdots &
        (f^{(\mu)})^T
    \end{bmatrix}^T
\end{displaymath}
and
\begin{align*}
    \widehat{E}_1 &\vcentcolon= Z_\mathrm{d}^T E, & 
    \widehat{A}_1 &\vcentcolon= Z_\mathrm{d}^T A, & 
    \widehat{A}_2 &\vcentcolon= Z_\mathrm{a}^T \mathcal{A}_\mu, \\
    \widehat{f}_1 &\vcentcolon= Z_\mathrm{d}^T f, &
    \widehat{f}_2 &\vcentcolon= Z_\mathrm{a}^T g_{\mu}, & 
    \widehat{f}_3 &\vcentcolon= Z_\mathrm{v}^T g_{\mu},
\end{align*}
we obtain the solution equivalent strangeness-free system
\begin{equation}
    \label{eqn:ltvDAE:sfree}
    \begin{bmatrix}
        \widehat{E}_1(t)\\
        0\\
        0
    \end{bmatrix}\dot{\xi}(t) = \begin{bmatrix}
        \widehat{A}_1(t)\\
        \widehat{A}_2(t)\\
        0
    \end{bmatrix}\xi(t) + \begin{bmatrix}
        \widehat{f}_1(t)\\
        \widehat{f}_2(t)\\
        \widehat{f}_3(t)
    \end{bmatrix}.
\end{equation}
\begin{remark}
    \label{rem:sfreeReformulationBlockStructure}
    In the behavior case for a control system, where the state variable is given as $\xi = \begin{smallbmatrix}
        \state^T & \inpVar^T
    \end{smallbmatrix}^T$, the matrix functions $E$ and $A$ have a block column structure, where the second block column corresponds to the control. Since the constructed coefficients $\widehat{A}_1$ and $\widehat{A}_2$ are obtained by transformations of the derivative array from the left, the block column structure of $A$ is retained in these matrices. Moreover, the strangeness-free reformulation does not depend on derivatives of the control $\inpVar$.
\end{remark}

If we further allow transformations of the solution space via a pointwise nonsingular matrix function, then we can obtain the following solvability result for the \DAE~\eqref{eqn:dae:ltv}.
\begin{theorem}
    \label{thm:LTV:daered}
    Under some constant rank assumptions the \DAE~\eqref{eqn:dae:ltv} is equivalent, in the sense that there is a change of basis in the solution space via a pointwise nonsingular matrix function, to a \DAE of the form
    \begin{align*}
        \dot{\xi}_1(t) &= \widehat{A}_{13}(t)\xi_3 + \widehat{f}_1(t), \\
        0 &= \xi_2(t) + \widehat{f}_2(t),  \\
        0 &= \widehat{f}_3(t),     
    \end{align*}
    where $A_{13}\in \mathcal{C}(\timeInt,\R^{d,N-d-a})$ and $\widehat{f}_1\in \mathcal{C}(\timeInt,\R^{d}), \widehat{f}_2\in \mathcal{C}(\timeInt,\R^{a}), \widehat{f}_3\in \mathcal{C}(\timeInt,\R^{v})$ are determined from $g_\mu$.
    \begin{enumerate}
        \item[(i)] If $f\in \mathcal{C}^{\mu+1}(\timeInt,\R^L)$, then \eqref{eqn:dae:ltv} is solvable if and only if $\widehat f_3=0$. 
        \item[(ii)] An initial value is consistent if and only if in addition the condition $\xi_2(t_0)=-\widehat{f}_2(t_0)$ is implied by the initial condition.
        \item[(iii)] The initial value problem is uniquely solvable if and only if in addition $N-d-a=0$.
\end{enumerate}
\end{theorem}

%%%%%%%%%%%%%%%%%%%%%%%%%%%%%%%%%%%%%%%%%%%%%%%%
\subsection{Linear time-invariant DAE systems}
\label{sec:LTIdae}

In principle, we can perform the analysis for the linear time-varying case also in the case of general constant coefficient linear \DAE systems 
\begin{equation}
    \label{daeko}
    E\dot{\xi}(t) = A\xi(t) + f(t),\qquad \xi(t_0) = \xi_0,
\end{equation}
with matrices $E,A\in\R^{L,N}$, also referred to \emph{linear time-invariant} (\LTI) \DAE systems. However, in this setting 
it is common to work with an equivalence transformation and a corresponding canonical form. For notational convenience, for the next result we also allow complex valued matrices in~\eqref{daeko} and work in the field of complex numbers. 

We call the matrix pencils $s E_i-A_i$ with $E_i,A_i\in\C^{L,N}$, $i=1,2$ \emph{(strongly) equivalent}, if there exists nonsingular matrices $S\in\C^{L,L}$ and $T\in\C^{N,N}$ such that
\begin{displaymath}
    S(\lambda E_1 - A_1)T = \lambda E_2 - A_2\qquad \text{for all $\lambda\in\C$}.
\end{displaymath}
In this case, we write $\lambda E_1-A_1 \sim \lambda E_2 - A_2$.
The associated canonical form is given by the Kronecker canonical form, see \eg \cite{Gan59b}.
\begin{theorem}[Kronecker canonical form]
    \label{th:kcf}
    Let $E,A\in \C^{L,N}$. Then 
    \begin{equation*}
        \lambda E-A \sim \diag(\mathcal{L}_{\epsilon_1},\ldots,\mathcal{L}_{\epsilon_p},
\mathcal{L}^T_{\eta_1},\ldots,\mathcal{L}^T_{\eta_q},
\mathcal{J}_{\rho_1}^{\lambda_1},\ldots,\mathcal{J}_{\rho_r}^{\lambda_r},\mathcal{N}_{\sigma_1},\ldots,\mathcal{N}_{\sigma_s}),
    \end{equation*}
where the block entries have the following properties:
\begin{enumerate}
\item[(i)] Every entry $\mathcal{L}_{\epsilon_j}$ is a bidiagonal block of size $\epsilon_j \times (\epsilon_j+1)$, $\epsilon_j\in\N_0$,
of the form
\begin{equation*}
\lambda\left[\begin{array}{cccc}
1&0\\&\ddots&\ddots\\&&1&0
\end{array}\right]-\left[\begin{array}{cccc}
0&1\\&\ddots&\ddots\\&&0&1
\end{array}\right].
\end{equation*}
\item[(ii)]
Every entry ${\mathcal L}^T_{\eta_j}$ is a bidiagonal block of size
$({\eta_j+1})\times {\eta_j}$, $\eta_j\in{\mathbb N}_0$,
of the form
\begin{equation*}
\lambda\left[\begin{array}{ccc}
1\\0&\ddots\\&\ddots&1\\&&0
\end{array}\right]-
\left[\begin{array}{ccc}
0\\1&\ddots\\&\ddots&0\\&&1
\end{array}\right].
\end{equation*}
\item[(iii)]
Every entry ${\mathcal J}_{\rho_j}^{\lambda_j}$ is a Jordan block of size
${\rho_j}\times{\rho_j}$, $\rho_j\in\N$, $\lambda_j\in\C$,
of the form
\begin{equation*}
\lambda\left[\begin{array}{cccc}
1\\&\ddots\\&&\ddots\\&&&1
\end{array}\right]-
\left[\begin{array}{cccc}
\lambda_j&1\\&\ddots&\ddots\\&&\ddots&1\\&&&\lambda_j
\end{array}\right].
\end{equation*}
\item[(iv)]
Every entry ${\mathcal N}_{\sigma_j}$ is a nilpotent block of size
${\sigma_j}\times {\sigma_j}$, $\sigma_j\in{\mathbb N}$,
of the form
\begin{equation*}
\lambda\left[\begin{array}{cccc}
0&1\\&\ddots&\ddots\\&&\ddots&1\\&&&0
\end{array}\right]-
\left[\begin{array}{cccc}
1\\&\ddots\\&&\ddots\\&&&1
\end{array}\right].
\end{equation*}
\end{enumerate}
The Kronecker canonical form is unique up to permutation of the blocks.
\end{theorem}

\begin{remark}
If the matrices are real-valued and we want to stay within the field of real numbers, then only real-valued transformation matrices $S,T$ may be used. The corresponding canonical form is called the \emph{real Kronecker canonical form}. Here, the blocks $\mathcal{J}_{\rho_j}^{\lambda_j}$ with $\lambda_j\in\C \setminus\R$ are in real Jordan canonical form instead, but the other
blocks are as in the complex case.
\end{remark}

A value $\lambda_0\in\C$ is called \emph{(finite) eigenvalue} of $\lambda E-A$ if 
\begin{displaymath}
    \rank(\lambda_0E-A)<\max_{\alpha\in\C} \rank(\alpha E-A).
\end{displaymath}
If zero is an eigenvalue of $\lambda A-E$, then $\lambda_0=\infty$ is said to be an eigenvalue of $\lambda E-A$ .
The blocks $\mathcal{J}_{\rho_j}$ correspond to
finite eigenvalues and the blocks~$\mathcal{N}_{\sigma_j}$  to the
eigenvalue $\infty$.
The size of the largest block ${\mathcal N}_{\sigma_j}$ is
called the \emph{(Kronecker) index} $\nu$ of the pencil $\lambda E-A$, where, by convention,  $\nu=0$ if $E$ is invertible.
A finite eigenvalue is called \emph{semisimple} if the largest Jordan block $\mathcal J_{\rho_j}$ associated with this block has $\rho_j=1$.
The pencil $\lambda E-A$ is called \emph{regular} if $N=L$ and
$\det(\lambda_0 E-A)\neq 0$ for some 
$\lambda_0 \in \C$. 
For regular pencils $\lambda E-A$ the Kronecker canonical form simplifies to the \emph{Weierstra\ss} canonical form.

\begin{theorem}[Weierstraß canonical form]
    \label{thm:WCF}
    Assume that the pencil $\lambda E-A$ with matrices $E,A\in\C^{N,N}$ ($E,A\in\R^{N,N})$ is regular. Then
    \begin{equation}
        \label{eqn:WCF}
        \lambda E-A \sim \lambda \begin{bmatrix}
            I & 0\\
            0 & \mathcal{N}
        \end{bmatrix} - \begin{bmatrix}
            \mathcal{J} & 0\\
            0 & I
        \end{bmatrix},
    \end{equation}
    where $\mathcal{J}$ and $\mathcal{N}$ are in Jordan (real Jordan) canonical form and $\mathcal{N}$ is nilpotent.
\end{theorem}
If the pencil is not regular then there may not exist a solution of~\eqref{daeko} or it may not be unique, see for instance \Cref{ex:DAEvanishingEquations}, while in the regular case one has the following theorem, see~\cite{KunM06} for the complex case.
\begin{theorem}\label{th:daeko}
Consider a regular matrix pencil $\lambda E-A$ of real square matrices $E,A$ and let~$S$ and~$T$ be nonsingular matrices which transform~\eqref{daeko} to its real Weierstra\ss\ canonical form~\eqref{eqn:WCF}, \ie
\begin{equation*}
    SET= \begin{bmatrix} I&0\\0&\mathcal{N}\end{bmatrix},\qquad
    SAT= \begin{bmatrix} \mathcal{J}&0\\0&I\end{bmatrix},\qquad
    Sf= \begin{bmatrix} f_1\\ f_2\end{bmatrix},
\end{equation*}
where $\mathcal{J},\mathcal{N}$ are in real Jordan canonical form and $\mathcal{N}$ is nilpotent of nilpotency index $\nu$. Set
\begin{equation*}
    T^{-1}\xi = \begin{bmatrix} \xi_1\\ \xi_2\end{bmatrix}, \qquad
    T^{-1}\xi_0= \begin{bmatrix} \xi_{1,0}\\ \xi_{2,0}\end{bmatrix} 
\end{equation*}
with analogous partitioning. If $f\in \mathcal{C}^\nu(\timeInt,\R^{N})$, then the \DAE~\eqref{daeko} is solvable. 
An initial value is consistent if and only if
\begin{displaymath}
    \xi_{2,0}=-\sum_{i=0}^{\nu-1}\mathcal{N}^i f_2^{(i)}(t_0).
\end{displaymath}
In particular, the set of consistent initial values $\xi_0$ is nonempty and every initial value problem with consistent initial condition is uniquely solvable.
\end{theorem}
\begin{remark}
\label{rem:indexrel}
To clarify the difference between the (Kronecker) index and the strangeness index, observe that a regular \LTI \DAE system has (Kronecker) index $\nu>0$, then its strangeness index is $\mu=\nu-1$ and if $\nu=0$, then also $\mu=0$, see also \cite{Meh15}.
\end{remark}

%%%%%%%%%%%%%%%%%%%%%%%%%%%%%%%%%%%%%%%%%%%%%%%%%%%%%%%%%
\section{Control concepts for general DAE systems}
\label{sec:controlaspects}
In this section we discuss different aspects related to control theory of general \DAE systems. Most of our discussion focuses on linear descriptor systems of the form
\begin{subequations}
    \label{eqn:descriptor:linear}
    \begin{align}
        \label{eqn:descriptor:linear:state}
        E\dot{\state} &= A\state + B\inpVar,\\
        \label{eqn:descriptor:linear:output}
        \outVar &= C\state,
    \end{align}
\end{subequations}
with either
\begin{itemize}
    \item[--] matrices $E,A\in\R^{\stateDim,\stateDim}$, $B\in\R^{\stateDim,m}$, $C\in\R^{p,n}$ for the \LTI case, and
    \item[--] matrix functions $E,A\colon\timeInt\to\R^{\stateDim,\stateDim}$, $B\colon\timeInt\to\R^{\stateDim,m}$, $C\colon\timeInt\to\R^{p,n}$ for the \emph{linear time-varying} (\LTV) case.
\end{itemize}

\begin{remark}
    \label{rem:feedthroughRemoval}
In general, the descriptor system~\eqref{eqn:descriptor:linear} may also include a feedthrough term, \ie, the output equation~\eqref{eqn:descriptor:linear:output} is given as
\begin{displaymath}
    \outVar = C\state + D\inpVar
\end{displaymath}
with a suitable matrix or matrix function $D$. However, in the \DAE context, we can rewrite the descriptor system~\eqref{eqn:descriptor:linear} without the feedthrough term, as follows. Consider any decomposition $D = D_{\mathrm{c}}D_{\mathrm{b}}$ and the extended system matrices or matrix functions,
\begin{align*}
    \widehat{E} &\vcentcolon= \begin{bmatrix}
        E & 0\\
        0 & 0
    \end{bmatrix}, & \widehat{A} &\vcentcolon= \begin{bmatrix}
        A & 0\\
        0 & -I
    \end{bmatrix}, & \widehat{B} &\vcentcolon= \begin{bmatrix}
        B\\D_{\mathrm{b}}
    \end{bmatrix}, & \widehat{C} &\vcentcolon= \begin{bmatrix}
        C & D_{\mathrm{c}}
    \end{bmatrix}.
\end{align*}
Then the solution of the associated descriptor system contains as a part the solution of the descriptor system with feedthrough term.
\end{remark}

\subsection{Feedback regularization}
\label{sec:feedbackRegularization}

As we have seen in \Cref{sec:modelClass}, a \DAE system may not be regular, \ie, there may not be any initial values such that the initial value problem has a solution, or a solution for a consistent initial value may not be unique. To deal with this situation, we first discuss how to regularize a descriptor system via instantaneous, proportional (linear) state or output feedback, \ie, via feedback laws of the form
\begin{equation}
    \label{eqn:feedback}
    \inpVar = F_1\state+w\qquad\text{or}\qquad
    \inpVar = F_2\outVar+w,
\end{equation}
respectively, with suitable matrices or matrix functions $F_1$ and $F_2$. After applying such a feedback, the closed-loop system matrices, respectively matrix functions, are given as $\widetilde{E} \vcentcolon= E$ and
\begin{equation*}
    \widetilde{A}_1 \vcentcolon= A+BF_1\qquad\text{and}\qquad \widetilde{A}_2 \vcentcolon= A+BF_2C,
\end{equation*}
respectively. 

We start our analysis for the \LTI case and recall important conditions for controllability and observability. If the matrix $E$ in~\eqref{eqn:descriptor:linear} is nonsingular, then the well-known Hautus lemma, see \eg \cite{Dai89}, asserts that the \LTI descriptor system~\eqref{eqn:descriptor:linear} is controllable if and only if
\begin{align}
    \label{eqn:C1}
    \rank  \begin{bmatrix} \lambda E - A & B \end{bmatrix} = n\qquad \text{for all $\lambda\in\C$}.
\end{align}
If $E$ is singular, then the situation is more involved and we need the following conditions, taken for instance from \cite{Dai89,BunBMN99}.

\begin{definition}
    Consider the \LTI descriptor system~\eqref{eqn:descriptor:linear} and let~$S_\infty$ be a matrix with columns that span the kernel of $E$.
    \begin{enumerate}
        \item[(i)] The system~\eqref{eqn:descriptor:linear} is called \emph{controllable at $\infty$} or \emph{impulse controllable} if
             \begin{align}
                \label{eqn:C2}
                \rank \begin{bmatrix} E & A S_\infty & B \end{bmatrix} &= n.
            \end{align}
        \item[(ii)] The system~\eqref{eqn:descriptor:linear} is called \emph{strongly controllable}, if it is impulse controllable and \eqref{eqn:C1} is satisfied.
        \item[(iii)] The system~\eqref{eqn:descriptor:linear} is called \emph{strongly stabilizable}, if it is impulse controllable and \eqref{eqn:C1} holds for all $\lambda \in\C$ with $\real(\lambda)\geq 0$. 
    \end{enumerate}
\end{definition}

The corresponding dual conditions with respect to the output equation are given as
\begin{align}        
    \label{eqn:O1}
    \rank\begin{bmatrix}\lambda E^T-A^T & C^T\end{bmatrix} &= \stateDim, \\        \label{eqn:O2}
    \rank\begin{bmatrix} E^T & A^T T_{\infty} & C^T \end{bmatrix} &= \stateDim,
\end{align}
respectively, where $T_\infty$ is a matrix that spans the kernel of $E^T$.
\begin{definition}\label{def:concon}
    Consider the \LTI descriptor system~\eqref{eqn:descriptor:linear} and let~$T_\infty$ be a matrix with columns that span the kernel of $E^T$.
    \begin{enumerate}
        \item[(i)] The system~\eqref{eqn:descriptor:linear} is called \emph{observable at~$\infty$} or \emph{impulse observable} if condition~\eqref{eqn:O2} is satisfied.
        \item[(ii)] The system~\eqref{eqn:descriptor:linear} is called \emph{strongly observable} if it is impulse observable and if \eqref{eqn:O1} holds for all $\lambda\in\C$. 
        \item[(iii)] The system~\eqref{eqn:descriptor:linear} is called \emph{strongly detectable} if it is impulse observable and if \eqref{eqn:O1} holds for all $\lambda \in \C$ with $\real(\lambda)\geq 0$.
    \end{enumerate}
\end{definition}

A system that satisfies conditions~\eqref{eqn:C1} and \eqref{eqn:O1} is called \emph{minimal}.

Conditions \eqref{eqn:O1} and \eqref{eqn:O2} are preserved under non-singular equivalence transformations as well as under state and output feedback. More precisely, if the system  satisfies \eqref{eqn:O1} and \eqref{eqn:O2}, then for any non-singular $U \in\R^{\stateDim,\stateDim}$, $V \in \R^{\stateDim,\stateDim}$, and any $F_1\in\R^{m,\stateDim}$ and $F_2\in\R^{m,p}$,
the system with coefficients $(\widetilde{E}, \widetilde{A}, \widetilde{B}, \widetilde{C})$ satisfies the same condition for all of the following three choices:
\begin{align*}
    \widetilde{E} &= U E V, & \widetilde{A} &= U A V, & \widetilde{B} &= U B, \\
    \widetilde{E} &= E, & \widetilde{A} &= A + B F_1, & \widetilde{B} &= B,\\
    \widetilde{E} &= E, & \widetilde{A} &= A + B F_2 C, & \widetilde{B} &= B.
\end{align*}
Analogous invariance properties hold for \eqref{eqn:O1} and \eqref{eqn:O2}. Further details and properties of \LTI \DAE systems are discussed in \cite{BerR13}.

Note, however, that regularity or non-regularity of the pencil and the (Kronecker) index are in general not preserved under state or output feedback, 
respectively. On the contrary, feedback of the form~\eqref{eqn:feedback} may be used to regularize the system, as detailed in the following theorem taken from \cite{BunBMN99}.

\begin{theorem}
    \label{thm:feedbackRegularization}
    Consider the \LTI descriptor system~\eqref{eqn:descriptor:linear}.
    \begin{enumerate}
        \item[(i)] If~\eqref{eqn:descriptor:linear} is impulse controllable, \ie, condition~\eqref{eqn:C2} is satisfied, then there exists a suitable linear state feedback matrix $F_1$ such that $\lambda E-(A+B F_1)$ is regular and of (Kronecker) index $\nu\leq 1$. 
        \item[(ii)] If~\eqref{eqn:descriptor:linear} is impulse controllable and impulse observable, \ie, the conditions~\eqref{eqn:C2} and~\eqref{eqn:O2} hold, then there exists a linear output feedback matrix $F_2$ such that the pencil $\lambda E-(A+BF_2 C)$ is regular and of (Kronecker) index $\nu\leq 1$.
    \end{enumerate}
\end{theorem}

\begin{remark}
    \label{rem:delayedFeedback}
    Although instantaneous feedback is a convenient theoretical approach, it may suffer from the fact that signals have to be measured first, and some calculations have to be carried out, thus resulting in an intrinsically necessary time delay. If this time delay cannot be ignored in the modeling phase, then for some $\tau>0$, the feedback takes the form
    \begin{displaymath}
        \inpVar(t) = F_1\state(t-\tau)+w(t)\qquad\text{or}\qquad
        \inpVar(t) = F_2\outVar(t-\tau)+w(t),
    \end{displaymath}
    thus rendering the closed-loop system a delay \DAE. However, the \DAE can be regularized with delayed feedback if and only if it can be regularized with instantaneous feedback, see \cite{TreU19,Ung20b} for further details. Nevertheless, we always assume that the feedback delay can be ignored in the modeling phase within this survey.
\end{remark}

For the feedback regularization in the \LTV and nonlinear case, we follow 
\cite{CamKM12}, and use the behavior approach as introduced in~\eqref{eqn:behaviorState}. In more detail, for the \LTV descriptor system~\eqref{eqn:descriptor:linear}, we form the (matrix) functions
\begin{equation}
    \label{eqn:descriptor:ltv:behavior}
    \xi \vcentcolon= \begin{bmatrix} \state \\ \inpVar \end{bmatrix},\quad 
    \mathcal{E} \vcentcolon= \begin{bmatrix} E & 0 \end{bmatrix},\quad
    \mathcal{A} \vcentcolon= \begin{bmatrix} A & B \end{bmatrix}.
\end{equation}
Ignoring the fact that~$\xi$ is composed of parts that may have quite different orders of differentiability, we form the derivative array~\eqref{eq:derivativeArray} and follow the approach presented in \Cref{sec:timevar}. In more detail, we construct matrices $\widehat{\mathcal{E}}_1$, $\widehat{\mathcal{A}}_1$, and $\widehat{\mathcal{A}}_2$, such that the system
\begin{equation}
    \label{eqn:descriptor:ltv:sfree:v1}
    \begin{bmatrix}
        \widehat{\mathcal{E}}_1(t)\\
        0\\
        0
    \end{bmatrix}\dot{\xi}(t) = \begin{bmatrix}
        \widehat{\mathcal{A}}_1(t)\\
        \widehat{\mathcal{A}}_2(t)\\
        0
    \end{bmatrix}\xi(t)
\end{equation}
is solution equivalent to~\eqref{eqn:descriptor:linear} and strangeness-free. Since the matrix functions are obtained solely by transformations of the derivative array from the left, the partitioning of the matrices as introduced in~\eqref{eqn:descriptor:ltv:behavior} is retained in~\eqref{eqn:descriptor:ltv:sfree:v1}, see also \Cref{rem:sfreeReformulationBlockStructure}. In particular, the state $\state$ and the input function $\inpVar$ are not mixed, such that we can rewrite~\eqref{eqn:descriptor:ltv:sfree:v1} as
\begin{equation}
    \label{eqn:descriptor:ltv:sfree:v2}
    \begin{bmatrix}
        \hat{E}_1(t)\\
        0\\
        0
    \end{bmatrix}\dot{\state}(t) = \begin{bmatrix}
        \hat{A}_1(t)\\
        \hat{A}_2(t)\\
        0
    \end{bmatrix}\state(t) + \begin{bmatrix}
        \hat{B}_1(t)\\
        \hat{B}_2(t)\\
        0
    \end{bmatrix}\inpVar(t),
\end{equation}
with $\widehat{\mathcal{E}}_1 = \begin{bmatrix}
    \widehat{E}_1 & 0
\end{bmatrix}$, $\widehat{\mathcal{A}}_1 = \begin{bmatrix}
    \widehat{A}_1 & \widehat{B}_1
\end{bmatrix}$, and $\widehat{\mathcal{A}}_2 = \begin{bmatrix}
    \widehat{A}_2 & \widehat{B}_2
\end{bmatrix}$.

\begin{remark}
Note that we have constructed~\eqref{eqn:descriptor:ltv:sfree:v1} such that the system
is strangeness-free (with respect to the combined state variable~$\xi$). Since~\eqref{eqn:descriptor:ltv:sfree:v2} is simply obtained by rewriting~\eqref{eqn:descriptor:ltv:sfree:v1}, it is also strangeness-free with respect to~$\xi$. However, it may not be strangeness-free with respect to the original state variable $\state$ (in the sense that we assume $u$ to be given). To distinguish this subtlety in the following, we say that a descriptor system is strangeness-free \emph{as a free system}, if it is strangeness-free with respect to $\state$ for given input~$\inpVar\equiv0$.
\end{remark}

To theoretically analyze the regularizability via feedback control, we use the following condensed form, see \cite{KunMR01}. 

\begin{theorem}
    \label{thm:controlcanform}
    Consider the \LTV descriptor system~\eqref{eqn:descriptor:linear} and assume that the corresponding behavior system defined in~\eqref{eqn:descriptor:ltv:behavior} has a well-defined strangeness index with strangeness-free form~\eqref{eqn:descriptor:ltv:sfree:v1}. Then, under some constant rank assumptions, there exist pointwise nonsingular matrix functions  $S_{\mathrm{z}}\in\mathcal{C}(\timeInt,\R^{\stateDim,\stateDim})$, $T_{\mathrm{z}}\in\mathcal{C}(\timeInt,\R^{\stateDim,\stateDim})$, $S_{\mathrm{y}}\in\mathcal{C}(\timeInt,\R^{\outputDim,\outputDim})$, $T_{\mathrm{u}}\in\mathcal{C}(\timeInt,\R^{\inputDim,\inputDim})$, such that setting 
    \begin{align*}
        \state &= T_{\mathrm{\state}} \begin{bmatrix} \state_1^T & \state_2^T & \state_3^T & \state_4^T\end{bmatrix}^T, &
        \inpVar &= T_{\mathrm{\inpVar}} \begin{bmatrix} \inpVar_1^T & \inpVar_2^T \end{bmatrix}^T, &
        \outVar &= S_{\mathrm{\outVar}} \begin{bmatrix} \outVar_1^T & \outVar_2^T\end{bmatrix}^T,
    \end{align*}
    and multiplying~\eqref{eqn:descriptor:ltv:sfree:v2} by appropriate matrix functions from the left, yields
    a transformed control system of the form
    \begin{subequations}
        \label{concf3}
        \begin{align}
            \label{concf3A} \dot{\state}_1 &= A_{13}(t)\state_3 + A_{14}(t)\state_4 + B_{12}(t)\inpVar_2, &&d \\
            \label{concf3B} 0 &= \state_2 + B_{22}(t)\inpVar_2, && a-\phi\\
            \label{concf3C} 0 &= A_{31}(t)\state_1 + \inpVar_1, && \phi\\
            \label{concf3D} 0 &= 0, && v\\
            \label{concf3E} \outVar_1 &= \state_3, && \omega\\
            \label{concf3F} \outVar_2 &= C_{21}(t)\state_1 + C_{22}(t)\state_2, && p-\omega
        \end{align}
    \end{subequations}
    where the number at the end of each block equation denotes the number of equations within this block.
\end{theorem}

\begin{corollary}
    \label{cor:consis}
    Let the assumptions be as in \Cref{thm:controlcanform}. Furthermore, let the quantities~$\phi$ and~$\omega$ defined in \eqref{concf3} be constant. Then the following properties hold.
    \begin{itemize}
        \item[(i)] The \LTV system~\eqref{eqn:descriptor:linear} is consistent. The equations \eqref{concf3D} describe redundancies in the system that can be omitted.
        \item [(ii)] If $\phi=0$, then for a given input function~$\inpVar$, an initial value is consistent if and only if it implies \eqref{concf3B}. Solutions of the corresponding initial value problem will in general not be unique.
        \item[(iii)] The system is regular and strangeness-free (as a free system) if and only if $v=\phi=0$ and $d+a=n$.
    \end{itemize}
\end{corollary}
Analogous to the constant coefficient case we can use proportial feedback to modify some of the system properties. However, the following result, see \cite[Thm.~3.80]{KunM06}, states that some properties stay invariant. 
\begin{theorem}
    \label{th:invfb}
    Consider the \LTV descriptor system~\eqref{eqn:descriptor:linear} and suppose that the assumptions of \Cref{thm:controlcanform} are satisfied. Then, the characteristic values $d$, $a$, and $v$ are invariant under proportional state feedback and proportional output feedback.
\end{theorem}
The strangeness index (as a free system) as well as the regularity of the system can, however, be modified by proportional feedback, cf.~\cite{KunMR01}. 
\begin{corollary}
    \label{cor:regul}
    Let the assumptions of \Cref{cor:consis} hold. 
    \begin{enumerate}
        \item[(i)] There exists a state feedback $\inpVar = F\state + w$ such that the closed-loop system
    \begin{equation*}
        E\dot{\state} = (A+BF)\state + Bw
    \end{equation*}
    is regular (as a free system) if and only if $v=0$ and $d+a=\stateDim$.
        \item[(ii)] There exists an output feedback $\inpVar = F\outVar+w$ such that the closed-loop system
        \begin{equation*}
            E\dot{\state} = (A+BFC)\state + Bw
        \end{equation*}
        is regular (as a free system) if and only if $v=0$, $d+a=n$, and $\phi=\omega$.
    \end{enumerate}
\end{corollary}

A similar local result is also available for nonlinear descriptor systems of the form \eqref{eqn:nonlinDAE:sfree}, see \cite{CamKM12} for details.

%%%%%%%%%%%%%%%%%%%%%%%%%%%%%%%%%%%%%%%%%%%%%%%%%%%%%%
\subsection{Stability}
\label{sec:stability}
One of the key questions in control is whether a system can be stabilized via feedback control. In this section we therefore recall the stability theory for \emph{ordinary differential equations} (\ODEs) and discuss how these concepts are generalized to \DAE systems. 
The classical stability concepts for \ODEs are as follows, see, \eg, \cite{HinP05}. Consider an \ODE of the form
\begin{equation}
\label{ODE}
\dot{\state} = f(t,\state), \qquad t\in \timeInt_\infty= [t_0,\infty]
\end{equation}
and denote the solution satisfying the initial condition $\state(t_0)=\state_0$ by $\state(\cdot;t_0,\state_0)$.
\begin{definition}\label{def:stabtraj}
A solution~$\state(\cdot;t_0,z_0)$ of~\eqref{ODE} is called
\begin{itemize}
\item [(i)]
\emph{stable} if for every $\varepsilon >0$ there exists $\delta>0$ such that for all $\hat{\state}_0\in\R^{\stateDim}$ with $\|\hat{\state}_0-\state_0\|<\delta$
\begin{itemize}
\item the initial value problem \eqref{ODE} with initial condition $\state(t_0)=\hat{\state}_0$ is solvable on~$\timeInt_\infty$ and
\item the solution $\state(t;t_0,\hat{\state}_0)$ satisfies
$\|\state(t;t_0,\hat{\state}_0)-\state(t;t_0,\state_0) \|<\varepsilon$ on~$\timeInt_\infty$;
\end{itemize}
\item [(ii)] \emph{asymptotically stable} if
it is stable and there exists $\varrho>0$ such that for all $\hat{\state}_0\in\R^{\stateDim}$ with $\|\hat{\state}_0-\state_0\|<\rho$
\begin{itemize}
\item the initial value problem~\eqref{ODE} with initial condition
$\state(t_0)=\hat{\state}_0$ is solvable on~$\timeInt_\infty$  and
\item the solution $\state(t;t_0,\hat{\state}_0)$ satisfies
$\lim_{t\to \infty}\| \state(t;t_0,\hat{\state}_0) -\state(t;t_0,\state_0) \|=0$;
\end{itemize}
\item [(iii)] \emph{exponentially stable} if it is stable and
\emph{exponentially attractive}, \ie, if there exist $\delta>0$, $L>0$, and $\gamma>0$ such that for all $\hat{\state}_0\in\R^{\stateDim}$ with $\|\hat{\state}_0-\state_0\|<\delta$
\begin{itemize}
\item the initial value problem \eqref{ODE} with initial condition
$\state(t_0)=\hat{\state}_0$ is solvable on~$\timeInt_\infty$ and
\item the solution satisfies the estimate
\begin{equation*}
\|\state(t;t_0,\hat{\state}_0) - \state(t;t_0,\state_0) \|<L\mathrm{e}^{-\gamma(t-t_0)} \text{~on~}\timeInt_\infty.
\end{equation*}
\end{itemize}
If $\delta$ does not depend on $t_0$, then we say the solution is \emph{uniformly (exponentially) stable}.
\end{itemize}
\end{definition}
By shifting the arguments  we may assume that the reference solution is the trivial solution $\state(t;t_0,\state_0)=0$. 

\begin{remark}
    To analyze the stability of a given \ODE systems (finite or infinite-dimensional) is  analytically and computationally very challenging, see \eg~\cite{Adr95,DieRV97,DieV02b,HinP05,LaS76}.
\end{remark}

For \DAE systems 
\begin{equation*}
    %\label{DAE}
    F(t,\state,\dot{\state})=0,\qquad t \in \timeInt_\infty
\end{equation*}
the stability concepts in \Cref{def:stabtraj} essentially carry over. However, when perturbing a consistent initial value, it may happen that the perturbed initial value is not consistent anymore. Then the solution (if one allows discontinuities in the part of the state vector that is not differentiated) has a discontinuous jump that transfers the solution to the constraint manifold. For a strangeness-free \DAE this would not be a problem because such a jump does not destroy the stability properties. If, however, the strangeness index is bigger than zero 
then, due to the required differentiations, the solution may only exist in the distributional sense, see for instance \cite{KunM06,RabR96a,Tre13}.
\begin{example} 
\label{ex:DAE:index2}
    Consider the homogeneous linear time-invariant \DAE from \cite{DuLM13},
\begin{displaymath}
    \dot{\state}_1= \state_2,\qquad 0=-\state_1-\varepsilon \state_2.
\end{displaymath}
If~$\varepsilon >0$ then the \DAE is strangeness-free and has the solution 
\begin{displaymath}
    \state_1(t) = \mathrm{e}^{-\varepsilon^{-1} t}\state_1(0),\qquad \state_2(t)= -\varepsilon^{-1}\mathrm{e}^{-\varepsilon^{-1} t}\state_1(0).
\end{displaymath}
With a consistent initial value $\state_2(0)= -\varepsilon^{-1} \state_1(0)$, the solution is asymp\-to\-ti\-cally stable but this limit would not exist for $\varepsilon\to 0$ except if $\varepsilon^{-1} \state_1(t)$ is bounded for $t\to 0$. If $\varepsilon=0$ then the \DAE has strangeness index one and for the solution  $\state_1=0$, $\state_2=\dot{\state}_1=0$ the initial value $\state_1(0)$ is restricted as well. For $\state_1(0)=1$ then $\state_1$ exists and is the discontinuous function that jumps from $1$ to $0$ at $t=0$ and $\state_2$ would only be representable by a delta distribution. Finally, if $\varepsilon <0$ then the solution is unstable.
\end{example}
The stability analysis and computational methods for \DAE systems, therefore, assumes uniquely solvable strangeness-free systems. If the system is not strangeness-free, then one first performs a strangeness-free reformulation as discussed in \Cref{sec:DAEtheory}.

For a strangeness-free system then a solution of the system is called \emph{ stable, asymptotically stable, (uniformly) exponentially stable}, respectively, if it satisfies the corresponding condition in \Cref{def:stabtraj}.
Then many analytical results and computational methods can be extended to the case of strangeness-free \DAE systems, see \cite{KunM07,LinM09,LinM11a,LinMV11,LinM14}.

For \LTI \ODE systems 
\begin{equation}\label{linode}
    \dot{\state} =A\state,
\end{equation}
with $A\in \R^{\stateDim,\stateDim}$, it is well-known, see \eg \cite{Adr95}, that the system is asymptotically (and also uniformly exponentially) stable if all the eigenvalues are in the open left half of the complex plane and stable if all the eigenvalues are in the closed left half plane and the eigenvalues on the imaginary axis are semisimple, \ie the associated Jordan blocks have size at most one. 

The stability analysis can also be carried out via the computation of a Lyapunov function given by $ V(\state)=\tfrac{1}{2} \state^T X\state $, where for stability 
$X=X^T>0$ is a  solution of the \emph{Lyapunov inequality}
\begin{equation}\label{lya-ineq} 
    A^T X+ X A \leq 0,
\end{equation} 
and for asymptotic stability it  is a positive definite solution of the strict inequality, $ A^T X+ X A < 0$, see, \eg, \cite{Kai80}. 

The spectral characterization of stability for \ODE systems can be generalized to \LTI \DAE systems
\begin{equation}\label{stabdaeko}
    E\dot{\state}=A\state,
\end{equation}
with $E,A\in\R^{\stateDim,\stateDim}$, see  \eg  \cite{DuLM13}.
\begin{theorem}\label{th:daestab}
Consider the \DAE~\eqref{stabdaeko} with a regular pencil $\lambda E-A$ of (Kronecker) index at most one. The trivial solution $\state=0$ then has the following stability properties:
\begin{itemize}
\item [(i)]
If all finite eigenvalues have non-positive real part
and the eigenvalues  on the imaginary axis  are semisimple, then the trivial solution~$\state=0$ is stable.
\item [(ii)]
If  all finite eigenvalues have negative real part, then the trivial solution~$\state=0$ is 
uniformly and thus exponentially and asymptotically stable.
\end{itemize}
\end{theorem}

For \LTV ordinary differential-equations
\begin{equation}
	\label{odevarhom}
	\dot{\state} = A(t)\state,
%\quad t\in{\mathbb I}=[t_0,\infty),
\end{equation}
the different stability properties are characterized by means of the fundamental solution
$\Phi(\cdot,t_0)\in \mathcal{C}^1(\timeInt,\R^{\stateDim,\stateDim})$ that satisfies
\begin{equation}
	\label{fundsol}
    \tfrac{\partial}{\partial t} \Phi(t,t_0) = A(t) \Phi(t,t_0),\qquad
    \Phi(t_0,t_0)=I_n
\end{equation}
such that $\state(t;t_0,\state_0)=\Phi(t,t_0)\state_0$, see \eg \cite{Adr95}. 

\begin{theorem}
    \label{th:prolinode}
    Consider the \LTV \ODE~\eqref{odevarhom} with fundamental solution~$\Phi$ as in~\eqref{fundsol}. The trivial solution of the \LTV \ODE~\eqref{odevarhom}
    \begin{itemize}
        \item [(i)] is stable if and only if there exists a constant $L>0$ with $\| \Phi(t,t_0) \|\leq L$ on~$\timeInt$;
        \item [(ii)] is asymptotically stable if and only if $\| \Phi(t,t_0) \|\to 0$ for $t\to \infty$;
        \item [(iii)] is exponentially stable if there exist $L>0$ and $\gamma>0$ such that 
        	\begin{displaymath}
        		\| \Phi(t,t_0) \|\leq L e^{-\gamma(t-t_0)} \text{ on~$\timeInt$}.
        	\end{displaymath}
    \end{itemize}
\end{theorem}

To obtain the results that extend this characterization to \LTV \DAE systems
\begin{equation}
    \label{daevarhom}
    E(t)\dot{\state} = A(t)\state,
\end{equation}
assume again that the initial value problem associated with~\eqref{daevarhom} has a unique solution for every consistent initial value and is strangeness-free. If the system is not strangeness-free then one first performs the transformation to strangeness-free form as in \Cref{sec:DAEtheory}.

For a regular strangeness-free system~\eqref{daevarhom} there exist pointwise orthogonal  matrix functions $S\in \mathcal{C}(\timeInt,\R^{\stateDim,\stateDim})$, $T\in\mathcal{C}^1(\timeInt,\R^{\stateDim,\stateDim})$ such that
\begin{equation}
    \label{pwunitary}
    SET= \begin{bmatrix}E_{11} &0\\0&0\end{bmatrix},\qquad
    SAT-SE\dot{T} = \begin{bmatrix} A_{11}&A_{12}\\ A_{21} &A_{22}\end{bmatrix},
\end{equation}
with $E_{11},A_{22}$ pointwise nonsingular, and $\state=V\begin{smallbmatrix}\state_1\\
\state_2\end{smallbmatrix}$.
Under the condition of a bounded matrix function $A_{22}^{-1}A_{21}$, 
one obtains the algebraic equation $\state_2= -A_{22}^{-1} A_{21} \state_1$ and the so-called inherent \ODE associated with \eqref{daevarhom} given by
\begin{equation}
    \label{inhodehom}
    \dot{\state}_1= E_{11}^{-1}(A_{11}-A_{12} A_{22}^{-1} A_{21})\state_1.
\end{equation}
It is then clear that for the different stability concepts
to extend to \DAE systems it is necessary that \eqref{inhodehom} satisfies the corresponding stability conditions.

\begin{remark}
    \label{rem:spectralintervals}
For \ODE systems there is also a well-known extension of the spectral stability analysis via the computation of Lyapunov, Bohl and Sacker-Sell spectral intervals. These results have been extended to \DAE systems in \cite{Ber12,LinM09,LinM12,LinM14,LinMV11}. We will not discuss this topic here further, but just mention that it is computationally highly expensive. 
\end{remark}

For general autonomous nonlinear \ODE systems $\dot{\state}=f(\state)$ the fundamental approach to analyze the stability properties is to compute a Lyapunov function~$V(\state)$ such that $\dot{V}(\state)$ is negative definite in neighborhood of the solution $\state$. If such a Lyapunov function exists, then the equilibrium solution $
\state=0$ is asymptotically stable, see \eg \cite{LaSL61,Adr95}. This approach can be used as well for general strangeness-free \DAE systems by reducing the system to the inherent \ODE. 

%%%%%%%%%%%%%%%%%%%%%%%%%%%%%%%%%%%%%%%%%%%%%%%%%%%%%%
\subsection{Stabilization}
\label{sec:stabilization}
Since in physical systems the stability of a  solution is typically a crucial property, it is important to know how a stable system behaves under
disturbances or uncertainties in the coefficients and how an unstable system can be stabilized with the help of available feedback control.
Let us consider this question first for \LTI control problems of the form \eqref{eqn:descriptor:linear} and ask whether it is possible to achieve stability or asymptotic stability via proportional state or output feedback. 

We have seen in \Cref{thm:feedbackRegularization} that for strongly stabilizable systems there exist an $F\in\R^{m,\stateDim}$ such that the pair $(E,A+BF)$ is regular and of index at most one, and for strongly stabilizable
and strongly detectable systems there exist an $F\in\R^{m,p}$ such that the pair $(E,A+BFC)$ is regular and of (Kronecker) index at most one. In the construction of stabilizing feedbacks we can therefore assume that such a (preliminary) state or output feedback has been performed, and therefore that the pair $(E,A)$ is regular and of (Kronecker) index at most one.

The calculation of stabilizing feedback control laws can then be performed via an optimal control approach (see also the forthcoming \Cref{sec:ocpdae}),
by minimizing the cost functional
\begin{equation}
    \label{stabcostfunctional}
    \mathcal{J}(\state, \inpVar) = \tfrac{1}{2} \int^{\infty}_{t_0} \begin{bmatrix} \state \\ \inpVar \end{bmatrix}^T \begin{bmatrix} W_{\mathrm{z}} & S \\ S^T & W_{\mathrm{u}} \end{bmatrix}
\begin{bmatrix} \state \\ \inpVar \end{bmatrix}\dt
\end{equation}
subject to the constraint \eqref{eqn:descriptor:linear}. We
could have also used the output function~$\outVar$ instead of the state function~$\state$ by
inserting $\outVar = C\state$ and modifying the weights accordingly. The following results, which are based on the Pontryagium maximum principle, are taken from \cite{Meh91}.

\begin{theorem}
    \label{th:stabneccon}
    Consider the optimal control problem to minimize \eqref{stabcostfunctional} subject to the constraint \eqref{eqn:descriptor:linear} with a pair $(E,A)$ that is regular and of (Kronecker) index at most one. Suppose that a continuous solution $\inpVar^\star$ to the optimal control problem exists and let $\state^\star$ be the solution of \eqref{eqn:descriptor:linear} with this input function. Then, there exists a Lagrange multiplier function $\lambda \in \mathcal{C}^1(\timeInt,\R^{\stateDim})$ 
    such that $\state^\star$, $\inpVar^\star$, and $\lambda$ satisfy the boundary value problem 
    \begin{equation}
        \label{neccond}
        \begin{bmatrix}
            0 & E & 0\\
            -E^T & 0 & 0\\
            0 & 0 & 0
        \end{bmatrix}\begin{bmatrix}
            \dot{\lambda}\\
            \dot{\state}\\
            \dot{\inpVar}
        \end{bmatrix} = \begin{bmatrix}
            0 & A & B\\
            A^T & W_{\mathrm{z}} & S\\
            B^T & S^T & W_{\mathrm{u}}
        \end{bmatrix}\begin{bmatrix}
            \lambda\\
            \state\\
            \inpVar
        \end{bmatrix},
    \end{equation}
    with boundary conditions
    \begin{equation} 
        \label{stabtwoptbc}
        E^\dagger E \state(t_0)=\state_0,\qquad \lim_{t\to\infty} E^T\lambda(t)=0,
    \end{equation}
    where $E^\dagger$ denotes the Moore-Penrose inverse of $E$.
\end{theorem}
\begin{theorem}
    \label{th:stabsuffcond}
    Suppose that $\state^\star$, $\inpVar^\star$, $\lambda$ satisfy the boundary value problem \eqref{neccond}, \eqref{stabtwoptbc} and suppose, furthermore, that the matrix
    \begin{displaymath}
        \begin{bmatrix} W_{\mathrm{z}} & S\\S^T & W_{\mathrm{u}}\end{bmatrix}
    \end{displaymath}
    is positive semi-definite. Then
    \begin{displaymath}
        \mathcal{J}(\state,\inpVar) \geq \mathcal{J}(\state^\star,\inpVar^\star)
    \end{displaymath}
    for all $\state$ and $\inpVar$ satisfying \eqref{eqn:descriptor:linear}.
\end{theorem}

The solution of the optimality boundary value problem~\eqref{neccond} with boundary conditions~\eqref{stabtwoptbc} will yield the optimal control $\inpVar$ and the corresponding optimal state $\state$. However, in many real-world applications one would like the optimal control to be a state feedback. 
A sufficient condition for this to hold is that the matrix pencil associated with \eqref{neccond} is regular of (Kronecker) index at most one and has no purely imaginary eigenvalue. If the matrix $W_{\mathrm{u}}$ is positive definite and  $(E,A,B)$ is strongly stabilizable, then this can be guaranteed, see \cite{Meh91}, and we 
can proceed as follows.
Recall that we have assumed that the pair $(E,A)$ is regular and of (Kronecker) index at most one. Then the coefficients $E,A,B,W_{\mathrm{\state}},S,W_{\mathrm{\inpVar}}$
can be transformed such that $(E,A)$ is in  Weierstra\ss\ canonical form~\eqref{eqn:WCF}, \ie
\begin{align*}
    %\label{trftriplea}
    PEQ &= \begin{bmatrix} I&0\\0& 0 \end{bmatrix}, &
    PAQ &= \begin{bmatrix} J&0\\0&I\end{bmatrix}, &
    PB &= \begin{bmatrix} B_1\\ B_2\end{bmatrix},
\end{align*}
with transformed cost function
\begin{equation*}
    %\label{costmatrices}
    Q^T W_{\mathrm{\state}} Q= \begin{bmatrix} W_{11} & W_{12} \\ W_{21} & W_{22} \end{bmatrix},\quad  Q^T S = \begin{bmatrix} S_1 \\ S_2 \end{bmatrix}.
\end{equation*}
Setting
\begin{align*}
    %\label{trtripleb}
    Q^{-1}\state &=\vcentcolon \begin{bmatrix} \state_1\\ \state_2\end{bmatrix}, &
    Q^{-1}\lambda &=\vcentcolon  \begin{bmatrix} \lambda_1\\ \lambda_2\end{bmatrix}, &
    Q^{-1}\state_0 &=\vcentcolon \begin{bmatrix} {\state}_{1,0}\\{\state}_{2,0}\end{bmatrix} 
\end{align*}
and reordering equations and unknowns we obtain the transformed boundary value problem
\begin{equation}\label{trfoptcon}
\begin{bmatrix} \phantom{-}0  & I & 0 & 0 & 0\\
           -I & 0 & 0 & 0 & 0\\
           \phantom{-}0 & 0 & 0 & 0 & 0\\
           \phantom{-}0  & 0 & 0 & 0 & 0\\
           \phantom{-}0 & 0 & 0 & 0 & 0\end{bmatrix} \begin{bmatrix} \dot{\lambda}_1\\
                                         \dot{\state}_1\\
                                         \dot{\lambda}_2\\
                                         \dot{\state}_2\\
                                         \dot{\inpVar}\end{bmatrix} =
   \begin{bmatrix} 0 &J & 0 & 0 & B_{1}\\
            J^T & W_{11} & 0 & W_{12} & S_{1}\\
            0 & 0 & 0 & I & B_{2}\\
            0 &W_{12}^T & I & W_{22} & S_{2}\\
            B_1^T & S_1^T & B_2^T & S_2^T  & W_{\mathrm{\inpVar}}\end{bmatrix}
   \begin{bmatrix} \lambda_1\\
            \state_1\\
            \lambda_2\\
            \state_2\\
            \inpVar\end{bmatrix},
\end{equation}
with boundary conditions
\begin{equation}
    \label{redtwoptbca}
    \state_1({t_0})={\state}_{1,0},\qquad 
\lim_{t\to \infty}\lambda_1({t})=0.
\end{equation}
Solving the third and fourth equation in \eqref{trfoptcon} gives
\begin{displaymath}
    \state_2= -B_2 \inpVar \qquad\text{and}\qquad  
    \lambda_2=-W_{12}^T \state_1 -W_{22} \state_2 -S_2 \inpVar.
\end{displaymath}
 Inserting these in the other equations gives the reduced optimality system
\begin{equation}
    \label{redtrfoptcon}
    \begin{bmatrix} \phantom{-}0  & I & 0 \\
           -I  & 0 &0 \\ \phantom{-}0&0& 0 \end{bmatrix} \begin{bmatrix} \dot{\lambda}_1\\
                                         \dot{\state}_1\\
                                         \dot{\inpVar}\end{bmatrix} =
   \begin{bmatrix} 0 &J  & B_{1}\\
            J^T & W_{11} & \tilde{S}_{1}\\
            B_1^T & \tilde{S}_1^T  & \tilde{W}_{\mathrm{\inpVar}} \end{bmatrix}
   \begin{bmatrix} \lambda_1\\
            \state_1\\
            \inpVar\end{bmatrix},
\end{equation}
with $\tilde S_1=S_{1}-W_{12} B_2$, $\tilde{W}_{\mathrm{\inpVar}} = W_{\mathrm{\inpVar}}-S_2^TB_2-B_2^TS_2 +B_2^TW_{22} B_2$,
and boundary conditions \eqref{redtwoptbca}. This is the classical optimality condition associated with the \ODE constraint $\dot{\state}_1=J\state_1 +B_1 \inpVar$ and the cost matrix
\begin{displaymath}
\tilde{\mathcal{W}}=\begin{bmatrix} W_{11} & \tilde{S}_1 \\ \tilde{S}_{1}^T & \tilde{W}_{\inpVar} \end{bmatrix},
\end{displaymath}
for which the standard theory for optimal control with \ODEs constraints can be applied, see, \eg, \cite{Meh91}.   

With an ansatz $\lambda_1 =X\state_1$ with $X=X^T$ the optimal control takes the form of a state feedback
\begin{displaymath}
    \inpVar = F\state_1=-\tilde{W}_{\mathrm{\inpVar}}^{-1} (B_1^TX+\tilde{S}_1^T)\state_1.
\end{displaymath}
Inserting $\lambda_1=X \state_1$ in~\eqref{redtrfoptcon} we obtain the system
\begin{align*}
    \dot{\state}_1 &=(J-B_1 \tilde{W}_{\mathrm{\inpVar}}^{-1}(\tilde{S}_1+B_1^T X))\state_1,\\
-X \dot{\state}_1&= ((J-B_1 \tilde{W}_{\mathrm{\inpVar}}^{-1}\tilde{S}_1^T)^TX +W_{11}- \tilde {S}_1 \tilde{W}_{\mathrm{\inpVar}}^{-1} \tilde{S}_1^T)\state_1,\\
 0 &= \lim_{t\to \infty}X\state_1(t)
\end{align*}
A sufficient condition for this system to have a solution is, see \cite{Meh91}, that we can find a positive semi-definite solution $X=X^T$ to the algebraic Riccati equation
\begin{displaymath}
    0 =W_{11}+XJ+J^TX -(B_1^TX+\tilde{S}_1^T)^T \tilde {W}_{\mathrm{\inpVar}}^{-1}(B_1^TX + \tilde{S}_1^T). 
\end{displaymath}
If $\tilde{W}_{\mathrm{\inpVar}}$ is not invertible then there are further restrictions on the boundary conditions that may defer the solvability of the boundary value problem, see \cite{Meh91} for details and the forthcoming \Cref{sec:OCPH}.

%%%%%%%%%%%%%%%%%%%%%%%%%%%%%%%%%%%%%%%%%%%%%%%%%
\subsection{Passivity}\label{sec:passive}
Another important property of control systems  is the concept of \emph{passivity}. Let us first introduce a passivity  definition for  general \DAE control systems of the form \eqref{eqn:descriptorSystem} with a state space $\mathcal{Z}$, input space $\mathcal{U}$, and output space $\mathcal{Y}$. See \cite{ByrIW91a} for the definition for \ODE systems. 

To introduce this definition we consider a positive definite and quadratic \emph{storage function} 
$\hamiltonian\colon \mathcal{Z}\to \R$
as well as a \emph{supply function}
$\mathcal{S}\colon \mathcal{Y}\times \mathcal{U}\to \R$
satisfying 
\begin{align*}
    \mathcal{S}(0,\inpVar) &=0\ \text{ for all } \inpVar\in \mathcal{U},\\
    \mathcal{S}(\outVar,0) &=0\ \text{ for all } \outVar\in \mathcal{Y}.
\end{align*}
\begin{definition}\label{def:passive}
An autonomous \DAE of the from~\eqref{eqn:descriptorSystem} is called \emph{(strictly) passive} with respect to the storage function $\hamiltonian(\state)$ and the supply function $\mathcal{S}(\outVar,\inpVar)$ if there exists a positive semi-definite (positive definite) function 
$\Phi\colon \mathcal{X} \to \R$, such that for any $\inpVar\in \mathcal{U}$ and for any $t_0< t_1$ the equation 
\begin{equation}
    \label{eqn:baleq}
    \hamiltonian(\state(t_1)) - \hamiltonian(\state(t_0))=
    \int_{t_0}^{t_1} \mathcal{S}(\outVar(s),\inpVar(s)) - \Phi(\state(s))\,\mathrm{d}s
\end{equation}
holds for all $(\outVar,\inpVar)\in \mathcal{Y}\times\mathcal{U}$.
\end{definition}
Equation \eqref{eqn:baleq} is called \emph{storage energy balance equation} and  directly implies that the \emph{dissipation inequality}
\begin{equation*}
    %\label{eqn:dissineq}
    \hamiltonian(\state(t_1)) - \hamiltonian(\state(t_0))\leq
    \int_{t_0}^{t_1} \mathcal{S}(\outVar(s),\inpVar(s))\,\mathrm{d}s
\end{equation*}
holds for all $(\outVar,\inpVar)\in \mathcal{Y}\times\mathcal{U}$.

For \LTI \ODE systems of the form 
\begin{align*}
    \dot{\state} &= A\state +B\inpVar, \\
    \outVar & = C\state+ D\inpVar,
\end{align*}
passivity can be characterized, see \cite{Wil71}, via the existence of a positive definite solution $X=X^T$ of a linear matrix inequality, the \emph{Kalman--Yakubovich--Popov inequality}
\begin{equation} \label{KYP-LMI}
W(X) \vcentcolon= \begin{bmatrix}
-X A - A^TX & C^T - X B \\
C- B^TX & D+D^T
\end{bmatrix}\geq 0.
\end{equation}
For strict passivity this inequality has to be strict.
Note that \eqref{KYP-LMI} generalizes  the Lyapunov inequality~\eqref{lya-ineq}, which is just the leading block, and hence (strict) passivity directly implies (asymptotic) stability. 

The relationship between passivity and the linear matrix inequality~\eqref{KYP-LMI} has been extended to  \LTI \DAE
systems in \cite{ReiRV15,ReiV15}.

Since passive systems are closely related to port-Hamiltonian systems, we will come back to this topic in \Cref{sec:generalDAE}.

%%%%%%%%%%%%%%%%%%%%%%%%%%%%%%%%%%%%%%%%%%%%%%%%%
\subsection{Optimal control}
\label{sec:ocpdae}

An important task in control theory is the solution of optimal control problems that minimize a cost functional subject to an \ODE or \DAE system. The optimal control theory for general nonlinear \DAE systems was presented in \cite{KunM08}. In this section we recall these general results. 

Consider the optimal control problem to minimize a cost functional
\begin{equation*}
    %\label{DaeOC}
    \mathcal{J}(\state,\inpVar) = \mathcal{M}(t_{\mathrm{f}})+\int_{t_0}^{t_{\mathrm{f}}}
\mathcal{K}(t,\state(t),\inpVar(t)) \dt
\end{equation*}
subject to a constraint given by an initial value problem
associated with a nonlinear \DAE system 
\begin{equation*}
    %\label{DaeConstr}
    F(t,\state,\inpVar,\dot{\state})=0, \quad \state(t_0) = \state_0.
\end{equation*}
We can rewrite this problem in the behavior representation, see the discussion in \Cref{sec:DAEtheory}, with $\xi=\begin{bmatrix} \state^T & \inpVar^T \end{bmatrix}^T$, and then study the optimization problem
\begin{equation}
    \label{DaeOCbeh}
    \mathcal{J}(\xi) = \mathcal{M}(\xi(t_{\mathrm{f}})) + \int_{t_0}^{t_{\mathrm{f}}}
\mathcal{K}(t,\xi(t)) \dt={\min !}
\end{equation}
subject to the constraint
\begin{equation}
    \label{DaeConstrbeh}
    F(t,\xi,\dot{\xi})=0,\qquad \begin{bmatrix} I_n & 0 \end{bmatrix} \xi(t_0) = \state_0. %\begin{bmatrix} I_n & 0 \end{bmatrix} \xi_0.
\end{equation}
If \Cref{hyp:nonLin:nonRegular} holds, then for this system we (locally via the implicit function theorem) have a strangeness-free reformulation, cf.~\cite{KunM06}, as
\begin{equation}
    \label{newnlDaeConstr}
    \begin{aligned}
        \dot{\state}_1 &=\mathcal{L}(t,\state_1,\inpVar), \qquad \state_1(t_0)=\state_{1,0},\\
        \state_2 &= \mathcal{R}(t,\state_1,\inpVar),
    \end{aligned}
\end{equation}
and the associated cost function reads
\begin{equation}
    \label{newnlDaeOC}
    \mathcal{J}(\state_1,\state_2,\inpVar)= \mathcal{M}(\state_1(t_{\mathrm{f}}),\state_2(t_{\mathrm{f}}))+
\int_{t_0}^{t_{\mathrm{f}}}
\mathcal{K}(t,\state_1,\state_2,\inpVar) \dt.
\end{equation}
For this formulation, the  necessary optimality conditions in the space 
\begin{displaymath}
    \mathbb{W} \vcentcolon=
    \mathcal{C}^1(\timeInt,\R^{d})\times
    \mathcal{C}(\timeInt,\R^{a})\times
    \mathcal{C}(\timeInt,\R^{m}),
\end{displaymath}
are presented in the following theorem. We refer to \cite{KunM08} for the original presentation and the proof.
\begin{theorem}
\label{thnl}
Let $\xi$ be a local solution of \eqref{DaeOCbeh} subject to
\eqref{DaeConstrbeh} in the sense that $(\state_1,\state_2,\inpVar)\in\mathbb{W}$ is a local solution of~\eqref{newnlDaeOC} subject to \eqref{newnlDaeConstr}. Then there exist
unique Lagrange multipliers $(\lambda_1,\lambda_2,\gamma)\in\mathbb{W}$
such that $(\state_1,\state_2,\inpVar,\lambda_1,\lambda_2,\gamma)$ solves the
boundary value problem  
\begin{align*}
\label{nlnecoptconblock}
    \dot{\state}_1 &= \mathcal{L}(t,\state_1,\inpVar),\ \state_1(t_0)=\state_{1,0},\\
    \state_2 &= \mathcal{R}(t,\state_1,\inpVar),\\
    \dot{\lambda}_1 &= \tfrac{\partial}{\partial {\state_1}}\mathcal{K}(t,\state_1,\state_2,\inpVar)^T-
\tfrac{\partial}{\partial {\state_1}}\mathcal{L}(t,\state_1,\state_2,\inpVar)^T\lambda_1 - \tfrac{\partial}{\partial {\state_1}}\mathcal{R}_{\state_1}(t,\state_1,\inpVar)^T
\lambda_1,\\
&\qquad \qquad \qquad \qquad \qquad \qquad
\lambda_1(t_{\mathrm{f}})=-\tfrac{\partial}{\partial {\state_1}}{\mathcal M} (\state_1(t_{\mathrm{f}}),\state_2(t_{\mathrm{f}}))^T\\
    0 &= \tfrac{\partial}{\partial {\state_2}}\mathcal{K}(t,\state_1,\state_2,\inpVar)^T+\lambda_2,\\
    0 &= \tfrac{\partial}{\partial {\inpVar}}\mathcal{K}(t,\state_1,\state_2,\inpVar)^T-\tfrac{\partial}{\partial {\inpVar}}\mathcal{L}(t,\state_1,\inpVar)^T\lambda_1-
\tfrac{\partial}{\partial {\inpVar}}\mathcal{R}(t,\state_1,\inpVar)^T\lambda_2,\\
    \gamma &=\lambda_1(t_0).
\end{align*}
\end{theorem}
\Cref{thnl} is a local result based on the implicit function theorem that has to be modified to turn it into a computationally feasible procedure, cf.~\cite{KunM08}, and which can be substantially strengthened for the minimization of quadratic cost functionals 
\begin{equation}
    \label{qcostfunct}
    \mathcal{J}(\state,\inpVar) = \tfrac{1}{2} \state(t_{\mathrm{f}})^T M \state(t_{\mathrm{f}})+ \tfrac{1}{2} \int_{t_0}^{t_{\mathrm{f}}}(\state^T W_{\mathrm{\state}} \state + 2 \state^T S \inpVar +\inpVar^T W_{\mathrm{\inpVar}} \inpVar)\dt,
\end{equation}
with matrix functions $ W_{\mathrm{\state}}=W_{\mathrm{\state}}^T\in  \mathcal{C}(\timeInt,\R^{\stateDim,\stateDim})$,
$W_{\mathrm{\inpVar}}=W_{\mathrm{\inpVar}}^T\in \mathcal{C}(\timeInt,\R^{m,m})$,
$S\in\mathcal{C}(\timeInt,\R^{\stateDim,m})$, and   $M=M^T\in\R^{\stateDim,\stateDim}$, subject to \LTV \DAE constraints  
\begin{equation}
    \label{lindae}
    E \dot{\state} = A \state+ B \inpVar + f,\qquad \state(t_0)=\state_0,
\end{equation}
with $E\in \mathcal{C}(\timeInt,\R^{\stateDim,\stateDim})$,
$A \in \mathcal{C}(\timeInt,\R^{\stateDim,\stateDim})$,
$ B\in \mathcal{C}(\timeInt,\R^{\stateDim,m})$, $f\in \mathcal{C}(\timeInt,\R^{\stateDim})$,
$\state_0\in\R^{\stateDim}$, 
$\inpVar\in\mathbb{U}\vcentcolon=\mathcal{C}(\timeInt,\R^{m})$, and $f\in  \mathcal{C}(\timeInt,\R^{\stateDim})$. 
Using the property that for the Moore-Penrose inverse $E^\dagger$ of $E$, we have 
\begin{displaymath}
    E \dot{\state} = E E^\dagger E \dot{\state} = E \ddt (E^\dagger E \state)-E\ddt (E^\dagger E) \state,
\end{displaymath}
we interpret \eqref{lindae} as
\begin{equation*}
    %\label{str-free}
    E \ddt (E^\dagger E \state)=(A+E\ddt (E^\dagger E)) \state +B \inpVar+f,\qquad
    (E^\dagger E \state)(t_0)=\state_0.
\end{equation*}
This allows the particular solution space, see \cite{KunM96a},
\begin{equation*}
    %\label{solspace}
    \mathbb{Z}\vcentcolon= \mathcal{C}^1_{E^\dagger E }(\timeInt,\R^{\stateDim}) \vcentcolon=\left \{ \state\in  \mathcal{C}(\timeInt,\R^{\stateDim}) \mid
E^\dagger E \state \in \mathcal{C}^1(\timeInt,\R^{\stateDim}) \right \},
\end{equation*}
which takes care that the differentiability of the state variable is only required in this restricted space. Note that with this definition, we slightly extend our solution concept from \Cref{def:solutionConcept}.

Using this solution space, in \cite{KunM08} the following necessary optimality condition was derived.
\begin{theorem}
    \label{th:optsys}
    Consider the optimal control problem \eqref{qcostfunct} subject to \eqref{lindae} with a consistent initial condition. Suppose that \eqref{lindae} is strangeness-free as a behavior system and that the range of $M$ is contained in the $\cokernel$ of $E(t_{\mathrm{f}})$.
    If $(\state,\inpVar)\in \mathbb{Z}\times\mathbb{U}$ is a solution to this optimal control problem, then there exists a Lagrange multiplier function $\lambda \in \mathbb{Z}$, such that $(z,\lambda,u)$ satisfy the boundary value problem 
    \begin{subequations}
        \label{optbvp}
        \begin{align}
            E \ddt (E^\dagger E \state) &=(A+E\ddt (E^\dagger E)) \state +B\inpVar+f,\ 
            \label{optbvp:a} \\
            E^T \ddt (E E^\dagger \lambda) &= W_{\mathrm{\state}} \state +S\inpVar -(A+E E^\dagger\dot{E})^T \lambda, \  \label{optbvp:b}\\
            0 &= S^T\state + W_{\mathrm{\inpVar}} \inpVar -B^T \lambda, \label{optbvp:c}\\
            (E^\dagger E \state)(t_0) &=\state_0, \qquad
            (E E^\dagger  \lambda)(t_{\mathrm{f}}) =- E^\dagger(t_{\mathrm{f}})^TM\state(t_{\mathrm{f}}).  \label{optbvp:BC}
        \end{align}
\end{subequations}
\end{theorem}

\begin{remark}\label{rem:other optimization}
If we have an output equation, then we can also consider the cost functional \eqref{qcostfunct} with the state replaced by the output. This problem can be treated analogously by inserting the output equation in the cost functional and renaming the coefficients. 
The same approach with a modified cost functional can also be used for the case that we want to optimally drive the solution to a reference function $\tilde{\state}$. 
\end{remark}

It \cite{KunM11}, see also \cite{Bac06,KurM04}, it was demonstrated that it is in general not possible to drop the assumptions of \Cref{th:optsys}
and instead consider the formal optimality system
\begin{equation}
    \label{formoptsys}
    \begin{aligned}
        E \dot{\state} &= A\state+B\inpVar+f, & \state(t_0)&=\state_0,\\
        \ddt (E^T \lambda) &= W_{\mathrm{\state}} \state+S\inpVar-A^T\lambda, & (E^T\lambda)(t_{\mathrm{f}}) &=-M \state(t_{\mathrm{f}}),\\
        0 &= S^T \state+ W_{\mathrm{\inpVar}} \inpVar- B^T \lambda.
    \end{aligned}
\end{equation}
However, if this formal optimality system has a unique solution, then the state~$\state$ and the control~$\inpVar$ are correct but the optimal Lagrange multiplier may be different. In this case, under some further assumptions, also a sufficient condition has been shown in \cite{KunM11} generalizing results from \cite{Bac06}. 

\begin{theorem}
    \label{thm_sufcon}
    Consider the optimal control problem \eqref{qcostfunct} subject to \eqref{lindae} with a consistent initial condition and suppose that in the cost functional  \eqref{qcostfunct} we have that the matrix functions $\begin{smallbmatrix} W_{\state} & S\\ S^T & W_{\mathrm{\inpVar}}\end{smallbmatrix}$ and $M$ are (pointwise) positive semi-definite. If $(\state^\star,\inpVar^\star,\lambda)$ satisfies the formal optimality system \eqref{formoptsys}, then for any $(\state,\inpVar)$ satisfying \eqref{lindae} we have
    \begin{displaymath}
        \mathcal{J}(\state,\inpVar)\geq \mathcal{J}(\state^\star,\inpVar^\star).
    \end{displaymath}
\end{theorem}

Numerically, the solution of the boundary value problem (which always is \DAE ) is a challenge, in particular for large-scale problems. For the case that $E=I$ and if $W_{\mathrm{\inpVar}}$ is positive definite, which means that the optimality system is strangeness-free, then a classical approach that is successfully employed in many applications is to resolve \eqref{optbvp:c} for $\inpVar$, and then to decouple the state equation \eqref{optbvp:a} and the adjoint equation  \eqref{optbvp:b} via the solution of a Riccati differential equation. 

If some further conditions hold, then the Riccati approach can also be carried out for \DAE systems, see \cite{KunM11}.
If the constraint system with $\inpVar=0$ is regular, strangeness-free, and $E$ has constant rank, then, using \Cref{thm:smoothFullRankDecomposition}, there
exist pointwise orthogonal $ P\in\mathcal{C}(\timeInt,\R^{\stateDim,\stateDim})$ and $Q\in  \mathcal{C}^1(\timeInt,\R^{\stateDim,\stateDim})$
such that
\begin{equation}\label{nf}
    \begin{gathered}
    \tilde{E} = PEQ = \begin{bmatrix} E_{11} & 0 \\ 0 & 0 \end{bmatrix},\quad
    \tilde{A} = PAQ -PE\dot{Q} = \begin{bmatrix} A_{11} & A_{12} \\ A_{21} & A_{22} \end{bmatrix},\\
    \tilde{B} = PB =\begin{bmatrix} B_{1}  \\ B_{2} \end{bmatrix},\ 
    \tilde{W}_{\mathrm{\state}} = Q^T W_{\mathrm{\state}} Q = \begin{bmatrix} W_{11} & W_{12} \\ W_{21} & W_{22} \end{bmatrix},\
    \tilde{f} = Pf = \begin{bmatrix} f_{1} \\ f_{2} \end{bmatrix},\\
    \tilde{S} = Q^T S = \begin{bmatrix} S_{1}  \\ S_{2} \end{bmatrix},\ 
    \state = Q\tilde{\state} = \begin{bmatrix} \state_1 \\ \state_2 \end{bmatrix}, \
    \state_0 = Q\tilde{\state}_0 = \begin{bmatrix} {\state}_{0,1} \\ {\state}_{0,2} \end{bmatrix},
    \end{gathered}
\end{equation}
with $E_{11}\in \mathcal{C}(\timeInt,\R^{d,d})$ and
$A_{22}\in \mathcal{C}(\timeInt,\R^{a,a})$ pointwise nonsingular.
Forming the formal optimality system associated with this transformed system and rearranging the equations, the following theorem is proved in \cite{KunM11}.
\begin{theorem}
    \label{thindex}
    The \DAE in \eqref{optbvp} is regular and strangeness-free if and only if
    \begin{equation*}
        %\label{hatr}
        \hat{W}_\inpVar = \begin{bmatrix} 0 & A_{22} & B_2  \\
        A_{22}^T & W_{22}& S_2\\
        B_2^T & S_2^T & W_{\mathrm{\inpVar}} \end{bmatrix}
    \end{equation*}
    is pointwise nonsingular, where we used the notation of~\eqref{nf}.
\end{theorem}

If $\hat{W}_\inpVar$ is pointwise nonsingular, then
\begin{equation*}
    %\label{algpart}
    \begin{bmatrix} -\lambda_2\\ z_2 \\  u  \end{bmatrix} = 
 -\hat{W}_\inpVar^{-1} \left ( \begin{bmatrix} 0 & A_{21}\\
A_{12}^T & W_{21} \\
B_1^T & S_1^T\end{bmatrix}
\begin{bmatrix} -\lambda_1 \\ z_1 \end{bmatrix} +
 \begin{bmatrix} f_2 \\ 0 \\ 0 \end{bmatrix}\right ).
\end{equation*}
The remaining equations can be written as
\begin{multline*}
    %\label{dynpart}
    \begin{bmatrix} E_{11} \dot{\state}_1 \\ \ddt((-E_{11}^T) (-\lambda_1))
 \end{bmatrix} = \begin{bmatrix} 0 & A_{11} \\ A_{11}^T & W_{11}\end{bmatrix}
 \begin{bmatrix} -\lambda_1 \\ \state_1  \end{bmatrix}
    +\begin{bmatrix} 0 & A_{12} & B_1\\
A_{21}^T & W_{21}^T & S_1\end{bmatrix}
\begin{bmatrix} -\lambda_2\\ \state_2 \\  \inpVar  \end{bmatrix} +
\begin{bmatrix} f_1 \\ 0 \end{bmatrix}.
\end{multline*}
Defining
\begin{align*}
    F_1 &\vcentcolon= E_{11}^{-1} \left ( A_{11}-\begin{bmatrix} 0 & A_{12} & B_1\end{bmatrix} \hat{W}_\inpVar^{-1} \begin{bmatrix} A_{21}^T & W_{21}^T & S_1 \end{bmatrix}^T\right ),\\
    G_1 &\vcentcolon= E_{11}^{-1} \begin{bmatrix} 0 & A_{12} & B_1\end{bmatrix} \hat{W}_\inpVar^{-1}
    \begin{bmatrix} 0 & A_{12} & B_1 \end{bmatrix}^T E_{11}^{-T},\\
    H_1 &\vcentcolon= W_{11}- \begin{bmatrix} A_{21}^T & W_{21}^T & S_1 \end{bmatrix} \hat{W}_\inpVar^{-1} \begin{bmatrix} A_{21}^T & W_{21}^T & S_1 \end{bmatrix}^T,\\
    g_1 &\vcentcolon= E_{11}^{-1}\left( f_1- \begin{bmatrix}
0 & A_{12} & B_1 \end{bmatrix} \hat{W}_\inpVar^{-1} \begin{bmatrix} f_2^T & 0 & 0 \end{bmatrix}^T\right),\\
    h_1 &\vcentcolon= -\begin{bmatrix} A_{21}^T & W_{21}^T & S_1 \end{bmatrix} \hat{W}_\inpVar^{-1} \begin{bmatrix} f_2^T & 0 & 0 \end{bmatrix}^T,
\end{align*}
we obtain the boundary value problem 
\begin{subequations}
    \label{hambvp}
    \begin{align}
    \dot{\state}_1 &= F_1 \state_1 + G_1 (E_{11}^T \lambda_1) + g_1, &
    \state_1(t_0) &= \state_{0,1}, \\
    \ddt (E_{11}^T \lambda_1) &= H_1 \state_1 - F_1^T (E_{11}^T \lambda_1) + h_1, &
    (E_{11}^T \lambda_1)(t_{\mathrm{f}}) &= -M_{11}\state_1(t_{\mathrm{f}}).
    \end{align}
\end{subequations}
Making the ansatz $E_{11}^T \lambda_1 = X_{11}\state_1+v_1$,
one can solve the two initial value problems
\begin{equation*}
    %\label{Riceq}
    \dot{X}_{11} + X_{11}F_1 +F_1^T X_{11}+ X_{11}G_1X_{11}-H_1=0,
    \quad X_{11}(t_{\mathrm{f}})=-M_{11},
\end{equation*}
and
\begin{equation*}
    %\label{inhomo}
    \dot{v}_1 + X_{11} G_1v_1 +F_1^Tv_1+X_{11}g_1-h_1=0, \quad v_1(t_{\mathrm{f}})=0,
\end{equation*}
to obtain $X_{11}$ and $v_1$ and to decouple the solution of \eqref{hambvp}. In \cite{KunM11}, a Riccati approach is also obtained directly for the original optimality system \eqref{optbvp} by the modified ansatz
\begin{equation}\label{genricans}
    \begin{aligned}
        \lambda &=X Ez+v= X E E^\dagger E\state +v,\\
        \ddt (EE^\dagger\lambda) &=\ddt (EE^\dagger X) E\state+ (EE^\dagger X)\dot{E} E^\dagger E \state +(EE^\dagger X)E \ddt (E^\dagger E x)+\ddt (E^\dagger E v),
    \end{aligned}
\end{equation}
where
\begin{equation*}
    %\label{ricspace}
    X\in\mathcal{C}^1_{EE^\dagger}(\timeInt, \R^{\stateDim,\stateDim}),\quad
    v\in \mathcal{C}^1_{EE^\dagger} (\timeInt, \R^{\stateDim})
\end{equation*}
to fit to the solution spaces for~$\state$ and~$\lambda$. If $W_{\mathrm{\inpVar}}$ is invertible, then introducing the notation
\begin{equation*}
    %\label{newHamdef}
    F \vcentcolon= A- B  W_{\mathrm{\inpVar}}^{-1} S^T, \quad
    G \vcentcolon= B W_{\mathrm{\inpVar}}^{-1} B^T,\quad 
    H \vcentcolon= W- S W_{\mathrm{\inpVar}}^{-1} S^T,
\end{equation*}
yields two initial value problems for the Riccati \DAE
\begin{equation}
    \label{genRiceq}
    \begin{aligned}
    \ddt (E^T X E) + E^T XF+ F^T XE + E^T X G X E -H &=0, \\
    (E^T X E)(t_{\mathrm{f}}) &= -M,
    \end{aligned}
\end{equation}
and 
\begin{equation*}
    %\label{geninhom}
    \begin{aligned}
	\ddt ( E^T v) +E^T X G v+ F^T v +E^TX f&=0,\\
	(E^Tv)(t_{\mathrm{f}})&=0.
	\end{aligned}
\end{equation*}
For this to be solvable, we must have  $M=E(t_{\mathrm{f}})^T \tilde M E( t_{\mathrm{f}})$
with suitable~$\tilde M$ and $H=E^T \tilde H E$ with suitable~$\tilde H$. 

The major advantage of the ansatz via Riccati equations is that the resulting control can be directly expressed as a feedback control, \eg using~\eqref{genricans}, we get 
\begin{equation*}
    %\label{eqn:feedback:Riccati}
    \inpVar= W_{\mathrm{\inpVar}}^{-1}(B^T \lambda- S^T\state)= W_{\mathrm{\inpVar}}^{-1}(B^T X E E^\dagger E-S^T) \state +v.
\end{equation*}
A similar result is also obtained if the cost functional is formulated in terms of the output and using an output feedback. 

Note that in the \LTI \DAE case further results have been obtained, that allow the use of efficient numerical techniques for the computation of the solutions to \eqref{genRiceq} via eigenvalue methods, see \cite{Meh91}.

In this section we have recalled several properties for general descriptor systems. In the following section we study the  special class of port-Hamiltonian descriptor systems and show that the structure of the systems ensures many improved  properties.
%
%%%%%%%%%%%%%%%%%%%%%%%%%%%%%%%%%%%%%%%%%%%%%%%%%%%%%%
\section{Port-Hamiltonian descriptor systems}
\label{sec:phdae}
To fulfill as many points on our wish-list as possible, we will not use general descriptor systems, but energy-based modeling within the class of (dissipative) port-Hamiltonian (\pH) systems and their generalization to descriptor systems. 

\subsection{Nonlinear (dissipative) port-Hamiltonian descriptor systems}
We start our exposition  by introducing the general model class of (dissipative) \pH descriptor systems, or \pH differential-algebraic equation (\pHDAE) systems, introduced in \cite{MehM19}.

\begin{definition}[\pH descriptor system, \pHDAE]
	\label{def:pHDAE}
	Consider a time interval $\timeInt$, a state space $\stateSpace\subseteq\R^\stateDim$, and an extended space $\extSpace \vcentcolon= \timeInt\times\stateSpace$.  Then a \emph{(dissipative) port-Hamiltonian descriptor system} (\pHDAE) is a descriptor system of the form
	\begin{subequations}
		\label{eqn:pHDAE}
		\begin{align}
			E(t,\state)\dot{\state} + r(t,\state) &= (J(t,\state)-R(t,\state))\eta(t, \state) + (G(t,\state) - P(t,\state))\inpVar,\\
			\outVar &= (G(t,\state) + P(t,\state))^T \eta(t, \state) + (S(t,\state) - N(t,\state))\inpVar,
		\end{align}
	\end{subequations}
	with state $\state\colon \timeInt \to\stateSpace$,  input $u\colon \timeInt \to \R^m$,  output $y\colon \timeInt \to\R^m$, where
	\begin{align*}
		r,\eta&\in\calC(\extSpace,\R^{\ell}),
		& E&\in\calC(\extSpace,\R^{\ell, \stateDim}),
		& J,R&\in \calC(\extSpace,\R^{\ell,\ell})\\
		G,P&\in\calC(\extSpace,\R^{\ell, m}),
		& S,N&\in\calC(\extSpace,\R^{m, m}),
	\end{align*}
	 and an associated function $\hamiltonian\in\calC^1(\extSpace,\R)$, called the \emph{Hamiltonian} of~\eqref{eqn:pHDAE}.  Furthermore, the following properties must hold:
	 \begin{enumerate}
	 	\item[(i)] The matrix functions
	 		\begin{subequations}
	 		\label{eqn:pHDAE:prop1}
	 		\begin{align}
	 			\structureMatrix &\vcentcolon= \begin{bmatrix}
	 				J & G\\
	 				-G^T & N
	 			\end{bmatrix}\in\calC(\extSpace,\R^{(\ell+m),(\ell+m)}),\\
	 			\dissipationMatrix &\vcentcolon= \begin{bmatrix}
	 				R & P\\
	 				P^T & S
	 			\end{bmatrix}\in\calC(\extSpace,\R^{(\ell+m),(\ell+m)}),
	 		\end{align}
	 		\end{subequations}
	 		called the \emph{structure matrix} and \emph{dissipation matrix}, respectively,
	 		satisfy $\structureMatrix = -\structureMatrix^T$ and $\dissipationMatrix = \dissipationMatrix^T \geq 0$ in $\mathcal S$.
	 	\item[(ii)] The Hamiltonian satisfies
	 		\begin{equation}
	 			\label{eqn:energyCondition}
%	 			\tfrac{\partial}{\partial \state}
\tfrac{\partial}{\partial {\state}}\hamiltonian(t,\state) = E^T(t,\state)\eta(t,\state) \quad\text{and}\quad \tfrac{\partial}{\partial t} \hamiltonian(t,\state) = \eta^T(t,\state) r(t,\state)
	 		\end{equation}
	 		in $\mathcal S$ along any solution of~\eqref{eqn:pHDAE}.
	 \end{enumerate}
	 If the \pH descriptor system has no inputs and outputs, i.e., if $G,P\equiv 0$ and the output equation is omitted, then we refer to \eqref{eqn:pHDAE} as \emph{(dissipative) Hamiltonian differential-algebraic equation (\dHDAE)}.
\end{definition}		

Note that in the literature and also in this survey the adjective \emph{dissipative} is typically omitted, and we follow this tradition, even if the system has a  dissipative part $R$.

\begin{remark}\label{rem:autono}
In many applications the coefficients of \pHDAE systems are not explicitly depending on time. This is not really a restriction, since we can always make a system of the form \eqref{eqn:pHDAE}
autonomous by introducing the combined state
$\widehat{z} \vcentcolon= [\state^T, t]^T$ and  reformulating the \pHDAE~\eqref{eqn:pHDAE} as
\begin{subequations}
\label{eqn:pHDAE:autonomous}
\begin{align*}
			\begin{bmatrix}
				E(\widehat{\state}) & r(\widehat{\state})\\
				0 & 1
			\end{bmatrix}\dot{\widehat{\state}} &= \begin{bmatrix}
				J(\widehat{\state})-R(\widehat{\state}) & 0\\
				0 & 0
			\end{bmatrix}\begin{bmatrix}
				\eta(\widehat{\state})\\
				0\\
			\end{bmatrix} + \begin{bmatrix}
				G(\widehat{\state}) - P(\widehat{\state}) & 0\\
				0 & 1
			\end{bmatrix}\begin{bmatrix}
				\inpVar\\
				1
			\end{bmatrix},\\
			\begin{bmatrix}
				\outVar\\
				0
			\end{bmatrix} &= \begin{bmatrix}
				G(\widehat{\state}) + P(\widehat{\state}) & 0\\
				0 & 1
			\end{bmatrix}^T \begin{bmatrix}
				\eta(\widehat{\state})\\
				0
			\end{bmatrix} + \begin{bmatrix}
				S(\widehat{\state}) - N(\widehat{\state}) & 0\\
				0 & 0
		\end{bmatrix}\begin{bmatrix}
				\inpVar\\
				1
			\end{bmatrix},
\end{align*}
\end{subequations}
which is again a \pHDAE, since
\begin{displaymath}
	\tfrac{\partial}{\partial \widehat{\state}} \hamiltonian(\widehat{\state}) = \begin{bmatrix}
		\tfrac{\partial}{\partial {\state}} \hamiltonian(t,\state)\\
		\tfrac{\partial}{\partial t} \hamiltonian(t,\state)
	\end{bmatrix} = \begin{bmatrix}
		E(\widehat{\state}) & r(\widehat{\state})\\
		0 & 1
	\end{bmatrix}^T \begin{bmatrix}
		\eta(\widehat{\state})\\
		0
	\end{bmatrix}.
\end{displaymath}
\end{remark}

\begin{remark}\label{rem:alongsol}
	For many properties of \pHDAEs that we discuss later, it is sufficient to require the properties~\eqref{eqn:pHDAE:prop1} and~\eqref{eqn:energyCondition} in \Cref{def:pHDAE} to hold only along any solution of~\eqref{eqn:pHDAE}, thus further extending the model class. Nevertheless, to simplify the presentation, we work with the definition as presented here.
\end{remark}

\begin{remark}\label{rem:H=I}
	If $E = I_{\stateDim}$ is the identity matrix, $r\equiv 0$, and the coefficients do not explicitly depend on time $t$, then \Cref{def:pHDAE} reduces to the well-known classical representation for \ODE \pH systems, called \pHODEs in the following, as for instance presented in \cite{SchJ14}.
\end{remark}

\begin{remark}\label{rem:dirac}
	In the literature, \pH systems are often described via a Dirac structure; see \cite{SchJ14} and this approach has also been extended to descriptor systems, see  \cite{MehM19,Sch13,SchM18}, and the forthcoming \Cref{sec:dirac}. In this survey, however, we mostly focus on the dynamical systems point of view, which is prevalent in the simulation and control context.
\end{remark}

In many applications,  additional properties of the Hamiltonian such as convexity or non-negativity, may further strengthen the properties of \pHDAEs, see  \Cref{sec:properties}.
We thus make the following definition.

\begin{definition} \label{def:nonham}
	The Hamiltonian for a \pHDAE of the form~\eqref{eqn:pHDAE} is called \emph{non-negative}, if
	\begin{equation*}
		\hamiltonian(t,\state(t)) \geq 0\qquad\text{for all $(t,\state)\in\extSpace$ with $\state$ being a solution of~\eqref{eqn:pHDAE}}.
	\end{equation*}
\end{definition}

Although satisfied in many applications, it may seem artificial from a mathematical point of view to require the Hamiltonian to be non-negative. This is however not a restriction, since any Hamiltonian that is bounded from below can be recast as a non-negative Hamiltonian by adding its infimum along any behavior solution.
In the following we therefore always assume that the Hamiltonian is non-negative.

\begin{remark}\label{rem:colo}
The particular structure of the ports with equal dimensions and the described structure ensure that inputs and outputs are \emph{co-located} or \emph{power-conjugated}.
This enables for easy power-conserving interconnection of \pHDAE systems, see the forthcoming \Cref{sec:interconnect}. However, in many applications, one has specific quantities that one can observe and others that one can use for control, and these are not necessarily co-located or power-conjugated. To allow classical control techniques as well as interconnectability, one can extend the inputs and outputs to obtain a  power-conjugated formulation.
Even if these variables are not explicitly used, they typically have a physical meaning in the context of supplied energy. We will demonstrate this with examples later on, see \eg \Cref{sec:power,sec:gasnetwork}. 
\end{remark}

There are different generalizations of  \Cref{def:pHDAE} to infinite-dimensional systems, \eg one can formulate operator \pHDAE systems via semigroup theory, introduce formal Dirac structures, or follow a gradient flow approach. In the final section of this survey, we present an incomplete list of references discussing different aspects of infinite-dimensional \pHDAE systems.
In this survey, we focus mainly on the finite-dimensional case. We assume that a space-discretization via Galerkin projection is performed for an infinite-dimensional case. For infinite-dimensional examples, we mimic the finite-dimensional properties, which are then preserved under Galerkin projection (see the forthcoming \Cref{sec:invariance}). We refer to the examples in \Cref{sec:examples} for further details. 

\subsection{Linear pHDAE systems}
\label{sec:linearDAE} 
Important special subclasses of the general class of \pHDAE systems, are \LTV and \LTI \pHDAE systems. Such a general class with a quadratic Hamiltonian was introduced in \cite{BeaMXZ18}. 
\begin{definition}[Linear \pHDAE, quadr.~Hamiltonian]
	\label{def:pHDAE:linear}
	A linear time-varying descriptor system of the form
	\begin{subequations}
		\label{eqn:pHDAE:linear}
		\begin{align}
			E(t)\dot{\state} + E(t)K(t)\state &=(J(t)-R(t))Q(t)\state + (G(t)-P(t))\inpVar,\\
			\outVar &= (G(t)+P(t))^T Q(t)\state + (S(t)-N(t))\inpVar,
		\end{align}
	\end{subequations}
	with
	\begin{gather*}
		E,Q \in \calC^1(\timeInt,\R^{\ell,\stateDim}), \quad
		J,R,K \in \calC(\timeInt,\R^{\ell,\ell}), \quad
		 G,P\in \calC(\timeInt,\R^{\ell, m}),\\
		S = S^T\!, N=-N^T \in \calC(\timeInt,\R^{m, m})
	\end{gather*}
	is called a \emph{linear \pHDAE with quadratic Hamiltonian}
	\begin{equation}
	 		\label{eqn:Hamiltonian:linear}
	 		\hamiltonian \colon \timeInt\times \R^{\stateDim}\to \R,\qquad (t,\state) \mapsto \tfrac{1}{2}\state^T Q^T(t)E(t)\state,
	 \end{equation}
	 if the following properties are satisfied.
	\begin{enumerate}
	 	\item[(i)] The differential-operator
	 		\begin{equation}
	 			\calL \vcentcolon= Q^T E\ddt - (Q^T JQ - Q^T EK)\colon %\calD(\calL) \subseteq
\calC^1(\timeInt,\R^{\stateDim})\to \calC(\timeInt,\R^{\stateDim})\label{skewL}
	 		\end{equation}
	 		is \emph{skew-adjoint}, \ie we have $Q^T E\in \calC^1(\timeInt,\R^{\stateDim,\stateDim})$ and for all $t\in\timeInt$,
	 		\begin{align*}
	 			Q^T(t)E(t) &= E^T(t)Q(t),\qquad\text{and}\\
	 			\ddt \left(Q^T(t)E(t)\right) &= Q^T(t)\left[E(t)K(t) - J(t)Q(t)\right] + \left[E(t)K(t)-J(t)Q(t)\right]^T Q(t).
	 		\end{align*}
	 	\item[(ii)] The matrix function
	 		\begin{equation}\label{Wdef}
	 			\dissipationMatrix \vcentcolon= \begin{bmatrix} Q & 0\\
	 			0 & I_m\end{bmatrix}^T\begin{bmatrix}
	 				R & P\\
	 				P^T & S
	 			\end{bmatrix}\begin{bmatrix} Q & 0\\
	 			0 & I_m\end{bmatrix}\in \calC(\timeInt,\R^{(\stateDim+m),(\stateDim+m)})
	 		\end{equation}
	 		is positive semi-definite, \ie $W(t) = W^T(t) \geq 0$ for all $t\in\timeInt$.
	 \end{enumerate}
\end{definition}	
If the Hamiltonian is quadratic, and the coefficients are not depending explicitly on the state $z$, then \Cref{def:pHDAE,def:pHDAE:linear} are closely related. 
Starting from \Cref{def:pHDAE:linear} we may set
\begin{displaymath}
	r(t,\state) \vcentcolon= E(t)K(t)\state\qquad\text{and}\qquad \eta(t,\state) \vcentcolon= Q(t)\state.
\end{displaymath}
Using the skew-adjointness of $\calL$ in \Cref{def:pHDAE:linear}\,(i)
we then obtain
\begin{align*}
	\tfrac{\partial}{\partial {\state}} \hamiltonian(t,\state) = E^T(t)Q(t)\state = E^T \eta(t,\state)
\end{align*}
and
\begin{align*}
	\tfrac{\partial}{\partial t} \hamiltonian(t,\state) &= \state^T Q^T(t) E(t)K(t)\state - \tfrac{1}{2}\state^T Q^T(t)(J(t)+J^T(t))Q(t)\state\\
	&= \eta^T(t,\state)r(t,\state) - \tfrac{1}{2}\state^T Q^T(t)(J(t)+J^T(t))Q(t)\state.
\end{align*}
Thus, if additionally $J$ is skew-symmetric, \ie $J(t) = -J^T(t)$, then the Hamiltonian satisfies the requirements~\eqref{eqn:energyCondition} from \Cref{def:pHDAE}. Similarly, we notice that the positive semi-definiteness of the dissipation matrix function in \Cref{def:pHDAE:linear} is slightly more general than its counterpart in \Cref{def:pHDAE}.  On the other hand, the requirement for the matrix functions $E$ and $Q$ to be continuously differentiable, is a sufficient condition to obtain a continuously differentiable Hamiltonian as required in \Cref{def:pHDAE}.

Another important special class is that of (\LTI) \pHDAE systems with quadratic Hamiltonian,  which is commonly used in closed-loop and  data-based control applications as well as linear stability analysis. 
\begin{definition}[\LTI \pHDAE, quadr. Hamiltonian]
	\label{def:pHDAE:LTI}
	A descriptor system of the form
	\begin{subequations}
		\label{eqn:pHDAE:LTI}
		\begin{align}
			E\dot{\state} &= (J-R)Q\state + (G-P)\inpVar,\\
			\outVar &= (G+P)^T Q\state + (S-N)\inpVar,
		\end{align}
	\end{subequations}
	with matrices $E,Q\in\R^{\ell, \stateDim},J,R\in\R^{\ell, \ell}$, $G,P\in\R^{\ell, m}$, and $S,N\in\R^{m, m}$ is called  \emph{linear time invariant %(\LTI) 
	\pHDAE with (quadratic) Hamiltonian}
	\begin{equation}
		\label{eqn:Hamiltonian:LTI}
		\hamiltonian\colon\R^{\stateDim}\to\R,\qquad \state \mapsto \tfrac{1}{2} \state^T Q^T E\state,
	\end{equation}
	if the matrices
	\begin{align*}
		\structureMatrix &\vcentcolon= \begin{bmatrix}
			J & G\\
			-G^T & N
		\end{bmatrix}\in\R^{(\ell+m),(\ell+m)}\\ %\qquad\text{and}\\
		\dissipationMatrix &\vcentcolon= \begin{bmatrix}
			Q & 0\\
			0 & I_{\inputDim}
		\end{bmatrix}^T \begin{bmatrix}
			R & P\\
			P^T & S
		\end{bmatrix}\begin{bmatrix}
			Q & 0\\
			0 & I_{\inputDim}
		\end{bmatrix}\in\R^{(\stateDim+m),(\stateDim+m)}
	\end{align*}
	satisfy $\structureMatrix = -\structureMatrix^T$ and $\dissipationMatrix = \dissipationMatrix^T \geq 0$.
\end{definition}

Having introduced the general modeling concept of (dissipative) \pH descriptor systems, we now discuss two modeling simplifications, namely removing the $Q$ factor in linear \pHDAE systems and removing the feedthrough term $(S-N)u$.

%%%%%%%%%%%%%%%%%%%%%%%%%%%%%%%%%%%%%%%%%%%%%
\subsection{Removing the \texorpdfstring{$Q$}{Q} factor in linear pHDAE systems}
\label{sec:noQ}
In many applications (time-varying or time-invariant) one has $\ell=\stateDim$ and that $Q = I_{\stateDim}$ is the identity matrix in \Cref{def:pHDAE:linear,def:pHDAE:LTI}. In this case $E$ is the Hessian of the Hamiltonian. This representation has many advantages: All the coefficients appear linearly in \eqref{eqn:pHDAE:LTI}, which greatly simplifies the analysis and also the perturbation theory. Also in many cases this leads to a convexification of the representation, cf.~\cite{Egg19,FriL71}.
In the following we will show how the factor $Q$ can be removed,
see \cite{BeaMXZ18,MehMW21}.

If $Q$ has pointwise full column rank in \eqref{eqn:pHDAE:linear}, then the state equation can be  multiplied with $Q^T$ from the left, yielding a system with the same solution set given by
\begin{align*}
	Q^T E\dot{\state} + Q^T EK\state &= Q^T (J-R)Q\state + Q^T (G-P)\inpVar,\\
	\outVar &= (G+P)^T Q\state + (S-N)\inpVar.
\end{align*}
Then setting
$\widetilde{E} \vcentcolon= Q^T E$, $\widetilde{J} \vcentcolon=Q^T JQ$, $\widetilde{R} \vcentcolon= Q^T R Q$, $\widetilde{G} \vcentcolon= Q^T G$, and $\widetilde{P} \vcentcolon= Q^T P$, the transformed system
\begin{align*}
	\widetilde{E}\dot{\state} + \widetilde{E}K\state &=(\widetilde{J}-\widetilde{R})\state + (\widetilde{G}-\widetilde{P})\inpVar,\\
	\outVar &= (\widetilde{G}+\widetilde{P})^T\state + (S-N)\inpVar
\end{align*}
is again a \pHDAE, but now has $\widetilde{Q}=I_{\stateDim}$ and hence $\widetilde{E}=\widetilde{E}^T$.

If $Q$ is not of full rank then the situation is more complex.
If $Q$ has constant rank in $\timeInt$, then, using a smooth full rank decomposition (\Cref{thm:smoothRankRevealingDecomposition}), there exist pointwise orthogonal matrix functions $U\colon\timeInt\to \R^{\ell,\ell}$ and
$V\colon\timeInt\to\R^{\stateDim,\stateDim}$ of the same smoothness as $Q$ such that
\begin{gather*}
    U^T QV = \begin{bmatrix} Q_{11} & 0\\ 0 & 0\end{bmatrix},\qquad
    U^T EV = \begin{bmatrix} E_{11} & E_{12} \\ E_{21} & E_{22} \end{bmatrix},\\
    U^T(J-R)U = \begin{bmatrix}J_{11}-R_{11} & J_{12}-R_{12} \\ J_{21}-R_{21} & J_{22}-R_{22} \end{bmatrix},
\end{gather*}
where the $(1,1)$ block in all three block matrices is square of size $r=\rank(Q)$ and $Q_{11}$ is pointwise invertible.
Since $Q^T E=E^T Q$, we get $Q_{11}^T E_{11}=E_{11}^T  Q_{11}$ and $E_{12}=0$, and the transformed system, with $\begin{bmatrix} \state_1^T & \state_2^T\end{bmatrix}^T=V^T\state$, is given by 
\begin{multline*}
	\resizebox{.97\linewidth}{!}{$\begin{aligned}
    \begin{bmatrix} E_{11} & 0 \\ E_{21} & E_{22}\end{bmatrix} \begin{bmatrix} \dot{\state}_1 \\ \dot{\state}_2 \end{bmatrix} = \left( \begin{bmatrix} (J_{11}-R_{11}) Q_{11} & 0 \\  (J_{21}-R_{21}) Q_{11} & 0\end{bmatrix}-\begin{bmatrix} E_{11} & 0 \\ E_{21} & E_{22}\end{bmatrix} \begin{bmatrix} K_{11} & K_{12} \\ K_{21} & K_{22} \end{bmatrix} \right)
\begin{bmatrix} \state_1 \\ \state_2\end{bmatrix} + \begin{bmatrix}G_1-P_1\\G_2-P_2\end{bmatrix} \inpVar,
\end{aligned}$}
\end{multline*}
where $\state_1$ is of size $r$ and $\state_2$ of size $n-r$.
By the \pHDAE structure it then follows that $E_{11} K_{12}=0$
and the resulting subsystem
\begin{align*} 
    E_{11} \dot{\state}_1 &= \left((J_{11}-R_{11}) Q_{11}  - E_{11} K_{11}\right) \state_1 +(G_1-P_1) \inpVar, \\
    \outVar &= (G_1+P_1)^T Q_{11} \state_1 + (S-N) \inpVar
\end{align*}
is a \pHDAE with $Q_{11}$ square and nonsingular, which determines $\state_1$ independent of $\state_2$. In particular, we can multiply by $Q_{11}^T$ as discussed before.

However, for given $\state_1$ and $\inpVar$, the remaining \DAE system for $\state_2$
\begin{multline*}
  E_{22} \dot{\state}_2 = \begin{bmatrix} E_{21} & E_{22} \end{bmatrix}  \begin{bmatrix} K_{12} \\  K_{22} \end{bmatrix}  \state_2+(J_{21}-R_{21})Q_{11} \state_1 \\ 
  - \begin{bmatrix} E_{21} & E_{22} \end{bmatrix} \begin{bmatrix} K_{11}  \\ K_{21}  \end{bmatrix}  \state_1 
 -E_{21} \dot{\state}_1+ (G_2-P_2) \inpVar,
\end{multline*}
has no apparent  structure. This is not  a problem, since the variable $z_2$ does not contribute to the Hamiltonian. Actually, further equations, as well as state, input and output variables can always be added to a \pHDAE system if they do not contribute to the Hamiltonian.

\begin{remark}
\label{rem:Qnotfull}
In the case that $Q$ is not of full rank, even in the case 
of \pHODE systems the solution can grow unboundedly. 
In \cite{MehMW18}, the  Hamiltonian \ODE system 
\begin{displaymath}
    \begin{bmatrix} \dot{\state}_1 \\ \dot{\state}_2 \end{bmatrix} = JQ\begin{bmatrix}\state_1\\\state_2\end{bmatrix} = \begin{bmatrix} 0&-1\\ 1&0\end{bmatrix}\begin{bmatrix} 1&0\\ 0&0\end{bmatrix} \begin{bmatrix}\state_1\\\state_2\end{bmatrix},\quad
    \begin{bmatrix}\state_1(0)\\\state_2(0)\end{bmatrix} = \begin{bmatrix}\state_{1,0}\\\state_{2,0}\end{bmatrix}
\end{displaymath}
with Hamiltonian $\hamiltonian=\tfrac {1}{2} \state_1^2$ is presented. It has the solution $\state_1=\state_{1,0}$, $\state_2=\state_{2,0} +t \state_{1,0}$ and thus has linear growth and thus is not stable.

Here the first equation $\dot \state_1=0$ is a \pHODE with Hamiltonian $\hamiltonian=\tfrac {1}{2} \state_1^2$, while the second equation $\dot \state_2 =z_1$ has no specific structure and $\state_2$ does not contribute to the Hamiltonian.
\end{remark}

\begin{remark}
\label{rem:eggerremoveQ}
A similar approach of generating a representation without a~$Q$ factor for nonlinear and even infinite-dimensional evolution equations in a weak formulation has been presented in \cite{Egg19}, where  
for applications in linear generalized gradient systems, it is discussed that even in the case that $Q=I$, it may be more convenient to use a representation that reverses the roles of $E$  and $J-R$.  
Note that if in $E\dot{\state} =(J-R)\state$ both $E$  and $J-R$ are (pointwise) invertible, then by multiplying with $E^{-1}$,
and setting $\tilde \state =(J-R) \state$, $(J-R)^{-1}=\widetilde{J} -\widetilde{R}$, the equivalent new system 
\begin{displaymath}
    (\widetilde{J} - \widetilde{R}) \dot{\tilde \state} = \widetilde{E}\tilde \state
\end{displaymath}
has $\widetilde R\geq 0$ and $\widetilde E>0$.
\end{remark}

In view of the observations concerning the term $Q$, one should avoid introducing a term $Q$ in the representation already on the modeling level and rather work with an $E$ in front of the derivative.

%%%%%%%%%%%%%%%%%%%%%%%%%%%%%%%%%%%%%%%%%%%%%%%%%%%
\subsection{Removing the feedthrough term in linear \pHDAE systems}
\label{sec:nofeed}

In many \pHDAE models, there is no feedthrough term $(S-N)u$ and in this case, by the semi-definiteness of the dissipation matrix $\dissipationMatrix$, also $P=0$. 
If this is not the case, then (under some constant rank assumptions) one can always remove the feedthrough term by extending the state space. However, the simple construction presented in \Cref{rem:feedthroughRemoval} may destroy the \pH structure. In this subsection, we, therefore, discuss how such an extension is possible while preserving the \pHDAE structure.

Consider the linear time-varying or time-invariant \pHDAE in \eqref{eqn:pHDAE:LTI}, or \eqref{eqn:pHDAE:linear}, respectively, and assume that $D\vcentcolon=S-N$ has constant rank. 
Under this assumption by \Cref{thm:smoothRankRevealingDecomposition} there exists a pointwise orthogonal matrix function $U_D$,  
such that
\begin{displaymath}
    D = U_D \begin{bmatrix}D_1 & 0 \\ 0 & 0\end{bmatrix} U_D^T,
\end{displaymath}
with $D_1$ pointwise nonsingular. By construction, the symmetric part of $D_1$ is pointwise positive semi-definite. 
Setting, with analogous partitioning, 
\begin{align*}
    (G-P)U_D &= \begin{bmatrix}G_1-P_1 &  G_2-P_2\end{bmatrix}, &
    U_D^T \inpVar &= \begin{bmatrix} \inpVar_1 \\ \inpVar_2 \end{bmatrix}, &
    U_D^T \outVar &= \begin{bmatrix} \outVar_1 \\ \outVar_2 \end{bmatrix},
\end{align*}
the system can be written as
\begin{subequations}
	\label{reduced}
	\begin{align}
        E \dot{\state} &= (J-R)\state +(G_1-P_1) \inpVar_1 + (G_2-P_2) \inpVar_2,\\
        \outVar_1 &= (G_1 + P_1)^T \state + D_1 \inpVar_1, \\
        \outVar_2 &= (G_2 + P_2)^T \state.
    \end{align}
\end{subequations}
Using the positive semi-definiteness of the matrix (function)~$\dissipationMatrix$ in \Cref{def:pHDAE:linear,def:pHDAE:LTI}, we immediately obtain $P_2=0$.  Let us introduce the new variable $\state_2\vcentcolon= D_1 \inpVar_1 + P_1^T \state$ to obtain the extended system
\begin{align*}
    \begin{bmatrix} E & 0 \\ 0 & 0 \end{bmatrix}\begin{bmatrix}\dot{\state}\\ \dot{\state}_2\end{bmatrix} &= \begin{bmatrix} J-R & 0 \\ D_1^{-1}P_1^T & -D_1^{-1} \end{bmatrix} \begin{bmatrix} \state \\ \state_2 \end{bmatrix} + \begin{bmatrix}G_1 -P_1\\ I\end{bmatrix} \inpVar_1 + \begin{bmatrix} G_2 \\ 0 \end{bmatrix} \inpVar_2,\nonumber \\
    \outVar_1 &= \begin{bmatrix} G_1^T & I \end{bmatrix}\begin{bmatrix} z \\ z_2 \end{bmatrix},\nonumber\\ 
    \outVar_2 &= G_2^T \state . 
\end{align*}
Note that by this extension the Hamiltonian and the output have not changed, they are just formulated in different variables. Then, by multiplying  the state equation with the nonsingular matrix (function) $\begin{smallbmatrix}
        I & P_1 \\ 
        0 & I
    \end{smallbmatrix}$
from the left, we obtain the extended descriptor system
\begin{equation}
	\label{extended:pH}
	\begin{aligned}
	\mathcal{E}\dot{\xi} &= (\mathcal{J}-\mathcal{R})\xi + \mathcal{G}\inpVar,\\
	\outVar &= \mathcal{G}^T\xi
	\end{aligned}
\end{equation}
with extended state $\xi = [\state^T,\state_2^T]^T$ and matrices
\begin{align*}
	\mathcal{E} &\vcentcolon= \begin{bmatrix}
		E & 0\\
		0 & 0
	\end{bmatrix}, & 
	\mathcal{J} &\vcentcolon= \begin{bmatrix}
		J + \tfrac{1}{2}\left(P_1D_1^{-1}P_1^T - \left(P_1D_1^{-1}P_1^T\right)^T\right) & -P_1D_1^{-1}\\
		D_1^{-1}P_1^T & -\tfrac{1}{2}\big(D_1^{-1}-D_1^{-T}\big)
	\end{bmatrix},\\
	\mathcal{G} &\vcentcolon= \begin{bmatrix}
		G_1 & G_2\\
		I & 0
	\end{bmatrix}U_D^T, & 
	\mathcal{R} &\vcentcolon= \begin{bmatrix}
		R - \tfrac{1}{2}\left(P_1D_1^{-1}P_1^T + \left(P_1D_1^{-1}P_1^T\right)^T\right) & 0\\
		0 & \tfrac{1}{2}\big(D_1^{-1}+D_1^{-T}\big)
	\end{bmatrix}.
\end{align*}

\begin{theorem} 
    Consider a linear time-varying or constant coefficient \pHDAE of the form~\eqref{reduced} and the extended system \eqref{extended:pH}. Then both systems have the same input-output relation and Hamiltonian and the extended system without feedthrough is again a \pHDAE.
\end{theorem}

\begin{proof}
	It remains to show that $\mathcal{R}$ is (pointwise) positive semi-definite. Since the positive semi-definiteness of $\mathcal{R}$ is equivalent to the positive semi-definiteness of the symmetric part  of
	\begin{displaymath}
		\begin{bmatrix}
		R - P_1D_1^{-1}P_1^T & 0\\
		0 & D_1^{-1}
	\end{bmatrix},
	\end{displaymath}
	the claim is an immediate consequence of the positive semi-definiteness of the dissipation matrix~\eqref{Wdef} and the Schur complement.
	Note that due to the special form  of the coefficient of the derivative, the changes of basis do not introduce extra derivative terms. 
\end{proof}

Thus, in the following we often assume that $S,N=0$, which then also implies $P=0$. We should be aware, however, that the matrix $D_1$ may be ill-conditioned with respect to inversion, so from a numerical point of view the removal of the feedthrough term may be not advisable.

\begin{remark}\label{rem:addorrem}
The extension of the \pHDAE system by algebraic equations and variables that do not contribute to the
Hamiltonian is the counterpart to 
\Cref{rem:Qnotfull}, where equations that do not contribute to the Hamiltonian can be separated from the system.
\end{remark}

%%%%%%%%%%%%%%%%%%%%%%%%%%%%%%%%%%%%%%%%%%%%%%%%%%%%%%%
\section{Applications and examples}
\label{sec:examples}

In this section, we illustrate the generality and wide applicability of the model class of \pHDAE systems introduced in \Cref{sec:phdae} with several examples from different application areas. For further examples we refer to \cite{RasCSS20} and the references therein.

\subsection{RLC circuit} \label{sec:circuit}
An RLC circuit can be modeled as a directed graph with incidence matrix
\begin{displaymath}
    \calA = \begin{bmatrix} 	\calA_r & \calA_c & \calA_\ell & \calA_v & \calA_i \end{bmatrix}
\end{displaymath}
conveniently partitioned into components associated with resistors, capacitors,  inductors, voltage sources, and current sources; see~\cite{Fre11} for further details. Let $V$ denote the vector of voltages at the nodes (except for the ground node at which the voltage is zero). Furthermore, let $I_\ell$, $I_v$, $I_i$ denote the vectors of currents along the edges for the inductors, voltage sources, and current sources, respectively, while $V_v$ and $V_i$ denote the vectors of voltages across the edges for the voltage sources and current sources. Using Kirchhoff's current and voltage law combined with the so-called branch constitutive relations yield a \pHDAE (in the spirit of \Cref{def:pHDAE:LTI} with $Q = I_n$, $P=0$, $S=N=0$)
\begin{align*}
	\begin{bmatrix}
		\calA_c\mathsf{C}\calA_c^T & 0 & 0\\
		0 & \mathsf{L} & 0\\
		0 & 0 & 0
	\end{bmatrix}\begin{bmatrix}
		\dot{V}\\
		\dot{I}_{\ell}\\
		\dot{I}_v
	\end{bmatrix} &= \begin{bmatrix}
		-\calA_r \mathsf{R}^{-1} \calA_r^T & -\calA_\ell & -\calA_v\\
		\calA_\ell & 0 & 0\\
		\calA_v & 0 & 0
	\end{bmatrix}\begin{bmatrix}
		V\\
		I_{\ell}\\
		I_v
	\end{bmatrix} + \begin{bmatrix}
		\calA_i & 0\\
		0 & 0\\
		0 & -I
	\end{bmatrix}\begin{bmatrix}
		-I_i\\
		V_v
	\end{bmatrix}\\
	\begin{bmatrix}
		V_i\\
		-I_v
	\end{bmatrix} &=  \begin{bmatrix}
		\calA_i & 0\\
		0 & 0\\
		0 & -I
	\end{bmatrix}^T \begin{bmatrix}
		V\\
		I_{\ell}\\
		I_v
	\end{bmatrix}
\end{align*}
and Hamiltonian
\begin{displaymath}
	\hamiltonian(V,I_\ell,I_v) = V^T \calA_c\mathsf{C}\calA_c^T V + I_\ell^T \mathsf{L} I_\ell,
\end{displaymath}
which is associated with the stored energy in the capacitors and inductors.
Here, the positive definite matrices $\mathsf{R}$, $\mathsf{C}$, and $\mathsf{L}$ are defined via the defining properties of the resistors, capacitors, and inductors.
Note that in this case we have an \LTI \pHDAE with quadratic Hamiltonian, $Q=I$ and no feedthrough term.

For general circuits, recently in \cite{NedPS21} a new \pHDAE formulation has been suggested, which allows, in particular, for a structural index analysis and for more efficient and structurally robust implementations than classical modified nodal analysis. Another recent development is the formulation of dynamic iteration schemes for coupled \pHDAE systems  and their use in circuit simulation in \cite{GunBJR21}.

%%%%%%%%%%%%%%%%%%%%%%%%%%%%%%%%%%%%%%%%%%%%%%%%
\subsection{Power networks}
\label{sec:power}
A major application of \pHDAE modeling arises in power network applications. Consider the following simple model of an electrical circuit in \Cref{fig:circuit}, which is presented in \cite{MehM19}.
In this model $L>0$ is an inductor, $C_1,C_2>0$ are capacitors $R_\mathrm{G},R_\mathrm{L},R_\mathrm{R}>0$ are resistances, and $E_\mathrm{G}$ a controlled voltage source. This circuit can serve as a surrogate model of a DC generator ($E_\mathrm{G}$,$R_\mathrm{G}$), connected to a load ($R_\mathrm{R}$ with a transmission line and given by $C_1,C_2,L,R_\mathrm{L}$). In real-world power networks one would have a large number of generators (including wind turbines and solar panels) and loads representing customers.
\begin{figure}[ht]
  \centering
  \begin{circuitikz}[scale=1.5,/tikz/circuitikz/bipoles/length=1cm]
    \draw (0,0) node[ground] {} to[american controlled voltage source,invert,v>=$E_\mathrm{G}$] (0,1) to[R=$R_\mathrm{G}$,i>=$I_\mathrm{G}$] (0,2) to (1,2) to[L=$L$,i>=$I$] (2,2) to[R=$R_\mathrm{L}$] (3,2) to (4,2) to[R=$R_\mathrm{R}$,i<=$I_\mathrm{R}$] (4,0) node[ground] {};
    \draw (1,2) to[C=$C_1$,v>=$V_1$,i<=$I_1$,*-] (1,0) node[ground] {};
    \draw (3,2) to[C,l_=$C_2$,v^>=$V_2$,i<_=$I_2$,*-] (3,0) node[ground] {};
  \end{circuitikz}
  \caption{Simple DC power network example}\label{fig:circuit}
\end{figure}
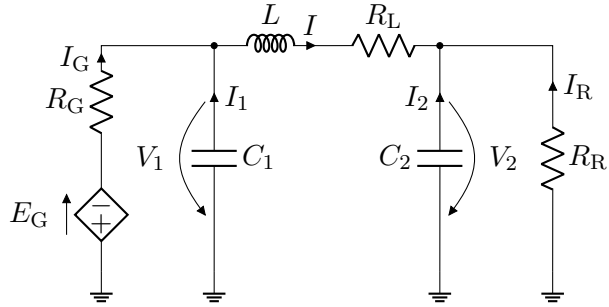
With a quadratic Hamiltonian describing the energy stored in the inductor and the two capacitors
\begin{equation}
  \mathcal H(I,V_1,V_2) = \tfrac{1}{2} LI^2 + \tfrac{1}{2} C_1V_1^2 + \tfrac{1}{2} C_2V_2^2,
\end{equation}
a formulation  as  \LTI \pHDAE has  the form
\begin{subequations}\label{eq:circuit2}
  \begin{align}
    E\dot{\state} &= (J-R)\state + G\inpVar, \\
    \outVar &= G^T \state,
  \end{align}
\end{subequations}
with $\state=\begin{bmatrix} I &V_1&V_2&I_\mathrm{G}&I_\mathrm{R}\end{bmatrix}^T $, $\inpVar=E_\mathrm{G}$, $\outVar=I_\mathrm{G}$, $E=\diag(L,C_1,C_2,0,0)$, and $G=e_4=\begin{bmatrix} 0 & 0 & 0 &1 & 0\end{bmatrix}^T$, 
\begin{equation*}
  J = \begin{bmatrix} 0 & -1 & 1 & 0 & 0 \\ 1 & 0 & 0 & -1 & 0 \\ -1 & 0 & 0 & 0 & -1 \\ 0 & 1 & 0 & 0 & 0 \\ 0 & 0 & 1 & 0 & 0\end{bmatrix}, \quad
  R = \begin{bmatrix}R_\mathrm{L} & 0 & 0 & 0 & 0 \\ 0 & 0 & 0 & 0 & 0 \\ 0 & 0 & 0 & 0 & 0 \\ 0 & 0 & 0 & R_\mathrm{G} & 0 \\ 0 & 0 & 0 & 0 & R_\mathrm{R}\end{bmatrix}.
\end{equation*}
If the generator is shut down (\ie $E_\mathrm{G}=0$), then the system approaches an equilibrium solution  for which $\ddt\hamiltonian(\state)=0$, so that  $I=I_\mathrm{G}=I_\mathrm{R}=0$, and then $\state=0$.

This system can also be considered as control problem that has the control task that a consumer (represented by a resistance $R_\mathrm{R}$) receives a fixed amount of power $P=R_\mathrm{R} I_\mathrm{R}^2$. This can be achieved by controlling the voltage of the generator $E_\mathrm{G}$, so that the solution converges to the values $I_\mathrm{R}=
%I_\mathrm{R}^\star\vcentcolon=
-\sqrt{P/R_\mathrm{R}}$,
$I=I_\mathrm{G}\equiv-I_\mathrm{R}$,
$V_1\equiv(R_\mathrm{R}+R_\mathrm{L})I_\mathrm{R}$,
$V_2\equiv R_\mathrm{R}I_\mathrm{R}$,
and $E_\mathrm{G}\equiv-(R_\mathrm{R}+R_\mathrm{L}+R_\mathrm{G})I_\mathrm{R}$.

%%%%%%%%%%%%%%%%%%%%%%%%%%%%%%%%%%%%%%%%%%%%%%%%%%
\subsection{Stokes and Navier-Stokes equation}
\label{sec:navierstokes}
A classical example of a partial differential equation which, after proper space discretization leads to a \pHDAE, see \eg~\cite{EmmM13}, are the incompressible or nearly incompressible Navier-Stokes equations describing the flow of a Newtonian fluid in a domain $\Omega$,
\begin{align*}
\tfrac{\partial v}{\partial_t} - \nu \Delta v + (v\cdot\nabla) v + \nabla p &=
f \qquad \text{in } \Omega \times \mathbb T ,\\
\nabla^T v &=0 \qquad \text{in } \Omega \times \mathbb T,
\end{align*}
together with suitable initial and boundary conditions, see \eg \cite{Tem77}.
When one linearizes around a prescribed stationary vector field $v_{\infty}$, then one
obtains the linearized Navier-Stokes equations,
\begin{align*}
\tfrac{\partial v}{\partial_t}  - \nu \Delta v + (v_{\infty}\cdot\nabla) v +
(v\cdot\nabla) v_{\infty} + \nabla p &= f
\qquad \text{in } \Omega \times \timeInt,
\nonumber \\
\nabla^T v &=0 \qquad \text{in } \Omega \times \timeInt.
\end{align*}
If $v_{\infty}$ is also constant in space then $(v\cdot\nabla) v_{\infty}=0$ and one obtains the Oseen equations. If also the term $(v_{\infty}\cdot\nabla) v$ is neglected one obtains the Stokes equation. 
Performing a finite element discretization in space, see for instance \cite{Lay08},
a Galerkin projection leads to a \dHDAE of the form
\begin{equation}
    \label{eq:P-instat-op-general}
    \begin{bmatrix} M & 0 \\ 0 & 0 \end{bmatrix}
    \begin{bmatrix} \dot{v} \\ \dot{p} \end{bmatrix} =
    \left(\begin{bmatrix} A_S & B\\ -B^T & 0\end{bmatrix} -
    \begin{bmatrix} -A_H & 0\\ 0 & -C\end{bmatrix}\right)
    \begin{bmatrix} v\\p\end{bmatrix}+\begin{bmatrix}f\\ 0\end{bmatrix},
\end{equation}
where $M=M^T>0$ is the mass matrix, $A_S=-A_S^T$, $-A_H=-A_H^T\geq 0$ are the skew-symmetric and the symmetric part of the discretized and linearized convection-diffusion operator, $B^T$ is the discretized divergence operator, which we assume to be normalized so that it is of full row rank, and $-C=-C^T>0$ is a stabilization term, typically of small norm, that is needed for some
finite element spaces, see \eg 
\cite{Ran00}. 
The variables $v$ and $p$ denote the discretized velocity and pressure, respectively, and $f$ is a forcing or control term.

This becomes a \pHDAE by adding an output equation $\outVar = f^T v$ and an appropriate Hamiltonian, see \cite{AltS17}.
Other possible inputs and outputs that are not necessarily co-located can be chosen, \eg by different boundary conditions (added to the system via the trace operator and suitable Lagrange multipliers) or measurement points for the velocities or pressures.

%%%%%%%%%%%%%%%%%%%%%%%%%%%%%%%%%%%%%%%%%%%%%
\subsection{Multiple-network poroelasticity}\label{sec:poro}
Biot's poroelasticity model for quasi-static deformation, \cite{Bio41}, describes porous materials fully saturated by a viscous fluid. Typical applications include geomechanics \cite{Zob10}, and biomedicine \cite{SobEWC12}. The effect of different fluid compartments can be accounted for with the theory of multiple-network poroelasticity \cite{BaiER93}.
For instance, in the investigation of cerebral edema, see \cite{TulV11}, one distinguishes different blood cycles (arterial, arteriole/capillary, venous) and a cerebrospinal fluid, giving a total of $m=4$ fluid compartments.
The complete model is given by a coupled system of (nonlinear) partial differential equations (\PDEs)
\begin{subequations}
\label{eqn:poro}
\begin{align}
  -\nabla \cdot \big( \sigma (u) \big) + \sum_{i=1}^m \nabla (\alpha_i p_i) &= f,\\
  \tfrac{\partial}{\partial_t}\Big(\alpha_i \nabla \cdot u + \tfrac{1}{M} p_i \Big)- \nabla \cdot \Big( \tfrac{\kappa_i(\nabla\cdot u)}{\nu_i} \nabla p_i \Big) - \sum_{j\neq i} \beta_{ij} (p_i-p_j) &= g_i,
\end{align}
with unknown displacements $u$, unknown pressure variables $p_i$ ($i=1,\ldots,m$) for the different fluid compartments, the Biot-Willis fluid-solid coupling coefficients $\alpha_i$, Biot modulus $M$,  fluid viscosities $\nu_i$, (nonlinear) hydraulic conductivities $\kappa_i = \kappa_i(\nabla\cdot u)$,  network transfer coefficients $\beta_{ij}$, volume-distributed external forces $f$, and injection $g_i$. The stress-strain relation is given by
\begin{equation*}
  \sigma(u) \vcentcolon= 2\mu\, \varepsilon (u) + \lambda\, (\nabla \cdot u)\, \mathcal{I}, \qquad
  \varepsilon(u) \vcentcolon= \tfrac{1}{2}\, \big(\nabla u + (\nabla u)^T \big)
\end{equation*}
with the Lam\'e coefficients $\mu$ and $\lambda$ and the identity tensor~$\mathcal{I}$.  For simplicity, we consider the system with Dirichlet boundary conditions
\begin{align}
	u = u_\mathrm{b}\qquad\text{and}\qquad p_i = p_{i,\mathrm{b}}\qquad\text{on } \timeInt \times \partial\Omega.
\end{align}
\end{subequations}
Following \cite{AltMU21c} (see also \cite{EggS21}),  a \pHDAE of the mixed finite-element discretization of~\eqref{eqn:poro} is given as
\begin{align*}
	\begin{bmatrix}
		0 & 0 & 0 & 0 & 0\\
		0 & K_\mathrm{u} & 0 & 0 & 0\\
		0 & 0 & M_\mathrm{p} & 0 & 0\\
		0 & 0 & 0 & 0 & 0\\
		0 & 0 & 0 & 0 & 0
	\end{bmatrix}\begin{bmatrix}
		\dot{w}_h\\
		\dot{u}_h\\
		\dot{p}_h\\
		\dot{\lambda}_{\mathrm{u},h}\\
		\dot{\lambda}_{\mathrm{p},h}
	\end{bmatrix} = \begin{bmatrix}
		0 & -K_\mathrm{u} & D^T & B_\mathrm{u}^T & 0\\
		K_\mathrm{u} & 0 & 0 & 0 & 0\\
		-D & 0 & -K_\mathrm{p}(u_h) &  0 & B_\mathrm{p}^T\\
		-B_\mathrm{u} & 0 & 0 & 0 & 0\\
		0 & 0 & -B_\mathrm{p} & 0 & 0
	\end{bmatrix}\begin{bmatrix}
		w_h\\
		u_h\\
		p_h\\
		\lambda_{\mathrm{u},h}\\
		\lambda_{\mathrm{p},h}
	\end{bmatrix} + \begin{bmatrix}
		f_h\\0\\g_h\\\dot{u}_{\mathrm{b},h}\\p_{\mathrm{b},h}
	\end{bmatrix},
\end{align*}
with positive definite mass and stiffness matrices $M_\mathrm{p}$, $K_\mathrm{u}$, and $K_\mathrm{p}(u_h)$. Let us emphasize that in this representation we may use the boundary conditions as additional inputs (added to the system via the trace operator and suitable Lagrange multipliers) such that the system may be controlled via its boundary.

\subsection{Pressure waves in gas network}
\label{sec:gasnetwork}
The propagation of pressure waves on acoustic time scales through a network of  gas pipelines is modeled in \cite{BroGH11}, see also \cite{EggK18,EggKLMM18}, via a linear infinite-dimensional \pHDAE system on a finite directed and connected graph $\G=(\V,\E)$
with vertices $v \in \V$ and edges $e \in \E$ that correspond to the
pipes of the physical network. Denote by $p^e(t,x)$ the pressure and by $f^e(t,x)$ the mass flux in pipe $e$, 
and using equations for the conservation of mass and the balance of momentum 
\begin{align*}
a^e \tfrac{\partial}{\partial t} p^e + \tfrac{\partial}{\partial x} f^e &= 0             && \text{on } e \in \E, \ t>0,   \\
b^e \tfrac{\partial}{\partial t} f^e + \tfrac{\partial}{\partial x} p^e + d^e q^e &= 0   && \text{on } e \in \E, \ t>0, 
\end{align*}
where the coefficients $a^e$, $b^e$ encode properties of the fluid and the pipe, and~$d^e$ models the damping due to friction at the pipe walls.
The coefficients are assumed to be positive and, for ease of presentation, constant on every pipe $e$.
To model the conservation of mass and momentum at the junctions at inner vertices $v \in \Vi$ of the graph, where several pipes $e \in \E(v)$ are connected, one requires Kirchhoff's law for the flow as well as continuity of the pressure, \ie
\begin{align*}
\sum\nolimits_{e \in \E(v)}  n^e(v) f^e(v) &= 0 && \text{for all } v \in \Vi, \ t>0,\\
p^e(v) &= p^{e'}(v) && \text{for all } e,e' \in \E(v), \ v \in \Vi, \ t>0.
\end{align*}
Here $n^e(v)\in \{+1,- 1\}$, depending on whether the pipe $e$ starts or ends at the vertex $v$. The time dependent quantities $m^e(v)$ and $p^e(v)$ denote the respective functions evaluated at the vertex $v$.
At the boundary vertices $v \in \Vb = \V \setminus \Vi$, we define co-located  ports of the network by using the pressure %
\begin{align*}
    p^e(v) &= u_v \qquad  \text{for } v \in \Vb, \ e \in \E(v), \ t>0
\end{align*}
as input at $v \in \Vb$, and the mass flux 
\begin{align*}
y_v &= -n^e(v) f^e(v), \qquad v \in \Vb, \ e \in \E(v), \ t>0
\end{align*}
as output. We further define initial functions
\begin{align*}
p(0)=p_0, \qquad  f(0)=f_0 \qquad \text{on } \E.
\end{align*}
Note that typically gas is only inserted at some nodes of the network and extracted at other ends. Thus one could also use different input and output variables at the external nodes that may not necessarily be co-located.  The discussed formulation however, allows easy interconnection, and still the variables have a physical interpretation. 

In \cite{EggK18} several important properties have been shown. These include the existence of unique classical solutions for sufficiently smooth initial data $p_0$, $f_0$, as well as global conservation of mass, and that this system has \pHDAE structure.
Space-discretization via a structure-preserving mixed finite element method leads to a block-structured linear time-invariant \pHDAE system
\begin{subequations}\label{eq:ph1}
\begin{align}
E\dot{\state} &=(J-R)\state+G\inpVar, \qquad \state(t_0)= \state_0, \\ 
\outVar&=G^T \state, 
\end{align}
\end{subequations}
with
\begin{gather*}
    P\vcentcolon=0,\quad S-N\vcentcolon=0,\quad G\vcentcolon=\begin{bmatrix} 0 \\ G_2 \\ 0 \end{bmatrix},\quad \state \vcentcolon= \begin{bmatrix}\state_1 \\ \state_2 \\ \state_3 \end{bmatrix},\\
E\vcentcolon=\begin{bmatrix} M_1 & 0 & 0 \\ 0 & M_2 & 0 \\ 0 & 0 & 0 \end{bmatrix},\quad
J\vcentcolon=\begin{bmatrix} 0 & -J_{12} & 0 \\ J_{12}^T & 0 & J_{32}^T \\ 0 & -J_{32} & 0\end{bmatrix}, \quad R\vcentcolon=\begin{bmatrix} 0 & 0 & 0 \\ 0 & R_{22} & 0 \\ 0 & 0 & 0\end{bmatrix}.
\end{gather*}
Here  $\state_1\colon \R\to \R^{n_1}$ represents the discretized pressures, $\state_2\colon\R\to \R^{n_2}$  the discretized fluxes, while $\state_3\colon\R\to \R^{n_3}$ is a Lagrange multiplier vector introduced to penalize the violation of the space-discretized constraints. The coefficients $M_1=M_1^T$, $M_2=M_2^T$, and $R_{22}=R_{22}^T$ are positive definite, the matrix $J_{32}$ has full row rank and $\begin{bmatrix} J_{12}^T & J_{32}^T \end{bmatrix}$ has full column rank. The discretized Hamiltonian is given by $\hamiltonian(\state)=\tfrac{1}{2} \state^T E\state= \tfrac{1}{2} (\state_1^T M_1 \state_1+\state_2^T M_2 \state_2)$. Note that the Lagrange multiplier does not contribute to the Hamiltonian.

\subsection{Multibody systems}
\label{sec:robot}
Another natural class of applications arises in multibody dynamics. In \cite{Hou94a} the model of a two-dimensional three-link mobile manipulator was derived, see also \cite{BunBMN99} for details. After linearizing around a stationary solution one obtains a control system
\begin{equation*}
    \label{eq:mblin}
    \begin{aligned}
	M \ddot{p}  &= - D\dot{p} - Kp +Z^T \lambda+ G_1 \inpVar, \\
	0 & = Zp ,
    \end{aligned}
\end{equation*}
where $p$ is the vector of positions, $Zp=0$ is the linearized position constraint, its violation is penalized by a Lagrange multiplier vector $\lambda$, and $G_1\inpVar$ is the control force applied at the actuators. The mass and stiffness matrices $M=M^T, K=K^T$ are positive definite and the damping matrix $D=D^T$ is positive semi-definite.

This \DAE has the first and second time derivative of $Zp=0$ as hidden algebraic constraints and it is typically necessary to use a regularization procedure to make the system better suited for numerical simulation and control, see \eg \cite{EicF98,KunM06,RabR00,Sim13}. One possibility is to replace the original constraint by its time derivative $0=-Z \dot p$. By adding a tracking output  $y=G_1^T \dot{p}$, see \eg \cite{HouM94}, and transforming to first order form by introducing
\begin{displaymath}
    \state = \begin{bmatrix} \state_1 \\ \state_2 \\ \state_3 \end{bmatrix} \vcentcolon= \begin{bmatrix} \dot{p} \\ p \\ \lambda  \end{bmatrix},
\end{displaymath}
one obtains a linear time-invariant\ \pHDAE system of the form~\eqref{eqn:pHDAE:LTI} with 
\begin{align*}
    E &\vcentcolon= \begin{bmatrix} M & 0 & 0 \\ 0 & K & 0 \\ 0 & 0 & 0 \end{bmatrix}, &
    R &\vcentcolon= \begin{bmatrix} D &  0 & 0 \\
0 & 0 & 0 \\ 0 & 0  & 0 \end{bmatrix}, &
    G &\vcentcolon= \begin{bmatrix} G_1 \\ 0 \\ 0 \end{bmatrix}, \\
    Q &\vcentcolon= \begin{bmatrix} I & 0 & 0 \\ 0 & I & 0 \\ 0 & 0 & I \end{bmatrix}, &
    J &\vcentcolon= \begin{bmatrix} 0 & -K &  Z^T \\ K & 0 & 0 \\ -Z & 0 & 0 \end{bmatrix}, &
    P &\vcentcolon=0,\ S-N\vcentcolon=0.
\end{align*}
The quadratic Hamiltonian~\eqref{eqn:Hamiltonian:LTI} is given by 
\begin{displaymath}
    \hamiltonian(\state)= \frac{1}{2} \begin{bmatrix} z_1 \\ z_2 \end{bmatrix}^T \begin{bmatrix} M & 0 \\ 0 &  K \end{bmatrix} \begin{bmatrix}z_1 \\ z_2\end{bmatrix}.
\end{displaymath}
Note that the Lagrange multiplier does not contribute to the Hamiltonian.

\subsection{Brake squeal}\label{sec:brake}

Disc brake squeal is a frequent and annoying phenomenon. 
\begin{figure}	
	\centering
	\includegraphics[width=0.4\linewidth]{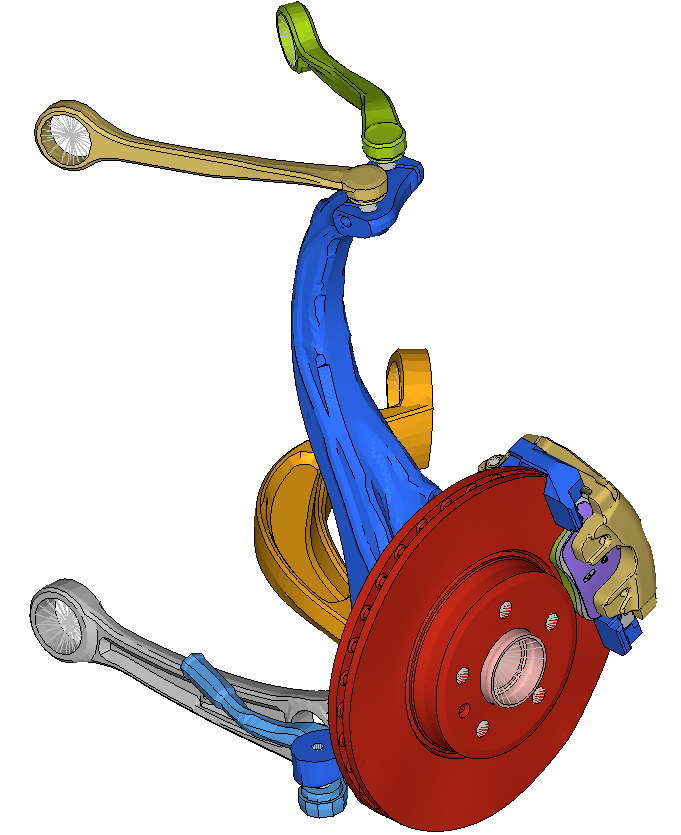}\quad \includegraphics[width=0.4\linewidth]{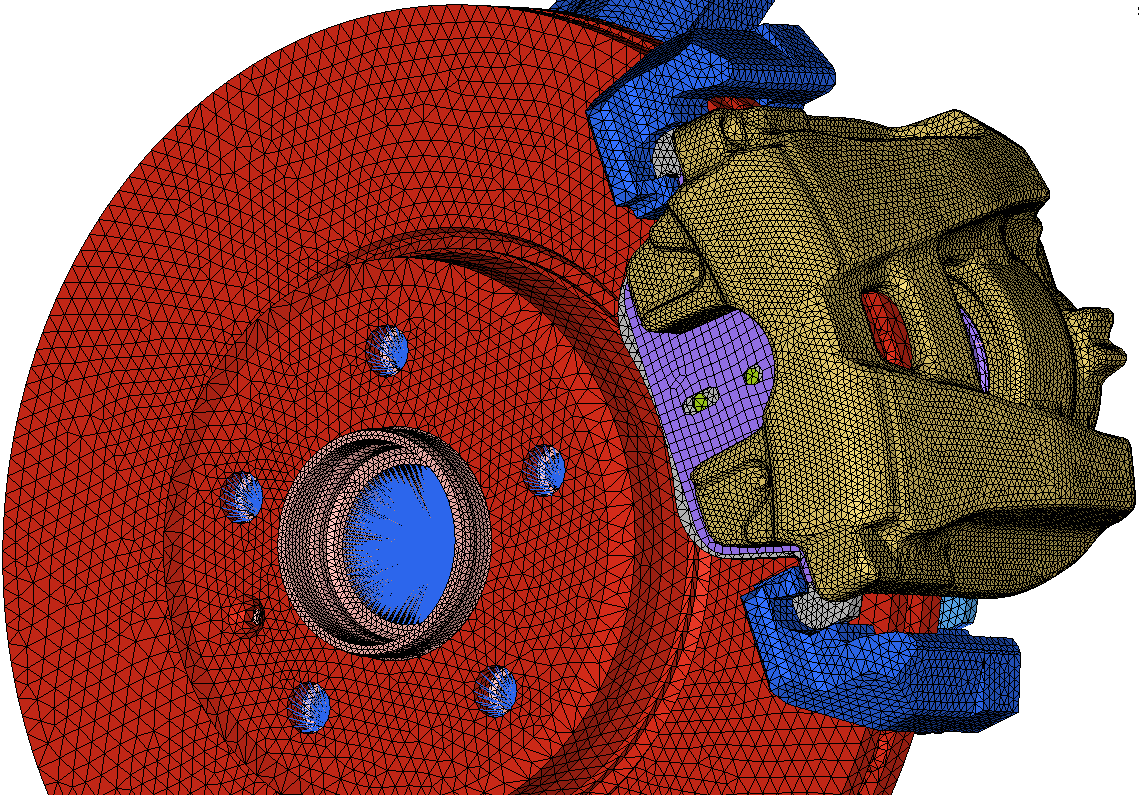}
	\caption{Finite element element model of disk brake}
	\label{fig:breakSqueal}
\end{figure}
In \cite{GraMQSW16} a very large finite element model of a brake system was derived that includes friction as well as circulatory and gyroscopic effects. This has the form
\begin{displaymath}
    M\ddot{{q}}+\big(C_1+\tfrac{\omega_{\mathrm{ref}}}{\omega}C_R+\tfrac{\omega}{\omega_{\mathrm{ref}}} C_G\big)\dot{q}+\big(K_1+K_R+\big(\tfrac{\omega}{\omega_{\mathrm{ref}}}\big)^2 K_{G}\big){q}=f,
\end{displaymath}
where $q$ is a vector of finite element  coefficients,
$M=M^T\geq 0$ is a singular mass matrix, $C_{1}=C_{1}^T$ models material damping, $C_{G}=-C_{G}^T$ models gyroscopic effects, $C_{R}=C_{R}^T$ models friction-induced damping and is typically generated from measurements, $K_{1}=K_1^T\geq 0$ is a stiffness matrix,
$K_{R}$ has no symmetry structure and models circulatory effects, while $K_{G}=K_G^T$ is a geometric stiffness matrix. One of many parameters is $\omega$, the rotational speed of the disk scaled by a reference velocity $\omega_{\mathrm{ref}}$. 

Experiments indicate that there is a  
subcritical Hopf bifurcation (in the parameter $\omega$), when eigenvalues  of the associated quadratic parametric eigenvalue problem
\begin{displaymath}
(\lambda(\omega)^{2}M+\lambda(\omega) (G(\omega)+D(\omega))+(K(\omega)+N(\omega))v(\omega)=0
\end{displaymath}
cross the imaginary axis.
Here $C_1+\tfrac{\omega_{\mathrm{ref}}}{\omega}C_R+\tfrac{\omega}{\omega_{\mathrm{ref}}} C_G= G(\omega) +D(\omega)$ and $ K_1+K_R+(\tfrac{\omega}{\omega_{\mathrm{ref}}})^2 K_{G}=K(\omega)+N(\omega)$ are split into their symmetric and skew-symmetric parts.

By writing the system in first order formulation, it can be expressed as a 
perturbed \dHDAE system $E \dot{\state} =(J-R_{\mathrm{D}})\state-R_{\mathrm{N}}\state$, with
\begin{align*}
E &= \begin{bmatrix} M & 0 \\ 0 & K\end{bmatrix}, &  J &= \begin{bmatrix} -G  &  -(K+\tfrac{1}{2}N)\\ (K+\tfrac{1}{2}N^T ) & 0 \end{bmatrix}, \\ 
    R_{\mathrm{D}} &= \begin{bmatrix} D  & 0  \\ 0 & 0 \end{bmatrix}, &
    R_{\mathrm{N}} &= \begin{bmatrix} 0  &  -\tfrac{1}{2}N  \\ -\tfrac{1}{2}N^T & 0 \end{bmatrix}.
\end{align*}
Instability and squeal arise only from the perturbation term $R_{\mathrm{N}}\state$, which is associated with the brake force restricted to the finite element nodes on the brake pad. 

Performing a linear stability analysis by solving the eigenvalue problem for an industrial problem,  incorporating the perturbation term $R_{\mathrm{N}} $ via a homotopy parameter $E \dot{\state} = (J-R_{\mathrm{D}})\state-\alpha R_{\mathrm{N}}\state$, $\alpha \in [0,1]$, it has been determined in \cite{Bec18} that for $\alpha=0$ the spectral abscissa, \ie, the maximal real part of all eigenvalues, is $-5.0462e-06$ and for $\alpha=.1$ it is already $2.0336e-05$, \ie, the unperturbed problem is already close to a problem with positive real part eigenvalues. 
The application task then is to design the brake in such a way (\eg~by including damping devices, so-called shims) that the unperturbed problem $E \dot{\state} = (J-R_{\mathrm{D}})\state$ is such that the perturbation $R_{\mathrm{N}}\state$ does not lead to eigenvalues in the right half plane, or at least make sure that they have a small real part.

%%%%%%%%%%%%%%%%%%%%%%%%%%%%%%%%%%%%%%%%%%%%%%%%%%%%%
\section{Properties of pHDAE systems}
\label{sec:properties}

In this section we discuss several general properties of the model class of \pHDAE systems and show why they are a very good candidate for our modeling  wishlist.
%and analyze to what extend our wishlist is fulfilled.

\subsection{Power balance equation and dissipation inequality}\label{sec:PBEDI}
A key property of \pHDAE systems that shows the strong rooting in the underlying physical principles is the power balance equation and the associated dissipation inequality, see also \Cref{sec:passive}. 
\begin{theorem}\label{th:pbe}
The \pHDAE~\eqref{eqn:pHDAE} satisfies the power balance equation
\begin{equation*}
		\ddt \hamiltonian(t,\state(t)) = - \begin{bmatrix}
			\eta(t,\state)\\
			\inpVar
		\end{bmatrix}^T \dissipationMatrix(t,\state) \begin{bmatrix}
			\eta(t,\state)\\
			\inpVar
		\end{bmatrix} + \outVar^T \inpVar
\end{equation*}
along any behavior solution of the \pHDAE~\eqref{eqn:pHDAE}.
\end{theorem}

\begin{proof}
Let $\begin{bmatrix} \state^T & \inpVar^T & \outVar^T\end{bmatrix}^T$ be a behavior solution of the \pHDAE~\eqref{eqn:pHDAE}. Using the structural properties of the \pHDAE system we obtain
\begin{align*}
		\ddt \hamiltonian(t,\state)
		&= \tfrac{\partial}{\partial t} \hamiltonian(t,\state) + \left(\tfrac{\partial}{\partial {\state}} \hamiltonian(t,\state)\right)^T \dot{\state}
		= \eta(t,\state)^T\left( E(t,\state)\dot{\state} + r(t,\state)\right)\\
		&= \eta(t,\state)^T \left(\left(J(t,\state)-R(t,\state)\right)\state + (B(t,\state)+P(t,\state) - 2P(t,\state))u\right)\\
		&= -\begin{bmatrix}
			\eta(t,\state)\\
			u
		\end{bmatrix}^T \begin{bmatrix}
			R(t,\state) & P(t,\state)\\
			P(t,\state)^T & S(t,\state)
		\end{bmatrix}\begin{bmatrix}
			\eta(t,\state)\\
			u
		\end{bmatrix} + y^T u\\
		&= -\begin{bmatrix}
			\eta(t,\state)\\
			u
		\end{bmatrix}^T W(t,z)\begin{bmatrix}
			\eta(t,\state)\\
			u
		\end{bmatrix} + y^T u		.
\end{align*}
\end{proof}
Using the fact that  $\dissipationMatrix\geq 0$ along any solution of~\eqref{eqn:descriptorSystem}, we immediately obtain  that the \pHDAE system satisfies the \emph{dissipation inequality}
\begin{equation}\label{eq:dissIneq}
    \mathcal H(t_1,z(t_1)) - \mathcal H(t_0,z(t_0)) \leq \int_{t_0}^{t_1}y(s)^Tu(s)\,\mathrm{d}s.
\end{equation}

The power balance equation and the dissipation inequality are obvious in all the examples described in \Cref{sec:examples}. In all cases the dissipation term $R$ is positive semi-definite. In the disk brake for the unforced system, without applying the brake force, cf.~\Cref{sec:brake}, this is also the case. The perturbation term with moves eigenvalues to the right half plane is due to the external force that can be interpreted as a supplied energy via the term  $\outVar^T\inpVar$.

%%%%%%%%%%%%%%%%%%%%%%%%%%%%%%%%%%%%%%%%
\subsection{Invariance under transformations and projection}
\label{sec:invariance}

Another essential property of the class of \pHDAE systems is the invariance under different equivalence transformations, see \cite{BeaMXZ18,MehM19,Mor22}.  

Let us begin  with  general state space-transformations and consider $\widetilde{\stateSpace}\subseteq \R^{\widetilde{\stateDim}}$ and define $\widetilde{\extSpace}\vcentcolon= \timeInt \times \widetilde{\stateSpace}$ with elements $(t,\widetilde{\state})\in\widetilde{\extSpace}$. Let $\varphi\in \calC^1(\widetilde{\extSpace},\stateSpace)$ and let $U\in\calC(\widetilde{\extSpace},\R^{\ell,\ell})$ be invertible. Define
\begin{align*}
	\psi\colon \widetilde{\extSpace} \to \extSpace,\qquad (t,\widetilde{\state}) \mapsto (t,\varphi(t,\widetilde{\state}))
\end{align*}
and
\begin{align*}
	\widetilde{E} &\vcentcolon= U^T(E\circ \psi)\tfrac{\partial}{\partial {\state}} \varphi, & \widetilde{J} &\vcentcolon= U^T (J\circ \psi) U, & \widetilde{R} &\vcentcolon= U^T (R\circ \psi)U,\\
	\widetilde{G} &\vcentcolon= U^T(G \circ \psi),  & \widetilde{P} &\vcentcolon= U^T(P \circ \psi),  & \widetilde{\eta} &\vcentcolon= U^T(\eta \circ \psi), \\
	\widetilde{S} &\vcentcolon= S\circ \psi, & \widetilde{N} &\vcentcolon= N\circ \psi, & \widetilde{r} &\vcentcolon= U^T(r\circ \psi) + (E\circ \psi)\tfrac{\partial}{\partial t} \varphi),
\end{align*}
together with a transformed Hamiltonian
\begin{equation}
	\label{eqn:Hamiltonian:transformed}
	\widetilde{\hamiltonian} \vcentcolon= \hamiltonian \circ \psi.
\end{equation}
We have the following transformation result taken from~\cite{MehM19}.
\begin{theorem}
	Consider a \pHDAE~\eqref{eqn:pHDAE} together with the transformed descriptor system
	\begin{subequations}
		\label{eqn:pHDAE:transformed}
		\begin{align}
			\widetilde{E}(t,\widetilde{\state})\dot{\widetilde{\state}} + \widetilde{r}(t,\widetilde{\state}) &= (\widetilde{J}(t,\widetilde{\state})-\widetilde{R}(t,\widetilde{\state}))\widetilde{\eta}(t, \widetilde{\state}) + (\widetilde{G}(t,\widetilde{\state}) - \widetilde{P}(t,\widetilde{\state}))\inpVar,\\
			\outVar &= (\widetilde{G}(t,\widetilde{\state}) + \widetilde{P}(t,\widetilde{\state}))^T \widetilde{\eta}(t, \widetilde{\state}) + (\widetilde{S}(t,\widetilde{\state}) - \widetilde{N}(t,\widetilde{\state}))\inpVar,
		\end{align}
	\end{subequations}
and transformed Hamiltonian $\widetilde{\hamiltonian}$ in~\eqref{eqn:Hamiltonian:transformed}. Then the following properties hold:
	\begin{enumerate}
		\item[(i)] The descriptor system~\eqref{eqn:pHDAE:transformed} is a \pHDAE.
		\item[(ii)] If $\varphi(t,\cdot)$ is a local diffeomorphism for all $t\in\timeInt$, then to any behavior solution $(\widetilde{\state},\inpVar,\outVar)$ of \eqref{eqn:pHDAE:transformed} there corresponds a behavior solution $(\state,\inpVar,\outVar)$ of \eqref{eqn:pHDAE} with $\state(t)=\varphi(t,\widetilde{\state}(t))$.
		\item[(iii)] If $\varphi(t,\cdot)$ is a global diffeomorphism for all $t\in\timeInt$, then there is a one-to-one correspondence between a behavior solution $(\widetilde{\state},\inpVar,\outVar)$ of \eqref{eqn:pHDAE:transformed} and a behavior solution $(\state,\inpVar,\outVar)$ of \eqref{eqn:pHDAE} with $\state(t)=\varphi(t,\widetilde{\state}(t))$.
	\end{enumerate}
\end{theorem}

\begin{proof}
The transformed DAE system is obtained from \eqref{eqn:pHDAE} by setting $\state=\varphi(t,\tilde{\state})$, pre-multiplying with $U^T$ and inserting $UU^{-1}$ in front of $\state$ in the first equation.
It is clear that if $(\widetilde{\state},\inpVar,\outVar)$ is a solution of \eqref{eqn:pHDAE:transformed}, then $(\state,\inpVar,\outVar)$ is a behavior solution of the original system and if $\varphi(t,\cdot)$ is a global diffeomorphism, then we can apply the inverse transformation 
to get a behavior solution $(\widetilde{\state},\inpVar,\outVar)$ for any  behavior solution $(\state,\inpVar,\outVar)$ of the original system.

To show that \eqref{eqn:pHDAE:transformed} is still a \pHDAE, we must check the defining conditions.
By substitution, we get
\begin{align*}
    \widetilde{\dissipationMatrix} 
    &= \begin{bmatrix}\widetilde{R} & \widetilde{P} \\ \widetilde{P}^T & \widetilde{S}\end{bmatrix} 
    = \begin{bmatrix} U^T (R\circ\varphi)U & U^T(P\circ\varphi) \\ (P\circ\varphi)^T U & S\circ\varphi\end{bmatrix}\\
    &= \begin{bmatrix} U & 0 \\ 0 & I\end{bmatrix}^T (W\circ\varphi) \begin{bmatrix} U & 0 \\ 0 & I\end{bmatrix} \geq 0,
\end{align*}
since positive semi-definiteness is invariant under congruence. %\qedhere
We also get 
\begin{align*}
    \tfrac{\partial}{\partial \widetilde{\state}}{\widetilde{\hamiltonian}}(t,\widetilde{\state}) &= (\tfrac{\partial}{\partial \widetilde{\state}}\varphi)^T\left (\tfrac{\partial}{\partial \state} \hamiltonian\circ\varphi\right )
      = (\tfrac{\partial}{\partial \widetilde{\state}}{\varphi})^T \left (E^T \state\circ\varphi\right ) \\
      &= (\tfrac{\partial}{\partial \widetilde{\state}}{\varphi})^T(E^T\circ\varphi)UU^{-1}(\state\circ\varphi)
      = \widetilde{E}^T\widetilde{\state},
\end{align*}
and
\begin{align*}
    \frac{\partial {\widetilde{\hamiltonian}}}{\partial t}(t,\widetilde{\state}) &=
      \frac{\partial {\hamiltonian}(t)}{\partial t}\circ\varphi + \left( \tfrac{\partial}{\partial \state} \hamiltonian \circ\varphi\right)^T \frac{\partial{\varphi}}{\partial t}
      = \state^T r\circ\varphi + \left (\state^T E\circ\varphi \right )\frac{\partial \varphi}{\partial t}\\
      &= (\state\circ\varphi)^T\left (r\circ\varphi + (E\circ\varphi)\frac{\partial{\varphi}}{\partial t}\right )
      = \widetilde{\state}^T U^T U^{-T}\tilde r = \widetilde{\state}^T\tilde{r}.\qedhere
\end{align*}
\end{proof}

For linear \pHDAE systems with quadratic Hamiltonian this invariance takes the following form,
see \cite{BeaMXZ18}. 
\begin{theorem}
\label{th:pHDAEinv}
Consider a linear \pHDAE system of the form~\eqref{eqn:pHDAE:linear} with  quadratic Hamiltonian \eqref{eqn:Hamiltonian:linear}. Let $U\in \calC(\timeInt,\R^{\ell,\ell})$ and $V\in \calC^1(\timeInt,\R^{\stateDim,\stateDim})$ be pointwise nonsingular in $\timeInt$.
Then the transformed \DAE
\begin{align*}
    \tilde{E} \dot{\tilde{\state}} &= [(\tilde{J} -\tilde{R})\tilde{Q} -\tilde{E}\tilde{K}]\tilde{\state}  + (\tilde{G} -\tilde{P})\inpVar, \\
    \outVar &= (\tilde{G}+\tilde{P})^T\tilde{Q} \tilde{\state} + (S+N)\inpVar,
\end{align*}
with
\begin{align*}
    \tilde{E} &\vcentcolon=U^T E V , &
    \tilde{Q} &\vcentcolon=U^{-1} Q V ,&
    \tilde{J} &\vcentcolon=U^T J U ,\\
    \tilde{R} &\vcentcolon= U^T R U ,&
    \tilde{G} &\vcentcolon= U^T G ,&
    \tilde{P} &\vcentcolon= U^T P ,\\
    \tilde{K} &\vcentcolon= V^{-1} K V +{V}^{-1}\dot{V},&
    \state &\phantom{\vcentcolon}= V\tilde{\state}
\end{align*}
is still a \pHDAE system with the same quadratic Hamiltonian \begin{displaymath}
    \tilde{\hamiltonian}(\tilde{\state})\vcentcolon=\frac{1}{2} \tilde{\state}^T \tilde{Q}^T \tilde{E} \tilde{\state} = \hamiltonian(\state).
\end{displaymath}
\end{theorem}
\begin{proof} 
The transformed \DAE system is obtained from the original \DAE system by setting  $\state=V\tilde{\state}$, pre-multiplying with $U^T$, and by inserting $U U^{-1}$ in front of $Q$.
The transformed operator corresponding to $\calL$ in \eqref{skewL} is given by
\begin{displaymath}
    \calL_V \vcentcolon= \tilde{Q}^T \tilde{E} \ddt-\tilde{Q}^T(\tilde{J}\tilde{Q}-\tilde{E}\tilde{K}).
\end{displaymath}
Because
\begin{displaymath}
    \tilde{Q}^T\tilde E=V^T Q^T EV, \quad
    \tilde{Q}^T\tilde{J}\tilde{Q} = V^T Q^T JQV, \quad
    \tilde{Q}^T\tilde{E}V^{-1}\dot{V} = V^T Q^T E\dot{V},
\end{displaymath}
we have that $\calL_V$ is again skew-adjoint. 
It is then straightforward to show that $\tilde{\hamiltonian}(\tilde{\state}) = \hamiltonian(\state)$.
\begin{align*}
    \tilde{W} &= \begin{bmatrix}\tilde{Q}^T \tilde{R}\tilde{Q} & \tilde{Q}^T\tilde{P}\\
    \tilde{P}^T\tilde{Q} & S\end{bmatrix}
    = \begin{bmatrix} V&0\\0&I \end{bmatrix}^T \dissipationMatrix\begin{bmatrix} V&0\\0&I \end{bmatrix},
\end{align*}
where $\dissipationMatrix$ is as defined in \eqref{Wdef}. Because $\dissipationMatrix(t)$ is positive semi-definite for all $t\in \timeInt$, so is $\tilde{W}(t)$.
\end{proof}

Note that even if $K=0$ in an \LTV \pHDAE system with quadratic Hamiltonian, after the transformation given in \Cref{th:pHDAEinv} 
the extra term $-\tilde{E}\tilde{K}$ with $\tilde{K}=V^{-1}\dot{V}$ will appear, and if an orthogonal change of basis is carried out in a system with $K=0$ then 
$\tilde{K}=V^{-1}\dot{V}$ is skew-symmetric. However, even if $K\neq0$, it has been shown in \cite{BeaMXZ18} that this term can be removed via a change of basis transformation that does not change the quadratic Hamiltonian.
\begin{lemma}
\label{pHDAErev}
Consider a \pHDAE system
\begin{align*}
    \tilde{E} \dot{\tilde{\state}} &= [(\tilde{J} -\tilde{R})\tilde{Q} -\tilde{E}\tilde{K})]\tilde{\state}  + (\tilde{G} -\tilde{P})\inpVar, \\
    \outVar &= (\tilde{G}+\tilde{P})^T\tilde{Q}\tilde{\state}     + (S+N)\inpVar
\end{align*}
with Hamiltonian $\tilde{\hamiltonian}(\tilde{\state}) = \tfrac{1}{2} \tilde{\state}^T \tilde{Q}^T \tilde{E} \tilde{\state}$ and $\tilde{K}\in \calC(\timeInt,\R^{\stateDim,\stateDim})$.
If $V_{\tilde{K}} \in \calC^1(\timeInt,\R^{\stateDim,\stateDim})$ is a pointwise invertible solution of the matrix differential equation
$\dot{V} = V \tilde{K}$ with the initial condition $V(t_0)=I$, then setting $\tilde{\state} = V_K^{-1}\state$ and  defining
\[
    E \vcentcolon= \tilde{E} V_K^{-1}, \
    Q \vcentcolon= \tilde{Q} V_K^{-1}, \
    G \vcentcolon= \tilde{G},\
    J \vcentcolon= \tilde{J}, \ 
    R \vcentcolon= \tilde{R}, \
    P \vcentcolon= \tilde{P}, 
\]
then the system
\begin{align*}
    E\dot{\state} &= ( J - R ) Q \state + (G-P)\inpVar, \\
    \outVar &= (G+ P)^T Q\state + (S+N)\inpVar
\end{align*}
is again a \pHDAE with the same Hamiltonian $\hamiltonian(\state) = \tilde{\hamiltonian}(\tilde{\state})$.
\end{lemma}
\begin{proof}
For a given matrix function $\tilde{K}$, the system $\dot{V} = V\tilde{K}$ always has a solution $V_K$ that is pointwise invertible.
The remainder of the proof follows by reversing the proof of \Cref{th:pHDAEinv} with $U=I$ and using that $\dot{V}_K V_K^{-1} = -V_K\ddt (V_K^{-1})$.
\end{proof}

Another important observation for linear \pHDAE systems of the form \eqref{eqn:pHDAE:linear} (with  $Q=I$) that is  important in the context of space discretization and model reduction is that the \pHDAE structure is invariant under Galerkin projection. 
\begin{corollary}
    \label{cor:pHDAEinv}
    Consider a \pHDAE system of the form \eqref{eqn:pHDAE:linear}  with  quadratic Hamiltonian \eqref{eqn:Hamiltonian:linear} and assume that $K=0$. If $V\in\R^{\stateDim, k}$ for some $k\in\mathbb{N}$, then the projected system in the variable $z=V\tilde z$,
\begin{align*}
    \widetilde{E} \dot{\widetilde{\state}} &= (\widetilde{J} -\widetilde{R})\widetilde{\state} + (\widetilde{G} -\widetilde{P})\inpVar, \\
    y &= (\widetilde{G}+\widetilde{P})^T \widetilde{\state} + (S +N)\inpVar,
\end{align*}
with projected matrix functions
\begin{align*}
    \widetilde{E} &\vcentcolon= V^T E V, &
    \widetilde{J} &\vcentcolon= V^T J V, &
    \widetilde{R} &\vcentcolon= V^T R V, &
    \widetilde{G} &\vcentcolon= V^T G, &
    \widetilde{P} &\vcentcolon= V^T P
\end{align*}
is still a \pHDAE with projected Hamiltonian $\widetilde{\hamiltonian}(\widetilde{\state}) = \tfrac{1}{2}\widetilde{\state}^T \widetilde{E}\widetilde{\state}$.
\end{corollary}

\begin{remark}\label{rem:galerkin}
Not that  for  \LTV \pHDAE systems with $K\neq 0$, strictly speaking, the invariance under Galerkin projection is violated, except if one incorporates a term $ VV^T$ between $E$ and $K$ 
so that the projected coefficient $K$  is approximated  by $\widetilde K= V^T ( K-\dot V) V$.
\end{remark}
The discussed invariance of a \pHDAE system under transformations allows to simplify the representation further. These simplifications are discussed in \Cref{sec:normalForm} for linear time-varying \pHDAE systems.

%%%%%%%%%%%%%%%%%%%%%%%%%%%%%%%%%%%%%%%%%%%%%%%%%%%%%%%%
\subsection{Geometric description of \pHDAE systems}\label{sec:dirac}
Port-Hamiltonian systems are often described through differential geometric structures known as Dirac structures, see \eg \cite{SchJ14}.
This viewpoint is extremely helpful in many aspects, in particular in the construction of time-discretization methods, see \Cref{sec:timedisc}. 
%We present first the basic definitions.
%
\begin{definition}
  Let $\mathcal F$ be a Euclidian vector space and $\mathcal E\vcentcolon=\mathcal F^*$ its dual space. Let $\bilp {\cdot,\cdot}$ be a bilinear form on $\mathcal F\times\mathcal E$ defined via
  \begin{equation*}
    \bilp{(f_1,e_1),(f_2,e_2)} \vcentcolon= \dualp{e_1}{f_2} + \dualp{e_2}{f_1},
  \end{equation*}
 where $\dualp{\cdot}{\cdot}$ is the standard duality pairing.
  Then a linear subspace $\mathcal D\subseteq\mathcal F\times\mathcal E$, such that $\mathcal D=\mathcal D^{\independent}$ with respect to $\bilp{\cdot\,,\cdot}$ is called \emph{Dirac structure} on $\mathcal F\times\mathcal E$.
If $(f,e)\in\mathcal D$, then $f$ and $e$ are called \emph{flow} and \emph{effort}, respectively.
\end{definition}
In finite dimension, one has a Dirac structure if $\dim\mathcal D=\dim\mathcal F$ and 
\begin{displaymath}
    \dualp{e}{f}=0 \qquad\text{ for all $(f,e)\in\mathcal D$},
\end{displaymath} 
see \cite{SchJ14}.

The concept of a Dirac structure can also be generalized to vector bundles. Let for this $\oplus$ denote the Whitney sum between vector bundles, see \eg \cite{Lan12}.

\begin{definition}
  Consider a state space $\mathcal Z$ and a vector bundle $\mathcal V$ over $\mathcal Z$ with fibers $\mathcal V_z$.
  A \emph{Dirac structure} over $\mathcal V$ is a subbundle $\mathcal D \subseteq \mathcal V \oplus \mathcal V^*$ such that, for all $z\in\mathcal Z$, $\mathcal D_z\subseteq\mathcal V_z\times\mathcal V_z^*$ is a linear Dirac structure.
\end{definition}
This definition generalizes the \emph{modulated Dirac structures}, see \cite{SchJ14},  where $\mathcal V=T\mathcal Z$ is the tangent bundle to $\mathcal Z$.
To associate a Dirac structure to the \pHDAE system~\eqref{eqn:pHDAE}, we assume that the system is autonomous, cf.~\Cref{rem:autono}, and first show the following lemma.
\begin{lemma}\label{lem:dirac}
Consider an autonomous  \pHDAE system~\eqref{eqn:pHDAE} with skew-symmetric $J:\mathcal Z\to\mathcal L(\mathcal V_z^*,\mathcal V_z)$. If $\mathcal D\subseteq\mathcal V\oplus\mathcal V^*$ is a vector subbundle with fibers defined by
  \begin{equation}\label{eqn:diracstructure}
    \mathcal D_z=\{(f,e)\in\mathcal V_z\times\mathcal V_z^*:f+J(x)e=0\},
  \end{equation}
then $\mathcal D$ is a Dirac structure.
\end{lemma}
\begin{proof}
  For generic $(f,e)\in\mathcal D_z$ and $(f',e')\in\mathcal V_z\times\mathcal V_z^*$, we have
  \begin{align*}
    \bilp{(f,e),(f',e')} &= \dualp{e}{f'} + \dualp{e'}{f} = \\
    &= \dualp{e}{f'}-\dualp{e'}{J e} = \dualp{e}{f'+Je'}.
  \end{align*}
  We show that $(f',e')\in\mathcal D_z$ if and only if $\dualp{e}{f'+Je'}=0$ for all $(f,e)\in\mathcal D_z$.
  If $(f',e')\in\mathcal D_z$, then $\dualp{e}{f'+Je'}=0$ holds for any $e\in\mathcal E$.
  If $(f',e')\notin\mathcal D_z$, then $f'+Je'\neq0$ and so there exists $e\in\mathcal E$ such that $\dualp{e}{f'+Je'}=1$, but $(f,e)\in\mathcal D_z$ with $f=-Je$.
\end{proof}
A Dirac structure for general \pHDAE systems is then constructed as follows, see \cite{MehM19}.
\begin{theorem}
\label{thm:diracStructure}
    Consider an autonomous  \pHDAE system of the form~\eqref{eqn:pHDAE}. Define the flow fiber $\mathcal{V}_{\state} \vcentcolon= \mathcal{F}_{\state}^\mathrm{s}\times\mathcal{F}_{\state}^{\mathrm{p}}\times\mathcal{F}_{\state}^{\mathrm{d}}$ for all $\state\in\mathcal{Z}$, where 
    \begin{align*}
        \mathcal{F}_{\state}^{\mathrm{s}} &\vcentcolon= E(\state)T_{\state} \mathcal{Z}\subseteq\R^\ell \text{ is the storage flow fiber},\\
        \mathcal{F}_{\state}^{\mathrm{p}} &\vcentcolon=\R^m \text{ is the port flow fiber, and}\\ \mathcal{F}_{\state}^{\mathrm{d}} &\vcentcolon=\R^{\ell+m} \text{ is the dissipation flow fiber.}
    \end{align*}
  Let us partition $f=(f_{\mathrm{s}},f_{\mathrm{p}},f_{\mathrm{d}})\in\mathcal{V}$ and $e=(e_\mathrm{s},e_\mathrm{p},e_\mathrm{d})\in\mathcal{V}^*$.
  Then the subbundle $\mathcal{D}\subseteq\mathcal{V}\oplus\mathcal{V}^*$ with
  \begin{align*}
    \mathcal{D}_{\state} = \left \{ (f,e)\in\mathcal{V}_{\state}\times\mathcal{V}_{\mathrm{\state}}^* \mid
    f + \begin{bmatrix}\Gamma(z) & I_{\ell+m} \\ -I_{\ell+m} & 0\end{bmatrix} e = 0
    \right \}
  \end{align*}
  is a Dirac structure over $\mathcal{V}$.
  Furthermore, the system of equations
  \begin{equation}\label{eq:feCond}
    \begin{aligned}
      f_{\mathrm{s}} &= -E(z)\dot{\state}, &\ e_{\mathrm{s}} &= \eta(z), \\
      f_{\mathrm{p}} &= y, &\ e_\mathrm{p} &= u, \\
      e_{\mathrm{d}} &= -W(z)f_d, &\ (f,e)&\in\mathcal{D}_{\state}
    \end{aligned}
  \end{equation}
  is equivalent to the original \pHDAE, and $\dualp{e}{f}=0$ represents the power balance equation.
\end{theorem}
\begin{proof} 
    \Cref{lem:dirac} implies that  $\mathcal D$ is a Dirac structure. Writing  \eqref{eqn:pHDAE} in compact form 
    \begin{equation*}
        \begin{bmatrix}E(\state)\dot{\state} \\ -\outVar\end{bmatrix} = \begin{bmatrix} (\Gamma(\state)-W(\state))\end{bmatrix}\begin{bmatrix} \eta(\state) \\ \inpVar\end{bmatrix},
    \end{equation*}
    it follows that $(f,e)\in\mathcal{D}_{\state}$ can be written as
    \begin{equation*}
        \begin{bmatrix}-f_\mathrm{s} \\ -f_\mathrm{p}\end{bmatrix} = \Gamma(\state)e_\mathrm{s} + e_\mathrm{p}, \qquad
        f_\mathrm{d} = (e_\mathrm{s},e_\mathrm{p}).
    \end{equation*}
    Together with the conditions \eqref{eq:feCond} this is equivalent to
    \begin{align*}
      \begin{bmatrix} E(\state)\dot{\state} \\ -\outVar\end{bmatrix} &= (\Gamma(\state)-W(\state))\begin{bmatrix}\eta(\state) \\ \inpVar\end{bmatrix}, \\
      f &= \{-E(\state)\dot{\state},\, \outVar,\; \begin{smallbmatrix}\eta(\state) \\ \inpVar\end{smallbmatrix}\}, \\
      e &= \{\eta(\state),\; \inpVar,\; -W(\state)\begin{smallbmatrix}\eta(\state) \\ \inpVar\end{smallbmatrix}\},
    \end{align*}
    which is exactly the compact form of the \pHDAE.
    Finally, note that the equation $\dualp{e}{f}=0$ can be written as
    \begin{align*}
      0 &= \dualp{\eta(\state)}{-E(\state)\dot{\state}} + \dualp{\inpVar}{\outVar} + \dualp{-W(\state)\begin{bmatrix}\eta(\state)\\ \inpVar\end{bmatrix}}{\begin{bmatrix}\eta(\state)\\ \inpVar\end{bmatrix}} = \\
      &= \ddt\hamiltonian(\state) + \outVar^T\inpVar - \begin{bmatrix}\eta(\state)\\ \inpVar\end{bmatrix}^TW(\state)\begin{bmatrix}\eta(\state)\\ \inpVar\end{bmatrix},
    \end{align*}
    which is the power balance equation.
\end{proof}

Note that, if we conversely want to retrieve a \pHDAE system from a Dirac structure, then the additional conditions \eqref{eq:feCond} and the definition of $\hamiltonian(\state)$ are needed.

\begin{remark}
    \label{rem:geoflow}
    Since \dHDAE systems generalize classical Hamiltonian systems, an immediate question is whether the associated flow has a geometric structure, such as symplecticity or generalized orthogonality when there is no dissipation, see \cite{HaiLW02}. While for \pHODE systems this is well established,  see \eg \cite{CelH17,SchJ14}, for \LTV \dHDAE systems this has only been established recently in \cite{Sch19}, and in a more general setting in \cite{KunM22_ppt}.
\end{remark}

\begin{remark}
\label{rem:MasVanderschaft}
In the linear time-invariant case an extension that addresses different Lagrange and Dirac structures and their relation has been introduced in \cite{SchM18}. 
Ignoring the dissipation term as well as the inputs and outputs, the corresponding equation has the form
\begin{displaymath}
    KP \dot{\state} = LS \state
\end{displaymath}
where $S,P,L,S \in \R^{\stateDim,\stateDim}$ satisfy $S^TP=P^TS$ with $\rank \begin{smallbmatrix} S\\ P \end{smallbmatrix}=\stateDim$, 
\ie the columns of $\begin{smallbmatrix} S\\ P \end{smallbmatrix}$,
form a \emph{Lagrangian subspace},  
and $K^TL=-L^TK$ with $\rank \begin{smallbmatrix} K\\ L \end{smallbmatrix} = \stateDim$,
\ie the columns of $\begin{smallbmatrix} K\\ L \end{smallbmatrix}$ are associated with a Dirac structure. A further generalization is discussed in \cite{GerHR21}.

Clearly, if $K=S=I$, then $P=P^T$ and $S=-S^T$ and we are in the case of \dHDAE systems with $E=E^T=P$, $S=-S^T=J$, where the extra condition  $E\geq 0$ has to be assumed. If $K=P=I$ then we are in the classical case of \pHODE systems with $S=S^T=Q$ and $L=-L^T=J$.
\end{remark}

\begin{remark}\label{rem:telegen}
The representation of \pHODE or \pHDAE systems via a Dirac structure is an extension of Tellegen's theorem, see \eg \cite{DuiMSB09} and it also shows the relation to implicit Lagrange systems, see \eg \cite{YosM06a,YosM06b}, as well as gradient systems, see \eg
\cite{OetG97,Oet06}.
\end{remark}

%%%%%%%%%%%%%%%%%%%%%%%%%%%%%%%%%%%%%%%%%%%%%%%%%%%%%%%%%%%%%%%%%
\subsection{Structure-preserving interconnection}
\label{sec:interconnect}
Another key property of the \pH model class that is particularly important for modularized, network-based modeling across physical domains is that it is preserved under  interconnection.
To see this, consider two autonomous \pHDAEs (cf.~\Cref{rem:autono}) 
\begin{align*}
  E_i\dot{\state}_i &= (J_i-R_i)\eta_i + (B_i-P_i)\inpVar_i, \\
  \outVar_i &= (B_i+P_i)^T\eta_i + (S_i-N_i)\inpVar_i,
\end{align*}
of the form \eqref{eqn:pHDAE} with Hamiltonians $\mathcal H_i$, for $i=1,2$, and assume that inputs and outputs  satisfy a linear interconnection relation 
\begin{displaymath}
    \begin{bmatrix} M_{11} & M_{12} \\ M_{21} & M_{22} \end{bmatrix}\begin{bmatrix} \inpVar_1\\ \inpVar_2 \end{bmatrix} +
\begin{bmatrix} L_{11} & L_{12} \\ L_{21} & L_{22} \end{bmatrix}\begin{bmatrix} y_1\\ y_2 \end{bmatrix} = \begin{bmatrix} 0 \\ 0 \end{bmatrix}.
\end{displaymath}
Then the interconnected system can be written as a \pHDAE of the form
\begin{align*}
    \mathcal{E}\dot{\state} &= (\mathcal{J}-\mathcal{R})\eta + \mathcal{G}\inpVar,\\
    \outVar &= \mathcal{G}^\top\eta
\end{align*}
with 
\begin{align*}
    \state &= \begin{bmatrix}
        \state_1^T &
        \state_2^T &
        \state_3^T &
        \state_4^T &
        \state_5^T &
        \state_6^T
    \end{bmatrix}^T,\\
    \eta &= \begin{bmatrix}
        \eta_1^T & \eta_2^T & \eta_3^T & \eta_4^T & \eta_5^T & \eta_6^T & 0 & 0
    \end{bmatrix}^T
\end{align*}
with new state variables 
$\state_3\vcentcolon= \eta_3 \vcentcolon= \inpVar_1$, $\state_4\vcentcolon= \eta_4 \vcentcolon= \inpVar_2$,  $\state_5 \vcentcolon= \eta_5 \vcentcolon= \outVar_1$,  $\state_6 \vcentcolon= \eta_6 \vcentcolon= \outVar_2$, matrix functions 
\begin{align*}
    \mathcal{E} &= \begin{bmatrix} 
        E_1 & 0 & 0 & 0 & 0 & 0\\ 
        0 & E_2 & 0 & 0 & 0 & 0\\ 
        0 & 0 & 0 & 0 & 0 & 0\\ 
        0 & 0 & 0 & 0 & 0 & 0\\ 
        0 & 0 & 0 & 0 & 0 & 0\\ 
        0 & 0 & 0 & 0 & 0 & 0\\ 
        0 & 0 & 0 & 0 & 0 & 0\\ 
        0 & 0 & 0 & 0 & 0 & 0
    \end{bmatrix}, &
    \mathcal{G} &= \begin{bmatrix} 0 & 0 \\ 0 & 0 \\0 & 0 \\0 & 0 \\ I& 0 \\ 0& I\\ 0 & 0 \\ 0& 0\end{bmatrix},
\end{align*}
and
\begin{align*}
    \mathcal{J}-\mathcal{R} = \begin{bmatrix} \phantom{-}J_1-R_1  & 0 & G_1 -P_1&0  &0 & 0 & 0 & 0\\
   0& \phantom{-}J_2-R_2 & 0 & G_2-P_2 & 0 & 0 & 0 & 0\\ -G_1^T-P_1^T & 0 & N_{11}-S_{11} & 0 & I & 0 &-M_{11}^T & -M_{21}^T \\
   0 & -G_2-P_2^T & 0 & N_{22}-S_{22} & 0 & I & -M_{12}^T & -M_{22}^T \\
   0 & 0 & -I &0 & 0 & 0 & -L_{11}^T & -L_{21}^T \\
   0 & 0 &0 &-I & 0 & 0 & -L_{12}^T & -L_{22}^T \\
   0 & 0 & M_{11} & M_{12} & L_{11} & L_{12} & 0 & 0 \\
    0 & 0 & M_{21} & M_{22} & L_{21} & L_{22} & 0 & 0 \end{bmatrix},
\end{align*}
and combined Hamiltonian  $\hamiltonian=\hamiltonian_1+\hamiltonian_2$.
It is clear that the structural conditions of the coefficients are still satisfied. 

Unfortunately, due to the extension by extra state variables the dimension of the state space may substantially increase.
However, if we assume that the interconnection is power-preserving (\eg~if $Mu+Ny=0$ defines a Dirac structure for $(\outVar,\inpVar)$), cf.~\Cref{sec:dirac}, then index reduction and removal of certain parts, see \Cref{sec:normalForm}, can usually be applied to make the system smaller.

If we restrict ourselves to linear \pHDAEs of the form
\begin{align*}
    E_i\dot{\state}_i &= (J_i-R_i)\state_i + G_i\inpVar_i,\\
    \outVar_i &= G_i^T\state_i,
\end{align*}
see \Cref{sec:noQ,sec:nofeed}, then we can achieve the interconnection in a more condensed form. Assume an ouput-feedback of the form $\inpVar = F\outVar + w$ 
with aggregated variables $\inpVar \vcentcolon= \begin{bmatrix}\inpVar_1^T & \inpVar_2^T\end{bmatrix}^T$ and $\outVar \vcentcolon= \begin{bmatrix} \outVar_1^T & \outVar_2^T\end{bmatrix}^T$. Define $\state \vcentcolon= \begin{bmatrix} \state_1^T & \state_2^T\end{bmatrix}^T$, $J\vcentcolon= \diag(J_1,J_2)$, $R\vcentcolon= \diag(R_1,R_2)$, $G = \diag(G_1,G_2)$. Then the coupled system has the form
\begin{align*}
    E\dot{\state} &= (J-R+GFG^T)\state + Gw,\\
    \outVar &= G^T\state,
\end{align*}
which is a \pHDAE, whenever $R-GF_{\mathrm{sym}}G^T$ is positive semi-definite, where $F_{\mathrm{sym}} \vcentcolon= \tfrac{1}{2}(F+F^T)$. A sufficient condition to retain the \pH structure is thus to require that $F_{\mathrm{sym}}$ is negative semi-definite corresponding to a potentially dissipative component of the interconnection.

%=============================================================%
\section{Condensed forms for dHDAE and pHDAE systems}
\label{sec:normalForm}
To analyze the solution behavior of \dHDAE or \pHDAE systems, it is convenient to study canonical or condensed forms and to reformulate the system by removing high index and redundant parts, as it was done for general \DAE systems in \Cref{sec:DAEtheory}. But the general condensed forms do not reflect the structure and, in particular, canonical forms like the Kronecker or Weierstra{\ss} form (cf.~\Cref{th:kcf,thm:WCF}) are obtained under transformations that may be arbitrarily ill-conditioned. 

In this section, to overcome some of these disadvantages, we present condensed forms for \dHDAE and \pHDAE systems under (pointwise) orthogonal transformations that preserve the structure. The resulting condensed forms are close to normal forms and display all the important information, but they are typically not canonical.

%=============================================================%
\subsection{Condensed forms for dHDAE systems}
We first present a condensed form for linear time-varying \dHDAE systems of the form (leaving off the argument $t$)
\begin{equation}
    \label{timevardhdae}
    E \dot{\state} = (J-R - EK)\state,\qquad \state(t_0)=\state_0,
\end{equation} 
which is a special case of the  form for \pHDAE systems presented in \cite{Sch19}. 

\begin{lemma} 
\label{lem:varregul}
Under some constant rank assumptions, for a \dHDAE system of the form~\eqref{timevardhdae}, there exists a pointwise real orthogonal matrix function $Z\in\mathcal{C}(\timeInt,\R^{\stateDim,\stateDim})$, such that with $\check{\state}= Z^T\state$ and multiplying the system with $Z^T$ from the left, the transformed system
\begin{equation}
    \label{trfdhadae}
    \check{E} \dot{\check{\state}} = (\check{J} -\check{R} -    \check{E}\check{K})\check{\state},\qquad \check{\state}(t_0) = \check{\state}_0,
\end{equation}
with 
\begin{align*} 
    \check{E} &\vcentcolon= Z^T EZ, &
    \check{J} &\vcentcolon= Z^T JZ, &
    \check{R} &\vcentcolon= Z^T R Z, &
    \check{K} &\vcentcolon= Z^T K Z- Z^T \dot{Z}, 
\end{align*}
is still a \dHDAE system with matrix functions in the block form 
\begin{equation}
    \label{stcvardh}
	\begin{gathered}
        \check{E} = \begin{bmatrix}
            E_{11} & E_{12} & 0 & 0 &0\\
            E_{21} & E_{22} & 0 & 0 &0\\
            0 & 0 & 0 &0 &0\\
            0 & 0 & 0 &0 &0\\
            0 & 0 & 0 &0&0
        \end{bmatrix}, \quad \check{J}-\check{R} = \begin{bmatrix}
            J_{11}-R_{11} &  J_{12}-R_{12}& J_{13}-R_{13} & J_{14}&0\\
            J_{21}-R_{21} & J_{22}-R_{22} & J_{23}-R_{23} & 0 &0\\
            J_{31}-R_{31} & J_{32}-R_{32} & J_{33} -R_{33} & 0 & 0\\
            J_{41} & 0 & 0  & 0 & 0\\
            0& 0 &  0 & 0& 0
        \end{bmatrix},\\
        \check{K} = \begin{bmatrix}
            K_{11} & K_{12} & 0 & 0 & 0\\
            K_{21} & K_{22} & 0 & 0 & 0\\
            K_{31} & K_{32} & K_{33} & K_{34} & K_{35}\\
            K_{41} & K_{42} & K_{43} & K_{44} & K_{45}\\
            K_{51} & K_{52} & K_{53} & K_{54} & K_{55}
        \end{bmatrix}
    \end{gathered}
\end{equation}
and block sizes $n_1=n_4,n_2,n_3,n_5$. (Note that blocks may be void).
The matrix function $\begin{smallbmatrix} E_{11}& E_{12}\\ E_{21} & E_{22} \end{smallbmatrix}$ is pointwise symmetric positive
definite, 
the matrix functions $J_{33}-R_{33}$, with $R_{33}\geq 0$ and $J_{41}=-J_{14}^T$ are pointwise nonsingular, and the block 
\begin{displaymath}
\begin{bmatrix}
E_{11} & E_{12} \\
E_{21} & E_{22} \end{bmatrix}\begin{bmatrix}
K_{11} & K_{12} \\
K_{21} & K_{22} \end{bmatrix}
\end{displaymath}
is pointwise skew-symmetric.
\end{lemma}

\begin{proof}
We present a constructive proof that is similar to that for the constant-coefficient case in \cite{AchAM21} and can be directly implemented as a numerical algorithm. The details are presented in \Cref{alg:staircaseDHDAE}. We make the following remarks for the algorithm. First, the smooth rank revealing decompositions can be computed using \Cref{thm:smoothRankRevealingDecomposition}. Second, the zero block structure in~$\hat{R}$ in \textbf{Step 2} follows from the positive semi-definiteness of $R$ and that the second and third block row in $\hat{K}$ do not destroy this structure since the multiplication by $\hat{E}$ puts zeros in these positions. Third, note that the transformed matrix function $K$ in \textbf{Step 3} does not destroy the structure. The extra zeros in the first two rows and the skew-symmetry of the extra term arising from the change of basis follow from the skew-adjointness of the operator $\check E\ddt -(\check J-\check E\check K)$ in \Cref{def:pHDAE:linear} (for the case that $Q=I$).
\end{proof}

\begin{algorithm}
    \caption{Staircase Algorithm for linear \dHDAE}
    \label{alg:staircaseDHDAE}
    \begin{flushleft}
    	\textbf{Input:} Pair of \dHDAE matrix functions $(E,J-R-EK)$\\
    \textbf{Output:} Matrix function $Z$ and condensed block form~\eqref{stcvardh}
    \end{flushleft}
    
    \begin{description}
        \item[Step 1] Assume that $E$ has pointwise constant rank $n_1$ in $\timeInt$. Then perform a smooth full rank decomposition
        \begin{equation}
            \label{firststep}
            E = Z_1 \begin{bmatrix} \tilde{E}_{11} & 0 \\ 0 & 0 \end{bmatrix} Z_1^T,
        \end{equation}
        with $Z_1$ pointwise orthogonal and $\tilde{E}_{11}$ pointwise positive definite of size $\tilde{n}_1 \times \tilde{n}_1$ (or $\tilde{n}_1=0$).
        Set $\tilde{E} \vcentcolon= Z_1^T E Z_1$, $\tilde{J} \vcentcolon= Z_1^T J Z_1$, $\tilde{R} \vcentcolon= Z_1^T R Z_1$, and $\tilde{K} \vcentcolon= Z_1^T(KZ_1+\dot{Z}_1)$ with
        \begin{gather*}
            \tilde{J} = \begin{bmatrix}
                \tilde{J}_{11} & -\tilde{J}_{21}^T \\
                \tilde{J}_{21} & \phantom{-}\tilde{J}_{22}
            \end{bmatrix}, \quad
            \tilde{R} = \begin{bmatrix}
                \tilde{R}_{11} & \tilde{R}_{21}^T \\
                \tilde{R}_{21} & \tilde{R}_{22} 
            \end{bmatrix}, \quad
            \tilde{K} = \begin{bmatrix}
                \tilde{K}_{11} & \tilde{K}_{12} \\
                \tilde{K}_{21} & \tilde{K}_{22} 
            \end{bmatrix}.
        \end{gather*}
        \item[Step 2] If $\tilde{n}_1<n$ then, assuming constant rank of $\tilde{J}_{22}-\tilde{R}_{22}$ in $\timeInt$, apply a full rank decomposition under pointwise orthogonal congruence 
        \begin{displaymath}
            Z_{22}^T(\tilde{J}_{22}-\tilde{R}_{22})Z_{22} = \begin{bmatrix} \tilde{\Sigma}_{22} & 0 \\ 0 & 0\end{bmatrix},
        \end{displaymath}
        with $\tilde{\Sigma}_{22}$ of size $\tilde{n}_2\times\tilde{n}_2$ pointwise invertible or $\tilde{n}_2=0$. Define the matrix functions $Z_2\vcentcolon=\diag(I,Z_{22})$, $\hat{E} \vcentcolon= Z_2^T\tilde{E}Z_2$, $\hat{J} \vcentcolon= Z_2^T\tilde{J}Z_2$, $\hat{R} \vcentcolon= Z_2^T\tilde{R}Z_2$, and $\hat{K} \vcentcolon= Z_2^T \tilde{K}Z_2 - Z_2^T\dot{Z}_2$
        with
        \begin{gather*}
            \qquad\qquad\hat{J} = \begin{bmatrix}
                \hat{J}_{11} & -\hat{J}_{21}^T & -\hat{J}_{31}^T\\
                \hat{J}_{21} & \phantom{-}\hat{J}_{22} & \phantom{-}0\\
                \hat{J}_{31} & \phantom{-}0 & \phantom{-}0
            \end{bmatrix}, \quad
            \hat{R} = \begin{bmatrix}
                \hat{R}_{11} & \hat{R}_{21}^T & \phantom{\hat{R}}0\\
                \hat{R}_{21} & \hat{R}_{22} & \phantom{\hat{R}}0\\
                0 & 0 & \phantom{\hat{R}}0
            \end{bmatrix}, \quad
            \hat{K}  \begin{bmatrix}
                \hat{K}_{11} & \hat{K}_{12} & \hat{K}_{13} \\
                \hat{K}_{21} & \hat{K}_{22} & \hat{K}_{23} \\
                \hat{K}_{31} & \hat{K}_{32} & \hat{K}_{33},
            \end{bmatrix}
        \end{gather*}
        and $\hat{J}_{22}-\hat{R}_{22}$ pointwise nonsingular.
        \item[Step 3] If $\tilde{n}_3 \vcentcolon= n -\tilde{n}_1 -\tilde{n}_2>0$ then, assuming that $\hat{J}_{31}$ has constant rank in $\timeInt$,  perform a pointwise full rank decomposition
        \begin{displaymath}
            \hat{J}_{31} = U_{31} \begin{bmatrix} \Sigma_{31} & 0\\ 0 & 0 \end{bmatrix} V^T_{31},
        \end{displaymath}
        with~$\Sigma_{31}$ of size $n_1 \times n_1$ pointwise nonsingular (or $n_1 =0$). Set
        \begin{displaymath}
            Z_3 \vcentcolon= \begin{bmatrix}
                V_{31}^T & & \\
                & I & \\
                & & U_{31}
            \end{bmatrix}
        \end{displaymath}
        Then $\check{E} \vcentcolon= Z_3^T \hat{E}Z_3$, $\check{J}\vcentcolon= Z_3^T\hat{J}Z_3$, $\check{R} \vcentcolon= Z_3^T\hat{R}Z_3$, $\check K= Z_3^T \tilde K  Z_3-  Z_3^T \dot Z_3$
have the desired form with $n_2:=\tilde n_1 -n_1$, $n_3 :=\tilde n_2$, $n_4=n_1$, $n_5:=\tilde n_3 -n_4$. 
        \item[Step 4] Set $Z \vcentcolon= Z_3Z_2Z_1$
    \end{description}
\end{algorithm} 

From the condensed form we immediately obtain a characterization of existence and uniqueness of solutions.
\begin{corollary}
Consider a \dHDAE initial value problem of the form \eqref{timevardhdae} in the  normal form \eqref{stcvardh}.
\begin{itemize}
    \item [(i)] The initial value problem \eqref{timevardhdae} is uniquely solvable (for consistent initial values)  if and only if  $n_5=0$. If $n_5\neq 0$, then ${z}_5$ can be chosen arbitrarily.
    \item [(ii)] The solution of the initial value problem is not unique (for consistent initial values) if and only if the  matrix functions $E, J, R, EK$ have a common kernel.
    \item [(iii)] The strangeness index is $0$ if and only if $n_1=n_4=0$.
    Otherwise the strangeness index is $1$ if and only if $n_1=n_4>0$.
\end{itemize}
\end{corollary}
\begin{proof} We prove each item separately.
\begin{enumerate}
    \item[(i)] Note that the last equation can be omitted and the variable $\state_5$ is arbitrary. So the solution is unique if and only if $n_5=0$.
    \item[(ii)] This follows trivially from (i). 
    \item[(iii)] If $n_4>0$, then the fourth equation states that $z_1=0$ and the first equation yields that $z_4$ depends via the term $E_{12} \dot z_2$ on the derivative of $z_2$ and hence if $n_4\neq 0$, then the system has strangeness index $1$. Finally the separated subsystem 
\begin{align*}
    %\label{indexonesub}
\qquad\qquad\begin{bmatrix}
E_{22} & 0 \\
0 & 0 \end{bmatrix}
    \begin{bmatrix} \dot{\state}_2 \\ \dot{\state}_3 \end{bmatrix} = 
\left (\begin{bmatrix}
 J_{22}-R_{22} & J_{23}-R_{23} \\
 J_{32}-R_{32} & J_{33} -R_{33}
\end{bmatrix}  -
\begin{bmatrix}
E_{21} K_{12} +E_{22} K_{22}&  0 \\
0 & 0 \end{bmatrix}
\right )
\begin{bmatrix} \state_2 \\ \state_3 \end{bmatrix}
\end{align*}
has $E_{22}$ positive definite and $J_{33} -R_{33}$ invertible and hence is of strangeness index zero.\qedhere
\end{enumerate}
\end{proof}

\begin{remark}\label{rem:constrank}
In the staircase Algorithm we have made several constant rank assumptions. If these are not satisfied, then we can  partition the time interval $\timeInt$ into smaller subintervals where the ranks are constants and record the points where these rank changes happen.  For a more detailed discussion of such \emph{hybrid \DAE systems} with rank changes in the characteristic values, see
\cite{HamM08,KunM06,KunM18}. 
\end{remark}
In the linear time-invariant case (with $K=0$, $Q=I$) we have the following corollary.
\begin{corollary} \label{cor:regul:pH}
For every \LTI \dHDAE system of the form 
\begin{equation}\label{dhdae}
    E\dot{\state} = (J-R)\state,
\end{equation}
with $E, J,R\in \R^{\stateDim,\stateDim}$, $J=-J^T$, $R=R^T\geq 0$, and $E^T=E \geq 0$, there exists a real orthogonal matrix $Z\in \R^{\stateDim,\stateDim}$, such that 
with 
\begin{align*}
    \check{E} &\vcentcolon= Z^TEZ, &
    \check{J} &\vcentcolon= Z^TJZ, &
    \check{R} &\vcentcolon= Z^TRZ, &
    \check{\state} &\vcentcolon= Z^T\state,
\end{align*}
the system $\check E\dot {\check z} =(\check J -\check R) \check z$ is still a \dHDAE with block matrices 
\begin{align}
    \label{stcdh}
    \check{E} = \begin{bmatrix}
E_{11} & E_{12} & 0 & 0 &0\\
E_{21} & E_{22} & 0 & 0 &0\\
0 & 0 & 0 &0 &0\\
0 & 0 & 0 &0 &0\\
0 & 0 & 0 &0&0\end{bmatrix},\quad \check{J}-\check{R} = \begin{bmatrix}
J_{11}-R_{11} &  J_{12}-R_{12}& J_{13}-R_{13} & J_{14}&0\\
J_{21}-R_{21} & J_{22}-R_{22} & J_{23}-R_{23} & 0 &0\\
J_{31}-R_{31} & J_{32}-R_{32} & J_{33} -R_{33} & 0 & 0\\
J_{41} & 0 & 0  & 0 & 0\\
0& 0 &  0 & 0& 0\end{bmatrix},
\end{align}
and block sizes $n_1=n_4,n_2,n_3,n_5$. Note that some of the blocks may be void. The matrix $\begin{smallbmatrix} E_{11}& E_{12}\\ E_{21} & E_{22} \end{smallbmatrix}$ is symmetric positive
definite, and the matrices $J_{33}-R_{33}$ and $J_{41}=-J_{14}^T$ are nonsingular.
\end{corollary}
\begin{proof}
The proof follows directly from the time varying case and setting $K=0$.
\end{proof}
We also have the corresponding existence and uniqueness result formulated with the (Kronecker) index.
\begin{corollary}
Consider a \dHDAE initial value problem of the form \eqref{dhdae} in the  staircase form \eqref{stcdh}.
\begin{itemize}
    \item[(i)] The initial value problem \eqref{dhdae} is uniquely solvable (for consistent initial values)  if and only if  $n_5=0$. If $n_5\neq 0$, then $\tilde{z}_5$ can be chosen arbitrarily.
    \item[(ii)] The pencil $\lambda E- (J-R)$ is non-regular if and only if the three coefficients $E,J,R$ have a common kernel.
    \item[(iii)] The (Kronecker) index is zero if and only if $n_1=n_4=0$ and $n_3=0$. It is one if and only if $n_1=n_4=0$ and $n_3>0$, and otherwise the (Kronecker) index is two if and only if $n_1=n_4>0$.
\end{itemize}
\end{corollary}
\begin{proof}
The proof is as in the variable coefficient case (with $K=0$), just observing the different counting between strangeness and Kronecker index, see \cite{Meh15}.
\end{proof}
\begin{remark}
\label{rem:common}
The fact that the \dHDAE system is not uniquely solvable if and only if the coefficients have a common nullspace is remarkable. It has significant importance in model evaluation since it allows to check this (pointwise) via the singular value decomposition. Moreover, in the constant-coefficient case, this allows to compute the distance to the nearest singular pencil, see \cite{GugM21_ppt,MehMW21}. For general pencils, and even for pencils with just the symmetry structure of a \dHDAE system, this property does not hold. Consider the pencil $\lambda E-J$ with 
\begin{displaymath}
    E=E^T= \begin{bmatrix} 0& 0 & 1 \\ 0 & 0 & 0 \\ 1 & 0 & 0 \end{bmatrix} \qquad\text{and}\qquad J=-J^T= \begin{bmatrix} 0 & 0 & 0 \\ 0 & 0 & 1 \\ 0 & -1 & 0\end{bmatrix}.
\end{displaymath}
The pencil is not regular, since $\det (\lambda E-J)=0 $ for any $\lambda\in\C$. However, $E$ and $J$ do not a have a common nullspace. The problem here is that $E=E^T$ is not semi-definite. 
In this case the singularity arises  from higher dimensional singular blocks in the Kronecker form.
\end{remark}

\begin{remark} 
\label{rem:index2} As the proof of \Cref{lem:regul} shows, the numerical computation of the condensed forms \eqref{stcvardh} or \eqref{stcdh} for a given \dHDAE system requires a sequence of three (smooth) full rank decompositions. Unfortunately, these  rank decisions may be sensitive under perturbations; see \eg \cite{ByeMX07} where the construction of general staircase forms and the challenges are discussed. However, in  contrast to general unstructured staircase forms, we see that for \dHDAE systems the number of steps is limited to three and often (as in all the examples discussed in \Cref{sec:examples}) the first step of performing a full rank decomposition of $E$ is not necessary.

Note also that the maximal strangeness index is one (the maximal (Kronecker) index of a \LTI \dHDAE is two), which is of great advantage in iterative solution methods, see~\cite{GudLMS21_ppt} and \Cref{sec:linsystem}, as well as time-discretization methods for \DAE systems, see \cite{HaiW96,KunM06}, and \Cref{sec:timedisc}.
\end{remark}

\begin{example}
\label{ex:flowcondensed}
To illustrate the construction of the condensed forms, consider the \dHDAE resulting from the fluid flow example discussed in \Cref{sec:navierstokes}. In more detail, consider an instationary incompressible fluid flow prescribed in terms of velocity $v\colon\Omega\times \timeInt\to\R^2$ and pressure $p\colon\Omega\times\timeInt\to\R$ on the spatial domain $\Omega=(0,1)^2$ with boundary $\partial \Omega$ for the time period $\timeInt = [0,T]$, that is driven by external forces $f\colon\Omega\times\timeInt\to\R^2$ and has dynamic viscosity $\nu>0$, see \Cref{sec:navierstokes}. The system is closed by non-slip boundary conditions and an appropriate initial value $v^0$ for the velocity. Spatial discretization by a  finite difference method on a uniform staggered grid with the semi-discretized  velocity  $v_h(t)\in\R^{n_v}$ and pressure vectors $p_h(t)\in\R^{n_p}$, $t\in \timeInt$, leads to a \pHDAE system for the state $\state = \begin{bmatrix} v_h^T & p_h^T\end{bmatrix}^T$  given by
\begin{equation}
    \label{StokesDAE}
    \begin{aligned}
    \begin{bmatrix}I & 0\\0&0\end{bmatrix}\begin{bmatrix}\dot v_h \\ \dot p_h \end{bmatrix} &=\left(\begin{bmatrix}A_S & B \\-B^T&0\end{bmatrix}-\begin{bmatrix}A_H&0\\0&0\end{bmatrix}\right)
\begin{bmatrix}v_h \\  p_h \end{bmatrix} + \begin{bmatrix}G_1\\0\end{bmatrix}\inpVar,\\
\outVar &= \begin{bmatrix}G_1^T & 0\end{bmatrix}\begin{bmatrix}v_h \\  p_h \end{bmatrix}.
\end{aligned}
\end{equation}
The matrix $B^T$ has full row rank if the freedom in the pressure is removed. The initial conditions are $v_h(0)=v_{h}^0$ and consistently $p_h(0)=p_{h}^0$. The input~$\inpVar$ with input matrix $ G_1\in\R^{n_v\times m}$ results from the external forces. System~\eqref{StokesDAE} is an \LTI \pHDAE of (Kronecker) index two. To obtain the condensed form one does not have to carry out the first and second step of \Cref{alg:staircaseDHDAE} but only \textbf{Step 3}, \ie, the splitting of the discrete divergence operator $B^T$  by performing, \eg a Leray projection as in \cite{HeiSS08} or a full rank decomposition,
\begin{align*}
    B^TV= \begin{bmatrix}B_1 & 0 \end{bmatrix}, 
\end{align*}
with nonsingular matrix $B_1$ and an orthogonal matrix $V\in\mathbb{R}^{n_p,n_p}$. Performing a congruence transformation  we get a system in condensed form
\begin{align*}
    \begin{bmatrix} I&0&0\\0&I&0\\0&0&0\end{bmatrix}
    \begin{bmatrix}\dot{\state}_1\\\dot{\state}_2\\\dot{\state}_3\end{bmatrix}
    &=\left(\begin{bmatrix}J_{11}&J_{12}& B_1\\ -J_{12}^T&J_{22}&0\\ -B_1^T&0&0\end{bmatrix}-\begin{bmatrix}R_{11}&R_{12}&0\\ R_{12}^T& R_{22}&0\\0&0&0\end{bmatrix}\right)\begin{bmatrix}\state_1\\\state_2\\\state_3\end{bmatrix} +\begin{bmatrix}
G_1\\G_2\\0\end{bmatrix}\inpVar,\\
    \outVar &=\begin{bmatrix}G_1^T&G_2^T&0\end{bmatrix}
\begin{bmatrix}\state_1\\\state_2\\\state_3\end{bmatrix}.
\end{align*}
This yields $\state_1=0$ as $B_1$ is invertible and the first equation yields the (hidden) algebraic constraint $ B_1 \state_3=(-J_{12}+R_{12})\state_2-G_1\inpVar$ as well as a consistency condition for the initial value which relates the initial condition for $\inpVar$ and $\state_2$ to that for $\state_3$.
\end{example}

\begin{example}\label{ex:poroindex}
    Consider the multiple-network poroelasticity problem discussed in \Cref{sec:poro}. Ignoring the boundary terms and permuting the rows and columns, we obtain a \dHDAE system of the form
    \begin{align}
        \label{eqn:poroStaircase}
        \begin{bmatrix}
        K_{\mathrm{u}} & 0 & 0\\
        0 & M_{\mathrm{p}} & 0\\
        0 & 0 & 0
        \end{bmatrix}\begin{bmatrix}
            \dot{\state}_1\\
            \dot{\state}_2\\
            \dot{\state}_3
        \end{bmatrix} = \begin{bmatrix}
            0 & 0 & K_{\mathrm{u}}\\
            0 & -K_{\mathrm{p}} & -D\\
            -K_{\mathrm{u}} & D^T & 0
        \end{bmatrix}\begin{bmatrix}
            \state_1\\
            \state_2\\
            \state_3
        \end{bmatrix},
    \end{align}
    with symmetric positive definite matrices $M_{\mathrm{p}}$, $K_{\mathrm{u}}$, and $K_{\mathrm{p}}$. Thus, \textbf{Step~1} and \textbf{Step~2} of \Cref{alg:staircaseDHDAE} are already done, and we see immediately that~\eqref{eqn:poroStaircase} is a \pHDAE of (Kronecker) index two. Notably, if one does not assume that the model is quasi-static, then the $(3,3)$ block in the matrix on the left-hand side in~\eqref{eqn:poroStaircase} is nonsingular and we end up with an implicit dissipative Hamiltonian \ODE. This details, that also within the \pHDAE framework, small perturbations may change the index. Nevertheless, in contrast to general \DAEs, the (Kronecker) index may be at most two.
\end{example}

\begin{example} 
    \label{ex:gasnf}
    In the gas network example presented in \Cref{sec:gasnetwork}, and more precisely for the \dHDAE obtained by setting $\inpVar=0$ in~\eqref{eq:ph1}, we know directly from the structure what the constraints are, since the system is almost in the form that would be obtained from the staircase algorithm. The first step and second step are already performed, $n_3=0$ and it remains to transform the matrix 
$\begin{smallbmatrix}0 & -J_{32}\end{smallbmatrix}$, which has full row rank, by an orthogonal transformation $\hat{U}_J$ to the form
$\begin{smallbmatrix} 0 & -J_{32}\end{smallbmatrix} \hat{U}_J= \begin{smallbmatrix} \hat{J}_{31} & 0 \end{smallbmatrix}$, with $\hat{J}_{31}$ square and nonsingular.
Setting $U_J \vcentcolon= \begin{smallbmatrix} \hat{U}_J & 0 \\ 0 & I \end{smallbmatrix}$
and forming $\hat{E}=U^T_J E U_J$, $\hat{J}=U^T_J J U_J$, 
$\hat{R}=U^T_J R U_J$ we get a \dHDAE system with $n_1=n_4=\rank J_{32}$, $n_3=n_5=0$, and $n_2=n-2n_1$. 
\end{example}

The simplified construction of the condensed forms for the other examples presented in \Cref{sec:brake} and \Cref{sec:robot} is analogous.

Although for linear time-varying \dHDAE systems we get a condensed form under pointwise orthogonal congruence transformations, in contrast to the constant-coefficient case, we would need time-varying changes of basis. Such a coordinate transformation requires derivatives of the basis transformation matrices, so a computational method needs to determine a smooth transformation. Such transformations can be determined via methods like a smooth singular value decomposition, cf.~\Cref{thm:smoothFullRankDecomposition,thm:smoothRankRevealingDecomposition}, or a $QR$ decomposition,
see \eg \cite{BunBMN91,DieE99,KunM91}. 
However, these methods require the solution of matrix differential equations (operating in the orthogonal group). This substantially increases the computational costs.

%--------------------------------------------------------------------------------------%
\subsection{Structure-preserving index reduction}
Although pointwise condensed forms help a  lot in the analysis, they are not practical in computational techniques. For general \DAE systems, one, therefore, proceeds differently and uses derivative arrays, see \cite{KunM96a,KunM06}, and \Cref{sec:DAEtheory}, to filter out a strangeness-free system by transformations that act only on the equations and their derivatives and avoid changes of basis. This approach will, however, in general, destroy the \dHDAE structure. 

An approach that achieves structure preservation on the basis of one smooth change of basis has been proposed in \cite{BeaMXZ18} for the case of \pHDAE systems that include the factor $Q$, using a technique that was introduced for self-adjoint systems in \cite{KunMS14}. 
We present this approach here for the case that $Q=I$ and assume that the system has a well-defined strangeness index and a unique solution for all consistent initial conditions. Under these assumptions, see \Cref{sec:DAEtheory}, from the (unstructured) derivative array, one can extract an algebraic equation of the form $\hat A_2 z=0$ that contains all the explicit or hidden constraint equations.
Then there exist a pointwise orthogonal $T \in \mathcal{C}^1(\timeInt, \R^{\stateDim,\stateDim})$ such that 
\begin{displaymath}
    \hat{A}_2 \begin{bmatrix} T_1 & T_2 \end{bmatrix} = \begin{bmatrix} 0 & \hat{A}_{22} \end{bmatrix},
\end{displaymath}
with $\hat{A}_{22}\in\mathcal{C}(\timeInt, \R^{a,a})$  pointwise nonsingular, so that the columns of the matrix function $T_1$ span the kernel of $\hat{A}_{22}$. 

Using the same proof as in \cite{KunMS14} for self-adjoint \DAE systems, it follows that the system consisting of the first $d=n-a$ rows and columns of $T^T E T$ is square nonsingular, and
therefore positive definite. 
With this in mind, the original \dHDAE can be transformed congruently with $T$, so that the \dHDAE structure is preserved and with 
\begin{gather*}
    \begin{bmatrix} E_{11} & E_{12}\\
E_{21} & E_{22} \end{bmatrix} = T^T E T,\qquad \state=T \begin{bmatrix} \state_1 \\ \state_2 \end{bmatrix},\\
\begin{bmatrix}
A_{11} & A_{12} \\ A_{21} & A_{22} \end{bmatrix} = T^T (J-R)T-(T^TET)(T^TK T-T^T \dot T),\ 
\end{gather*}
we obtain that $A_{22} \state_2=0$  so that $\state_2=0$. Inserting this, the remaining part of the first block equation 
\begin{displaymath}
    E_{11} \dot{\state}_1 =A_{11}\state_1, 
\end{displaymath}
is still a \dHDAE but now with $E_{11}$ positive definite and the Hamiltonian is unchanged since $\state_2=0$.

\begin{remark}\label{rem:nlhi}
For nonlinear pHDAE systems satisfying
\Cref{hyp:nonLin:nonRegular} with $\mu>0$, the corresponding local result follows directly via linearization and the implicit function theorem.
\end{remark}

\subsection{Condensed forms for  \pHDAE systems}\label{sec:phDAEs}
 In this section we extend the results presented in the previous section for \dHDAE systems to systems with inputs and outputs. 
 
For \LTI \pHDAEs of the form \eqref{eqn:pHDAE:LTI}, the extension of the condensed forms to \pHDAE systems was presented in \cite{BeaGM21} (unfortunately, some parts were omitted in the printed version; the correct condensed form is presented in the ArXiv version).
For a \LTV \pHDAE system~\eqref{eqn:pHDAE:linear} this result follows by a variation of the condensed form presented in \cite{Sch19}.

Let us first consider the following result that allows to remove the non-uniqueness part if there is any. Note that we again consider the case that $Q=I$, $S-N=0$, and $P=0$, cf.~\Cref{sec:noQ,sec:nofeed}.
\begin{lemma}
\label{lem:regul} 
For a \pHDAE of the form~\eqref{eqn:pHDAE:linear} with $Q=I$, $S-N=0$, and $P=0$, under some constant rank assumptions, there exists a pointwise orthogonal change of basis $V^{-1}\state=\vcentcolon\tilde{\state} = \begin{bmatrix} \tilde{\state}_1^T & \tilde{\state}_2^T & \tilde{\state}_3^T\end{bmatrix}^T$ such that the system has the form
\begin{subequations}
    \label{eqn:structstc}
    \begin{align}
    \begin{bmatrix} E_1 & 0 & 0 \\
 0 & 0 & 0\\ 0 & 0 & 0 \end{bmatrix} \begin{bmatrix} \dot{\tilde{\state}}_1 \\ \dot{\tilde{\state}}_2 \\ \dot{\tilde{\state}}_3 \end{bmatrix} & =
 \begin{bmatrix} J_1-R_1 -E_1 K_1& 0 & 0\\ 0 & 0 & 0\\ 0 & 0 & 0\end{bmatrix}\begin{bmatrix} \tilde{\state}_1 \\ \tilde{\state}_2 \\ \tilde{\state}_3 \end{bmatrix}
 + \begin{bmatrix} G_1 \\ G_2 \\ 0 \end{bmatrix} \inpVar,\\
 \outVar  &= \begin{bmatrix} G_1^T & G_2^T  & 0\end{bmatrix} \begin{bmatrix} \tilde{\state}_1 \\ \tilde{\state}_2 \\ \tilde{\state}_3 \end{bmatrix},
    \end{align}
\end{subequations}
where the system $ E_1\dot{\state}_1-(J_1-R_1-E_1K_1)\state_1+G_1 \inpVar$ has a unique solution for every sufficiently often differentiable $\inpVar$, and $G_2$ has full row rank. Furthermore, the subsystem
\begin{subequations}
    \label{eqn:subsystemstate}
    \begin{align}
    \begin{bmatrix} E_1 & 0  \\
 0 & 0   \end{bmatrix} \begin{bmatrix} \dot{\tilde{\state}}_1 \\ \dot{\tilde{\state}}_2  \end{bmatrix} & =
 \begin{bmatrix} J_1-R_1 -E_1K_1& 0 \\ 0 & 0  \end{bmatrix}\begin{bmatrix} \tilde{\state}_1 \\ \tilde{\state}_2  \end{bmatrix}
 + \begin{bmatrix} G_1 \\ G_2  \end{bmatrix}\inpVar,\\
 \outVar &= \begin{bmatrix} G_1^T & G_2^T \end{bmatrix} \begin{bmatrix} \tilde{\state}_1 \\ \tilde{\state}_2 \end{bmatrix},
    \end{align}
\end{subequations}
obtained by removing the third equation and the variable $\tilde{\state}_3$ is still a \pHDAE with the same Hamiltonian.
\end{lemma}
\begin{proof}
The proof follows by computing as in \eqref{firststep} a pointwise orthogonal matrix $V_1$  such that
\begin{gather*}
    V_1^T(J-R-EK)V_1- (V^T_1 E V_1) V_1^T\dot{V}_1 =
 \begin{bmatrix} J_1-R_1 -E_1 K_1& 0 \\ 0 & 0 \end{bmatrix},\\
 V_1^T  E V_1 = \begin{bmatrix} E_1 & 0  \\
 0 & 0 \end{bmatrix},\qquad V_1^T G = \begin{bmatrix} G_1 \\ \tilde{G}_2 \end{bmatrix}.
\end{gather*}
This $V_1$ exists by the condensed form \eqref{stcvardh}, by just combining the first four rows and columns into one block. 
Then a row compression of $\tilde G_2$ via a pointwise orthogonal  matrix $\tilde V_2$ (assuming constant rank) and a congruence transformation with  $V_2=\diag(I, \tilde V_2)$ is performed, so by a congruence transformation with $V= \diag(I, V_2)V_1$, we obtain the zero pattern  in \eqref{eqn:structstc}. Updating the output equation accordingly gives the desired form. 
\end{proof}

\begin{remark} Practically the constant rank assumptions that are required for the derivation of \eqref{stcvardh} can be reduced by performing only the transformation that splits off the common nullspace part of $E,J,R, EK$.
\end{remark}
\Cref{lem:regul} shows that we can remove redundant equations and variables that do occur in the system. In the following we assume that this reduction has already been performed. 
The next result presents a condensed form, which extends the form that was obtained in \cite{BeaGM21} for \LTI \pHDAE systems with $K=0$ and for \LTV \pHDAE systems in \cite{Sch19}.
\begin{lemma}
	\label{lem:c2o2} 
    Consider a linear time-varying \pHDAE as in \eqref{eqn:subsystemstate}. Then under some constant rank assumptions,  there exists a pointwise orthogonal basis transformation $V$ in the state space and $U$ in the control space such that in the new variables 
    \begin{align*}
        \hat{\state} = \begin{bmatrix} \hat{\state}_1^T & \hat{\state}_2^T & \hat{\state}_3^T & \hat{\state}_4^T & \hat{\state}_5^T & \hat{\state}_6^T \end{bmatrix}^T &= V^T \begin{bmatrix} \tilde{\state}_1^T & \tilde{\state}_2^T \end{bmatrix}^T \qquad\text{and}\\
        \hat{\inpVar} = \begin{bmatrix} \inpVar_1^T & \inpVar_2^T & \inpVar_3^T \end{bmatrix}^T &= U^T\inpVar
    \end{align*}
    the system has the form
    \begin{align*}
        \hat{E}\dot{\hat{\state}} &= (\hat{J}-\hat{R}-\hat{E}\hat{K})\hat{\state} + \hat{G}\hat{u},\\
        \hat{\outVar} &= \hat{G}^T\hat{\state},
    \end{align*}
    with
    \begin{subequations}
        \label{fullstcout}
        \begin{align}
        \hat{E} &\vcentcolon= \begin{bmatrix}
            E_{11} & E_{12} & 0 & 0 & 0 & 0\\
            E_{21} & E_{22} & 0 & 0 & 0 & 0\\
            0 & 0 & 0 & 0 & 0 & 0\\ 
            0 & 0 & 0 & 0 & 0 & 0\\  
            0 & 0 & 0 & 0 & 0 & 0\\  
            0 & 0 & 0 & 0 & 0 & 0
        \end{bmatrix}, \qquad \hat{G} \vcentcolon= \begin{bmatrix}
            G_{11} & G_{12} & G_{13}\\
            G_{21} & G_{22} & G_{23}\\
            G_{31} & G_{32} & G_{33}\\
            0 & G_{42} & G_{43}\\
            0 & 0 & G_{53}\\
            0 & 0 & G_{63}
        \end{bmatrix} \\
        \hat{J}-\hat{R} &\vcentcolon= \begin{bmatrix}
            J_{11}-R_{11} &  J_{12}-R_{12}& J_{13}-R_{13}&J_{14} & J_{15}&0\\
            J_{21}-R_{21} &  J_{22}-R_{22}& J_{23}-R_{23}&J_{24} & 0&0\\
            J_{31}-R_{21} & J_{32}-R_{32} & J_{33}-R_{33} & 0 & 0&0\\
            J_{41} & J_{42} & 0& 0 & 0 & 0\\ 
            J_{51} & 0&0  & 0  & 0 & 0\\ 
            0& 0 &  0 & 0& 0& 0
        \end{bmatrix},\\
        \hat{K} &\vcentcolon= \begin{bmatrix}
            K_{11} & K_{12} & K_{13} & K_{14} & 0 & 0\\
            K_{21} & K_{22} & K_{23} & K_{24} & 0 & 0\\
            0 & 0 & 0 & 0 & 0 & 0\\
            0 & 0 & 0 & 0 & 0 & 0\\
            0 & 0 & 0 & 0 & 0 & 0\\
            0 & 0 & 0 & 0 & 0 & 0
        \end{bmatrix}
        \end{align}
    \end{subequations}
where $E_{22}$, $J_{33}-R_{33}$, $J_{15}$, and $G_{42}$ and $G_{63}$  are pointwise invertible.
\end{lemma}
\begin{proof}
Starting from \eqref{eqn:subsystemstate}, in the first step (similarly as in \eqref{firststep}), one determines a pointwise orthogonal matrix function $\tilde{V}_1$ such that 
\begin{displaymath}
    \tilde{V}_1^T E_1 \tilde{V}_1 = \begin{bmatrix} 
        \tilde E_{11} & 0  \\
        0 & 0 
    \end{bmatrix}
\end{displaymath}
with $\tilde{E}_{11}>0$, and then performs a congruence transformation with the matrix function ${V}_1=\diag(\tilde{V}_1,I)$, yielding $\tilde{V}_1^T G_1 = \begin{smallbmatrix} \tilde{G}_1 \\ \tilde{G}_2\end{smallbmatrix}$ and
\begin{displaymath}
    V_1^T(J-R-EK) V_1-(V_1^T EV_1) V_1^T \dot{V}_1 = \begin{bmatrix} \tilde{J}_{11}-\tilde{R}_{11}-\tilde{E}_{11} \tilde{K}_{11} & \tilde{J}_{12}-\tilde{R}_{12} \\ \tilde{J}_{21}-\tilde{R}_{21} & \tilde{J}_{22}-\tilde{R}_{22}  \end{bmatrix}.
\end{displaymath}
Next, under the assumption of a constant rank, compute a smooth full rank decomposition
\begin{displaymath}
    \tilde{V}_2^T (\tilde{J}_{22}-\tilde{R}_{22}) \tilde{V}_2 = \begin{bmatrix} \widehat{J}_{22} -\widehat{R}_{22} & 0 \\ 0& 0 \end{bmatrix},
\end{displaymath}
where $\widehat{J}_{22} -\widehat{R}_{22}$ is invertible and $\widehat{R}_{22}\geq 0$. Such a full rank decomposition exists, since $\widetilde{J}_{22}-\widetilde{R}_{22}$ has a positive semi-definite symmetric part. Then, defining $V_2 \vcentcolon= \diag(I, \tilde{V}_2, I)$, and applying an appropriate congruence transformation with $\widehat{V} \vcentcolon= V_1V_2$ yields
\begin{align*}
	\widehat{V}^T E \widehat{V} &=
\begin{bmatrix} \tilde E_{11} & 0 &0 &0 \\
 0 & 0 & 0 & 0\\ 0 & 0  &0 &0\\ 0 & 0  &0 &0
\end{bmatrix},\qquad \widehat{V}^TG = \begin{bmatrix} \widehat{G}_1 \\ \widehat{G}_2 \\ \widehat{G}_3 \\ \widehat{G}_4\end{bmatrix},\\
\widehat{V}^T(J-R-EK) \widehat{V} &= \begin{bmatrix} \widehat{J}_{11}-\widehat{R}_{11} -\widetilde E_{11} \widetilde K_{11} & \widehat{J}_{12}-\widehat{R}_{12} -\widetilde E_{11} \widetilde K_{12} & \widehat{J}_{13} -\widetilde E_{11} \widetilde K_{13}& 0\\ \widehat{J}_{21}-\widehat{R}_{21} & \widehat{J}_{22}-\widehat{R}_{22} &0 &0 \\ \widehat{J}_{31} & 0& 0 &0  \\ 0 & 0 & 0 & 0\end{bmatrix},
\end{align*}
where $\widehat{J}_{22}-\widehat{R}_{22}$ is invertible and $\widehat{G}_4$ has full row rank. Note that there is no contribution of $R$ in the third block column and row, which is due to the fact that $\widehat{V}^TR\widehat{V}$ is pointwise positive semi-definite.

Note also that the terms $\widetilde K_{12}$ and $\widetilde K_{13}$ arise due to the fact that the transformation from the right operates in  these block columns. Note further that $\widehat G_4$ has full row rank so it can be transformed by a change of of basis to be of the form $\begin{bmatrix} 0 & \bar{G}_{63}\end{bmatrix}$ with invertible $\bar{G}_{63}$. Combining this with a smooth full rank decomposition of the block $\widehat{G}_3$, one can perform a smooth transformation 
\begin{displaymath}
    \begin{bmatrix} \tilde{V}_3^T & 0 \\ 0 & I \end{bmatrix} \begin{bmatrix}\widehat{G}_3 \\ \widehat{G}_4 \end{bmatrix} U = \begin{bmatrix} 0 & \bar{G}_{42} &\bar{G}_{43}\\ 0 & 0 & \bar{G}_{53}\\
0 & 0 & \bar{G}_{63} \end{bmatrix},
\end{displaymath} 
with $\bar{G}_{42}$ and $\bar{G}_{63}$ square and pointwise nonsingular, where the number of rows in $\bar{G}_{63}$ is that of $\widehat{G}_4$. Applying an appropriate congruence transformation with $V_3=\diag(I, I,  \tilde{V}_3,I)$ one obtains block matrices
\begin{gather*}
    \begin{bmatrix} \bar{E}_{11}  & 0 & 0 & 0 &0\\
0 & 0 & 0&0& 0\\
 0 & 0 & 0 & 0 &0 \\ 0 & 0 & 0 &0  & 0\\ 0 & 0 & 0 &0  & 0\end{bmatrix},\qquad \begin{bmatrix} 
 \bar{G}_{11} & \bar{G}_{12} & \bar{G}_{13} \\
 \bar{G}_{21} & \bar{G}_{22} & \bar{G}_{23} \\
 0 & \bar{G}_{42} & \bar{G}_{43} \\
0 & 0 & \bar{G}_{53} \\
0 & 0 & \bar{G}_{63} \end{bmatrix},\\
\begin{bmatrix} \bar{J}_{11}-\bar{R}_{11}-\bar E_{11} \bar K_{11} &  \bar{J}_{12}-\bar{R}_{12}- \bar E_{11} \bar K_{12}& \bar{J}_{13}-\bar E_{11} \bar K_{13} & \bar{J}_{14}-\bar E_{11} \bar K_{14} &0\\
 \bar{J}_{21}-\bar{R}_{21} &  \bar{J}_{22}-\bar{R}_{22}& 0 & 0&0\\
 \bar{J}_{31}  & 0& 0 & 0 & 0\\  \bar{J}_{41}&0  & 0  & 0 &0\\ 0 & 0 & 0 & 0 & 0\end{bmatrix}.
\end{gather*}
As final step one computes a column compression of the full row rank matrix function~$\bar{J}_{41}$ and applies an appropriate congruence transformation. This yields the desired form.
\end{proof}

Since in the condensed form \eqref{fullstcout} the blocks $J_{51}$, $G_{42}$ and $G_{63}$ are pointwise invertible, it follows immediately that $u_3=0$ and that $\hat{\state}_1=-G_{53} u_3=0$ and $\hat{\state}_5$ is uniquely determined by  the other variables and their derivatives. These parts are associated with equations 
for which a regularization is necessary, see \Cref{sec:regulviaout}.
\begin{example} \label{ex:gasnf:2}
The construction in \Cref{ex:gasnf} can also be used to derive the condensed form for the \pHDAE, which yields a system of the form \eqref{fullstcout} in which the third, fourth, and sixth block row and the corresponding block columns do not occur. The system has (Kronecker) index two as a free system with~$\inpVar=0$. 
\end{example}

\begin{example}\label{ex:power}
The power network from \Cref{sec:power} is already in the condensed form \eqref{fullstcout}, where the first, fourth, fifth, and sixth block row and column do not occur, so the system has (Kronecker) index one as a free system with $\inpVar=0$.
\end{example}

\subsection{Regularization via output feedback}\label{sec:regulviaout}

Consider a \pHDAE system of the form \eqref{fullstcout} and denote the system that is obtained by removing the variables 
$\widehat{\state}_1$, $\widehat{\state}_4$ 
and the corresponding first and fifth equation by

\begin{subequations}\label{redsystem}
\begin{align}
\hat E\dot{\hat \state}&=(\hat J-\hat R-\hat E \hat K) \hat \state+\hat G\,\inpVar,\\
\hat \outVar&=\hat G^T \hat z.
\end{align}
\end{subequations}
System~\eqref{redsystem} can be viewed as the subsystem that is controllable and observable at $\infty$, see \Cref{def:concon}, since we have the following corollary.
\begin{corollary}\label{cor:c2o2}
For system \eqref{redsystem} there exists an output feedback
\begin{displaymath}
    \inpVar=-\hat{W}\hat{\outVar}+w
\end{displaymath}
with $\hat W+\hat{W}^T> 0 $, so that the resulting closed-loop system is a \pHDAE system
\begin{align*}
\hat E\dot{\hat \state}&= (\hat J-\hat R-\hat E \hat K-\hat G \hat W \hat G^T)\hat{\state} + \hat{G}w,\\
\hat \outVar&=\hat G^T \hat z,
\end{align*}
is regular and strangeness-free as a free system with $w=0$.
\end{corollary}
\begin{proof}
This follows directly from the structure of the system, and the fact that
$\hat E$ is already in a form where the kernel of $\hat{E}$ and $\hat{E}^T$ can be directly read off.
\end{proof}

\begin{remark}\label{rem:noncontrinf}
The condensed form \eqref{fullstcout} is a structured \pHDAE version of
the condensed forms in \cite{ByeGM97,ByeKM97}, which allow to remove parts of the system that cannot be made strangeness-free, uniquely solvable, or of (Kronecker) index at most one via (output) feedback. 
\end{remark}

A similar process of removing parts from a system and that cannot be made strangeness-free by (output) feedback has been discussed for general nonlinear \DAE systems in \cite{CamKM12}. The procedure can be applied directly to \pHDAE systems. 

%%%%%%%%%%%%%%%%%%%%%%%%%%%%%%%%%%%%%%%%%%
\subsection{Stability}\label{sec:phandstability}
In this subsection, we show that another important feature of \pHDAE systems is that physical properties like stability or passivity are directly available from the structure. Conversely, if a system is (asymptotically) stable, then it typically can be written as \pHDAE system. 

Beginning with \LTI \ODE systems, an immediate consequence of the Lyapunov characterization of stability  is the existence of a \dHDAE formulation of (asymptotically) stable \LTI \ODE systems, see 
\cite{GilMS18}.
\begin{corollary}\label{cor:daeform}
Consider the linear time-invariant system \eqref{linode} and
suppose that $X=X^T>0$ is a solution of the Lyapunov inequality $ A^T X+ X A \leq 0$, 
($ A^T X+ X A < 0$). Setting $J-R\vcentcolon= XA$ and $E\vcentcolon= X$
with $J=-J^T$ and $R=R^T$, then 
\begin{equation}\label{reformdhdae}
    E\dot{\state}= (J-R)\state    
\end{equation}
is a \dHDAE system with $R\geq 0$ ($R>0$).

Conversely every \dHDAE of the form \eqref{reformdhdae} with $E>0$, $R\geq 0$ ($R>0$) is (asymptotically) stable. 
\end{corollary}
\begin{proof}
The proof follows trivially 
since $-R=A^TE+EA=A^TX+XA$.
\end{proof}

For \dHDAE systems we have the following stability characterizing spectral properties, 
see \cite{MehMW18} also for an extended result that also deals with singular and high index \dHDAE systems.
\begin{theorem}\label{thm:singind}
Consider a \dHDAE of the form \eqref{dhdae} and suppose that the pencil $\lambda E-(J-R)$ is regular and of (Kronecker) index at most one.
\begin{itemize}
\item[(i)] If $\lambda_0\in\C$ is an eigenvalue of $\lambda E-(J-R)$, then $\operatorname{Re}(\lambda_0)\leq 0$.
\item[(ii)] If $\omega\in\R$ and $\lambda_0=i\omega$ is an eigenvalue of $\lambda E-(J-R)$, then
$\lambda_0$ is semisimple. Moreover, if the columns of $V\in\C^{\stateDim,k}$ form a basis of a regular deflating subspace of $\lambda E-(J-R)$ associated with the eigenvalue $\lambda_0$, then $RV=0$.
\end{itemize}
\end{theorem}

\begin{proof} \ 
\begin{itemize}
    \item[(i)] Let $\lambda_0\in\C$ be an eigenvalue of $\lambda E-(J-R)$ and let $v\neq 0$ be an eigenvector associated with
$\lambda_0$. Then we have $\lambda_0 Ev=(J-R)v$ and thus
\[
\lambda_0 v^H   Ev=v^H  Jv-v^H Rv.
\]
Considering the real parts of both sides of this equation, we obtain
\[
\operatorname{Re}(\lambda_0) v^H  Ev=-v^H  Rv,
\]
where we used the fact that $E$ and $R$ are symmetric and $J$ is skew-symmetric. If $Ev=0$, then also $(J-R)v=0$ which would imply that the pencil is singular. 
Hence, we have $Ev\neq 0$, and since $E$ is positive semi-definite, we obtain
$v^HEv>0$, which finally implies
\[
\operatorname{Re}(\lambda_0)=-\frac{v^H Rv}{v^HEv}\leq 0.
\]
\item[(ii)] We first prove the 'moreover' part.
For this, let the columns of $V\in\C^{n,k}$ form a basis of a regular deflating subspace of $\lambda E-(J-R)$ associated with the eigenvalue $\lambda_0=i\omega$, $\omega\in\R$, \ie there exists a matrix $W\in\C^{n,k}$ with full column rank such that
$EV=W$ and $(J-R)V=WT$, where $T\in\C^{k,k}$ only has the eigenvalue $i\omega$. Without loss of generality we may assume that $T=i\omega I_k+N$ is in Jordan canonical form,
where $N$ is strictly upper triangular.  Then  $V^H  (J-R)V=V^H  WT$,
and taking the Hermitian part on both sides we obtain
\[
0\geq -2V^H  RV=V^H WT+T^HW^HV.
\]
Since $R$ is positive semi-definite, it remains to show that $V^H RV=0$, because then we also have $RV=0$. For this, we show that
\begin{displaymath}
    V^H W=W^H V> 0.
\end{displaymath}
This follows, since first $V^HW= V^H EV\geq 0$. If there exists $x\neq 0$ such that $EVx=0$, then with $y=Vx =y_1+i y_2$, $y_1,y_2$ real, one has $Ey=0$. This implies that  $Ey_1=0$ and $y_2=0$. Hence,
\[
y\in\sspan (\ker E \cup \ker (J-R) )\subseteq    \sspan \bigg( \bigcup_{\lambda\in\mathbb{S}} \ker (\lambda E- (J-R))\bigg),
\]
with $\mathbb{S} \vcentcolon= (\mathbb C\cup\{\infty\})\setminus\{\lambda_0\}$, 
which contradicts the fact that the columns of $V$ span a regular deflating subspace associated with $\lambda_0$.

If $M$ is the inverse of the Hermitian positive definite square root of $V^HW$, then 
\[
M(V^H WT+T^H V^HW)M=M^{-1}TM+MT^H M^{-1}\leq 0.
\]
Moreover, 
\begin{align*}
\operatorname{trace}(M^{-1}TM+MT^H M^{-1})&=
\operatorname{trace}(M^{-1}TM)+\operatorname{trace}(MT^H M^{-1})\\
&=\operatorname{trace}(T+T^H)=\operatorname{trace}(N+N^H)=0,
\end{align*}
because $N$ has a zero diagonal. But this implies
\begin{displaymath}
    M^{-1}TM+MT^HM^{-1}=0
\end{displaymath} 
and hence also $-2V^H R V=0$, which finishes the proof of the 'moreover' part.

To show that $i\omega$ is a semisimple eigenvalue, it remains to show that the matrix
$T=i\omega I_k+N$  is diagonal, \ie, $N=0$.  Since purely imaginary eigenvalues of the system correspond to eigenvectors of the non-dissipative system ($R=0$) that are in the kernel of $R$, see \cite{MehMS16}, with $RV=0$ we get  $EV=W$ and $JV=WT$ which implies that $V^H JV=V^H WT$.
Then 
\[
M^{-1}TM=MV^H WTM=MV^H JVM,
\]
implies that $T$ is similar to a matrix which is congruent to $J$, i.e., $T$ is similar to 
a skew-symmetric matrix, which implies that  $N=0$. Thus, $i\omega$  is a semisimple eigenvalue of  $\lambda E-(J-R)$ and assertion (ii) is proved.
See \cite{MehMW18} for further details.
\end{itemize}
\end{proof}

Note that if we consider an asymptotically stable linear system \eqref{linode} and split $A=J-R$ in its skew-symmetric and symmetric part, then in general we do not have that $R>0$ and even for singular  $R\geq 0$  the system may be asymptotically stable, as is shown by the following characterization from \cite{AchAM21}.

\begin{lemma} \label{lem:Equivalence}
Consider the \LTI system~\eqref{ODE} with $A=J-R$, $J=-J^T$, $0\leq R=R^T$. Then the following conditions are equivalent.
\begin{itemize}
\item[(i)]
There exists a nonnegative integer $m_H$ such that %
\begin{equation*}\label{condition:KalmanRank}
 \rank \begin{bmatrix} R & J{R} & \cdots & J^{m_H} {R}\end{bmatrix} = \stateDim.
\end{equation*}
\item[(ii)]
There exists a nonnegative integer $m_H$ such that
\begin{equation*}\label{Tm:J-R}
 T_{m_H} \vcentcolon=\sum_{j=0}^{m_H} J^j R (J^T)^j > 0.
\end{equation*}
\item[(iii)]
No eigenvector of~$J$ lies in the kernel of~$R$.
\item[(iv)] We have $\rank \begin{bmatrix}\lambda I-J & R\end{bmatrix} =\stateDim$ for every $\lambda \in \C$, in particular for every eigenvalue~$\lambda$ of~$J$.
\end{itemize}
Moreover, the smallest possible~$m_H$ in~(i) and (ii) coincide. 
\end{lemma}

\begin{proof}
See \cite{AchAM21}.
\end{proof}
The smallest possible $m_H$ in (i) and (ii) of \Cref{lem:Equivalence} is called the \emph{hypocoercivity index} of $A$ and we have the following corollary.
\begin{corollary}\label{cor:asyhyp}
Consider the \LTI system~\eqref{ODE} with $A=J-R$, $J=-J^T$, $0\leq R=R^T$. Then the system is asymptotically stable if and only if the hypocoercivity index is  finite.
\end{corollary}

\begin{remark}\label{rem:shortdecay}
It has been shown in \cite{AchAC21} that if the hypocoercivity index $m_H$ is finite , then for the fundamental solution $e^{A t}\in \R^{n,n}$ of~\eqref{linode}, the short-time decay in the spectral norm is given by
\begin{equation*}
    %\label{short-t-decay}
\|\mathrm{e}^{At}\|_2
= 1 -ct^{2m_H+1} +\mathcal O(t^{2m_H+2})
\quad \text{for } t\to0+\,,
\end{equation*}
with a constant $c>0$.
\end{remark}

Analogous to the \ODE case, \Cref{thm:singind} implies that \dHDAE systems with regular pencils of (Kronecker) index at most one are stable, but they are not necessarily asymptotically stable. To characterize asymptotic stability, a hypocoercivity index and the corresponding Lyapunov inequality for \dHDAE systems is introduced in \cite{AchAM21} also for the \DAE case. For the proof
we use the following simplification of the staircase form~\eqref{stcdh}.
\begin{lemma}\label{lem:reduced}
Consider a \dHDAE of the form 
\eqref{dhdae} with regular matrix pencil $\lambda E-(J-R)$ in staircase form~\eqref{stcdh}. Then there exist nonsingular matrices~$L_1,L_2$ such that
\begin{align}\label{kcf:almost}
L_1 \check{E} L_2 &=
\begin{bmatrix}
\widehat{E}_{1,1} & 0 & 0 & 0  \\
0 & \widehat{E}_{2,2} & 0 & 0  \\
0 & 0 & 0 & 0  \\
0 & 0 & 0 & 0  
\end{bmatrix}, &
L_1 (\check{J} -\check{R}) L_2
&=\begin{bmatrix}
0 & 0 & 0& I  \\
0  & \widehat{J}_{2,2}-\widehat{R}_{22} & 0 & 0 \\
0 & 0 & I& 0  \\
-I & 0 & 0 & 0  
\end{bmatrix}.
\end{align}
The blocks satisfy $\widehat J_{2,2}=-\widehat J_{2,2}^T$, $\widehat E_{1,1}=\widehat E_{1,1}^T>0$, $\widehat E_{2,2}=\widehat E_{2,2}>0$, and $\widehat R_{22}=\widehat R_{22}^T\geq 0$.
\end{lemma}
\begin{proof}
The proof follows by block Gaussian elimination to create first the block diagonal structure of~$\widehat E$,  using the positive definite diagonal block $E_{22}$. This is followed by block Gaussian elimination using the nonsingular blocks  $J_{41}=-J_{14}^T$ and $J_{33}-R_{33}$ and then scaling these nonsingular blocks. Note that $\widehat E_{22}= E_{22}$ with $E_{22}$ as in~\eqref{stcdh}, and $\widehat{J}_{22}$, $\widehat{R}_{22}$ are skew-symmetric and symmetric part of the Schur complement obtained in this way, so the semi-definiteness of $\widehat{R}_{22}$ follows as in the proof of \Cref{lem:schur}. See \cite{AchAM21} for details. 
\end{proof}
If the \dHDAE system is transformed to the form \eqref{kcf:almost} with
transformed state vector $\state= \begin{bmatrix} \state_1^T &\state_2^T & \state_3^T & \state_4^T \end{bmatrix}^T$ partitioned according to the block structure, then we immediately obtain that $\state_1=0$, $\state_3=0$ and $ \state_4=0$, which gives restrictions in the initial values. 
Using the fact that  $\widehat{E}_{22}=\widehat{E}_{22}^T>0$, in 
\cite{AchAM21} then the hypocoercivity index of \eqref{dhdae} is defined as that of the underlying \ODE
\begin{equation}\label{underlyingODE}
 \dot{\xi}_2=
   \widehat E_{22}^{-1/2}( \widehat J_{22}-\widehat R_{22})\widehat E_{22}^{-1/2} \xi_2,\  
\end{equation}
with $\xi_2=\widehat{E}_{22}^{1/2} \state_2$.  
We have the following characterization of asymptotic stability, cf.~\cite{AchAM21}.
\begin{corollary}\label{cor:hypodae}
If a \dHDAE system of the form~\eqref{dhdae} has a regular pencil $\lambda E-(J-R)$ with (Kronecker) index at most one, and non-trivial dynamics with a finite hypocoercivity index, then for every consistent initial condition the solution is asymptotically stable.
\end{corollary}

\begin{remark}\label{rem:shorttime}
Using the transformation to the condensed form \eqref{kcf:almost} we see that the short-time decay is as for the \ODE case. This can be viewed as considering the decay in a semi-norm obtain by scaling the solution with the semi-definite matrix $E$.  To see this, let $Z$ be the transformation matrix to the  form~\eqref{kcf:almost}.
By assumption $n_1=n_4$ and  the solution~$\xi_2(t)$ of~\eqref{underlyingODE} is asymptotically stable for every initial value~$\state_2(0)$. The solution of the original system is then $\state =Z \begin{bmatrix} \xi_2^T & 0\end{bmatrix}^T$, hence for every consistent initial value
\begin{equation*} %\label{DAE:directEstimate}
 \| \state(t)\|^2
 =\| Z \xi(t) \|^2
 \leq \sigma_{\max}(Z) \|\xi_2(0)\|^2 \mathrm{e}^{-2\mu t}
 \leq \kappa(Z) \|\state(0)\|^2 \mathrm{e}^{-2\mu t} ,
\end{equation*}
where  $\sigma_{\max}(Z)$ is the largest singular value of~$Z$, $\mu>0$ is some exponential decay rate capturing the asymptotic stability of~\eqref{underlyingODE}, and $\kappa(Z)=\|Z\| \|Z^{-1}\|$ is the condition number of~$Z$.
\end{remark}

The condensed form \eqref{kcf:almost} also allows a characterization of (asymptotic) stability via a generalized 
Lyapunov equation. 
The following theorem is a simplified and real version of a result in \cite{AchAM21}.
\begin{theorem}\label{thm:genlya}
Consider a \dHDAE of the form~\eqref{dhdae} with regular matrix pencil~$\lambda E-A$ of Kronecker index at most two and finite hypocoercivity index. Then for every~$W\in\mathbb R^{n,n}$ the generalized Lyapunov equation
\begin{equation}\label{eq:genlya}
E^T X A+ A^T X E=-E^T W E
\end{equation}
has a solution. For all solutions~$X$ of~\eqref{eq:genlya}, the matrix~$E^T XE$ is unique. 
Moreover, if~$W$ is positive (semi-)definite, then every solution~$X$ of~\eqref{eq:genlya} is positive (semi-)definite on the image of~$P_\ell$, the spectral projection onto the left deflating subspace associated with the finite eigenvalues of~$\lambda E -A$.
\end{theorem}
\begin{proof}
Due to \Cref{thm:singind}, the eigenvalues are in the closed left half-plane, the eigenvalues on the imaginary axis are semi-simple, and the pencil is of  Kronecker index at most two.
But since the pencil is regular and has a finite hypocoercivity index, its finite spectrum lies in the open left half-plane.

For general linear \DAE systems with regular matrix pencil $\lambda E -A$ of (Kronecker) index at most two whose finite eigenvalues lie in the open left half-plane, the result then follows from \cite{Sty02a}.
\end{proof}

\begin{remark}\label{rem:extlya}
For \LTI \DAE systems the characterization  of stability via different generalized Lyapunov equations and the relation to \pHDAE systems  has recently been studied in different contexts \eg in a behavior context in \cite{GerH21}, via generalized Kalman-Yakubovich-Popov inequalities in \cite{ReiRV15,ReiV15}, or via linear relations in \cite{GerHR21}.
\end{remark}

All the discussed approaches are, if at all computationally feasible,  highly involved. Fortunately, as we show in the next section, the \dHDAE structure comes to help.

\subsection{Stability and passivity of general \dHDAE and \pHDAE systems}
\label{sec:generalDAE}

In this section we show that for general \dHDAE and \pHDAE systems the stability analysis is straightforward, since it will turn out that the associated Hamiltonian is a Lyapunov function. To show this, we will use the power balance equation from \Cref{th:pbe} and for the passivity the dissipation inequality~\eqref{eq:dissIneq}.

We have already seen in \Cref{sec:properties} that every \pHDAE can be easily made autonomous by turning it to the form \eqref{eqn:pHDAE:autonomous} without changing the Hamiltonian. The dissipation inequality \eqref{eq:dissIneq} then implies that the Hamiltonian $\hamiltonian$ is locally negative semi-definite in an equilibrium point $\state^\star$ and hence $\hamiltonian$ is a Lyapunov function. 
\begin{corollary}\label{cor:stab}
Consider an autonomous \dHDAE obtained from the \pHDAE~\eqref{eqn:pHDAE} by setting $\inpVar=0$ and omitting the output equation. If the system is regular and strangeness-free, then it is stable. Furthermore, in this case a sufficient condition for the system to be asymptotically stable is that $R(t,\state)>0$.
\end{corollary}

The dissipation inequality directly implies that strangeness-free \pHDAE systems are passive.
\begin{corollary}\label{cor:pass}
Consider an autonomous \pHDAE of the form~\eqref{eqn:pHDAE}. If the system is strangeness-free then it is passive. Furthermore, in this case a sufficient condition for the system to be strictly passive is that $W(t,\state)>0$ for all $(t,\state)\in \timeInt\times \mathcal{Z}$.
\end{corollary}

\begin{remark}\label{rem:examples}
For the examples in \Cref{sec:examples}  we directly have asymptotic stability for the circuit \eqref{eq:circuit2}, while for multi-body systems in \Cref{sec:robot}, the gas network problem in \Cref{sec:gasnetwork}, the poroelasticity problem in~\Cref{sec:poro} and the fluid-dynamics example in \Cref{sec:navierstokes} are not strangeness-free, but (asymptotic) stability is obtained after removing the algebraic parts that are associated with a strangeness index that is greater than zero, or (Kronecker) index greater than one. By an appropriate output feedback the circuit example of \Cref{sec:circuit} can be made to have (Kronecker) index one.
\end{remark}

\begin{example}\label{ex:stabbrake}
The disk brake example in \Cref{sec:brake} is in general not stable, but it may be stable if the perturbation term is small enough.
\end{example}

It is currently under investigation how to extend the results on hypocoercivity to the \LTV and nonlinear case to obtain a necessary and sufficient condition for asymptotic stability.

%%%%%%%%%%%%%%%%%%%%%%%%%%%%%%%%%%%%%%%%%%%%%%%%%%%%%%%%%%%%%%
\section{Model-order reduction}
\label{sec:MOR}

In this section, we discuss structure-preserving \emph{model-order reduction} (\MOR) methods for \pHDAE systems. The main idea is to replace the potentially high-dimensional \pHDAE with a low-dimensional \pHDAE surrogate model, such that the output error approximation for a given input is below some given tolerance. A standard approach in the \MOR literature, see for instance \cite{Ant05,QuaMN16,HesRS16,BenCOW17,AntBG20}, is to construct the surrogate model via Galerkin or Petrov-Galerkin projection. In more detail, for a regular descriptor-system of the form~\eqref{eqn:descriptorSystem} (\ie, we assume $\ell=\stateDim$ within this section), the projection-based surrogate is given as
\begin{subequations}
	\label{eqn:ROM}
	\begin{align}
		\label{eqn:ROM:stateEq}\widehat{F}(t,\widehat{\state}(t),\dot{\widehat{\state}}(t),\inpVar(t)) &= 0,\\
		\label{eqn:ROM:output}\widehat{\outVar}(t) - \widehat{G}(t,\widehat{\state}(t),\inpVar(t)) &= 0,
	\end{align}
\end{subequations}
with 
\begin{subequations}
\label{eqn:ROM:construction}
\begin{align}
    \widehat{F}(t,\widehat{\state},\dot{\widehat{\state}},\inpVar) &\vcentcolon = V_{\mathrm{r}}^T F(t,V_{\mathrm{\ell}}\widehat{\state},V_{\mathrm{\ell}},\dot{\widehat{\state}},\inpVar),\\
    \widehat{G}(t,\widehat{\state}(t),\inpVar(t)) &\vcentcolon= G(t,V_{\mathrm{r}}\widehat{\state}(t),\inpVar(t))
\end{align}
\end{subequations}
for matrices $V_{\mathrm{\ell}},V_{\mathrm{r}}\in\R^{\stateDim,r}$. The task of \MOR is (i) to construct suitable matrices $V_{\mathrm{\ell}}, V_{\mathrm{r}}\in\R^{\stateDim,r}$ in a numerically stable way, and (ii) to ensure that $\widehat{F}$ and $\widehat{G}$ in~\eqref{eqn:ROM} can be evaluated efficiently (without the need to evaluate terms in the full model dimension~$\stateDim$). In addition, \MOR strives to quantify the error of the \emph{reduced-order model}~(\ROM)~\eqref{eqn:ROM} and preserve important properties (such as stability or passivity) within the \ROM.

For general \DAE systems, even if the original system is of (Kronecker) index zero, a Galerkin projection may change the index, the regularity, or the stability properties of the free system (with~$\inpVar=0$).

\begin{example}
\label{ex:change}
Consider the implicit \ODE system
\begin{displaymath}
    \begin{bmatrix} 
        0 & -1 \\ -1 & 1 
    \end{bmatrix} \begin{bmatrix} \dot{\state}_1 \\ \dot{\state}_2 \end{bmatrix} = \begin{bmatrix} \varepsilon & 1 \\ 1 & 0 \end{bmatrix} \begin{bmatrix} \state_1 \\ \state_2 \end{bmatrix} + \begin{bmatrix} 1 \\ 0 \end{bmatrix} \inpVar.
\end{displaymath}
Then with $V_{\mathrm{\ell}}^T \vcentcolon= V_{\mathrm{r}}^T \vcentcolon= \begin{bmatrix} 1 & 0\end{bmatrix}$ we obtain the \ROM 
\begin{displaymath}
    0 = \varepsilon \state_1 + \inpVar,
\end{displaymath}
which now has (Kronecker) index $1$ for $\varepsilon>0$ and is even singular for $\varepsilon=0$. 
\end{example}
A key advantage of modeling with \pHODE and \pHDAE systems is that effects as in \Cref{ex:change} do not occur if the structure is not altered. Since \pHDAE systems are invariant under Galerkin projection (cf.~\Cref{cor:pHDAEinv}), the model class is ideal for projection-based discretization and \MOR methods. This, together with the invariance under interconnection, allows the construction of model hierarchies ranging from fine models for simulation and parameter studies to very course or surrogate models that can be used in control and optimization. 

\begin{remark}
    \MOR for \pHODEs is discussed for instance in 
    \cite{AfkH17,AfkH19,BeaG11,BorSF21,BreMS20,BreU21,BucGH21,ChaBG16, EggKLMM18, FujK07,GugPBV09,GugPBV12,IonA13,KawS18, Lil20,MosL20,PolS11, PolS12,SatS18,SchKL19,SchU18,SchV20,WolLEK10,WuHLM14,WuHLM18}. Let us emphasize that for \LTI systems, any passive system can be recast as a \pH system, see \cite{BeaMX15_ppt,BeaMV19}. Thus, also any passivity-preserving \MOR method can be used as a structure-preserving \MOR method for \LTI \pHODEs (with a potentially necessary post-processing step to construct the low-dimensional \pH representation). We exemplarily mention positive-real balanced trunction, see \cite{DesP84,GuiO13,IonS07,ReiS10a,ReiS10b}, and interpolation methods, \eg, \cite{Ant05b,Ant08,Fre00,IonRA08,Sor05}.
\end{remark}

In the following we focus solely on \LTI \pHDAE systems, since structure-preserving \MOR methods for general \pHDAE systems are still under investigation. In the following, we discuss different \MOR techniques and their use for \LTI \pHDAE systems. One important class are methods related to the reduction of the underlying Dirac structure and the associated power conservation. These are the effort and flow constraint reduction methods discussed in \Cref{sec:power_method}. Another major class are (Galerkin) projection methods that operate in the classical differential equation domain and make sure that the corresponding transfer functions in frequency domain is well approximated. These methods are the well-known moment matching (\Cref{sec:mm}) and tangential interpolation (\Cref{sec:tangential}). Before we present these methods, we provide some general considerations in the next subsection.

%%%%%%%%%%%%%%%%%%%%%%%%%%%%%%%%%%%%%%%%%%
\subsection{General considerations for LTI pHDAE systems}
\label{sec:MOR:LTI}

We assume that the \pHDAE has been reformulated in such a way that the free system (with $\inpVar=0$) is of (Kronecker) index at most one,
see \Cref{sec:regulviaout}, that $Q=I$ and that the system has no feedthrough term, see \Cref{sec:noQ,sec:nofeed}. This means that the system has the form
\begin{subequations}
    \label{eqn:strfreecontrol}
    \begin{align}
       E\dot{\state} &= (J-R)\state + G\inpVar, \\
       \outVar &= G^T \state,
    \end{align}
\end{subequations}
where the matrix pencil $\lambda E-(J-R)$ is regular and of (Kronecker) index at most one, $E=E^T\geq 0$, $R=R^T\geq 0$, and $J=-J^T$. In view of \Cref{cor:pHDAEinv} and the general Petrov-Galerin projection approach described above, a \ROM is constructed by choosing a suitable matrix $V\in\R^{\stateDim,r}$, setting $V_{\mathrm{r}} \vcentcolon= V_{\mathrm{\ell}} \vcentcolon= V$, and constructing the \ROM matrices as
\begin{subequations}
\label{eqn:strfreecontrol:ROM:matrices}
\begin{align}
    \widehat{E} &\vcentcolon= V^TEV \in\R^{r,r}, & \widehat{J} &\vcentcolon= V^TJV\in\R^{r,r},\\
    \widehat{R} &\vcentcolon= V^TRV \in\R^{r,r}, & \widehat{G} &\vcentcolon= V^TG\in\R^{r,m},
\end{align}
\end{subequations}
such that structure-preserving surrogate~\eqref{eqn:ROM} for~\eqref{eqn:strfreecontrol} is given as
\begin{subequations}
    \label{eqn:strfreecontrol:ROM}
    \begin{align}
       \widehat{E}\dot{\widehat{\state}} &= (\widehat{J}-\widehat{R})\widehat{\state} + \widehat{G}\inpVar, \\
       \widehat{\outVar} &= \widehat{G}^T \widehat{\state}.
    \end{align}
\end{subequations}
In particular, the \ROM for the \LTI case can be evaluated efficiently and independent of the full model dimension $\stateDim$, as soon as the matrices in~\eqref{eqn:strfreecontrol:ROM:matrices} are constructed.

\begin{remark}\label{rem:withq}
Note that the techniques that we describe below can also be extended to the case that $Q\neq I$. In this case, one uses a Petrov-Galerkin approach as described in~\eqref{eqn:ROM:construction}, \ie, different projection matrices from left and right. In more detail, if $Q$ is nonsingular, then the choice $V_{\ell} \vcentcolon= QV_{r}$ retains the \pHDAE structure in the \ROM. This strategy is also prevailing in the context of \pHODE systems, see for instance \cite{ChaBG16}, and even used to ensure stability-preservation in the context of \MOR for switched systems, see \cite{SchU18}.
\end{remark}

\begin{remark}
    Depending on the application at hand, the system matrices in~\eqref{eqn:strfreecontrol} may depend on additional parameters $\omega$. If these parameters are not fixed a-priori to a specific value, then one wants to preserve this parametric dependency in the \ROM. A standard assumption in the \MOR literature is, that the system matrices are available in a parameter-separable form, \ie,
    \begin{equation}
        \label{eqn:ROM:matrices:parameterSeparable}
        E(\omega) = \sum_{i=1}^K \gamma_i(\omega)E_i
    \end{equation}
    with scalar functions $\gamma_i$ and constant matrices $E_i\in\R^{\stateDim,\stateDim}$ for $i=1,\ldots,k$ (and similarly for the other matrices). In this case, the reduced matrices are simply obtained by reducing each $E_i$ separately. If the matrices are not in the parameter-separable form, or only with a very large $k$, then the (discrete) empirical interpolation method, see \cite{BarMNP04,ChaS10}, can be used instead. For more details, we refer to \cite{Haa17}.
\end{remark}

\begin{remark}
    One advantage of the projection-based approach is that besides the \pH structure, also the Hamiltonian is approximated with the same ansatz space. Thus, the general framework as discussed above not only preserves the \pH-structure but also retains information about the original Hamiltonian. However, reformulating the \pH system with a different Hamiltonian may be more amendable for \MOR. This is demonstrated in detail in \cite{BreMS20,BreU21} for \pHODEs and in \cite{BreS21} for \pHDAEs. Similar results are also achieved if the coefficients of the \ROM matrices are directly obtained by minimizing a suitable error function, see \cite{SchV20} for further details.
\end{remark}

Although this is not necessary in general, we often also perform another simplification that allows to clearly separate the dynamical part and the algebraic constraints. These parts have to be treated in a slightly different way and the reduction only takes place in the dynamical equations in order to assure that the model reduction does not violate the physical principles described by the constraints.
For this, let $V_0$ be an invertible matrix such that 
\begin{gather*}
    V_0^T E V _0 = \begin{bmatrix} E_{11} & 0 \\ 0 & 0 \end{bmatrix},\qquad 
    V_0^T (J-R) V_0 = \begin{bmatrix} J_{11}-R_{11} & J_{12}-R_{12}\\ J_{21}-R_{21} & J_{22}-R_{22} \end{bmatrix} \\ 
    V_0^TG = \begin{bmatrix} G_1 \\ G_2 \end{bmatrix},\qquad
    \begin{bmatrix} \state_1 \\ \state_2\end{bmatrix} = V_0^{-1}\state,
\end{gather*}
\ie, the transformed system is given as
\begin{subequations}
    \label{eqn:strfree:semiexpl}
        \begin{align}
        \begin{bmatrix}
            E_{11} & 0\\
            0 & 0
        \end{bmatrix}\begin{bmatrix}
            \dot{\state}_1\\
            \dot{\state}_2
        \end{bmatrix} &= \begin{bmatrix}
            \phantom{-}J_{11} - R_{11} & J_{12}-R_{12}\\
            -J_{12}^T - R_{12}^T & J_{22}-R_{22}
        \end{bmatrix}\begin{bmatrix}
            \state_1\\
            \state_2
        \end{bmatrix} + \begin{bmatrix}
            G_1\\
            G_2
        \end{bmatrix}\inpVar,\\
        \outVar &= \begin{bmatrix}
            G_1^T & G_2^T
        \end{bmatrix}\begin{bmatrix}
            \state_1\\
            \state_2
        \end{bmatrix}.
        \end{align}
\end{subequations}
The assumption that~\eqref{eqn:strfreecontrol} is of (Kronecker) index at most one, implies that $J_{22}-R_{22}$ is nonsingular. 
The decomposition can be easily obtained by first computing a full rank factorization of the positive semi-definite matrix $E$ using \eg
a singular value decomposition 
\begin{displaymath}
    E = U_0 \begin{bmatrix} \Sigma & 0 \\ 0 & 0 \end{bmatrix} V_0^T
\end{displaymath}
with invertible diagonal matrix $\Sigma$ and then forming
\begin{equation}
    \label{eqn:symmetricRankRevealingDecomposition}
    V_0^T E V_0= \begin{bmatrix} E_{11} & 0 \\ 0 & 0 \end{bmatrix}
\end{equation}
with $E_{11}=E^T_{11}>0$. 

\begin{remark}
    \label{rem:semiexplicit}
    If a semi-explicit representation with $E_{11} = I$ is required, then one can compute the Cholesky factorization $E_{11} = L_{11}L_{11}^T$ and perform another congruence transformation with $\widetilde{V}_0\vcentcolon= \diag(L_{11}^{-1},I)$. This yields, after renaming of the transformed matrices, the equivalent formulation
    \begin{subequations}
        \label{eqn:semiexpl}
        \begin{align}
            \label{eqn:semiexpl:state}\begin{bmatrix} I & 0 \\ 0 & 0 \end{bmatrix}
            \begin{bmatrix} \dot{\state}_1 \\ \dot{\state}_2\end{bmatrix} &= \begin{bmatrix} J_{11}-R_{11} & J_{12}-R_{12} \\ J_{21}-R_{12} & J_{22}-R_{22} \end{bmatrix}\begin{bmatrix} \state_1 \\ \state_2\end{bmatrix} +
            \begin{bmatrix} G_1 \\ G_2 \end{bmatrix} \inpVar\\
            \outVar &= \begin{bmatrix} G_1 \\ G_2 \end{bmatrix}^T \begin{bmatrix}\state_1 \\ \state_2\end{bmatrix},
        \end{align}
    \end{subequations}
    with $J_{22}-R_{22}$ nonsingular. For many \MOR techniques it is essential that the semi-explicit form~\eqref{eqn:semiexpl} is available. Fortunately, in many applications this can be done directly by exploiting the structure of the equations coming from the physical properties, see the examples in \Cref{sec:examples}. 
\end{remark}

Performing a Laplace transformation for the system~\eqref{eqn:strfreecontrol} yields the \emph{transfer function}
\begin{equation}
    \label{eqn:trfun}
    \transferFunc(s) = G^T(sE-J+R)^{-1}G,
\end{equation}
which can be used to assess the approximation quality of the \ROM via the $\mathcal{H}_2$ or $\mathcal{H}_\infty$ norm, see \eg \cite{AntBG20}. It is important to note that a singular $E$ implies that $\transferFunc$ may contain a polynomial term. In general, using the Weierstra\ss{} canonical form (cf.~\Cref{thm:WCF}), it is easy to see that the (Kronecker) index minus one defines an upper bound for the degree of the polynomial. Measuring the approximation error in the $\mathcal{H}_2$ norm thus requires that the polynomial part is matched exactly since otherwise, the error is unbounded. The situation is analogous for the $\mathcal{H}_\infty$ norm, except that the constant term in the polynomial does not need to be matched exactly. 

For simplicity of the presentation, we will only describe the single-input, single-output case, \ie, we assume $G\in \R^{\stateDim,1}$. All the algorithmic approaches can be easily extended to the multi-input multi-output case.

%%%%%%%%%%%%%%%%%%%%%%%%%%%%%%%%%%%%%%
\subsection{Power conservation based model order reduction}
\label{sec:power_method}

Two methods that carry out a \MOR for the Dirac structure representation are the effort and flow constraint reduction methods that were introduced for standard \pHODE systems in \cite{PolS12} and extended to \pHDAE systems in \cite{HauMM19}. The basic idea of these approaches is to find a suitable transformation for the dynamic part of the state~$\state_1$ that partitions the state into a part associated to the \ROM, denoted with~$\widehat{z}_1$, and a part that does not contribute much to the input-output behavior of the system, denoted with~$\widetilde{z}_1$. In more detail, one determines a matrix $V_1$ with orthonormal columns such that $\state_1 = V_1\begin{bmatrix} \widehat{z}_1^T & \widetilde{z}_1^T\end{bmatrix}^T$. Then one cuts the interconnection between the part of the energy storage port belonging to~$\widetilde{z}_1$ and the Dirac structure, such that no power is transferred. In this way, the power is exclusively exchanged via the energy storage of~$\widehat{z}_1$ and the part associated with~$\widetilde{z}_1$ is omitted. 

In more detail, following the general discussion about Dirac structures in \Cref{sec:dirac}, the relevant constitutive equations in term of \MOR are given as 
\begin{equation}
    \label{eqn:dirac:constitutive}
    -E\dot{\state} = f_\mathrm{s},\qquad e_\mathrm{s} = \state.
\end{equation}
\begin{remark}
    Recall that in general the constitutive equation for the effort variable is $e_\mathrm{s} = \eta(\state)$ in the nonlinear case, and $e_\mathrm{s} = Q\state$ in the linear case with quadratic Hamiltonian, see \Cref{thm:diracStructure} for further details. The methods that we will discuss can also be formulated for the more general case, see \cite{HauMM19}, but for the ease of presentation, we proceed here with $Q = I$ (cf.~\Cref{sec:noQ}).  
\end{remark}
Using the semi-explicit formulation~\eqref{eqn:semiexpl} and performing a congruence transformation with $V \vcentcolon= \diag(V_1,I)$, transforms the constitutive equations~\eqref{eqn:dirac:constitutive} as
\begin{equation}
    \label{eqn:fe}
    -\begin{bmatrix}
        I & 0 & 0\\
        0 & I & 0\\
        0 & 0 & 0
    \end{bmatrix}\begin{bmatrix}
        \dot{\widehat{\state}}_1\\
        \dot{\widetilde{\state}}_1\\
        \dot{\state}_2
    \end{bmatrix} = \begin{bmatrix}
        \widehat{f}_{\mathrm{s},1}\\
        \widetilde{f}_{\mathrm{s},1}\\
        f_{\mathrm{s},2}
    \end{bmatrix}, \qquad \begin{bmatrix}
        \widehat{e}_{\mathrm{s},1}\\
        \widetilde{e}_{\mathrm{s},1}\\
        e_{\mathrm{s},2}
    \end{bmatrix} = \begin{bmatrix}
        \widehat{\state}_1\\
        \widetilde{\state}_1\\
        \state_2
    \end{bmatrix}
\end{equation}
For the model reduction we have to identify the part that is influenced by the dissipation (the resistive port). For this we apply a symmetric full rank decomposition of $V^TRV$ to compute
\begin{align}
    \label{LowRankR}
    \begin{bmatrix}
        \widehat{R}_{11} & \widetilde{\widehat{R}}_{11} & \widehat{R}_{12}\\
        \widehat{\widetilde{R}}_{11} & \widetilde{R}_{11} & \widetilde{R}_{12}\\
        \widehat{R}_{21} & \widetilde{R}_{21} & R_{22}
    \end{bmatrix} = \begin{bmatrix}Z& \hat Z\end{bmatrix}\begin{bmatrix}R_1&0\\0&0\end{bmatrix}\begin{bmatrix}
Z^T\\ \hat Z^T\end{bmatrix}=ZR_1Z^T,
\end{align}
with $0<R_1=R_1^T\in\mathbb{R}^{\ell,\ell}$ and $Z\in\mathbb{R}^{\stateDim,\ell}$. Plugging \eqref{LowRankR} into the transformed system and introducing the associated flow and effort variables accordingly, \ie 
\begin{displaymath}
    f_\mathrm{d} = -R_1e_\mathrm{d},\qquad e_\mathrm{d} =Z^TV^TV^{-1}\state= \begin{bmatrix} \widehat{Z}_1^T & \widetilde{Z}_1^T & Z_2^T\end{bmatrix} \begin{bmatrix}
\widehat{e}_{\mathrm{s},1}\\\widetilde{e}_{\mathrm{s},1}\\e_{\mathrm{s},2}
\end{bmatrix},
\end{displaymath}
yields a \pHDAE with opened resistive port. Inserting the relations \eqref{eqn:fe} and introducing the external port variables $(f_\mathrm{p},e_\mathrm{p})=(\outVar,\inpVar)$, where 
\begin{displaymath}
    \outVar = (V G)^T (V^TV^{-1})\state = (V G)^Te_\mathrm{s} = \begin{bmatrix} \widehat{G}_1^T & \widetilde{G}_1^T & G_2^T\end{bmatrix} \begin{bmatrix} \widehat{e}_{\mathrm{s},1}\\ \widetilde{e}_{\mathrm{s},1}\\e_{\mathrm{s},2}
\end{bmatrix},
\end{displaymath}
 we obtain a new representation as
\begin{multline}
    \label{WholeDAE}
    -\begin{bmatrix}
        I & 0 & 0\\
        0 & I & 0\\
        0 & 0 & I\\
        0 & 0 & 0\\
        0 & 0 & 0
    \end{bmatrix} \begin{bmatrix}
        \widehat{f}_{\mathrm{s},1}\\
        \widetilde{f}_{\mathrm{s},1}\\
        f_{\mathrm{s},2}
    \end{bmatrix} = \begin{bmatrix}
        \widehat{J}_{11} & \widetilde{\widehat{J}}_{11} & \widehat{J}_{12}\\
        \widehat{\widetilde{J}}_{11} & \widetilde{J}_{11} & \widetilde{J}_{12}\\
        \widehat{J}_{21} & \widetilde{J}_{21} & J_{22}\\
        -\widehat{G}_1^T & -\widetilde{G}_1^T & -G_2^T\\
        -\widehat{Z}_1^T & -\widetilde{Z}_1^T & -Z_2^T
    \end{bmatrix} \begin{bmatrix}
        \widehat{e}_{\mathrm{s},1}\\
        \widetilde{e}_{\mathrm{s},1}\\
        e_{\mathrm{s},2}
    \end{bmatrix} \\
    + \begin{bmatrix}
        \widehat{Z}_1\\ 
        \widetilde{Z}_1\\
        Z_2\\0\\0
    \end{bmatrix}f_\mathrm{d} + \begin{bmatrix}
0\\0\\{0}\\0\\I\end{bmatrix}e_\mathrm{d} + \begin{bmatrix}
0\\0\\{0}\\I\\0\end{bmatrix}f_{\mathrm{p}}+
\begin{bmatrix}\widehat{G}_1\\
\widetilde{G}_1\\G_2\\0\\0\end{bmatrix}e_{\mathrm{p}}.
\end{multline}

With these preparations we are now ready to formulate the energy-based \MOR methods.
The main idea is to cut the interconnection 
\begin{equation}
    \label{eqn:discardedState}
    -\dot{\widetilde{\state}}_1 = \widetilde{f}_{\mathrm{s},1},\qquad \widetilde{e}_{\mathrm{s},1} = \widetilde{\state}_1
\end{equation}
between the energy storage corresponding to $\widetilde{\state}_1$ and the Dirac
structure, in such a way that no energy is transferred. The energy flow through the interconnection \eqref{eqn:discardedState} is set equal to zero by enforcing
\begin{equation}
    \label{eqn:cutDiracStructure}
    \widetilde{e}_{\mathrm{s},1}^T\widetilde{f}_{\mathrm{s},1} = 0\qquad\text{and}\qquad
    \widetilde{\state}_1^T\dot{\widetilde{\state}}_1 = 0.
\end{equation}
This can be achieved in two canonical choices, leading to two different \MOR methods that are discussed in the remainder of this subsection.

In the \emph{Effort Constraint Reduction Method} (\ECRM), we set $\widetilde{e}_{\mathrm{s},1} = 0$, which implies $\widetilde{\state}_1=0$. This choice thus immediately yields~\eqref{eqn:cutDiracStructure}. The reduced Dirac structure is obtained by inserting this relation and removing the second row in \eqref{WholeDAE}.
This yields the reduced \pHDAE  model 
\begin{subequations}
    \label{ECRM}
    \begin{align}
        \begin{bmatrix}
            I &0\\0&0
        \end{bmatrix}\begin{bmatrix}
            \dot{\widehat{\state}}_1\\
            \dot{\state}_2
        \end{bmatrix} &= \left(\begin{bmatrix}
            \widehat{J}_{11} & \widehat{J}_{12}\\
            \widehat{J}_{21} & J_{22}
        \end{bmatrix} - \begin{bmatrix}
            \widehat{R}_{11} & \widehat{R}_{12}\\ 
            \widehat{R}_{21} & R_{22}
        \end{bmatrix}\right)\begin{bmatrix}
            \widehat{\state}_1\\
            \state_2
        \end{bmatrix} + \begin{bmatrix}\widehat{G}_1\\G_2 \end{bmatrix}\inpVar,\\
        \outVar &= \begin{bmatrix}\widehat{G}_1^T & G_2^T\end{bmatrix}\begin{bmatrix}\widehat{\state}_1\\\state_2\end{bmatrix},
\end{align}
\end{subequations}
It remains to show that~\eqref{ECRM} is indeed port-Hamiltonian, which is easily established with \Cref{cor:pHDAEinv}, since~\eqref{ECRM} can be constructed via Galerkin projection.

\begin{remark}
    Note that the \ROM~\eqref{ECRM} is obtained by standard truncation, as is common in balancing type methods; see for instance \cite{GugA04}. The situation is different if the \pHDAE~\eqref{eqn:semiexpl} features a $Q$-term that is not identical to the identity. In this case, a simple truncation may destroy the \pH structure. Nevertheless, one can proceed similarly as above and rewrite the reduced Dirac structure (obtained by setting $\widetilde{e}_{s,1}=0$ and removing the second block row) as a \pHDAE. We refer to \cite{HauMM19} for further details.
\end{remark}

In the \emph{Flow Constraint Reduction Method} (\FCRM), the energy transfer between the energy-storing elements and the Dirac structure is cut by setting $\widetilde{f}_{\mathrm{s},1}=0$, which implies $\dot{\widetilde{\state}}_1=0$, and thus also~\eqref{eqn:cutDiracStructure}. Thus, $\widetilde{\state}_1$ is constant and can particularly be chosen as $\widetilde{\state}_1=0$. The second row in \eqref{WholeDAE} is then an algebraic equation which can be resolved for $\widetilde{e}_{\mathrm{s},1}$ if $\widetilde{J}_{11}$ is invertible, \ie 
\begin{equation}
    \label{eqn:FCRM:resolveDiscardedState}
    \widetilde{e}_{\mathrm{s},1} = -\widetilde{J}_{11}^{-1}\left(\widehat{\widetilde{J}}_{11}\widehat{e}_{\mathrm{s},1} + \widetilde{J}_{12}e_{\mathrm{s},2} + \widetilde{Z}_1 f_{\mathrm{d}} + \widetilde{G}_1 e_{\mathrm{p}}\right).
\end{equation}
Substituting~\eqref{eqn:FCRM:resolveDiscardedState} into~\eqref{WholeDAE} and removing the second block row, yields
\begin{multline}
    \label{WholeDAE:FCRM}
    -\begin{bmatrix}
        I & 0\\
        0 & I\\
        0 & 0\\
        0 & 0
    \end{bmatrix} \begin{bmatrix}
        \widehat{f}_{\mathrm{s},1}\\
        f_{\mathrm{s},2}
    \end{bmatrix} = \begin{bmatrix}
        \widehat{J}_{11} - \widetilde{\widehat{J}}_{11}\widetilde{J}_{11}^{-1}\widehat{\widetilde{J}}_{11} & \widehat{J}_{12} - \widetilde{\widehat{J}}_{11}\widetilde{J}_{11}^{-1}\widetilde{J}_{12}\\
        \widehat{J}_{21} - \widetilde{J}_{21}\widetilde{J}_{11}^{-1}\widehat{\widetilde{J}}_{11} & J_{22} - \widetilde{J}_{21}\widetilde{J}_{11}^{-1}\widetilde{J}_{12}\\
        -\widehat{G}_1^T + \widetilde{G}_1^T\widetilde{J}_{11}^{-1}\widehat{\widetilde{J}}_{11} & -G_2^T + \widetilde{G}_1^T\widetilde{J}_{11}^{-1}\widetilde{J}_{11}\\
        -\widehat{Z}_1^T + \widetilde{Z}_1^T\widetilde{J}_{11}^{-1}\widehat{\widetilde{J}}_{11} & -Z_2^T + \widetilde{Z}_1^T\widetilde{J}_{11}^{-1}\widetilde{J}_{12}
    \end{bmatrix} \begin{bmatrix}
        \widehat{e}_{\mathrm{s},1}\\
        e_{\mathrm{s},2}
    \end{bmatrix} \\
    + \begin{bmatrix}
        \widehat{Z}_1 - \widetilde{\widehat{J}}_{11}\widetilde{J}_{11}^{-1}\widetilde{Z}_1\\ 
        Z_2 - \widetilde{J}_{21}\widetilde{J}_{11}^{-1}\widetilde{Z}_1\\
        \widetilde{G}_1^T\widetilde{J}_{11}^{-1}\widetilde{Z}_1\\
        \widetilde{Z}_1^T\widetilde{J}_{11}^{-1}\widetilde{Z}_1
    \end{bmatrix}f_\mathrm{d} + \begin{bmatrix}
        0\\0\\0\\I
    \end{bmatrix}e_\mathrm{d} + \begin{bmatrix}
        0\\0\\I\\0
    \end{bmatrix}f_{\mathrm{p}} + \begin{bmatrix}
        \widehat{G}_1 - \widetilde{\widehat{J}}_{11}\widetilde{J}_{11}^{-1}\widetilde{G}_1\\
        G_2 - \widetilde{J}_{21}\widetilde{J}_{11}^{-1}\widetilde{G}_1 \\
        \widetilde{G}_1^T\widetilde{J}_{11}^{-1}\widetilde{G}_1\\
        \widetilde{Z}_1^T\widetilde{J}_{11}^{-1}\widetilde{G}_1
    \end{bmatrix}e_{\mathrm{p}}.
\end{multline}
The resulting \ROM then is again a \pHDAE system, but due to the elimination, it now has a feedthrough term (see the third block row in~\eqref{WholeDAE:FCRM}. We do not present the technical formulas here. For details, we refer to \cite{HauMM19}. In contrast to \ECRM, we immediately conclude that the \ROM obtained by \FCRM is not obtained via projection.

The reduced models obtained by \ECRM and \FCRM have similar properties but also major differences. Both methods have the same number of reduced states. The \ROM in \FCRM has an extra feedthrough term and requires the skew-symmetric matrix $\widetilde{J}_{11}$ to be invertible, which is impossible if it is a square matrix of odd size. If $\widetilde{J}_{11}$ is singular, then the procedure has to be modified, but a (rather technical) construction is possible to deal with this case.

The question that remains to be answered is how to choose the coordinates $\widehat{\state}_1$ and $\widetilde{\state}_2$ in an optimal way, which in general is an open problem. Instead, we present a balancing-inspired algorithm to perform the separation, which, of course, can also be used to compute a (numerically) minimal realization for the \pHODE. 
The details are presented in \Cref{alg:pHODEbalancing}; see also \cite{HauMM19}. We emphasize that the resulting \pHDAE is not balanced in the classical sense, but only inspired from standard balancing, see \cite{BreMS20,BorSF21} for other \pH structure-preserving balancing approaches.

\begin{algorithm}
    \caption{Structure-preserving balancing for \pHODEs}
    \label{alg:pHODEbalancing}
    \begin{flushleft}
    \textbf{Input:} \pHDAE~\eqref{eqn:semiexpl}\\
    \textbf{Output:} Balanced-like \pHDAE~\eqref{eqn:fe}\\[-.8em]
        \end{flushleft}
    
    \begin{description}
        \item[Step 1] Set $A_{11} \vcentcolon= J_{11}-R_{11}$.
        \item[Step 2] Compute solutions $\mathcal{P}_{11}$, $\mathcal{O}_{11}$ of the equations
            \begin{align*}
                A_{11}\mathcal{P}_{11}\mathcal{P}_{11}^T + \mathcal{P}_{11}\mathcal{P}_{11}^TA_{11}^T + G_1G_1^T &= 0,\\
                A_{11}^T\mathcal{O}_{11}\mathcal{O}_{11}^T + \mathcal{O}_{11}\mathcal{O}_{11}^TA_{11} + G_1G_1^T &= 0.
            \end{align*}
        \item[Step 3] Compute the singular value decomposition $U\Sigma W^T = \mathcal{P}_{11}^T\mathcal{O}_{11}$, and a QR-decomposition $V_1\mathcal{R} = PU$.
        \item[Step 4] Partition $V_1 = \begin{bmatrix} \widehat{V}_1 & \widehat{V}_2 \end{bmatrix}$ and perform a congruence transformation with the matrix $V\vcentcolon= \diag(V_1,I)$ to obtain the form~\eqref{eqn:fe}.
    \end{description}
\end{algorithm}

\subsection{Moment matching} 
\label{sec:mm} 
The \emph{moment matching} (\MM) method derives the \ROM using a Galerkin projection in such a way that the leading coefficients of the Taylor series expansion of the transfer function $\transferFuncRed$
at a given shift parameter $s_0\in\C\cup\{\infty\}$ of the reduced systems match those of the full-order system $\transferFunc$ at $s_0$.
For details of the \MM methods for \LTI \DAE systems we refer to \cite{Fre05} for $s_0\in\C$ and to \cite{BenS06} for $s_0=\infty$. The adaptation of these methods for \pHODE systems was developed in \cite{PolS10,PolS11}. Since the projection space that maps the original to the reduced problem is typically a Krylov subspace, constructed by using an Arnoldi or Lanczos iteration, see \eg, \cite{Bai02,Fre00,Gri97}, the resulting \MOR method is applicable to large-scale systems and numerically stable.

For a shift $\sigma_0\in \C$, a formal expansion of the transfer function $\transferFunc$ around~$s_0$, see \cite{Ant05}, %\citeasnoun{BenMS05} 
leads to
\begin{equation} 
    \label{eq:transfer}
     \transferFunc(s) = \sum_{i=0}^\infty m_i(\sigma_0-s)^i.
\end{equation}
The generalized moments $m_i$ can be written as
$m_i=G^T v_i$ with vectors $v_i$ that are determined recursively by solving the linear systems
\begin{subequations}
    \label{eqn:moments}
\begin{align}
    (\sigma_0 E -J+R) v_0 &= G,                        \label{eq:r0}\\
    (\sigma_0 E -J+R) v_i &= E v_{i-1},\quad i\geq 1,   \label{eq:r1}
\end{align}
\end{subequations}
and employing the Arnoldi-process \cite{Saa03} to generate an orthogonal basis for this Krylov subspace $\mathcal{V}=\text{span}\{v_0,\ldots,v_{r-1}\}$. Let the columns of~$V$ denote this orthonormal basis and construct the matrices for the \ROM as in~\eqref{eqn:strfreecontrol:ROM:matrices}. It is well-known that in this way the moments are matched up to level $r$, see \cite{Fre05,BenS06}. To ensure that the algebraic constraints are preserved in the \ROM, we exploit the semi-explicit form~\eqref{eqn:semiexpl} and construct the projection matrix only for the dynamic part, as in the following result taken from \cite{HauMM19}.

\begin{theorem}
\label{thm:momentMatching}
    Consider the \pHDAE~\eqref{eqn:strfree:semiexpl}. For given shift $\sigma_0\in\C$ compute the vectors $v_i$ for $i=0,\ldots,r-1$ as in \eqref{eqn:moments} and construct a matrix $\begin{bmatrix} V_1^T & V_2^T\end{bmatrix}^T$, partitioned accordingly to~\eqref{eqn:strfree:semiexpl} with orthonormal columns such that
    \begin{displaymath}
        \mathrm{span} \begin{bmatrix}
            V_1\\
            V_2
        \end{bmatrix} = \mathrm{span}\{v_0,\ldots,v_{r-1}\}.
    \end{displaymath}
    Then the \ROM
    \begin{align}
    \label{eqn:MM:ROM}
	\begin{aligned}
    \begin{bmatrix}
        V_1^T E_{11}V &0\\ 
        0&0
    \end{bmatrix}\begin{bmatrix}
        \dot{\widehat{\state}}_1\\
        \dot{\state}_2
    \end{bmatrix} & = \begin{bmatrix}
        V_1^T(J_{11}-R_{11})V_1 & V_1^T(J_{12}-R_{12})\\
        (J_{21}-R_{21})V_1 & J_{22}-R_{22}
    \end{bmatrix} \begin{bmatrix}   
        \widehat{\state}_1\\ 
        \state_2
    \end{bmatrix} + \begin{bmatrix}
        V_1^T G_1 \\G_2
    \end{bmatrix}\inpVar, \\
    \widehat{\outVar} &= \begin{bmatrix} G_1V_1 \\G_2
\end{bmatrix} \begin{bmatrix}\widehat{\state}_{1}\\ \state_2\end{bmatrix},\end{aligned}
    \end{align}
    retains the \pH structure and matches the first $r$ moments and the polynomial part of the transfer function.
\end{theorem}

\begin{remark}
    \label{rem:momentMatchingProjection}
    \Cref{thm:momentMatching} presents a seemingly easy solution to structure-preserving \MOR of  \pHDAE systems of (Kronecker) index one. Nevertheless, this may not be the maximal reduction that is possible, because redundant algebraic conditions cannot be removed, see \cite{MehS05} for further details.
\end{remark}

%%%%%%%%%%%%%%%%%%%%%%%%%%%%%%%%%%%%%
\subsection{Tangential interpolation}
\label{sec:tangential}
A fourth and very successful \MOR method for \LTI \ODE systems is the tangential interpolation method, see \cite{AntBG20} for the general theory and application. In contrast to moment matching, the transfer function and its derivatives are not interpolated at a single point but rather at multiple points. If the system has multiple inputs and outputs, the interpolation is typically only enforced along so-called tangential directions. The main motivation for this approach is the fact that an $\mathcal{H}_2$-reduced model interpolates the full-order model at several interpolation points along tangential directions, see \cite{AntBG20}. For different classes of \LTI \pHDAE systems, the method has been introduced in detail in \cite{BeaGM21}. We discuss the method for \pHDAE systems of the form~\eqref{eqn:strfreecontrol}. 

As in the previous section, we work with single-input single-output systems to ease the presentation, \ie, we assume $m=1$. All results can be extended to the multi-input multi-output case. For a prescribed set of interpolation frequencies $\sigma_1,\ldots,\sigma_r\in\C$, the goal is to construct a reduced \pHDAE system whose transfer function interpolates the transfer function of the original model at the prescribed frequency points, \ie, we want
\begin{equation}
    \label{eqn:interpolationConditions}
    \transferFunc(\sigma_i) = \transferFuncRed(\sigma_i)\qquad\text{for } i=1,\ldots,r.
\end{equation}
Following the moment matching approach from the previous subsection, we immediately obtain the following result for \LTI \pHDAE system of (Kronecker) index one.

\begin{theorem}
    \label{thm:rationalInterpolation:index1:v1}
    Consider the index-1 \pHDAE~\eqref{eqn:strfree:semiexpl}.
    For given interpolation points $\{\sigma_1,\ldots,\sigma_r\}\subseteq\C$ construct a matrix $\begin{bmatrix} V_1^T & V_2^T\end{bmatrix}^T\in\C^{\stateDim,r}$, partitioned accordingly, that satisfies
    \begin{equation*}
        \label{eqn:rationalInterpolationMatrix}
        \mathrm{span}\begin{bmatrix}
            V_1\\
            V_2
        \end{bmatrix} = \mathrm{span}\{
            \left(\sigma_1 E - J+R\right)^{-1}G, \ldots,  \left(\sigma_r E - J+R\right)^{-1}G\}.
    \end{equation*}
    Then the \ROM~\eqref{eqn:MM:ROM} retains the \pH structure, interpolates the original model at the interpolation points, and matches the polynomial part.
\end{theorem}

As discussed in the previous subsection, cf.~\Cref{rem:momentMatchingProjection}, the construction in \Cref{thm:rationalInterpolation:index1:v1} suffers from the fact that possible redundant algebraic equations are not removed. We thus present an alternative approach in the next theorem. Again, the main idea is to construct the \ROM via Galerkin projection such that the interpolation conditions~\eqref{eqn:interpolationConditions} are satisfied. Since, in general, such a \ROM will not match the polynomial part of the transfer function, we follow a strategy from \cite{MayA07} (see also \cite{GugSW13}) and modify the feedthrough term without violating the interpolation conditions. The corresponding result for \pHDAE systems from \cite{BeaGM21} is presented in the following theorem.

\begin{theorem}
    \label{thm:rationalInterpolation:index1:v2}
    Consider a \pHDAE~\eqref{eqn:strfree:semiexpl} with (Kroecker) index at most one.
    For given interpolation points $\{\sigma_1,\ldots,\sigma_r\}\subseteq\C$ construct a matrix $V \vcentcolon= \begin{bmatrix} V_1^T & V_2^T\end{bmatrix}^T\in\C^{\stateDim,r}$, partitioned accordingly, as in~\eqref{eqn:rationalInterpolationMatrix}. Define the matrices
    \begin{align*}
        \widehat{E} &\vcentcolon= V_1^TE_{11}V_1, & 
        \widehat{D} &\vcentcolon= -G_2^T(J_{22}-R_{22})^{-1}G_2, &
        \widehat{B} &\vcentcolon= V^TG + \mathds{1}\widehat{D},\\
        \widehat{C} &\vcentcolon= G^TV + \widehat{D}\mathds{1}^T, &
        \widehat{A} &\vcentcolon= V^T(J-R)V-\mathds{1}\widehat{D}\mathds{1}^T, &
        \widehat{J} &\vcentcolon= \tfrac{1}{2}(\widehat{A}-\widehat{A}^T), \\
        \widehat{R} &\vcentcolon= -\tfrac{1}{2}(\widehat{A}+\widehat{A}^T), &
         \widehat{P} &\vcentcolon= \tfrac{1}{2}(\widehat{C}^T - \widehat{B}), &
        \widehat{G} &\vcentcolon= \tfrac{1}{2}(\widehat{C}^T+B), \\
        \widehat{S} &\vcentcolon= \tfrac{1}{2}(\widehat{D}+\widehat{D}^T), &
        \widehat{N} &\vcentcolon= -\tfrac{1}{2}(\widehat{D}-\widehat{D}^T),
    \end{align*}
    with $\mathds{1} \vcentcolon= \begin{bmatrix} 1 & \cdots & 1\end{bmatrix}^T\in\R^{r}$.
    Then, the \ROM
    \begin{subequations}
        \label{eqn:interpROMpH}
        \begin{align}
            \widehat{E}\dot{\widehat{\state}} &= (\widehat{J}-\widehat{R})\state + (\widehat{G}-\widehat{P})\inpVar,\\
            \widehat{\outVar} &= (\widehat{G}+\widehat{P})^T\state + (\widehat{S}-\widehat{N})\inpVar,
        \end{align}
    \end{subequations}
    satisfies the interpolation conditions~\eqref{eqn:interpolationConditions} and matches the polynomial part of $\transferFunc$. If, in addition, the matrix $\begin{smallbmatrix}
        \widehat{R} & \widehat{P}\\
        \widehat{P}^T & \widehat{S}
    \end{smallbmatrix}$ is positive semi-definite, then~\eqref{eqn:interpROMpH} is a \pHDAE system.
\end{theorem}

\begin{remark}
    In general, the projection matrix $V$ in \Cref{thm:rationalInterpolation:index1:v1,thm:rationalInterpolation:index1:v2} is complex, and thus also the matrices in the \ROM are complex-valued. Nevertheless, if the interpolation points are closed under complex conjugation, then a state-space transformation can be used to find a real-valued realization. In practice, this can be done a-priori by choosing $V$ appropriately. For details, we refer to \cite{AntBG20}. A similar approach also applies to the \MM approach discussed in \Cref{thm:momentMatching}.
\end{remark}

\begin{remark}
    It is possible to extend these results to the \pHDAE systems of (Kronecker) index two. Such a generalization is discussed in detail in \cite{BeaGM21}.
\end{remark}

The crucial question that remains to be answered is the choice of the interpolation points $\sigma_1,\ldots,\sigma_r$. It is well-known, see for instance \cite{AntBG20}, that an $\mathcal{H}_2$-optimal reduced model interpolates the transfer function of the full-order model at the mirror images of the poles of the \ROM. In more detail, assume that the poles $\lambda_i\in\C$ for $i=1,\ldots,r$ of $\transferFuncRed$ are semi-simple. If $\transferFuncRed$ is an $\mathcal{H}_2$-optimal approximation, then
\begin{equation}
    \label{eqn:H2optimalityCond}
    \transferFunc(-\lambda_i) = \transferFuncRed(-\lambda_i)\quad\text{and}\quad
    \transferFunc'(-\lambda_i) = \transferFuncRed'(-\lambda_i)
    \quad\text{for } i=1,\ldots,r,
\end{equation}
where $\transferFunc'$ denotes the derivative with respect to $s$.
Since, these poles are not known a-priori, they cannot be used as interpolation points in \Cref{thm:rationalInterpolation:index1:v1,thm:rationalInterpolation:index1:v2}. Instead,  \cite{GugAB08} proposed a fixed-point iteration to resolve this problem, which is known as the \emph{iterative rational Krylov algorithm} (\IRKA). The main idea is to construct a \ROM via \Cref{thm:rationalInterpolation:index1:v1} or \Cref{thm:rationalInterpolation:index1:v2}, compute the poles of the transfer function and use its mirror images as the next set of interpolation points. This is repeated until convergence. For general unstructured descriptor systems, an Hermite interpolant can be constructed similarly as in \Cref{thm:rationalInterpolation:index1:v1,thm:rationalInterpolation:index1:v2}, see for instance \cite{GugSW13}. If we however preserve the \pH-structure as in \Cref{thm:rationalInterpolation:index1:v1,thm:rationalInterpolation:index1:v2}, then, in general, only a subset of the interpolation conditions~\eqref{eqn:H2optimalityCond} is satisfied, and hence, the resulting \ROM may not be optimal with respect to the $\mathcal{H}_2$-norm. Indeed, as our forthcoming numerical examples (see \Cref{sec:MOR:examples}) show, the approximation quality can be significantly improved, if a different Hamiltonian is used. 

One way to achieve such a reformulation with a Hamiltonian that is particularly amendable for \MOR is to adapt the strategy for passive \ODE systems discussed in \cite{BreU21} to the \pHDAE setting as follows. First, consider only the differential part of the \pHDAE~\eqref{eqn:strfree:semiexpl}, \ie, the implicit \pHODE
\begin{subequations}
    \label{eqn:strfree:truncatedODE}
    \begin{align}
        E_{11}\dot{\state}_1 &= (J_{11}-R_{11})\state_1 + G_1\inpVar,\\
        \outVar &= G_1^T\state_1.
    \end{align}
\end{subequations}
If~\eqref{eqn:strfree:truncatedODE} is not (numerically) minimal, compute a structure-preserving minimal realization, for instance via \Cref{alg:pHODEbalancing} or via the method described in \cite{BreMS20}. For the sake of notation, we assume that this step has already been done, \ie, we assume that~\eqref{eqn:strfree:truncatedODE} is already (numerically) minimal. Second, set $A_{11} \vcentcolon= (J_{11}-R_{11})E_{11}^{-1}$, and compute the minimizing solution $X_{11} = X_{11}^T > 0$ of the algebraic Riccati equation
\begin{displaymath}
    -A_{11}^TX_{11} - X_{11}A_{11} - (G_{11}^T - X_{11}G_{11})D_{11}(G_{11}-G_{11}^T-X_{11}) = 0.
\end{displaymath}
Then, construct the matrices
\begin{gather*}
    \tilde{E}_{11} \vcentcolon= X_{11}^{-1}, \qquad \tilde{J}_{11} \vcentcolon= \tfrac{1}{2}(A_{11}X_{11}^{-1}-X_{11}^{-1}A_{11}^T),\\ \tilde{R}_{11} \vcentcolon= -\tfrac{1}{2}(A_{11}X_{11}^{-1}+X_{11}^{-1}A_{11}^T),
\end{gather*}
and perform \MOR for the system 
\begin{subequations}
    \label{eqn:strfree:semiexpl:optimizedHamiltonian}
    \begin{align}
        \begin{bmatrix}
            \tilde{E}_{11} & 0\\
            0 & 0
        \end{bmatrix}\begin{bmatrix}
            \dot{\state}_1\\
            \dot{\state}_2
        \end{bmatrix} &= \begin{bmatrix}
            \phantom{-}\tilde{J}_{11} - \tilde{R}_{11} & J_{12}-R_{12}\\
            -J_{12}^T - R_{12}^T & J_{22}-R_{22}
        \end{bmatrix}\begin{bmatrix}
            \state_1\\
            \state_2
        \end{bmatrix} + \begin{bmatrix}
            G_1\\
            G_2
        \end{bmatrix}\inpVar,\\
        \outVar &= \begin{bmatrix}
            G_1^T & G_2^T
        \end{bmatrix}\begin{bmatrix}
            \state_1\\
            \state_2
        \end{bmatrix},
    \end{align}
\end{subequations}
provided that the matrix
\begin{displaymath}
    \begin{bmatrix}
        \tilde{R}_{11} & R_{12}\\
        R_{21} & R_{22}
    \end{bmatrix}
\end{displaymath}
is positive semi-definite.

\begin{remark}
    \label{rem:systemIdentification}
    Interestingly, a (generalized) state-space realization is not necessary to construct an interpolatory \ROM. As demonstrated in \cite{MayA07,AntLI17}, the \ROM can be constructed solely from the interpolation points $\sigma_i$ and associated measurements of the transfer function $\transferFunc$ and its derivative. Generalizations to models with structure are proposed for instance in \cite{SchUBG18}. First attempts to use frequency measurements to construct a low-dimensional \pHDAE, \ie, to use interpolation or least-squares approaches as a structure-inducing system identification framework, are presented in \cite{AntLI17,BenGV20,CheMH19_ppt,SchV20,Sch21,SchV21}. A first approach that works with time-domain data is presented in~\cite{ShaWK22}. Methods to analyze whether the available data is generated from a passive system are presented in \cite{RomMA17,RomBKA19,vanCPT21}.
\end{remark}

\subsection{Numerical examples}
\label{sec:MOR:examples}

To illustrate the performance of the discussed structure-preserving \MOR methods, we present a numerical example using a multibody system as described in \Cref{sec:robot}. The reported $\mathcal{H}_2$-norms are computed as in \cite{Sty06b}, using the \texttt{M-M.E.S.S.} Toolbox \cite{SaaKB21}.

A holonomically constrained mass-spring-damper system is a multibody problem that describes the one-dimensional dynamics of $g$ connected mass points in terms of their positions $q\colon\timeInt\to\R^g$, velocities $v\colon\timeInt\to\R^g$ and a Lagrange multiplier $\lambda\colon\timeInt\to\R$, see \Cref{MultiBodyPic}.
\begin{figure}[t]
    \centering
    \includegraphics[scale=0.4]{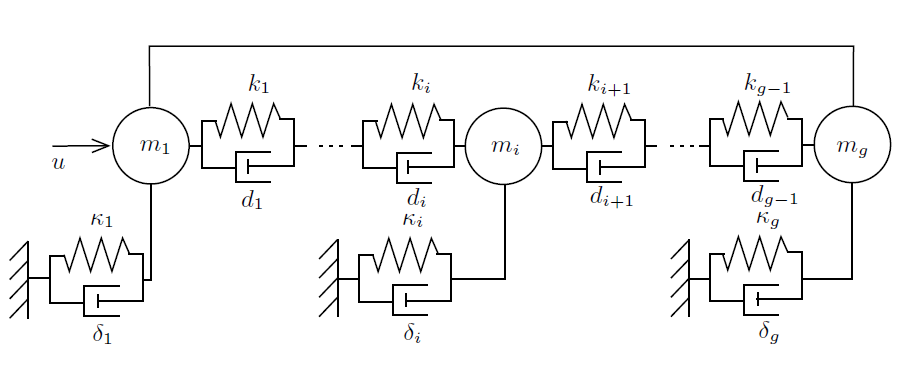}
    \caption{Damped mass-spring system with holonomic constraint taken from~\protect\cite{MehS05}.}
    \label{MultiBodyPic}
\end{figure}
Here the $i$th mass of the weight $m_i$ is connected to the $(i+1)st$ mass by a spring and a damper with constants $k_i>0$ and $d_i>0$, and also to the ground by a spring and a damper with the constants $\kappa_i>0$ and $\delta_i>0$, respectively. Furthermore, the first and the last mass points are connected by a rigid bar. The vibrations are driven by an external force $\inpVar\colon\timeInt\to\R$, the control input, acting on the first mass point. The resulting system has a mass matrix $M=\mathrm{diag}(m_1,\dots,m_g)$, symmetric positive definite tridiagonal stiffness and damping matrices $K$, $D\in \R^{g,g}$, a constraint matrix $C=\begin{bmatrix}1 & 0 & \dots & 0 & -1\end{bmatrix}\in \R^{1,g}$, and an input matrix $\tilde{G}=e_1\in \R^{g,1}$. Since $K$ and $D$ are symmetric positive definite, the problem can be formulated as a \pHDAE of (Kronecker) index two by replacing the algebraic constraint $Cq=0$ by its first derivative $Cv=0$ 
yielding the \pHDAE
\begin{align}
    \label{MultiBodyPH}
    \begin{bmatrix} K & 0 & 0\\ 0 & M & 0\\ 0&0&0\end{bmatrix} \begin{bmatrix}\dot{q}\\ \dot{v}\\ \dot{\lambda}\end{bmatrix} &= \left(\begin{bmatrix}0&K&0\\-K&0&-C^T\\
0&C&0\end{bmatrix}-\begin{bmatrix}0&0&0\\0&{D}&0\\0&0&0\end{bmatrix}\right)
\begin{bmatrix}
{q}\\v\\\lambda\end{bmatrix}+\begin{bmatrix}0\\ \tilde{G}\\0\end{bmatrix} \inpVar
\end{align}
with $\state \vcentcolon= \begin{bmatrix} q^T & v^T & \lambda^T\end{bmatrix}^T$ when adding an associated output equation $\outVar=G^T\state$.

The structure of the equations allows an easy construction of the condensed form required for the \MOR methods. Performing a full rank decomposition of $C$ as $CV = \begin{bmatrix} C_1 & 0\end{bmatrix}$ with $C_1$ invertible and an orthogonal matrix $V$, a congruence transformation yields the system
\begin{align*}
    \begin{bmatrix} K & 0 &0& 0\\ 0 & M_{11} & M_{12} & 0\\ 0& M_{12}^T&M_{22}&0\\0& 0&0&0\end{bmatrix}
    \begin{bmatrix}\dot{q}\\ \dot{v}_1\\ \dot{v}_2\\ \dot{\lambda}\end{bmatrix}&=
    \begin{bmatrix}0&K_1&K_2&0\\-K_1^T&-D_{11}&-D_{12}&-C_1^T
    \\-K_2^T&-D_{12}^T&-D_{22}&0\\0&C_1&0&0\end{bmatrix}
\begin{bmatrix}
{q}\\v_1\\v_2\\\lambda\end{bmatrix}+\begin{bmatrix}0\\G_1\\G_2\\0\end{bmatrix} \inpVar.
\end{align*}

The last equation $C_1v_1=0$ implies that $v_1=0$. Differentiating this equation and inserting it into the second equation yields the hidden constraint for the Lagrange multiplier
\begin{align*}
    {C_1}^{T}\lambda &= -M_{12}\dot{v}_2-K_1^T{q}-D_{12}v_2+G_1\inpVar,
\end{align*}
which imposes a consistency condition for the initial value.
The underlying \pHODE of size $n_1=2(g-1)$ together with the output equation are given by
\begin{subequations}
	\label{MBSuODE}
	\begin{align}
		\begin{bmatrix}
			K&0\\0&M_{22}
		\end{bmatrix}\begin{bmatrix}
			\dot{q}\\\dot{v}_2
		\end{bmatrix} &= \left(\begin{bmatrix}
			0&K_2\\-K_2^T&0
		\end{bmatrix} - \begin{bmatrix}
			0&0\\ 0&{D_{22}}
		\end{bmatrix}\right) \begin{bmatrix}
			{q}\\v_2\end{bmatrix}+\begin{bmatrix}
			0\\ G_2
		\end{bmatrix}u,\\
		y &= \begin{bmatrix}
			0 & G_2^T
		\end{bmatrix} \begin{bmatrix}{q}\\v_2\end{bmatrix}.
	\end{align}
\end{subequations}
If we permute the columns and rows such that the \pHODE~\eqref{MBSuODE} is in the leading blocks, then we can mimic the \MOR strategies from the previous subsections also for the \pHDAE systems of (Kronecker) index two, by only reducing the \pHODE~\eqref{MBSuODE}. Note that the lower-right $2\times 2$ block matrix has no skew-symmetric contribution and hence the relevant submatrix for the \FCRM is singular. We thus exclude \FCRM in the following.

For our numerical example, we choose a similar setting as in \cite{HauMM19} with parameters as listed in \Cref{tab:parametersMSD}.
\begin{table}[t]
    \centering
    \begin{tabular}{l@{\hspace{2em}}cccccccccc}
        \toprule
        \textbf{parameter} & $g$ & $m_i$ & $k_i$ & $d_i$ &  $\kappa_i$ & $\delta_i$ & $\kappa_1$ & $\kappa_g$ & $\delta_1$ & $\delta_g$\\
        \textbf{value} & $1000$ & $100$ & $2$ & $5$ & $2$ & $5$ & $4$ & $4$ & $10$ & $10$\\\bottomrule
    \end{tabular}
    \caption{Parameters for the holonomically constrained mass-spring-damper system}
    \label{tab:parametersMSD}
\end{table}
To compute a \ROM with the different methods, we pick $r\in\N$ and reduce only the \pHODE~\eqref{MBSuODE} while the algebraic part is not reduced, knowing that, in general, this is not optimal (cf.~\Cref{rem:momentMatchingProjection}). Whenever we report a reduced dimension~$r$, this means that one has to add the number of algebraic equations to the dimension of the \ROM. A frequency sweep for the different \ROMs with $r=10$ is presented in \Cref{fig:MBS:frequencyPlot} and relative $\mathcal{H}_2$-norms for different values of $r$ are reported in \Cref{fig:MBS:H2error}.
\begin{figure}[t]
    \centering
    \input{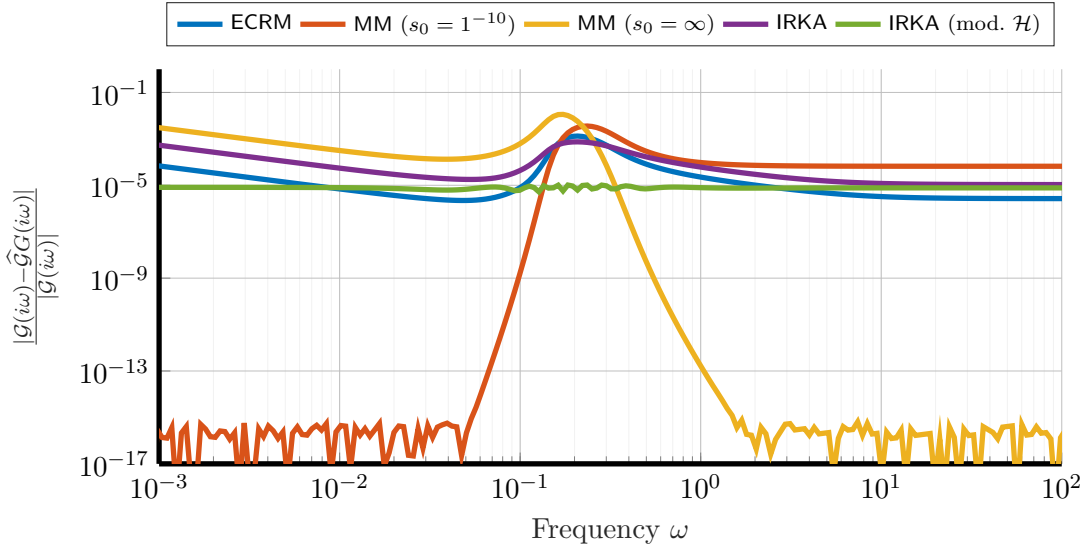}\\[-.8em]
    \caption{Relative errors of reduced transfer functions with $r=10$ plotted over the frequency for the mass-spring-damper system in the formulation~\eqref{MultiBodyPH}.}
    \label{fig:MBS:frequencyPlot}
\end{figure}
\begin{figure}[t]
    \centering
    \input{MSD-n2002-MORerrorDecay}\\[-.8em]
    \caption{Relative $\mathcal{H}_2$ errors of the \ROMs for the mass-spring-damper system~\eqref{MultiBodyPH} for different reduced orders.}
    \label{fig:MBS:H2error}
\end{figure}
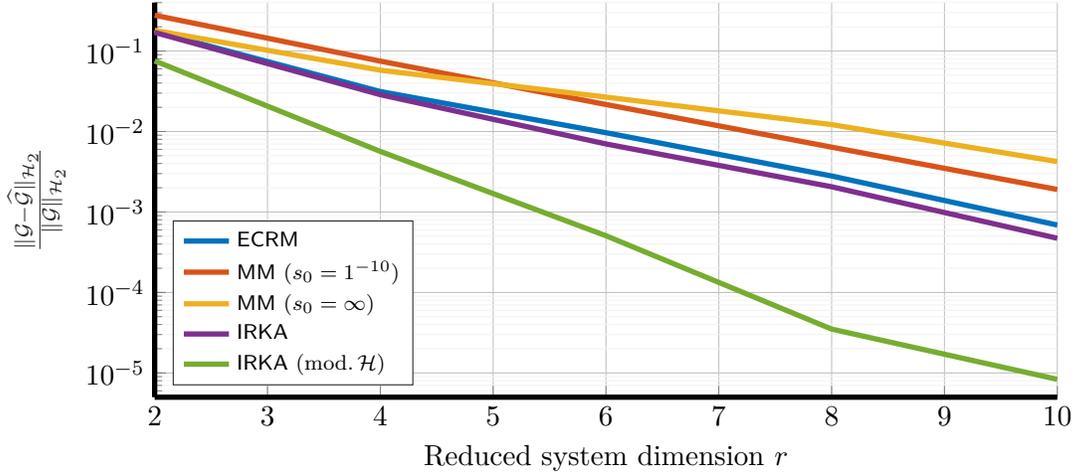
For the \ROMs with $r=10$, cf.~\Cref{fig:MBS:frequencyPlot}, we observe that \MM with shift $s_0=\infty$ and $s_0=10^{-10}$ yields outstanding approximations (errors of order $\mathcal{O}(10^{-15})$) for high and low frequencies, respectively. \ECRM and rational interpolation with interpolation points computed via~\IRKA as described in \Cref{sec:tangential}, in contrast, provides a uniformly good approximation quality of order $\mathcal{O}(10^{-3})$, independently of the chosen frequency. As already discussed before, the structure-preserving variant of \IRKA cannot satisfy all the necessary optimality conditions. To improve the situation, we also present the error for tangential interpolation via structure-preserving \IRKA with a modified Hamiltonian, denoted as \IRKA (mod.~$\hamiltonian$) in the figures, as described in~\eqref{eqn:strfree:semiexpl:optimizedHamiltonian}, which yields a significant improvement over the original formulation. This can also be seen in the relative $\mathcal{H}_2$-errors displayed in \Cref{fig:MBS:H2error}, where for $r\geq 6$, the \IRKA-reduced \pHDAE with modified Hamiltonian yields an approximation that is at least one order of magnitude better. 

%%%%%%%%%%%%%%%%%%%%%%%%%%%%%%%%%%%%%%%%%%%%%%%%%%%%%%%%%%%%%%%
\section{Temporal discretization and linear system solvers}
\label{sec:TimeDiscretization}
In this section, we discuss the time-discretization of \dHDAE and \pHDAE systems and the associated linear system solves. The numerical solution of initial and boundary value problems for general \DAE systems of the form \eqref{eqn:descriptorSystem:stateEq} has been an important research topic, see the monographs \cite{BreCP96,HaiLR89,HaiW96,LamMT13}. Following the approach discussed in \cite[Cha.~6]{KunM06}, we may assume that the \DAE is represented at every time-step in one of the strangeness-free forms \eqref{eqn:nonlinDAE:sfree} or \eqref{eqn:ltvDAE:sfree}, and then it has been shown that many implicit one-step and multi-step methods can be applied and have the same convergence order as for \ODE systems.  

\subsection{Time discretization for pHDAEs}
\label{sec:timedisc}
Most of the classical time-discretization techniques do not respect a given \dHDAE or \pHDAE structure in such a way that the time-discretized system still satisfies a power-balance equation or dissipation inequality. To analyze for which approaches this is guaranteed is an active research area that is proceeding in different directions. A very promising approach is the discretization of \pHDAE systems in such a way, that the time-discrete system satisfies a discrete version of the  power-balance equation, see \cite{CelH17,KotL18,KotML18} for \pHODE systems, and \cite{MehM19,Mor22} for \pHDAE systems. 
Another class of methods, particularly for non-dissipative \ODE methods, is based on energy preserving geometric integration, see \eg, \cite{CelMMOQW09,HaiLW02,QuiM08}.
The analysis and comparison of these techniques is a topic on its own, so we only briefly discuss the approach based on collocation methods for \pHDAEs of the form \eqref{eqn:pHDAE} in \cite{MehM19}.

Consider an autonomous \pHDAE of the form \eqref{eqn:pHDAE} with  a given input function
$\inpVar\colon\timeInt\to\R^m$, a consistent initial value $\state(t_0)=\state_0$, and suppose that we want to approximate the solution in a time interval $(t_0,t_{\mathrm{f}}=t_0+\tau)$ 
by a polynomial $\tilde{\state}(t)$
of degree at most $s$. For a collocation method, the polynomial $\tilde{\state}(t)$ is chosen such that $\tilde{\state}(t_i)=\state(t_i)$ satisfies the \pHDAE~\eqref{eqn:pHDAE} in the $s$ collocation points $t_i=t_0+\tau \gamma_i$ with $\gamma_i\in[0,1]$ for $i=1,\ldots,s$.

Let $\ell_i$ denote the $i$-th Lagrange interpolation polynomial with respect to the nodes $\gamma_1,\ldots,\gamma_s$, \ie
\begin{equation*}
  \ell_i(t) \vcentcolon= \prod_{\substack{j=1\\j\neq i}}^s\frac{t-\gamma_j}{\gamma_i-\gamma_j}.
\end{equation*}
Then for collocation methods one requires that
\begin{equation*}
    \dot{\tilde{\state}}(t_0+t\tau) = \sum_{i=1}^s\dot{\state}_i\ell_i(t), \qquad
    \tilde{\state}(t_0+t\tau) = \state_0 + \tau\sum_{j=1}^s\dot{\state}_j\int_0^t\ell_j(\sigma)\dsigma,
\end{equation*}
for the values $\dot{\state}_i\vcentcolon=\dot{\tilde{\state}}(t_i)$, and also
\begin{align*}
  \state_i &\vcentcolon=\tilde{\state}(t_i) = \state_0 + \tau\sum_{j=1}^s\alpha_{ij}\dot{\state}_j, &
  \state_f &\vcentcolon=\tilde{\state}(t_{\mathrm{f}}) = \state_0 + \tau\sum_{j=1}^s\beta_j\dot{\state}_j,
\end{align*}
where the coefficients $\alpha_{ij}\vcentcolon=\int_0^{\gamma_i}\ell_j(\sigma)\dsigma$ and $\beta_j\vcentcolon=\int_0^1\ell_j(\sigma) \dsigma$, $i,j=1\ldots s$ are the coefficients of the Butcher diagram of the associated Runge-Kutta method, see \cite{HaiLW02}.

To  preserve the \pHDAE  structure one uses  the Dirac structure $\mathcal{D}_\state$ as described in~\eqref{eqn:diracstructure} associated to the \pHDAE~\eqref{eqn:pHDAE} and defines the Dirac structure associated to the time-discretization 
\begin{align*}
  \mathcal{D}_{\state_i} &=
  \left\{(f^i,e^i)\in\mathcal V_{\state_i}\times\mathcal V_{\state_i}^* \; \left|\;
  f^i +
    \begin{bmatrix}
        \Gamma(\state_i) & I_{\ell+m} \\
        -I_{\ell+m} & 0
    \end{bmatrix} 
    e^i
    = 0
  \right.\right\},\ i=1,\ldots,s
\end{align*}
with $f^i=\begin{bmatrix}(f_\mathrm{s}^i)^T & \outVar_i^T &(f_\mathrm{d}^i)^T\end{bmatrix}^T$ and $e^i=\begin{bmatrix}(e_\mathrm{s}^i)^T,\inpVar_i^T,(e_\mathrm{d}^i)^T\end{bmatrix}^T$, where
\begin{align*}
%  \state_f &= \state_0 + \tau\sum_{i=1}^s\beta_i\dot{\state}_i, &
% \state_i &= \state_0 + \tau\sum_{j=1}^s\alpha_{ij}\dot{\state}_j,  \\
  f_\mathrm{s}^i &= -E(\state_i)\dot{\state}_i, & 
  e_\mathrm{s}^i &= \eta(\state_i), \\
  e_\mathrm{d}^i &= -W(\state_i)f_\mathrm{d}^i, &
  \inpVar_i &= \inpVar(\state_i).
\end{align*}
In this way one obtains a system that is equivalent to the approach of applying the collocation method and then computing discrete inputs and outputs $\inpVar_i, \outVar_i$, for $i=1,\ldots,s$, and then introduces the associated collocation flows, efforts, input and output as
\begin{align*}
  \tilde{f}_\mathrm{s}(t_0+\tau t) &= \sum_{i=1}^s f_\mathrm{s}^i\ell_i(t), &
  \tilde{e}_\mathrm{s}(t_0+\tau t) &= \sum_{i=1}^s e_\mathrm{s}^i\ell_i(t), \\
  \tilde{f}_\mathrm{d}(t_0+\tau t) &= \sum_{i=1}^s f_\mathrm{d}^i\ell_i(t),  &
  \tilde{e}_\mathrm{d}(t_0+\tau t) &= \sum_{i=1}^s e_\mathrm{d}^i\ell_i(t), \\
  \tilde{\outVar}(t_0+\tau t) &= \sum_{i=1}^s \outVar_i\ell_i(t), \ &
  \tilde{\inpVar}(t_0+\tau t) &= \sum_{i=1}^s \inpVar_i\ell_i(t).
\end{align*}
In this way the discrete values are in $\mathcal D_{\tilde z}$ in all collocation points $t_i$ and the discretization preserves  the structure. To see this, let us consider the evolution of the Hamiltonian $\hamiltonian$ along the collocation polynomial $\tilde{\state}(t)$. For this, let $\tilde {\hamiltonian}(t)\vcentcolon=\hamiltonian(\tilde{\state}(t))$. Then we have
\begin{equation*}
  \tilde{\hamiltonian}(t)- \tilde{\hamiltonian}(t_0) = \int_{t_0}^t\dot{\tilde{\hamiltonian}}(\sigma)\dsigma,
\end{equation*}
and in the collocation points the power balance equation is satisfied, \ie
\begin{equation*}
    \dot{\tilde{\hamiltonian}}(t_i) = \tfrac{\partial}{\partial {\state_i}}\tilde{\hamiltonian}^T\dot{\state}_i = \eta (\state_i)^TE(\state_i)\dot{\state}_i = -\dualp{e_\mathrm{s}^i}{f_\mathrm{s}^i} = \dualp{e_\mathrm{d}^i}{f_\mathrm{d}^i} + \dualp{\outVar_i}{\inpVar_i},
\end{equation*}
for $i=1,\ldots,s$.
Applying the quadrature rule associated to the collocation method to evaluate the integral, we get
\begin{subequations}
  \begin{align*}
    \tilde{\hamiltonian}(t_{\mathrm{f}}) - \tilde{\hamiltonian}(t_0) &= \tau\sum_{j=1}^s\beta_j\dot{\tilde{\hamiltonian}}(t_j) + \mathcal{O}(\tau^{p+1}) 
    = -\tau\sum_{j=1}^s\beta_j\dualp{e_\mathrm{s}^j}{f_\mathrm{s}^j} + \mathcal{O}(\tau^{p+1})\\
    &= \tau\sum_{j=1}^s\beta_j\dualp{e_\mathrm{d}^j}{f_\mathrm{d}^j} + h\sum_{j=1}^s\beta_j\dualp{\outVar_j}{\inpVar_j} + \mathcal{O}(\tau^{p+1}),
  \end{align*}
\end{subequations}
where $p\in\N$ is the approximation order of the quadrature rule.
With the same argument we get
\begin{subequations}
  \begin{align}
    \label{eq:discDissInt}
    \tau\sum_{j=1}^s\beta_j\dualp{e_\mathrm{d}^j}{f_\mathrm{d}^j} &= \int_{t_0}^{t_{\mathrm{f}}}\dualp{\tilde{e}_\mathrm{d}(\sigma)}{\tilde{f}_\mathrm{d}(\sigma)} \dsigma + \mathcal O(\tau^{p+1}), \\
    \label{eq:discPortInt}
    \tau\sum_{j=1}^s\beta_j\dualp{\outVar_j}{\inpVar_j} &= \int_{t_0}^{t_{\mathrm{f}}}\dualp{\tilde{\outVar}(\sigma)}{\tilde{\inpVar}(\sigma)} \dsigma + \mathcal{O}(\tau^{p+1}),
  \end{align}
\end{subequations}
and hence
\begin{equation*}
  \tilde{\hamiltonian}(t_{\mathrm{f}}) - \tilde{\hamiltonian}(t_0) = \int_{t_0}^{t_{\mathrm{f}}} \Big(\dualp{\tilde{e}_\mathrm{d}(\sigma)}{\tilde{f}_\mathrm{d}(\sigma)}
  + \dualp{\tilde{\outVar}(\sigma)}{\tilde{\inpVar}(\sigma)}\Big) \dsigma
  + \mathcal{O}(\tau^{p+1}).
\end{equation*}
If $p\geq 2s-2$, then  \eqref{eq:discDissInt} and \eqref{eq:discPortInt} are even exact,
and, if $\beta_j\geq0$ for $j=1,\ldots,s$, as is the case for many collocation methods, we have that $\tau\sum_{j=1}^s\dualp{e_\mathrm{d}^j}{f_\mathrm{d}^j}\leq0$, thus the discrete system satisfies the same qualitative behaviour as the continuous problem. 

If the Hamiltonian $\hamiltonian$ is quadratic, \ie
\begin{equation*}
  \hamiltonian(\state) = \tfrac{1}{2}\state^TE\state + v^T\state + c,
\end{equation*}
for some $E=E^T\in\R^{\stateDim,\stateDim}$, $v\in\R^n$ and $c\in\R$, then 
we have that $\tilde{\hamiltonian} = \hamiltonian(\tilde{\state})$ and
$\dot{\tilde{\hamiltonian}}$ are polynomials of degree $2s$ and degree $2s-1$, respectively. 
Using the well-known fact, see \eg~\cite{HaiW96}, that the maximum degree of exactness for quadrature rules with $s$ nodes is $2s-1$, and that it is attained only 
with Gau{\ss}-Legendre collocation methods,
it follows that for these methods the integration of $\dot{\tilde{\hamiltonian}}$ is exact, \ie
\begin{equation*}
  \begin{split}
    \tilde{\hamiltonian}(t_{\mathrm{f}})-\tilde{\hamiltonian}(t_0)
    &= \tau \sum_{j=1}^s\beta_j\dualp{e_\mathrm{d}^j}{f_\mathrm{d}^j} + \tau\sum_{j=1}^s\beta_j\dualp{\outVar_j}{\inpVar_j} = \\
    &= \int_{t_0}^{t_{\mathrm{f}}}\left(\dualp{\tilde{e}_\mathrm{d}(s)}{\tilde{f}_\mathrm{d}(s)} + \dualp{\tilde{\outVar}(\sigma)}{\tilde{\inpVar}(\sigma)}\right)\dsigma.
  \end{split}
\end{equation*}
Since furthermore, for these methods
we have $\beta_j\geq0$, it follows that the dissipation term is always non-positive, and we obtain the discrete version of the dissipation inequality
\begin{equation*}
  \tilde{\hamiltonian}(t_{\mathrm{f}})-\tilde{\hamiltonian}(t_0) \leq \tau \sum_{j=1}^s\beta_j \dualp{\outVar_j}{\inpVar_j} = \int_{t_0}^{t_{\mathrm{f}}} \dualp{\tilde{\outVar}(\sigma)}{\tilde{\inpVar}(\sigma)}\dsigma,
\end{equation*}
hence for quadratic Hamiltonians the \pHDAE structure is fully preserved.

\begin{example} 
	\label{ex:circuit} 
	Consider the numerical solution of the strangeness-free \dHDAE system, given by the power network presented in \Cref{sec:power}. We use the artificial constants $E_G=0$, $L=2$, $C_1=0.01$, $C_2=0.02$, $R_\mathrm{L}=0.1$, $R_\mathrm{G}=6$ and $R_\mathrm{R}=3$, see \cite{MehM19}. For the time integration we choose $\tau=0.01$ and $t_{\mathrm{f}} = 1$ and use the implicit midpoint rule, \ie, the Gau{\ss}-Legendre collocation method with $s=1$ stages and order $p=2$. The numerical result for the consistent initial value
	\begin{displaymath}
	    \state_0 = \sqrt{\tfrac{10}{R_\mathrm{R}}}
	    \begin{bmatrix}
	        1&
	        -R_{\mathrm{R}}-R_{\mathrm{L}}&
	        -R_{\mathrm{R}}&
	        -\tfrac{R_{\mathrm{R}}-R_{\mathrm{L}}}{R_\mathrm{R}}&
	        -1
	    \end{bmatrix}^T
	\end{displaymath}
	is presented in \Cref{fig:powerNetworkSimulation}. One observes that after an initial phase the state converges to zero and the Hamiltonian decreases monotonically.
\end{example}
\begin{figure}[ht]
  \centering
  \newcommand{\figurewidth}{\linewidth}
  \newcommand{\figureheight}{.5\linewidth}
  \input{RLCimplicitMidpoint}
  \caption{Evolution of state components (solid lines) and Hamiltonian (dashed line).}
  \label{fig:powerNetworkSimulation}
\end{figure}
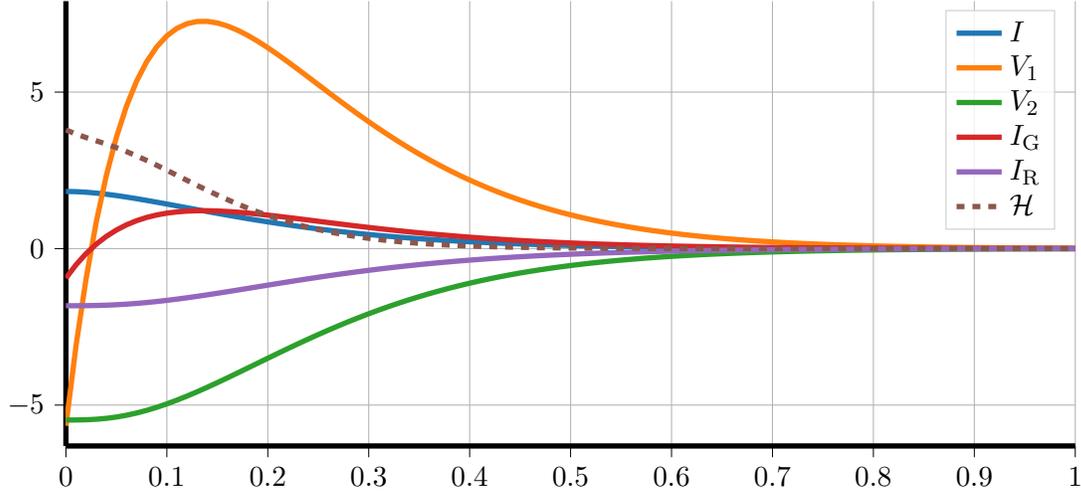

%%%%%%%%%%%%%%%%%%%%%%%%%%%%%%%%%%%%%%%%%%%%%%
\subsection{Linear System Solvers}\label{sec:linsystem}
In every step of an implicit time-discretization method  for a finite dimensional \pHDAE system, it is necessary to solve linear algebraic systems of the form
\begin{equation} \label{eqn:dHDAE-basic}
Ax=\Big(E+ \tau(R-J)\Big) x= b,
\end{equation}
where $\tau$ is the time-stepsize. The matrix $A$ can be  split in its symmetric and skew-symmetric part $A=H+S$, where  $H=\frac 12((E+\tau R)^T+(E+\tau R))\geq 0$ and $S=\frac 12((E+\tau R)^T-(E+\tau R))$. An analogous linear system structure occurs in discretized linear time-varying and nonlinear \pHDAE systems, in the construction of reduced models \cite{EggKLMM18},  and by multiplying some equations with $-1$ in optimization methods,see also \cite{GudLMS21_ppt,ManM19}.

The fact that  the symmetric part $H$ is positive semi-definite can be exploited in direct or iterative solution methods.
In the small and medium scale case we can make use of the following staircase form,  see~\cite{AchAM21}, which we present here in the real case.
\begin{lemma}\label{lem:sts}
Consider $A=H+S\in\mathbb R^{n,n}$, where $H=H^T\geq 0$ and $S=-S^T$. Then there exist a real orthogonal matrix $U\in\mathbb R^{n,n}$, and integers $n_1 \geq n_2 \geq \cdots \geq n_{r-1} \geq 0$ and $n_r \geq 0$, such that
\begin{equation}\label{eq:stcform}
U^T HU= \begin{bmatrix} H_{11} & 0 \\ 0 & 0\end{bmatrix},\
U^T SU =
\begin{bmatrix}
S_{11} & S_{12} &  & & 0 \\
S_{21} & S_{22} &  \ddots &  & 0 \\
&  \ddots & \ddots  & S_{r-2,r-1} &\vdots  \\
& &  S_{r-1,r-2}& S_{r-1,r-1} & 0\\
0 & \cdots & \cdots  & 0 & S_{r,r}
\end{bmatrix} ,
\end{equation}
where $H_{11}=H_{11}^T \in\R^{n_1,n_1}$ is positive definite, $S_{ii}=-S_{ii}^T\in\mathbb R^{n_i,n_i}$ for $i=1,\dots,r$,
and $S_{i,i-1}=-S_{i-1,i}^T=[\Sigma_{i,i-1}\;0]\in\R^{n_{i},n_{i-1}}$ with $\Sigma_{i,i-1}$
being nonsingular for $i=2,\dots,r-1$.
\end{lemma}
\begin{proof} We present the proof for completeness, see also \cite{GudLMS21_ppt}.
The result is trivial when $H$ is nonsingular (and thus positive definite), since in this case it holds with $U=I$, $r=2$, $n_1=n$, and $n_2=0$.

Let $H=H^T\geq 0$ be singular. We consider a full rank decomposition of $H$ with a real orthogonal $U$ such that
\begin{equation*}
    %\label{eqn:sim-1}
    U_1^T HU_1=\begin{bmatrix} \widehat{H}_{11} & 0\\ 0 & 0\end{bmatrix},
\end{equation*}
where we assume that $\widehat{H}_{11}=\widehat{H}_{11}^T\in\mathbb R^{n_1,n_1}$, with $0\leq n_1\leq n$, is void or positive definite.  Applying the same orthogonal congruence transformation to $S$ gives the matrix
\begin{equation}\label{eqn:sim-2}
\widehat{S}=U_1^T SU_1=\begin{bmatrix} \widehat{S}_{11} & \widehat{S}_{12}\\ \widehat{S}_{21} & \widehat{S}_{22}\end{bmatrix},
\end{equation}
where $\widehat{S}_{11}\in\R^{n_1,n_1}$, and $\widehat{S}_{21}=-\widehat{S}_{12}^T$, since $S$ is skew-symmetric. If $\widehat H_{11}$ is void or $\widehat{S}_{21}=0$, then the proof is complete. Otherwise, let
\[
\widehat{S}_{21}=W_2\begin{bmatrix}\Sigma_{21} & 0\\ 0 & 0\end{bmatrix}V_2^T
\]
be a singular value decomposition, where $\Sigma_{21}$ is nonsingular (and diagonal), and
$W_2\in\R^{n_1,n_1}$, $V_2\in\R^{n-n_1,n-n_1}$ are real orthogonal. We define $U_2\vcentcolon=\diag(V_2,W_2)\in\R^{n,n}$ and form
\[
U_2^T U_1^T HU_1U_2=\begin{bmatrix} V_2^T\widehat{H}_{11}V_2 & 0\\ 0 & 0\end{bmatrix},
\]
where $V_2^T\widehat{H}_{11}V_2\in\R^{n_1,n_1}$ is void or symmetric positive definite, and
\[
U_2^T U_1^T SU_1U_2=\begin{bmatrix} V_2^T\widehat{S}_{11}V_2 & V_2^T\widehat{S}_{12}W_2\\ W_2^T\widehat{S}_{21}V_2 & W_2^T\widehat{S}_{22}W_2\end{bmatrix}=
\begin{bmatrix}\widetilde{S}_{11} & \widetilde{S}_{12} & 0\\
\widetilde{S}_{21} & \widetilde{S}_{22} & \widetilde{S}_{23}\\
0 &\widetilde{S}_{32} & \widetilde{S}_{33}\end{bmatrix}
\]
with $\widetilde{S}_{21}=[\Sigma_{21}\;0]$. If $\widetilde{S}_{32}=0$ or is void then the proof is again complete. Otherwise we continue inductively and after finitely many steps we obtain a decomposition of the required form.
\end{proof}

When  the staircase form \eqref{eq:stcform} has been computed, then the transformed  linear system $(U^T AU)(U^T x)=U^T b$ can be solved using block Gaussian elimination.
\begin{lemma}\label{lem:schur}
Consider the matrix $U^T AU$ in \eqref{eq:stcform}.  Then there exist invertible lower and upper block bi-diagonal matrices
$L_s,R_s$, respectively,  such that
\[
L_s U^T A U R_s=\begin{bmatrix}
{H}_{11}+S_{11} & &  & &  \\
 & \mathcal{S}_1 &   &  &  \\
&  & \ddots  &  &  \\
& &  & \mathcal{S}_{r-2} & \\
 & &  & & \mathcal {S}_{r,r}
\end{bmatrix},
\]
with Schur complements $\mathcal{S}_1,\dots,\mathcal{S}_{r-2}$ that have  positive definite symmetric parts.
\end{lemma}
\begin{proof} A constructive proof via a sequence of block Gaussian elimination steps is given in \cite{GudLMS21_ppt}.
It relies on the fact that in every step (except the last one) the Schur complement has a symmetric part which is positive definite.
\end{proof}

The properties of the linear system resulting from the \dHDAE structure also have an immediate advantage in the context of iterative methods. If the symmetric part is positive definite then in~\cite{Wid78} it is suggested to solve, instead of $Ax=b$, the equivalent system
\begin{equation}\label{eqn:IpKsystem}
(I+K)x=\hat{b},\quad \mbox{where}\quad K=H^{-1}S,\quad \hat{b}=H^{-1}b.
\end{equation}
This transformation is a left preconditioning of the original system with its positive definite symmetric part, which defines the $H$-inner product 
\begin{displaymath}
    \langle x,y\rangle_H=y^T Hx.
\end{displaymath} 
This implies that one can construct optimal
Krylov subspace methods based on three-term recurrences for the system \eqref{eqn:IpKsystem}, see~\cite{Rap78,Wid78}.
If the symmetric part is semi-definite but singular, then one has to identify the nullspace,  which is actually easy in many applications.
The advantages of this approach and the fact that one obtains a rigorous convergence analysis and optimality conditions is discussed in detail in \cite{GudLMS21_ppt} and illustrated with several numerical examples including those discussed in \Cref{sec:examples}.

\begin{example}\label{ex:brake}
The finite element model of the disk brake discussed in \Cref{sec:brake} leads to a second order \DAE of the form
\[
M\ddot{p}+D\dot{p}+Kp=f,
\]
with $p$ the coefficient vector of displacements of the structure, with frequency dependent mass matrix $M=M^T>0$, damping matrix  $D=D^T\geq 0$, and stiffness matrix $K=K^T>0$, see~\cite{GraMQSW16} evaluated for $\omega_{\mathrm{ref}}=500$. For $f=0$, after  a first-order reformulation and discretization with the implicit mid-point rule one obtains a linear system with $\stateDim=9338$ and a positive definite symmetric part. 

As shown in \Cref{table:brake}, a preconditioned GMRES methods (preconditioned with the inverse of the symmetric part), even though using fewer iterations, takes a significantly longer time than the method in \cite{Wid78}, due to the full recurrences in the algorithm compared to three-term recurrences in Widlund's method. This effect becomes even more pronounced for smaller stepsizes. 
\begin{table}
    \centering
    \begin{tabular}{l@{\hspace{4em}}rr@{\hspace{4em}}rr}
        \toprule
         & \multicolumn{2}{c}{$\tau=0.001$} & \multicolumn{2}{c}{$\tau=0.0001$} \\
         method & time [s] & \# of iter. & time [s] & \# of iter\\\midrule
         Widlund	&  3.54 & 85 & 0.72 & 16  \\
		GMRES	& 31.10& 65&  10.51 &  13\\\bottomrule
    \end{tabular}
	\caption{Brake squeal problem. Run times and iteration numbers.} \label{table:brake}
\end{table}
\end{example}

%%%%%%%%%%%%%%%%%%%%%%%%%%%%%%%%%%%%%%%%%%%%%

\section{Control methods for pHDAE systems}
\label{sec:control}
One of the main advantages to introduce \pH descriptor systems is its direct base in control theory. In this section we therefore discuss classical control
applications for \pHDAE systems.

Consider a linear \pHDAE system as in \Cref{def:pHDAE:linear} and a linear output feedback $\inpVar=F(t)\outVar$. Then, we can write the system in behavior form
by introducing
$\xi=\begin{bmatrix} \state^T & \inpVar^T & \outVar^T \end{bmatrix}^T$
and new block matrices
\begin{align*}
    \mathcal{E} &\vcentcolon= \diag (E,0,0), &
    \mathcal{Q} &\vcentcolon= \diag (Q, I , I), &
    \mathcal{K} &\vcentcolon= \diag (K, 0 ,0),
\end{align*}
as well as
\begin{equation}
    \label{eqn:feebdack:closedLoopJR}
    \mathcal{J} = \begin{bmatrix} J & G & 0 \\ -G^T & N & I \\ 0 & -I & \tfrac{1}{2}(F-F^T) \end{bmatrix},\quad
    \mathcal{R} = \begin{bmatrix} R & P & 0 \\ P^T & S &0 \\ 0 & 0 & -\tfrac{1}{2}(F+F^T) \end{bmatrix},
\end{equation}
which gives the closed-loop descriptor system 
\begin{equation}
    \label{eq:ofdhdae}
    \mathcal{E}\dot{\xi} = ((\mathcal{J} -\mathcal{R}) \mathcal{Q} - \mathcal{E} \mathcal{K})\xi.
\end{equation}
This is a \dHDAE if and only if $-\tfrac{1}{2}(F+F^T)\geq 0$ pointwise.

Analogously, for the general nonlinear \pHDAE structure in \Cref{def:pHDAE}, we introduce
\begin{align*}
    \xi &\vcentcolon= \begin{bmatrix} \state^T &\inpVar^T & \outVar^T \end{bmatrix}^T,&
    \tilde{\eta} &\vcentcolon= \begin{bmatrix} \eta^T & \inpVar^T& \outVar^T \end{bmatrix}^T, &
    \tilde{r} &\vcentcolon= \begin{bmatrix} r^T & 0 & 0\end{bmatrix}^T,
\end{align*}
the Hamiltonian $\tilde{\hamiltonian}(t,\xi)\vcentcolon= \hamiltonian(t,\state)$ and matrix functions $\mathcal{E} \vcentcolon= \diag(E,0,0)$, and $\mathcal{J}$ and $\mathcal{R}$ as defined in~\eqref{eqn:feebdack:closedLoopJR}.
This gives the system in behavior form
\begin{equation}
    \label{eq:ofdhdaegen}
    \mathcal{E} \dot{\xi} + \tilde{r} = (\mathcal{J} -\mathcal{R})\tilde{\eta}(\xi)
\end{equation}
satisfying $\tfrac{\partial}{\partial \xi}\tilde {\mathcal H}=\mathcal E^T\tilde \eta$ and $\tfrac{\partial}{\partial t}\tilde {\mathcal H}=\tilde \eta^T\tilde r$ pointwise,
which is a \pHDAE if and only if  $\tfrac{1}{2}(F+F^T)\leq 0$ pointwise.

In this way, using a parameterization via the output feedback matrix (function) $F$, we can introduce control methods via the \dHDAE systems~\eqref{eq:ofdhdae} or~\eqref{eq:ofdhdaegen}, respectively.

%%%%%%%%%%%%%%%%%%%%%%%%%%%%%%%%%%%%%%%%%%%%%%%%%
\subsection{Robust stabilization/passivation} 
\label{sec:stabil}
We have seen in \Cref{sec:generalDAE} that for \dHDAE systems, stability and passivity can be easily characterized, while for asymptotic stability or strict passivity, in general, we have only sufficient conditions. 

Considering a linear \pHDAE system of the form \eqref{eqn:pHDAE:linear} with $Q=I$ that has no feedthrough term (cf.~\Cref{sec:nofeed}), and using a linear output feedback $\inpVar=F\outVar+w$  we obtain the closed-loop system
\begin{align*}
    E \dot{\state} +  E K \state &=(J-R +G FG^T)\state + Gw,\\
    \outVar &= G^T \state.
 \end{align*}
If $F=F_\mathrm{H}+F_\mathrm{S}$ is chosen to have a negative semi-definite symmetric part $F_H$, then the new dissipation coefficient becomes $\hat{R}=R- G F_\mathrm{H} G^T$ and the skew-symmetric part becomes $J-G F_\mathrm{S} G^T$. If $\hat{R}$ is positive definite then the system is asymptotically stable. The same approach can also be applied in the general nonlinear case, if an output feedback leads to a positive definite $R(t,\state)$ in \eqref{eqn:pHDAE}. We only have this sufficient condition to guarantee asymptotic stability, see \Cref{cor:stab}. At current, the problem to find a necessary and sufficient condition that guarantees that a \pHDAE system can be made asymptotically stable by output feedback is under investigation. 

The situation is much better understood in the case of \LTI \pHDAE systems, where we have a characterization via the hypocoercivity index being finite in~\Cref{cor:asyhyp}. Then in view of \Cref{thm:singind}~(ii), we have different options to obtain asymptotic stability. 

Again, if we can achieve $\hat{R}=R- G F_\mathrm{H} G^T>0$, then we have asymptotic stability. But we can also use the skew part $F_\mathrm{S}$ to change the eigenvectors of the pair $(E,J+GF_\mathrm{S} G^T)$ in such a way that no eigenvector is in the kernel of $R-G F_\mathrm{H} G^T$, or we can use a combination of both. It is clear from the classical theory of unstructured \DAE systems, see \Cref{sec:stabilization}, that if the system is strongly stabilizable and strongly detectable, then such an output feedback exists and can be computed by ignoring the structure and solving an optimal control problem.

In constructing an output feedback via optimal control methods, we have the freedom to choose the cost functional~\eqref{stabcostfunctional} and, 
furthermore, we also have some freedom in choosing the \pHDAE representation, which is not unique (see, for instance, the discussion at the end of \Cref{sec:tangential}). This flexibility can be used to make the resulting closed-loop system maximally robust against perturbations, which for general \LTI control systems recently has been an important research topic, see \eg, \cite{MehX00} and the references therein. For \pHODE systems, this topic has recently been of great importance by introducing measures like the distance to instability for the robustness of \pHODE representations, see \cite{AliMM20},  as well as their optimization \cite{GilMS18,GilS17}.
For \pHDAE systems, this is currently an active research topic.

The analogous question arises in the context of passivity. We have that a regular strangeness-free \pHDAE system is passive, but in general not necessarily strictly passive, since  $W$ in \eqref{Wdef} is only assumed positive semi-definite. To obtain strict passivity, it is necessary to consider the system in the formulation with feedthrough term and similarly we can analyze how to obtain robust representations as is done for \pHODE systems in \cite{BanMNV20,BeaMV19,MehV20}. For \pHDAE systems, this is again an active research topic.

%%%%%%%%%%%%%%%%%%%%%%%%%%%%%%%%%%%%%%%%%%%%%%%%%
\subsection{Optimal control}
\label{sec:OCPH}
Due to the many interesting properties of \pHDAE systems, one may investigate whether some extra advantages can be obtained also in the context of optimal control problems. 
Clearly one can just use the general results in \Cref{sec:ocpdae} and obtain the same optimality conditions. However, it has been observed in two recent papers, see \cite{PhiSFMW21,FauMPSW21}, that some surprising results arise for optimal control problems with \LTI \pHODE and \pHDAE, when as very special cost functional the supplied energy is minimized, \ie
\begin{equation}
    \label{energysupply}
    \tfrac{1}{2} \state(t_{\mathrm{f}})^T M \state(t_{\mathrm{f}}) + \tfrac{1}{2} \int^{t_{\mathrm{f}}}_{t_0} 2\outVar^T\inpVar\dt = \tfrac{1}{2} \state(t_{\mathrm{f}})^T M \state(t_{\mathrm{f}}) + \tfrac{1}{2} \int^{t_{\mathrm{f}}}_{t_0} 2\state^T G^T \inpVar\dt,
\end{equation}
subject to the constraint
\begin{equation}
    \label{eqn:pHconstraint}
    E\dot{\state} =  (J-R)\state+G\inpVar,\qquad E^\dagger E\state(t_0)  = \state_0
\end{equation}
and the output equation is $\outVar=G^T \state$. Note that in the cost functional \eqref{qcostfunct} we then have $W_{\mathrm{\state}}=0$, $W_{\mathrm{\inpVar}}=0$, and $S=G^T$.
Since we are in the \LTI case, we can insert the data into the optimality system~\eqref{optbvp} and obtain the following result.
\begin{corollary}
    \label{cor:neccon}
    Consider the optimal control problem to minimize \eqref{energysupply} subject to the constraint \eqref{eqn:pHconstraint}. Assume that the pair $(E,J-R)$ is regular and of (Kronecker) index at most one (as a free system with $u\equiv0$) and that $M$ is in $\cokernel E$. If $(\state,\inpVar) \in \mathbb{Z} \times \mathbb{U}$ is a solution to this optimal control problem, then there exists a Lagrange multiplier $\lambda \in \mathcal{C}^1_{E^\dagger E}(\timeInt,\R^{\stateDim})$, such that $(\state,\lambda,\inpVar)$ satisfy the boundary value problem
    \begin{equation}
        \label{pHneccond}
        \begin{bmatrix} 0 & E & 0 \\ -E^T & 0 & 0 \\ 0 & 0 & 0 \end{bmatrix} \begin{bmatrix} \dot{\lambda}\\ \dot{\state} \\ \dot{\inpVar} \end{bmatrix} = \begin{bmatrix} 0 & J-R & G\\
           (J-R)^T & 0 & G\\
           G^T& G^T& 0\end{bmatrix} \begin{bmatrix} \lambda \\ \state \\ \inpVar \end{bmatrix},\ 
    \end{equation}
    with boundary conditions
    \begin{equation*} 
        E^\dagger E\state(t_0) = \state_0,\qquad E E^\dagger \lambda(t_{\mathrm{f}}) = -(E^\dagger)^T M \state(t_{\mathrm{f}}).
    \end{equation*}
\end{corollary}

Use a full rank decomposition $E=U_E \begin{smallbmatrix} E_{11} & 0 \\ 0 & 0 \end{smallbmatrix} U_E^T$ with $E_{11}=E_{11}^T>0$ and transform the other coefficients accordingly as
\begin{gather*}
    U^T_E (J-R)U_E = \begin{bmatrix}
        L_{11} & L_{12}\\ L_{21} & L_{22} 
    \end{bmatrix},\quad
    U^T_EMU_E = \begin{bmatrix} M_{11} & M_{12} \\ M_{21} & M_{22} \end{bmatrix},\\ 
    U^T_E G = \begin{bmatrix} G_1\\ G_2 \end{bmatrix},\qquad
    \begin{bmatrix} \hat{\state}_1 \\ \hat{\state}_2 \end{bmatrix} = U_E^T\hat{\state},\qquad
    \begin{bmatrix} \hat \lambda_1 \\ \hat \lambda_2 \end{bmatrix} = U_E^T\hat \lambda,\qquad
    \begin{bmatrix} \hat \state_{1,0} \\ \hat \state_{2,0} \end{bmatrix} =U^T_E\state_0.
\end{gather*}
After some permutations we can express \eqref{pHneccond} in the form
\begin{equation} 
    \label{eqn:BVPtrf}
    \begin{bmatrix} 
        0 & E_{11} & 0 & 0 & 0\\
        -E_{11}^T & 0 & 0 & 0 & 0 \\
        0 & 0 & 0 & 0 & 0 \\
        0 & 0 & 0 & 0 & 0 \\
        0 & 0 & 0 & 0 & 0 
    \end{bmatrix} \begin{bmatrix} \dot{\lambda}_1\\ \dot{\hat \state}_1 \\ \dot{\hat \lambda}_2 \\ \dot{\hat \state}_2 \\ \dot{\inpVar} 
\end{bmatrix} = 
\begin{bmatrix} 0 & L_{11} & 0 &L_{12} & G_1\\
L_{11}^T & 0 & L_{21}^T & 0 & G_1 \\
0 & L_{21} & 0 & L_{22} & G_2 \\
L_{12}^T & 0 & L_{22}^T & 0 & G_2 \\
G_1^T & G_1^T & G_2^T & G_2^T & 0 \end{bmatrix}
\begin{bmatrix} \hat \lambda_1\\ \hat \state_1 \\  \hat \lambda_2 \\  \hat \state_2 \\ \inpVar \end{bmatrix},
\end{equation}
with boundary conditions $\hat \state_1(t_0)=\hat \state_{1,0}$ and
$\hat \lambda_1(t_{\mathrm{f}})=-E_{11}^{-1}M_{11} \hat \state_1(t_{\mathrm{f}})$,
and consistency condition $\hat \state_{2,0}=0$. Here we have used the condition that $M=M^T$ is in $\cokernel E$, which implies that $M_{12}=0$, $M_{22}=0$.

For the structured matrix pencil associated with 
\eqref{eqn:BVPtrf} there exists a condensed form under real orthogonal congruence transformations, which has been introduced in \cite{ByeMX07}, and from which the spectral properties, the (Kronecker) index and the regularity can be read off. If this pencil is regular, then we directly obtain that the system has (Kronecker) index one if and only if
\begin{displaymath}
\hat{W}_\inpVar = \begin{bmatrix} 0 & J_{22} -R_{22} & G_2 \\ -J_{22}-R_{22} & 0 & G_2 \\ G_2^T & G_2^T & 0 \end{bmatrix}
\end{displaymath}
is invertible, see \Cref{sec:timevar}. To simplify the algebraic equations, we can perform a congruence transformation with the orthogonal matrix  
\begin{displaymath}
    U_S = \begin{bmatrix} \tfrac{1}{\sqrt{2}}I & -\tfrac{1}{\sqrt{2}}I & 0 \\ \tfrac{1}{\sqrt{2}}I &  \tfrac{1}{\sqrt{2}}I  &0 \\ 0 & 0 &  I\end{bmatrix},
\end{displaymath}
\ie, we multiply the system with $U_S^T$ from the left 
and  set
\begin{equation*}
    \begin{bmatrix} \tilde \lambda_2 \\ \tilde  \state_2 \\ \inpVar \end{bmatrix} = \begin{bmatrix} \frac1{\sqrt{2}}(\hat \state_2+\hat \lambda_2) \\ \frac1{\sqrt{2}}(\hat \state_2-\hat \lambda_2) \\ \inpVar \end{bmatrix} = U_S^T \begin{bmatrix} \hat \lambda_2 \\ \hat \state_2 \\ \inpVar \end{bmatrix}. 
\end{equation*}
Then we get that 
\begin{equation*}
U_S^T \hat W_{\mathrm{\inpVar}} U_S = \begin{bmatrix}
-R_{22} & J_{22} & G_2 \\
-J_{22} & R_{22} & 0\\
G_2^T & 0 & 0 \end{bmatrix}
\end{equation*}
Clearly, for this to be invertible, we need that $G_2$ has full column rank, which implies that $u$ is fixed as a linear combination of $\hat \lambda_2$ and $\hat \state_2$. Considering the application examples in \Cref{sec:examples}, where we typically have $G_2=0$, we cannot expect the \DAE associated with the optimality system to be of (Kronecker) index one. Thus we are in the case of a singular control problem, see \eg, \cite{BryH18}.

The case that $\hat W_{\mathrm{\inpVar}}$ is not invertible
has been analyzed in \cite{SchPFWM21} for the \pHODE case and in \cite{FauMPSW21} for the \pHDAE case, where  $\mu E-(J-R)$ is regular and of (Kronecker) index at most one. In these papers, it is assumed that the image of the matrix $G$ does not intersect with the kernel of $R$ and that even though this is a singular control problem, the optimal solution is still a feedback control that can be obtained via the solution of a Riccati equation. We refer to \cite{Aro18} for the analysis of such problems. The extension of this analysis to the \LTV \pHDAE case and the case that this assumption is not valid is currently under investigation.

%%%%%%%%%%%%%%%%%%%%%%%%%%%%%%%%%%%%%%%%%%%%%%%%%%%%%%%%%%%%
\section{Summary and open problems}
\label{sec:summary}

This survey paper discusses the model class of port-Hamiltonian descriptor systems (differential-algebraic systems) for numerical simulation and control. We have demonstrated that this model class has many nice properties:  
\begin{itemize}
	\item It allows for automated modeling in a modularized and network-based fashion.
	\item It allows the coupling of mathematical models across different scales and physical domains.
	\item It incorporates the properties of the real (open or closed) physical system.
	\item It has nice algebraic, geometric, and analytical properties and allows analysis concerning existence, uniqueness, robustness, stability, and passivity.
	\item It is invariant under local variable (coordinate) transformations which leads to local canonical and condensed forms.
	\item Furthermore, it allows for structure-preserving (space-time) discretization and model reduction methods as well as fast solvers for the resulting linear and nonlinear systems of equations.
\end{itemize}

Despite the many promising results already available for \pHDAEs, there are still many open problems that are either under investigation or pose challenging problems to be tackled in the future. In the following, we present an incomplete list to stimulate further research.

Many of the control theoretical concepts presented within this paper rely on instantaneous feedback. While this is a convenient theoretical approach, it is not always possible to realize in applications, where the states or outputs first have to be measured, the control action computed, and then fed back to the system, thus resulting in a necessary intrinsic time delay; see also \Cref{rem:delayedFeedback}. Although some first results for \pHODEs with delays are available in the literature, see for instance \cite{SchFOR16}, a general model class for time-delayed \pHDAEs is not yet available and the results presented in this paper have to be extended to the time-delay case.

In terms of \MOR, the impact of the Hamiltonian on the approximation quality, see \cite{BreMS20,BreU21}, needs to be further investigated, in particular with an emphasis on nonlinear \MOR methods. 
Structure-preserving balancing methods are still under investigation; see \cite{BreS21} for some first results for \LTI systems. Another open problem in structure-preserving \MOR is the construction of optimal projection spaces, in the sense that they minimize the Kolmogorov $n$-widths, respectively the Hankel singular values, cf.~\cite{UngG19}. First gradient-based optimization procedures are discussed for \pHODEs in \cite{MosL20,SatS18,SchV20}. If the $n$-widths do not decay rapidly, then one cannot expect accurate \ROMs with a small dimension and current efforts in the reduction of transport-dominated phenomena, see for instance~\cite{BlaSU20} and the references therein, need to be adapted to the \pH framework.

A closely related topic to \MOR is identifying a \pHDAE realization from measurements. In view of large data sets, modern artificial intelligence approaches, and automatized machine learning methods used within the digital twin paradigm, this is an important topic requiring further research, with only a few available results for \LTI systems, see \Cref{rem:systemIdentification}.

It is an open problem to derive necessary and sufficient conditions under which output feedback can make a general nonlinear \pHDAE system asymptotically stable. A natural research question in this context would be to extend the concept of hypocoercivity to the \LTV and nonlinear case.

The characterization of distance measures for general \dHDAE systems, like the distance from an asymptotically stable \dHDAE system to a system which is only stable or the distance of a strictly passive \pHDAE system to the nearest system that is only passive, is an important research topic because it is a requirement for the design of real-world systems, as \eg, in \Cref{sec:brake}. Even if such a characterization is available, then one needs computational methods to compute these distances. In the large-scale setting, this is a challenging issue that can only be achieved with a clever combination with model reduction techniques; see \cite{AliMM20} for a first attempt. 

Another important research problem is the exact characterization of the relationship between passivity, positive realness, and the port-Hamiltonian structure for \pHDAEs, see  \cite{BeaMV19,BeaMX15_ppt} for the \pHODE case, as well as the characterization via Kalman-Yakobovich-Popov inequalities as it has been done for general \LTI \DAE systems in  \cite{ReiRV15,ReiV15}.

A research field that has so far not received much attention are discrete-time \pH descriptor systems. The primary research in this direction arises from discretizations of \pHDAE systems, see 
\cite{KotL18,MehM19} for the most recent approaches. Since discrete-time systems not only arise from discretization but also from sampling or realization, it is an open question how to properly define discrete-time \pH systems in a general way. This also concerns the stability and passivity analysis. An approach to modifying the concept of hypocoercivity for use in discrete-time systems is currently under investigation.

Since, in many cases, mathematical models are obtained from data via measurements, and the resulting parameters, as well as model coefficients (including the inputs and outputs), are only available with some uncertainty or randomness, it is an open question to include such uncertainties adequately into the \pHDAE framework and also to study robust control methods that deal with such uncertainties, see \cite{BreMS20,BreS21} for the modification of classical balanced truncation methods and \cite{Kar22} for the modification of robust control methods in the \pHODE case. Another research area is the perturbation theory for the problems and the error analysis of the relevant numerical methods.

As we have seen, also the area of optimal control for \pHDAE systems requires a lot of further research. This concerns, in particular, the optimal use of the structure for nonlinear systems, the choice of an appropriate cost functional, and also structure-preserving numerical methods in particular for large-scale problems.

Finally, we would like to discuss a topic that has only briefly been touched on in this survey: the extension of \pH modeling to partial differential equations that has been pursued in many different directions in recent years. One can take the approach to replace the coefficient matrices with linear differential operators, one can follow a differential geometric approach via the extension of Lagrange or Dirac structures, or one can formulate all classical partial differential equations from physics in the areas of elasticity, electromagnetism, fluid dynamics, structural mechanics, geomechanics, poroelasticity, gas or water transport, to name a few directions, in a  structure that resembles the \pHDAE structure via the given symmetries or differential forms. These efforts have been the topic of many recent research papers on modeling, numerical methods, optimization, and control. Discussing these developments would be a considerable survey on its own, mainly since the field is growing immensely fast. An essential question that is widely open is to incorporate the boundary appropriately into the structure, to obtain well-posed partial differential equations and so that they, on the one hand, can be treated as controls and, on the other hand, also be used for the interconnection of subsystems. Another important topic that is actively pursued is appropriate time space-time discretization methods like finite element or finite volume approaches that preserve the structure. 

Instead of discussing this further, we present the following incomplete list of references that describes many different research directions pursued:
\cite{AltMU21c,AltS17,AouCMA17,BaaCEJLLM09,BanSAZISW21,CarMP17,DuiMSB09, Egg19,EggK18,EggKLMM18,EnnS05,JacZ12,GayY18,GayY19,GrmO97,Kot19,KotML18, KurZSB10,LeGZM05,MacM09,MacMB05,MacSM04a,MacSM04b,MatH13,Mor22, MosZ18,MosMBM18,Oet06,OetG97,Ram19,SchM02,SchS17,SerMH19,Vil07,YosM06a,YosM06b}.

\subsection*{Acknowledgments}
The work of V.~Mehrmann has been supported by the Deutsche Forschungsgemeinschaft (DFG, German Research Foundation) CRC 910 \emph{Control of self-organizing nonlinear systems: Theoretical methods and concepts of application}: Project No.~163436311, CRC TRR 154 \emph{Mathematical modeling, simulation and optimization at the example of gas networks}: Project No.~239904186,  priority programs SPP 1984 \emph{Hybrid and multimodal energy systems: System theoretical methods for the transformation and operation of complex networks}: Project No.~361092219, and  SPP 1897, \emph{Calm, Smooth and Smart - Novel Approaches for Influencing Vibrations by Means of Deliberately Introduced Dissipation}: Project No.~273845692, DFG Excellence Cluster 2046 Math${}^+$: Project No.~390685689, as well as the BMBF (German Ministry of Education and Research) via the Project EiFer.

B.~Unger acknowledges funding from the DFG under Germany's Excellence Strategy -- EXC 2075 -- 390740016 and is thankful for support by the Stuttgart Center for Simulation Science (SimTech).

\bibliographystyle{plain-doi}
\bibliography{ph}

\end{document}

%% file: def.tex
\definecolor{myBlue}{RGB}{30,144,255} % dodger blue
\definecolor{myGreen}{RGB}{69,169,0} % chatreuse
\definecolor{myRed}{RGB}{165,12,42}
\definecolor{myOrange}{RGB}{225,92,22}
\definecolor{color0}{rgb}{0.121568,0.46666,0.70588}
\definecolor{color1}{rgb}{1,0.4980,0.0549}
\definecolor{color2}{rgb}{0.17254,0.62745,0.17254901}
\definecolor{color3}{rgb}{0.83921,0.15294,0.156862}
\definecolor{color4}{rgb}{0.580392,0.4039215,0.7411764}
\definecolor{color5}{rgb}{0,0,0}

\newcommand{\delete}[1]{ }

\def\N{\mathbb{N}}

\def\R{\mathbb{R}}
\def\C{\mathbb{C}}

\newcommand\dt{\,\text{d}t}
\newcommand{\dsigma}{\,\text{d}\sigma}

\newcommand{\ddt}{\ensuremath{\tfrac{\mathrm{d}}{\mathrm{d}t}} }

\DeclareMathOperator{\sspan}{span}

\DeclareMathOperator{\diag}{diag}

\DeclareMathOperator{\rank}{rank}

\DeclareMathOperator{\real}{Re}

\DeclareMathOperator{\corank}{corank}
\DeclareMathOperator{\cokernel}{cokernel}

\newcommand\independent{\protect\mathpalette{\protect\independenT}{\perp}}
\def\independenT#1#2{\mathrel{\rlap{$#1#2$}\mkern2mu{#1#2}}}

\makeatletter
\newsavebox{\@brx}
\newcommand{\llangle}[1][]{\savebox{\@brx}{\(\m@th{#1\langle}\)}%
  \mathopen{\copy\@brx\kern-0.5\wd\@brx\usebox{\@brx}}}
\newcommand{\rrangle}[1][]{\savebox{\@brx}{\(\m@th{#1\rangle}\)}%
  \mathclose{\copy\@brx\kern-0.5\wd\@brx\usebox{\@brx}}}
\makeatother
\newcommand{\bilp}[1]{\llangle #1 \rrangle}
\newcommand{\dualp}[2]{\left\langle\left. #1 \,\right|\, #2 \right\rangle}
\newcommand{\calA}{\ensuremath{\mathcal{A}} }

\newcommand{\calC}{\ensuremath{\mathcal{C}}}

\newcommand{\calL}{\ensuremath{\mathcal{L}} }

\newcommand{\V}{\ensuremath{\mathcal{V}}}

\newcommand{\eg}{\emph{e.g.}\xspace}
\newcommand{\ie}{\emph{i.e.}\xspace}

\newenvironment{smallbmatrix}{\left[\begin{smallmatrix}}{\end{smallmatrix}\right]}

\newcommand{\state}{z} 

\newcommand{\timeInt}{\mathbb{T}}
\newcommand{\stateSpace}{\mathcal{Z}}
\newcommand{\stateDim}{n}
\newcommand{\extSpace}{\mathcal{S}}
\newcommand{\inpVar}{u}
\newcommand{\outVar}{y}
\newcommand{\outputDim}{p}
\newcommand{\inputDim}{m}
\newcommand{\hamiltonian}{\mathcal{H}}
\newcommand{\structureMatrix}{\Gamma}
\newcommand{\dissipationMatrix}{W}

\newcommand{\G}{\mathcal{G}}
\newcommand{\E}{\mathcal{E}}
\newcommand{\Vi}{\mathcal V_i}
\newcommand{\Vb}{\mathcal V_b}

\newcommand{\transferFunc}{\mathcal{G}}
\newcommand{\transferFuncRed}{\widehat{\transferFunc}}

%% file: MSD-n2002-MORerrorDecay.tex
\definecolor{mycolor1}{rgb}{0.00000,0.44700,0.74100}%
\definecolor{mycolor2}{rgb}{0.85000,0.32500,0.09800}%
\definecolor{mycolor3}{rgb}{0.92900,0.69400,0.12500}%
\definecolor{mycolor4}{rgb}{0.49400,0.18400,0.55600}%
\definecolor{mycolor5}{rgb}{0.46600,0.67400,0.18800}%
\begin{tikzpicture}

\begin{axis}[%
width=.8\linewidth,
height=.35\linewidth,
at={(1.011in,0.642in)},
scale only axis,
grid=both,
grid style={line width=.1pt, draw=gray!10},
major grid style={line width=.2pt,draw=gray!50},
axis lines*=left,
axis line style={line width=2pt},
xmin=2,
xmax=10,
xlabel={Reduced system dimension $r$},
ymode=log,
ymin=5e-06,
ymax=4e-1,
yminorticks=true,
ylabel style={font=\color{white!15!black}},
ylabel={$\frac{\|\transferFunc-\transferFuncRed\|_{\mathcal{H}_2}}{\|\transferFunc\|_{\mathcal{H}_2}}$},
axis background/.style={fill=white},
legend style={at={(0.02,0.03)},
	anchor=south west, legend cell align=left, align=left, draw=white!15!black, font=\scriptsize},
]
\addplot [color=mycolor1, line width=\lineWidth]
  table[row sep=crcr]{%
2	0.174716715654157\\
4	0.0313718934554877\\
6	0.00966684259709893\\
8	0.00280201961514591\\
10	0.000688510021405809\\
};
\addlegendentry{\ECRM}

\addplot [color=mycolor2, line width=\lineWidth]
  table[row sep=crcr]{%
2	0.278537080405954\\
4	0.0748881769796802\\
6	0.0216644829966393\\
8	0.00639047311360533\\
10	0.00190529162445563\\
};
\addlegendentry{\MM ($s_0=1^{-10}$)}

\addplot [color=mycolor3, line width=\lineWidth]
  table[row sep=crcr]{%
2	0.180653480064026\\
4	0.057842513612914\\
6	0.0267559128564047\\
8	0.0121833740047787\\
10	0.00423101507710211\\
};
\addlegendentry{\MM ($s_0=\infty$)}

\addplot [color=mycolor4, line width=\lineWidth]
  table[row sep=crcr]{%
2	0.170268017063909\\
4	0.0286744814367013\\
6	0.00700529251579089\\
8	0.00205731690847185\\
10	0.000471613355849235\\
};
\addlegendentry{\IRKA}

\addplot [color=mycolor5, line width=\lineWidth]
  table[row sep=crcr]{%
2	0.0755181150923696\\
4	0.00565891471752187\\
6	0.000506705366750425\\
8	3.5069396992263e-05\\
10	8.33295829435264e-06\\
};
\addlegendentry{\IRKA (mod.\,$\hamiltonian$)}

\end{axis}
\end{tikzpicture}%

%% file: RLCimplicitMidpoint.tex
% This file was created with tikzplotlib v0.9.17.
\begin{tikzpicture}

\definecolor{color0}{rgb}{0.12156862745098,0.466666666666667,0.705882352941177}
\definecolor{color1}{rgb}{1,0.498039215686275,0.0549019607843137}
\definecolor{color2}{rgb}{0.172549019607843,0.627450980392157,0.172549019607843}
\definecolor{color3}{rgb}{0.83921568627451,0.152941176470588,0.156862745098039}
\definecolor{color4}{rgb}{0.580392156862745,0.403921568627451,0.741176470588235}
\definecolor{color5}{rgb}{0.549019607843137,0.337254901960784,0.294117647058824}

\begin{axis}[
height=\figureheight,
legend cell align={left},
legend style={fill opacity=0.8, draw opacity=1, text opacity=1, draw=white!80!black},
tick align=outside,
tick pos=left,
width=\figurewidth,
axis lines*=left,
axis line style={line width=2pt},
x grid style={white!70!black},
xmajorgrids,
xmin=0, xmax=1,
xtick style={color=black},
y grid style={white!70!black},
ymajorgrids,
ymin=-6.30579686236842, ymax=7.90613937022898,
ytick style={color=black}
]
\addplot [color0, line width=\lineWidth]
table {%
0 1.82574185835055
0.01 1.81936439413122
0.02 1.8012587018527
0.03 1.77335744899933
0.04 1.7373641005173
0.05 1.69477631593046
0.06 1.6469071810284
0.07 1.59490445928029
0.08 1.53976803344851
0.09 1.48236569426623
0.1 1.42344742043308
0.11 1.36365828250764
0.12 1.30355009247159
0.13 1.24359191074965
0.14 1.18417951323637
0.15 1.12564391235339
0.16 1.06825901828937
0.17 1.01224851931416
0.18 0.957792053364818
0.19 0.905030736933791
0.2 0.854072111610736
0.21 0.804994563403507
0.22 0.757851265157622
0.23 0.71267368797593
0.24 0.669474723482233
0.25 0.628251455047303
0.26 0.588987612677756
0.27 0.551655743134314
0.28 0.516219123974068
0.29 0.482633447581326
0.3 0.450848298844492
0.31 0.42080844793475
0.32 0.3924549776298
0.33 0.365726262787485
0.34 0.340558817895847
0.35 0.316888027095075
0.36 0.294648769670938
0.37 0.27377595274756
0.38 0.254204961749619
0.39 0.235872038150689
0.4 0.218714593066833
0.41 0.202671464384654
0.42 0.18768312432328
0.43 0.173691843613551
0.44 0.160641817828481
0.45 0.148479260811234
0.46 0.137152469614978
0.47 0.126611864888194
0.48 0.116810010204755
0.49 0.107701613446299
0.5 0.0992435129912304
0.51 0.0913946511466577
0.52 0.0841160369734948
0.53 0.0773707003979195
0.54 0.071123639271696
0.55 0.065341760837101
0.56 0.059993818867095
0.57 0.0550503475859113
0.58 0.0504835933275276
0.59 0.0462674447578384
0.6 0.0423773623692008
0.61 0.0387903078519725
0.62 0.0354846738554076
0.63 0.0324402145686496
0.64 0.0296379774804973
0.65 0.0270602366131468
0.66 0.0246904274693522
0.67 0.0225130838836054
0.68 0.0205137769252774
0.69 0.0186790559645469
0.7 0.0169963919797641
0.71 0.0154541231571268
0.72 0.0140414028096953
0.73 0.0127481496224092
0.74 0.0115650002124916
0.75 0.0104832639800846
0.76 0.00949488021183662
0.77 0.00859237739016207
0.78 0.00776883465276477
0.79 0.0070178453405305
0.8 0.00633348256683365
0.81 0.00571026673749347
0.82 0.00514313494788571
0.83 0.00462741218191924
0.84 0.00415878423659451
0.85 0.00373327229555293
0.86 0.00334720907530164
0.87 0.00299721646856371
0.88 0.00268018461037976
0.89 0.00239325229410106
0.9 0.00213378866620438
0.91 0.00189937613086943
0.92 0.001687794397444
0.93 0.00149700560623563
0.94 0.0013251404704787
0.95 0.00117048537479738
0.96 0.00103147037299397
0.97 0.00090665803051345
0.98 0.000794733059450255
0.99 0.00069449269645514
1 0.00060483777635512
};
\addlegendentry{$I$}
\addplot [color1, line width=\lineWidth]
table {%
0 -5.65979976088672
0.01 -3.10670460422025
0.02 -0.957693236193782
0.03 0.839467023613894
0.04 2.33065127360405
0.05 3.5561543467947
0.06 4.55136913819191
0.07 5.34737925861255
0.08 5.9714774462393
0.09 6.44761955960928
0.1 6.79682260337676
0.11 7.03751406575298
0.12 7.18583884408909
0.13 7.25592917725442
0.14 7.26014226874728
0.15 7.20926965459682
0.16 7.11272182957088
0.17 6.97869118083853
0.18 6.81429587886905
0.19 6.62570703148856
0.2 6.41826111058779
0.21 6.19655940511932
0.22 5.96455603289841
0.23 5.72563585236029
0.24 5.48268344959325
0.25 5.23814423205407
0.26 4.99407853530347
0.27 4.75220954024697
0.28 4.51396570348977
0.29 4.28051832059383
0.3 4.05281476962208
0.31 3.83160791896295
0.32 3.6174821278446
0.33 3.41087621913802
0.34 3.21210376112448
0.35 3.02137095710114
0.36 2.8387924083622
0.37 2.66440498665348
0.38 2.49818002616702
0.39 2.34003402209532
0.4 2.18983800233489
0.41 2.04742572079944
0.42 1.9126008046955
0.43 1.78514297379011
0.44 1.66481343694949
0.45 1.5513595598679
0.46 1.44451888777725
0.47 1.3440225978899
0.48 1.24959844825743
0.49 1.16097328251831
0.5 1.07787514356359
0.51 1.00003504338668
0.52 0.927188431228796
0.53 0.859076397518865
0.54 0.795446646978862
0.55 0.73605427057079
0.56 0.680662342654143
0.57 0.629042366762586
0.58 0.58097459075916
0.59 0.536248209758689
0.6 0.494661473085216
0.61 0.456021709635725
0.62 0.420145284325942
0.63 0.386857496779208
0.64 0.355992432066629
0.65 0.327392772099598
0.66 0.300909575199275
0.67 0.276402030408444
0.68 0.253737192257399
0.69 0.232789700936179
0.7 0.213441492151064
0.71 0.19558150034485
0.72 0.17910535843033
0.73 0.163915096717404
0.74 0.149918843300065
0.75 0.137030527804321
0.76 0.125169590076851
0.77 0.114260695112104
0.78 0.104233455268516
0.79 0.0950221606087265
0.8 0.0865655180107828
0.81 0.078806399534198
0.82 0.0716916003837271
0.83 0.0651716066922945
0.84 0.0592003732404863
0.85 0.0537351111414026
0.86 0.0487360854446581
0.87 0.0441664225503409
0.88 0.0399919272713393
0.89 0.0361809093393552
0.9 0.0327040191149799
0.91 0.0295340922344017
0.92 0.0266460029037153
0.93 0.0240165255356113
0.94 0.0216242044116931
0.95 0.0194492310461754
0.96 0.0174733289226675
0.97 0.0156796452746452
0.98 0.0140526495816061
0.99 0.0125780384563923
1 0.0112426466043982
};
\addlegendentry{$V_1$}
\addplot [color2, line width=\lineWidth]
table {%
0 -5.47722557505166
0.01 -5.47575385253951
0.02 -5.46885858968357
0.03 -5.45240714915964
0.04 -5.42374179148507
0.05 -5.38135238120608
0.06 -5.32460974493412
0.07 -5.25354939347703
0.08 -5.16869698511028
0.09 -5.07092830918133
0.1 -4.9613577496225
0.11 -4.84125018112844
0.12 -4.71195208595004
0.13 -4.57483838116263
0.14 -4.43127203574977
0.15 -4.28257405153975
0.16 -4.1300017968358
0.17 -3.97473402907726
0.18 -3.8178612336836
0.19 -3.66038014933965
0.2 -3.50319155295152
0.21 -3.34710054673149
0.22 -3.19281873074845
0.23 -3.04096776135643
0.24 -2.89208389302271
0.25 -2.74662318144911
0.26 -2.60496709223964
0.27 -2.46742831400556
0.28 -2.3342566196451
0.29 -2.20564465621248
0.3 -2.08173357366268
0.31 -1.96261842697132
0.32 -1.84835330564447
0.33 -1.73895616025701
0.34 -1.63441330806747
0.35 -1.53468360951653
0.36 -1.43970231499845
0.37 -1.34938458709527
0.38 -1.26362870781073
0.39 -1.18231898350915
0.4 -1.10532836248102
0.41 -1.0325207815112
0.42 -0.963753258672849
0.43 -0.898877749939372
0.44 -0.837742787204553
0.45 -0.780194915013018
0.46 -0.726079942801679
0.47 -0.675244028794461
0.48 -0.627534610924455
0.49 -0.58280119931709
0.5 -0.54089604398466
0.51 -0.501674690480378
0.52 -0.464996435357279
0.53 -0.430724692388024
0.54 -0.398727279636701
0.55 -0.368876636640777
0.56 -0.341049980166241
0.57 -0.315129406245205
0.58 -0.291001945495198
0.59 -0.268559578054098
0.6 -0.247699213844323
0.61 -0.228322643303929
0.62 -0.210336463189643
0.63 -0.193651981566019
0.64 -0.178185105644127
0.65 -0.163856215720487
0.66 -0.150590028090219
0.67 -0.138315449465483
0.68 -0.126965425118997
0.69 -0.116476782690649
0.7 -0.106790073340775
0.71 -0.0978494117045538
0.72 -0.0896023158961968
0.73 -0.0819995486272675
0.74 -0.0749949603388189
0.75 -0.0685453351003643
0.76 -0.0626102398984439
0.77 -0.0571518778222215
0.78 -0.0521349455517859
0.79 -0.0475264954653485
0.8 -0.0432958026031482
0.81 -0.0394142366575086
0.82 -0.0358551390991332
0.83 -0.0325937054984523
0.84 -0.029606873056809
0.85 -0.0268732133247186
0.86 -0.0243728300526514
0.87 -0.0220872620931355
0.88 -0.0199993912508708
0.89 -0.0180933549594632
0.9 -0.016354463648847
0.91 -0.0147691226560414
0.92 -0.0133247585231843
0.93 -0.0120097495204666
0.94 -0.0108133602273289
0.95 -0.00972568000280358
0.96 -0.0087375651749395
0.97 -0.00784058477960436
0.98 -0.00702696968042609
0.99 -0.00628956490403101
1 -0.00562178502790553
};
\addlegendentry{$V_2$}
\addplot [color3, line width=\lineWidth]
table {%
0 -0.943299960147786
0.01 -0.517784100703375
0.02 -0.15961553936563
0.03 0.139911170602316
0.04 0.388441878934008
0.05 0.592692391132449
0.06 0.758561523031985
0.07 0.891229876435426
0.08 0.995246241039883
0.09 1.07460325993488
0.1 1.13280376722946
0.11 1.17291901095883
0.12 1.19763980734818
0.13 1.2093215295424
0.14 1.21002371145788
0.15 1.2015449424328
0.16 1.18545363826182
0.17 1.16311519680642
0.18 1.13571597981151
0.19 1.10428450524809
0.2 1.06971018509797
0.21 1.03275990085322
0.22 0.994092672149735
0.23 0.954272642060048
0.24 0.913780574932208
0.25 0.873024038675678
0.26 0.832346422550579
0.27 0.792034923374495
0.28 0.752327617248295
0.29 0.713419720098972
0.3 0.675469128270348
0.31 0.638601319827158
0.32 0.6029136879741
0.33 0.568479369856336
0.34 0.535350626854081
0.35 0.503561826183523
0.36 0.473132068060368
0.37 0.444067497775579
0.38 0.416363337694505
0.39 0.390005670349219
0.4 0.36497300038915
0.41 0.34123762013324
0.42 0.318766800782584
0.43 0.297523828965018
0.44 0.27746890615825
0.45 0.258559926644649
0.46 0.240753147962875
0.47 0.224003766314983
0.48 0.208266408042906
0.49 0.193495547086385
0.5 0.179645857260599
0.51 0.166672507231112
0.52 0.1545314052048
0.53 0.143179399586477
0.54 0.132574441163145
0.55 0.122675711761797
0.56 0.113443723775692
0.57 0.10484039446043
0.58 0.096829098459861
0.59 0.0893747016264471
0.6 0.0824435788475371
0.61 0.0760036182726198
0.62 0.0700242140543248
0.63 0.0644762494632004
0.64 0.0593320720111058
0.65 0.0545654620165987
0.66 0.0501515958665469
0.67 0.046067005068073
0.68 0.0422895320429008
0.69 0.0387982834893622
0.7 0.0355735820251784
0.71 0.0325969167241407
0.72 0.0298508930717226
0.73 0.027319182786233
0.74 0.0249864738833452
0.75 0.0228384213007192
0.76 0.0208615983461428
0.77 0.0190434491853497
0.78 0.0173722425447536
0.79 0.0158370267681201
0.8 0.0144275863351315
0.81 0.0131343999223653
0.82 0.0119486000639555
0.83 0.0108619344487147
0.84 0.00986672887341538
0.85 0.00895585185689943
0.86 0.00812268090744402
0.87 0.00736107042505581
0.88 0.00666532121189088
0.89 0.00603015155655819
0.9 0.00545066985249766
0.91 0.00492234870573262
0.92 0.00444100048395356
0.93 0.00400275425593421
0.94 0.00360403406861651
0.95 0.0032415385076949
0.96 0.00291222148711225
0.97 0.00261327421243986
0.98 0.00234210826360202
0.99 0.00209633974273104
1 0.00187377443406736
};
\addlegendentry{$I_{\mathrm{G}}$}
\addplot [color4, line width=\lineWidth]
table {%
0 -1.82574185835055
0.01 -1.82525128417984
0.02 -1.82295286322786
0.03 -1.81746904971988
0.04 -1.80791393049503
0.05 -1.79378412706869
0.06 -1.77486991497804
0.07 -1.75118313115901
0.08 -1.72289899503676
0.09 -1.69030943639378
0.1 -1.65378591654084
0.11 -1.61375006037615
0.12 -1.57065069531668
0.13 -1.52494612705421
0.14 -1.47709067858325
0.15 -1.42752468384658
0.16 -1.37666726561193
0.17 -1.32491134302575
0.18 -1.27262041122787
0.19 -1.22012671644655
0.2 -1.16773051765051
0.21 -1.11570018224383
0.22 -1.06427291024948
0.23 -1.01365592045214
0.24 -0.964027964340903
0.25 -0.915541060483036
0.26 -0.868322364079881
0.27 -0.82247610466852
0.28 -0.7780855398817
0.29 -0.735214885404161
0.3 -0.693911191220892
0.31 -0.654206142323773
0.32 -0.616117768548158
0.33 -0.579652053419002
0.34 -0.544804436022489
0.35 -0.511561203172177
0.36 -0.479900771666151
0.37 -0.449794862365089
0.38 -0.421209569270243
0.39 -0.394106327836383
0.4 -0.368442787493672
0.41 -0.344173593837068
0.42 -0.321251086224283
0.43 -0.299625916646457
0.44 -0.279247595734851
0.45 -0.260064971671006
0.46 -0.24202664760056
0.47 -0.225081342931487
0.48 -0.209178203641485
0.49 -0.19426706643903
0.5 -0.18029868132822
0.51 -0.167224896826793
0.52 -0.154998811785759
0.53 -0.143574897462675
0.54 -0.132909093212233
0.55 -0.122958878880259
0.56 -0.11368332672208
0.57 -0.105043135415069
0.58 -0.0970006484983992
0.59 -0.0895198593513664
0.6 -0.0825664046147741
0.61 -0.0761075477679766
0.62 -0.0701121543965474
0.63 -0.0645506605220065
0.64 -0.0593950352147086
0.65 -0.0546187385734958
0.66 -0.0501966760300728
0.67 -0.0461051498218281
0.68 -0.0423218083729988
0.69 -0.0388255942302168
0.7 -0.0355966911135914
0.71 -0.0326164705681849
0.72 -0.0298674386320653
0.73 -0.0273331828757562
0.74 -0.0249983201129393
0.75 -0.0228484450334551
0.76 -0.0208700799661477
0.77 -0.0190506259407408
0.78 -0.0173783151839283
0.79 -0.0158421651551165
0.8 -0.0144319342010491
0.81 -0.0131380788858365
0.82 -0.0119517130330441
0.83 -0.0108645684994844
0.84 -0.00986895768560268
0.85 -0.0089577377749065
0.86 -0.00812427668421681
0.87 -0.00736242069771214
0.88 -0.00666646375028995
0.89 -0.00603111831982138
0.9 -0.00545148788294869
0.91 -0.00492304088534747
0.92 -0.00444158617439446
0.93 -0.00400324984015587
0.94 -0.00360445340910933
0.95 -0.00324189333426818
0.96 -0.00291252172497951
0.97 -0.00261352825986844
0.98 -0.00234232322680838
0.99 -0.00209652163467732
1 -0.00187392834263486
};
\addlegendentry{$I_{\mathrm{R}}$}
\addplot [color5, dashed, line width=\lineWidth]
table {%
0 3.7935
0.01 3.6581836686579
0.02 3.5482029354129
0.03 3.44560760354217
0.04 3.33976344476905
0.05 3.22508746423708
0.06 3.09939275744182
0.07 2.96269037120637
0.08 2.81633259652348
0.09 2.66241018063324
0.1 2.50333725344384
0.11 2.3415739657428
0.12 2.17944916765169
0.13 2.01905484374621
0.14 1.86219116692915
0.15 1.7103464672509
0.16 1.5647005376995
0.17 1.42614282386032
0.18 1.29529940310929
0.19 1.1725644315073
0.2 1.05813306081635
0.21 0.952033810114368
0.22 0.854159097923114
0.23 0.764293164358472
0.24 0.682136986866645
0.25 0.607330054756796
0.26 0.539469045488251
0.27 0.478123561352707
0.28 0.422849155481534
0.29 0.373197913683721
0.3 0.328726873072451
0.31 0.289004556975425
0.32 0.253615893617651
0.33 0.222165767476865
0.34 0.194281429923647
0.35 0.169613971831467
0.36 0.147839036715591
0.37 0.128656929606256
0.38 0.111792254905859
0.39 0.096993196291839
0.4 0.0840305334917906
0.41 0.0726964745293008
0.42 0.0628033677823823
0.43 0.0541823458155959
0.44 0.0466819423096095
0.45 0.0401667143651398
0.46 0.0345158948406024
0.47 0.0296220930328856
0.48 0.0253900567725529
0.49 0.0217355047317971
0.5 0.0185840343003842
0.51 0.015870107648912
0.52 0.0135361164600987
0.53 0.0115315241702622
0.54 0.00981208333946475
0.55 0.0083391248860145
0.56 0.00707891531550887
0.57 0.00600207769204559
0.58 0.00508307189361796
0.59 0.00429972962641257
0.6 0.0036328397115357
0.61 0.00306577927599918
0.62 0.00258418665580256
0.63 0.00217567203497586
0.64 0.00182956208631212
0.65 0.00153667513597717
0.66 0.00128912363645451
0.67 0.00108014099362819
0.68 0.000903930049169861
0.69 0.000755530765096101
0.7 0.000630704890829436
0.71 0.000525835612650629
0.72 0.000437840390096021
0.73 0.000364095373205469
0.74 0.000302369968559289
0.75 0.000250770281068005
0.76 0.000207690303038568
0.77 0.000171769852838536
0.78 0.000141858383324944
0.79 0.000116983885969537
0.8 9.63262111972693e-05
0.81 7.91942097639834e-05
0.82 6.50061749182454e-05
0.83 5.32731314768486e-05
0.84 4.35845766076314e-05
0.85 3.55963288235958e-05
0.86 2.9020187163881e-05
0.87 2.36151424316144e-05
0.88 1.9179917284151e-05
0.89 1.55466424832245e-05
0.9 1.25755032158017e-05
0.91 1.01502125473592e-05
0.92 8.1741891787791e-06
0.93 6.56733411455774e-06
0.94 5.26331594275374e-06
0.95 4.20728746922147e-06
0.96 3.35396770042612e-06
0.97 2.66603286084968e-06
0.98 2.1127684659965e-06
0.99 1.66894162930213e-06
1 1.31385891806296e-06
};
\addlegendentry{$\mathcal{H}$}
\end{axis}

\end{tikzpicture}